                  \def\version{15 April 2026}		

\documentclass[reqno,11pt]{amsart} 
\usepackage[pagebackref, colorlinks=true,linkcolor=blue,citecolor=blue]{hyperref}
\usepackage[utf8]{inputenc}
\usepackage{amsmath} 
\usepackage[mathscr]{eucal}
\usepackage{mathtools}
\usepackage{amssymb}
\usepackage{srcltx} 
\usepackage{dsfont}
\usepackage{hyperref}
\usepackage{color}
\usepackage{xcolor}
\usepackage{fancyvrb}
\usepackage{cancel} 
\numberwithin{equation}{section}
 
\def\d{{\rm d}} 
\def\eps{\varepsilon}

\def\L{\Lambda}

\newfam\Bbbfam 
\font\tenBbb=msbm10 
\font\sevenBbb=msbm7 
\font\fiveBbb=msbm5 
\textfont\Bbbfam=\tenBbb 
\scriptfont\Bbbfam=\sevenBbb 
\scriptscriptfont\Bbbfam=\fiveBbb

\newcommand{\Bcal}  {{\mathcal B}}
\newcommand{\Ccal}   {{\mathcal C }} 
\newcommand{\Dcal}   {{\mathcal D }} 
 
\newcommand{\Fcal}   {{\mathcal F }} 
\newcommand{\Gcal}   {{\mathcal G }} 
\newcommand{\Hcal}   {{\mathcal H }}

\newcommand{\Lcal}   {{\mathcal L }} 
\newcommand{\Mcal}   {{\mathcal M }} 
 
\newcommand{\Ocal}   {{\mathcal O }} 
\newcommand{\Pcal}   {{\mathcal P }} 
 
\newcommand{\Rcal}   {{\mathcal R }} 
\newcommand{\Scal}   {{\mathcal S }} 
\newcommand{\Tcal}   {{\mathcal T }} 
\newcommand{\Ucal}   {{\mathcal U }}

\newcommand{\Xcal}   {{\mathcal X }} 
\newcommand{\Ycal}   {{\mathcal Y }}

\newcommand{\Poi}{{\rm Poi}}

\newcommand{\e}   {{\operatorname e }}
\newcommand{\h}{{\rm h}}

\newcommand{\R}     {\mathbb{R}} 
\newcommand{\Z}     {\mathbb{Z}} 
\newcommand{\N}     {\mathbb{N}} 
\renewcommand{\P}   {\mathbb{P}} 
 
\newcommand{\E}     {\mathbb{E}}

\def\L{\Lambda} 
\newcommand{\smfrac}[2]{\textstyle{\frac {#1}{#2}}}

\newcommand*{\LPP}{{{\tt Q}}}
\newcommand*{\SPP}{{{\tt R}}}

\newcommand{\Loops}{\Lcal}

\newcommand{\Shreds}{\Scal}

\newtheorem{theorem}{Theorem}[section] 
\newtheorem{lemma}[theorem]{Lemma} 
\newtheorem{prop}[theorem] {Proposition} 
\newtheorem{cor}[theorem]  {Corollary} 
\newtheorem{defn}[theorem] {Definition}

\theoremstyle{definition}

\newtheoremstyle{rem}{1.5ex}{1.5ex}{\rmfamily}{} {\bfseries\rmfamily}{} {1.5ex}{}
\theoremstyle{rem}
\newtheorem{remark}[theorem]{{\bfseries Remark}}

\newcommand{\defeq}{\vcentcolon=}

\newcommand{\abs}[1]{\lvert#1\rvert}

\def\emptyset{\varnothing} 
\def\1{{\mathchoice {1\mskip-4mu\mathrm l}      
{1\mskip-4mu\mathrm l} 
{1\mskip-4.5mu\mathrm l} {1\mskip-5mu\mathrm l}}} 
\newcommand{\ssup}[1] {{{\scriptscriptstyle{({#1}})}}} 
\def\comment#1{} 
\newtheoremstyle{thm}{2ex}{2ex}{\itshape\rmfamily}{} 
{\bfseries\rmfamily}{}{1.7ex}{}

\renewcommand{\d}{{\rm d}} 
 
\newcommand{\Leb}{{\rm Leb}}

\newcommand{\supp}{{\operatorname {supp}}} 
\newcommand{\Ssup}[1]{^{\ssup{#1}}}
 
\newcommand{\dist}{{\operatorname {dist}}}

\newcommand{\Ccalcirc}{\Ccal^{\ssup\circlearrowleft}}

\definecolor{Red}{rgb}{1,0,0}

\definecolor{amethyst}{rgb}{0.6, 0.4, 0.8} 
\newcommand{\AZ}[1]{\textcolor{amethyst}{(\small AZ: #1)}}
\newcommand{\WK}[1]{{\textcolor{green!70!black}{(\small WK: #1)}}}

\begin{document} 
 
\title[The free energy of the interacting Bose gas, loops and interlacements]{\large The free energy of the interacting Bose gas:\\\medskip  a variational description \\\medskip with loops and interlacements}

\author[Wolfgang  K{\"o}nig]{Wolfgang  K{\"o}nig$^{1,2}$}

\maketitle 
 
\thispagestyle{empty}

\vspace{0.2cm} 

\centerline{\small(\version)} 
\vspace{.5cm}

\centerline{\textit{$^1$Weierstra\ss-Institut Berlin (WIAS),}}
\centerline{\textit{Anton-Wilhelm-Amo-Straße 39 in 10117 Berlin, Germany}, }
\centerline{{\tt koenig@wias-berlin.de}}\medskip
\centerline{\textit{$^2$Institut f\"ur Mathematik, Technische Universit\"at Berlin, Berlin, Germany}}

\bigskip\bigskip

\begin{quote} 
{\small {\bf Abstract:}} We consider the interacting Bose gas in the thermodynamic limit in a large box in $\R^d$ at positive temperature $1/\beta\in(0,\infty)$ with particle density $\sim\rho\in(0,\infty)$. We follow a path-integral approach and adopt from \cite {ACK10} a description of the free energy in terms of the {\it Brownian loop soup}, a Poisson point process consisting of Brownian bridges, also called loops or cycles. It is the objective of this paper to derive, for any values of $\beta$ and $\rho$, a formula for the limiting free energy with explicit control on the  particle numbers in the short and in the long loops. The latter are presumed to play the role of the condensate, according to Feynman's  \cite{F53} famous, vague suggestion, and they turn into {\it random interlacements} (bi-infinite, locally finite random processes in $\R^d$) in our formula. In \cite{ACK10} there was no concept that could describe the long loops; only small $\rho$ could be handled successfully.

In the present paper we represent the limiting free energy in terms of a variational formula, ranging over the set of all stationary point processes with loops and with interlacements, having each a given particle density, and minimizing the sum of the interaction energy and a characteristic entropy term. The latter is a new kind of a {\it specific relative entropy density} with respect to the reference process of loops (the Brownian loop soup), together with an independent Markov kernel describing collections of path shreds in large boxes. In $d\geq 3$, the latter can be seen as a projection of the {\em Brownian interlacement Poisson point process with $\beta$-spacing}. Our proof tool box  comes from large-deviation theory, both for the derivation of the formula for  the free energy and for the proof of the existence of the specific relative entropy.

The limiting variational formula has good compactness properties and may serve as a starting point to future deeper investigations, as it describes the interacting Bose gas in much smoother and more well-behaved terms. We give various approximations of that formula, in particular we show that it is equal to a version  without interlacements. A brief discussion about the conjectured relation to the famous {\it Bose--Einstein condensation} phenomenon is given as well.
\end{quote}


\bigskip\noindent 
{\it MSC 2020.} 60F10; 60J65; 82B10; 81S40

\medskip\noindent
{\it Keywords and phrases.} Interacting Bose gas, free energy, path-integral analysis, Brownian loop soup,  thermodynamic limit, Brownian bridges, random marked spatial point processes, long loops and short loops, Brownian interlacements, specific relative entropy density, large deviations, empirical stationary measure, variational formula.

\tableofcontents

\section{Introduction}\label{Intro}

\subsection{Background, ansatz and purpose}\label{sec-Background}
\noindent In this paper, we study the interacting Bose gas in the thermodynamic limit, a well-known quantum model from statistical physics that is of particularly high interest since decades because of the conjectured phase transition called {\em Bose--Einstein condensation (BEC)}. Its introduction is in terms of the symmetrised trace of the exponential of an interacting Hamilton operator, and its crucial properties are formulated in terms of spectral properties of the convolution operator formed out of the corresponding one-particle density matrix. Notably, the usual notion of a phase transition as the existence of a non-analyticity of the free energy as a function of temperature (or of the particle density) is not the core of the definition of BEC, but a much finer property called {\em off-diagonal long-range order (ODLRO)} of the mentioned symmetrised operator. The validity or non-validity of ODLRO is a widely open problem  of mathematical physics, and one is currently far from an understanding, in spite of many phenomenological ansatzes. The main conjecture, which is also in good coincidence with the famous experiments from 1995, is that ODLRO (and therefore BEC) is valid in dimensions $d\geq 3$ at sufficiently small temperature but not at other temperatures and not in other dimensions. In this paper, we are not going to study ODLRO and refer the reader for this question to the vast literature (e.g., a concise literature survey in Section 1.7 of \cite{KVZ23}). Observe that the free (i.e., non-interacting) case is much simpler and has been thoroughly understood, since all spatial properties easily wash out, and the model drastically simplifies. Furthermore several rescaled regimes have been successfully worked out as well \cite{FKSS20}. However, the interacting case in the thermodynamic limit is still widely open, and this is that we concentrate on here. 

Even though the phenomenon of ODLRO cannot be decided by looking at the limiting free energy alone, a comprehensive understanding of the free energy is obviously likely to contribute a lot to the understanding of BEC. Therefore we concentrate in the present paper on understanding the limiting free energy in a comprehensive and detailed way. We are doing this from the view point of a {\em random loop ensemble} using a {\em path-integral analysis}. Random-loop presentations of the interacting Bose gas are known at least since the 1950s and were first considered a mathematical tool only, but Feynman \cite{F53}  raised the, by now famous, vague advice to think about the loop length distribution as an order parameter that might have some physical relevance. This idea was phenomenologically taken up occasionally by mathematical physicists with strong probability background, and it inspired several investigations of related models without any physical relevance, but triggered also phenomenological discussions about possible relations between BEC and the loop length statistics. See Section~\ref{sec-literature} for a thorough survey on the literature on path-integral approaches to the Bose gas. In short words, the rough idea was coined that BEC should have much to do (if not even being equivalent to) with the occurrence of a macroscopic part of the particles in long loops. In several simplified models this conjecture was confirmed, but there were also counter-examples found. For the most important case of an interacting Bose gas in the thermodynamic limit, the question of the relation between BEC and macroscopic occupation of long loops has not been settled at all, even more, there is currently no mathematical tool box that might be able to attack this question.

The present paper changes this situation. We introduce a mathematical path-integral framework for describing all loops, distinguishing long and short ones,  in the interacting  random loop ensemble in the thermodynamic limit. The starting point is, as in \cite{ACK10}, a representation in terms of a Poisson point process (PPP) on loops, a version of the {\em Brownian loop soup}. In this paper, we carry out  a mesoscopic decomposition procedure of the macrobox into smaller subboxes, a large-deviation ansatz, and arrive at a variational description of the limiting free energy, which is our main result. The state spaces that arise along this way are spaces of {\em point processes with marks}, which comprehensively describe the entire configuration of loops. However, in order to describe also long loops in the limit, we need to enlarge the point process space. One innovative aspect (a novelty with respect to \cite{ACK10}) is that the long loops are described in term of point processes of  {\em interlacements}, which are bi-infinite and locally finite functions $\R\to\R^d$. Observe that such processes cannot be seen as marked point processes in $\R^d$. This extended probability space is difficult to handle, since it consists of very non-local quantities, which induce far-reaching correlations.  The variational description is in terms of the minimization of energy (pair interaction between all the legs) plus entropy (negative exponential rate of probabilities) for point processes of loops and interlacements with fixed given particle density in loops and in interlacements, respectively. The interlacement-part of the entropy, which also a novelty of the present paper, is relative with respect to an explicit reference point process kernel. Interestingly, this kernel is, in $d\geq 3$, reminiscent of a local version of  Sznitman's {\em Brownian interlacement PPP}, see the Appendix.

The arising variational description has good compactness properties, and the minimizing point process of loops and interlacements has the interpretation of the typical, the \lq optimal\rq, spatial distribution of all the particles of the Bose gas, organised in local and in global structures. Decisive is then the question for what values of the particle density the minimizer has a non-trivial interlacement-part. This is obviously a mathematically sound formulation of the distinction between the occurrence and non-occurrence of long Feynman-cycles. We conjecture that the interlacement particles directly relate (if they are not even equal) to the particles in the condensate. Conjecturally, the occurrence of interlacements will in future work turn out to be equivalent to the occurrence of BEC (i.e., ODLRO) and, in dimension $d\geq 3$, it will turn out occur for all large particle densities only. In the present paper, we leave both questions open, but we believe that the mathematical framework that we develop offers good chances to answer them in future work.

We introduce the interacting Bose gas in Section~\ref{sec-Bosegas}, present the starting point of our investigations (a rewrite in terms of the Brownian loop soup) in Section~\ref{sec-PPPdescription} and explain our ansatz in Section~\ref{sec-Purpose}. A quick survey on the remainder of the paper is at the end of Section~\ref{sec-Purpose}.

\subsection{The interacting Bose gas}\label{sec-Bosegas}
We consider an interacting bosonic many-body system in a large box in $\R^d$ at positive temperature $1/\beta\in(0,\infty)$ with fixed particle density $\rho\in(0,\infty)$ in the thermodynamic limit. To start with, denote by
\begin{equation}\label{Hamiltonian}
	\Hcal_N^{\ssup{\L ,{\rm bc }}}=-\sum_{i=1}^N \Delta_i+\sum_{1\leq i<j\leq N}v(x_i-x_j),\qquad x_1,\dots,x_N\in\L,
\end{equation}
the $N$-particle operator in a box $\L\subset \R^d$ with some boundary condition \lq bc\rq\ (later we will consider periodic and Dirichlet zero boundary conditions) with kinetic energy and pair interaction given by a measurable interaction functional $v\colon \R^d\to[0,\infty]$; see Assumption (V) in Section~\ref{sec-mainresult} for more specific assumptions.

As we are interested in {\em bosons}, we project the operator $\Hcal_N^{\ssup{\L ,{\rm bc }}}$ on the set of symmetric, i.e., permutation invariant, wave functions. Furthermore, we consider the particle system at positive temperature $1/\beta\in(0,\infty)$. The main object of the study of this paper is the \emph{partition function} of the system, that is, the trace of the operator $\e^{-\beta\Hcal_N^{\ssup{\L ,{\rm bc }}}}$ in $\L$ with symmetrisation:
\begin{equation}\label{partitionfunction}
	Z_N^{\ssup{\rm bc}}(\beta,\L)={\rm Tr}_+^{\ssup{\L ,{\rm bc }}}(\e^{-\beta\Hcal_N^{\ssup{\L ,{\rm bc }}}}),
\end{equation}
where the index $+$ denotes this symmetrisation, i.e., the projection on the set of permutation invariant wave functions. We will introduce the {\em particle density} $\rho\in(0,\infty)$ (the number of particles per unit volume), fix a centred box $\L_N$ of volume $\sim N/\rho$, and study the {\em limiting free energy per unit volume} in the {\em thermodynamic limit},
\begin{equation}\label{freeenergy}
	{\rm f}(\beta,\rho)=-\frac 1\beta\lim_{N\to\infty}\frac 1{|\L_N|}\log Z_N^{\ssup{\rm bc}}(\beta,\L_N).
\end{equation}
Under suitable conditions on the interaction potential, the existence of this limit is well-known (see, e.g., \cite[Thm.~3.58]{Rue69}) for many decades and also the fact that it does not depend on the boundary condition (at least for periodic and Dirichlet zero conditions). However,  an explicit or even interpretable formula is still lacking (with the exception of \cite[Cor.~1.3]{ACK10} for small values of $\rho$; see Section~\ref{sec-literature}). It is the main purpose of this paper to provide some. It will be of the form (see Corollary~\ref{cor-freeenergyVP})
$$
{\rm f}(\beta,\rho)=\frac 1\beta\inf\{F_U(P)+\h^{\ssup{\Lcal,\Scal}}(P)\colon P\in \Mcal_1^{\ssup{\rm s}}(\Lcal\times\Scal), \mathfrak N^{\ssup\ell}_U(P)=\rho\},
$$
where $\Mcal_1^{\ssup{\rm s}}(\Lcal\times\Scal)$ is the set of stationary point processes of loops and interlacements in $\R^d$, $F_U$ is the interaction energy in the unit box $U=[-\frac 12, \frac 12]^d$,  $\h^{\ssup{\Lcal,\Scal}}$ is an entropy functional introduced in Theorem~\ref{thm-specrelent}, and $\mathfrak N^{\ssup\ell}_U$ registers the particle density. Our actual main result, Theorem~\ref{thm-freeenergy}, in combination with Corollary~\ref{cor-LDPlongshortloops}, even makes the distinction between short and long loops rigorously visible; actually we will show that the short loops lead to the $\Lcal$-part and the long loops give rise to the $\Scal$-part.
This gives a mathematically rigorous meaning to what is often called the {\it Feynman cycles}.  Our proof is flexible enough to express several other characteristic quantities of the interacting Bose gas, like the density of loops having a given length or with various other properties.

It is the goal of this line of research to employ that formula as a platform for deeper analysis of the interacting Bose gas in view of the famous Bose--Einstein phase transition. A natural next step would be to derive a proof for the conjecture that, in dimension $d\geq 3$, for every large particle density, the minimising point process cannot contain loops only but must have a non-trivial interlacement part. We are not going to do that in this paper and devote future work to that. See Section~\ref{sec-VP} for more about that.

We will keep the inverse temperature $\beta\in(0,\infty)$ fixed for the entire paper and therefore omit it in most of the notation.

\subsection{Rewrite in terms of the Brownian loop soup}\label{sec-PPPdescription}

It is known since decades that the partition function $Z_N^{\ssup{\rm bc}}(\beta,\L)$ can be written in terms of the famous {\it Feynman--Kac formula} in terms of an ensemble of a random number of Brownian bridges (cycles) in $\L$ with total sum of lengths equal to $N$ and a pair-interaction between any two legs of the cycles. Such presentation is the starting point of our analysis, and we are going to introduce a version of it that was found in \cite{ACK10}, based on the classical text \cite{G70}. This formula is in terms of a marked interacting Poisson point process of Brownian cycles, sometimes called the Brownian loop soup. We need to introduce some notation, which will be crucial for the entire paper.

For $k\in\N$, let $\Ccal_k$ denote the space of  continuous functions $f\colon [0,k\beta]\to\R^d$, and $\Ccal\defeq \bigcup_{k\in\N} \Ccal_k$ the union of all these spaces. Let us write $\Ccal_k^{\ssup\circlearrowleft}=\{f\in\Ccal_k\colon f(k\beta)=f(0)\}$ for the set of loops in $\Ccal_k$, and $\Ccal^{\ssup\circlearrowleft}=\bigcup_{k\in\N}\Ccal_k^{\ssup \circlearrowleft}$ for their union. The \emph{length} of $f\in\Ccal_k$ is $\ell(f)\defeq k$. We call all the $k$ sites $f(0),f(\beta), f(2\beta),\dots, f((\ell(f)-1) \beta)$ the {\em particles} of $f$, hence $\ell(f)$ is the number of particles of $f$.  On $\R^d\times\Ccal^{\ssup\circlearrowleft}$ we are going to introduce our crucial reference measure, a Poisson point process (PPP) with a certain intensity measure. This process and variants thereof are often called the {\em Brownian loop soup}. The canonical {\em Brownian bridge measure} from $x$ to $y$ in time $ k \beta$ in the box $\L$ with boundary condition \lq bc\rq\ is defined by 
\begin{equation}\label{nnBBM}
	\mu^{\ssup{k,\L,{\rm bc}}}_{x,y}(A)=\frac{\P_x^{\ssup{k,\L,{\rm bc}}}(B\in A;B_{k\beta}\in\d y)}{\d y},\qquad A\subset\Ccal^{\ssup\circlearrowleft}_k\mbox{ measurable}.
\end{equation}
Here, $B=(B_t)_{t\in[0,k \beta]}$ is a Brownian motion in $\R^d$ with generator $\Delta^{\ssup{\L ,{\rm bc }}}$, starting from $x$ under $\P_x^{\ssup{k,\L,{\rm bc}}}$; the latter is the Brownian motion in $\L$  on the time interval $[0, k\beta]$ with this boundary condition. We will be considering periodic boundary condition or Dirichlet zero boundary condition and will consider boxes $ \L$ only.  

We now introduce a variant of Dirichlet zero boundary conditions for the bridge measure that will later make the analysis less cumbersome, but does not come from any generator with any boundary condition. We require that every particle is in $\L$, i.e., we restrict to the event $\{B_{l\beta}\in \L\ \forall l\in \{0,1,2,\dots,k\}\}$. This can be seen as Dirichlet zero boundary condition for the Gaussian random walk  $(B(0),B(\beta), B(2\beta),\dots,B(k\beta))$ plus interpolation with a free Brownian bridge between the times $0,\beta,2\beta,\dots,k\beta$. It will turn out to be a useful condition, since it perfectly harmonises with the crucial operator $\Pi_\L^{\ssup \Lcal}$ that we will introduce in Section~\ref{sec-PPPs}. We refer to this condition as {\em particle-boundary condition}, and we will write \lq {\rm per}\rq, \lq $0$\rq\, and \lq{\rm par}\rq, respectively, for each of these three boundary conditions. For each of these conditions, $ \mu^{\ssup{k,\L,{\rm bc}}}_{x,y}$ is a regular Borel measure on $\Ccal^{\ssup\circlearrowleft}_k$ with total mass equal to
\begin{equation}\label{Gaussian}
\mu_{x,y}^{\ssup{k,\L,{\rm bc}}}(\Ccal^{\ssup\circlearrowleft}_k)=g_{k\beta }^\ssup{\L,{\rm bc}}(x,y)=\frac {\P^{\ssup{k,\L,{\rm bc}}}_x(B_{k\beta}\in\d y)}{\d y}.
\end{equation}
If we do not have any boundary condition (and no set $\L$), then we write $\mu_{x,y}^{\ssup k}$ for the canonical  Brownian bridge measure from $x$ to $y$ in the time interval $[0, k\beta]$.

We denote by $\Mcal_{\N_0}(\Xcal)$ the set of simple point measures on a Polish space $\Xcal$. By $\Loops=\Mcal_{\N_0}(\R^d\times \Ccal^{\ssup\circlearrowleft})$ we denote the set of simple point measures $\omega=\sum_{x\in \zeta}\delta_{(x,f_x)}$ with $f_x(0)=x$ on $\R^d\times \Ccal^{\ssup\circlearrowleft}$, which we also see as the set of all marked point processes in $\R^d$ with marks in $\Ccal^{\ssup\circlearrowleft}$. We introduce a reference probability measure on $\Loops$, which is a \emph{Poisson point process} in $\R^d\times\Ccal^{\ssup\circlearrowleft}$:

\begin{defn}[The reference loop Poisson point process]\label{def-PPP}
For any  $\L\Subset \R^d$ and any boundary condition ${\rm bc}\in\{ {\rm per}, 0, {\rm par}\}$, we write $\LPP^{\ssup{\L,\rm bc}}$ for the distribution of the Poisson point process (PPP) $\omega_{\rm P}$ on $\L\times\Ccal^{\ssup\circlearrowleft}$  with intensity measure 
\begin{equation}\label{nudef}
	\nu_\L^{\ssup{\rm bc}}(\d x,\d f)\defeq \sum_{k\in\N}\frac{1}{k}\,\d x\otimes \mu_{x,x}^{\ssup{k,\L,{\rm bc}}}(\d f)\,\qquad x\in \L, f\in \Ccal^{\ssup\circlearrowleft}.
\end{equation}
For $\L$ replaced by $\R^d$, we write $\LPP$ for the PPP on $\R^d\times \Ccal^{\ssup\circlearrowleft}$ with intensity measure $\nu(\d x,\d f)=\sum_{k\in\N}\frac 1k \d x\otimes\mu_{x,x}^{\ssup k}(\d f)$.
\end{defn}

The process in Definition~\ref{def-PPP} was introduced to the study of the Bose gas in \cite{ACK10}. A close cousin of it  was already used for the study of other phenomena in statistical mechanics (e.g., conformal invariance in dimension two) in \cite{LW04} under the name {\em Brownian loop soup}. Here we rely on the recent adaptation in \cite{KVZ23} and refer proofs to Appendix A there. The projection on the loop part, i.e., $\{f_x\colon  x\in\zeta\}$, is what is called sometimes the Brownian loop soup in the inner sense.

Denote by
\begin{equation}\label{partnumbloop}
{\mathfrak N}_\L(\omega)=|\zeta\cap\L|\qquad\mbox{and}\qquad 	{\mathfrak N}_\Lambda^{\ssup \ell}(\omega)= \sum_{x\in\zeta\cap \L}  \ell(f_{x}),\qquad  \L\subset\R^d,
\end{equation}
the number of points of $\omega$ in $\L$ (i.e., the number of loops that start in $\L$), and the total number of particles of $\omega$ attached to points in $\L$, respectively.

We introduce a functional on $\Lcal$ that expresses the interaction between any two particles of loops that start within $ \L$.  For $\omega\in\Loops$ and $\L,\widetilde \L\Subset\R^d$, we define  $ \Phi_{\L,\widetilde \L}\colon\Lcal\to [0,\infty] $ by
\begin{equation}\label{interaction}
	\Phi_{\L,\widetilde \L}(\omega)=\sum_{x\in\zeta\cap\L, y\in\zeta\cap\widetilde\L}\mathfrak V(f_{x},f_{y}), \qquad \omega\in\Lcal, 
\end{equation}
where
\begin{equation}\label{Tdef}
	\mathfrak V(f_x,f_y)=\frac{1}{2}\sum_{i=1}^{\ell(f_x)}\sum_{j=1}^{\ell(f_y)} \1_{\{(x,i)\not=(y,j)\}}V(f_{x,i},f_{y,j}),
\end{equation}
and
\begin{equation}\label{Vdef}
V(f,g)=\int_0^\beta v(f(s)-g(s))\,\d s,
\end{equation}
and $f_i(\cdot)=f((i-1)\beta+\cdot)|_{[0,\beta]}$
is the $i$-th leg of a function $f\in\Ccal$.

We write $\mu(f)=\int f( \omega)\,\mu(\d\omega)=\langle \mu,f\rangle=\langle f,\mu\rangle$ for the integral of a function $f$ with respect to a measure $\mu$ whenever this makes sense. The following is Proposition 1.1 in \cite{ACK10}. 

\begin{prop}[The Bose gas in terms of the marked Poisson point process]\label{lem-rewrite} Fix $\beta\in(0,\infty)$. Let $ v\colon[0,\infty)\to (-\infty,\infty] $ be measurable and bounded from below, and let $ \L\subset\R^d$ be measurable with finite volume (assumed to be a torus for periodic boundary condition). Then, for any $N\in\N$, and $ {\rm bc}\in\{{\rm per},0\} $,
\begin{equation}\label{rewrite}
	Z_N^{\ssup{\rm bc}}(\L)
	=\e^{\nu_{\L}^{\ssup{\rm bc}}(\L\times \Ccal^{\ssup\circlearrowleft})}
	\LPP^{\ssup{\L,\rm bc}}\big[{\rm e}^{-\Phi_ {\L,\L}}\1\{{\mathfrak N}^{\ssup\ell}_\L=N\}\big].
\end{equation}
\end{prop}

That is, up the non-random term $\nu_{\L}^{\ssup{\rm bc}}(\L\times \Ccal^{\ssup\circlearrowleft})$, the partition function  is equal to the expectation over the Boltzmann factor $ {\rm e}^{-\Phi} $ of a marked Poisson point process in $\L$, restricted to a fixed total number of particles. See Figure \ref{Pix_BoseGas} for an illustration.
\begin{figure}[!htpb]
	\centering
\includegraphics[scale=0.4]{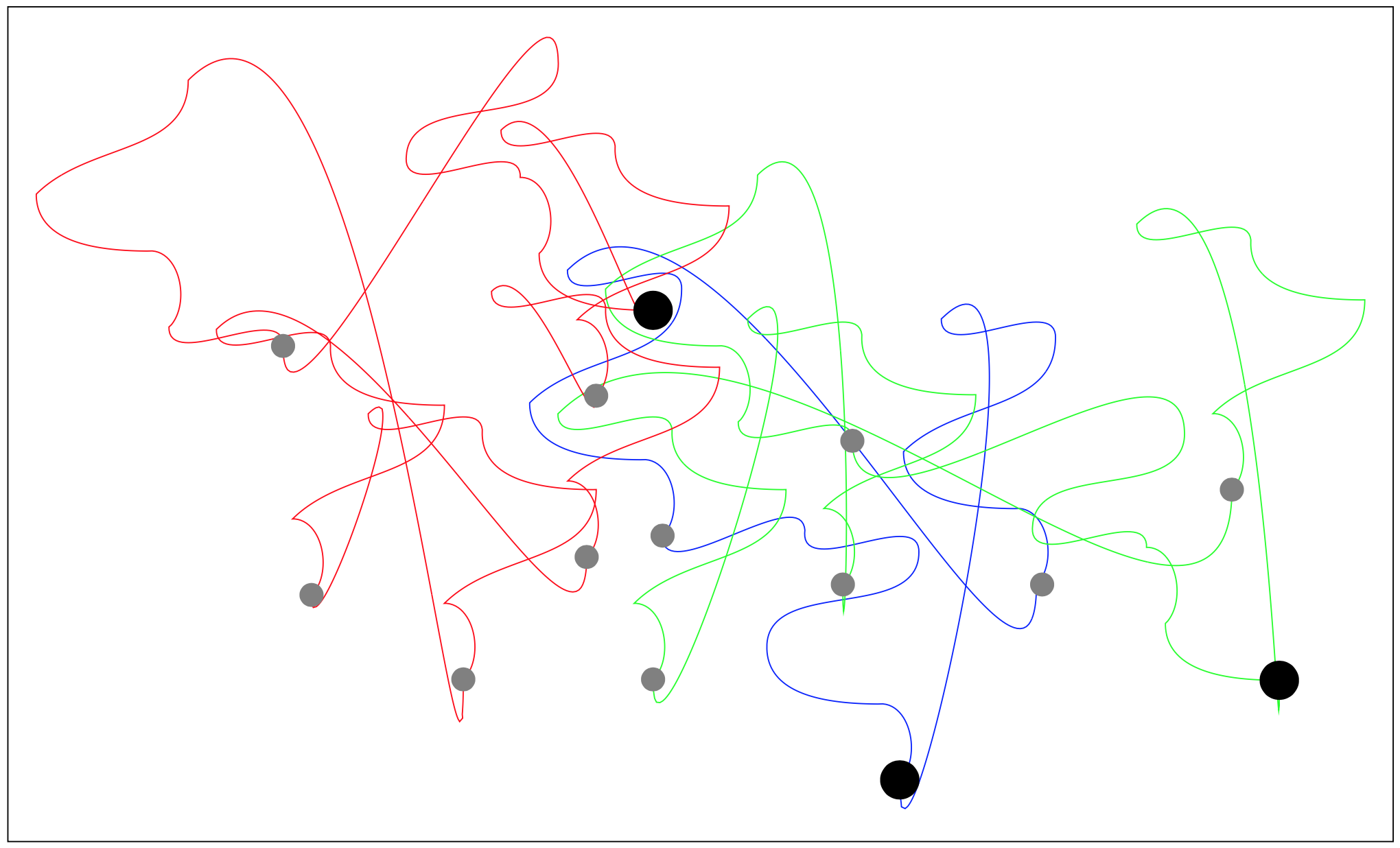}
	\caption{Illustration (taken from \cite{CJK23}) of a realisation of the Bose gas with 14 particles (grey and black bullets), organised in three Brownian bridges, attached to three Poisson points (black bullets). The red cycle has six particles, the blue and green ones each four.}
	\label{Pix_BoseGas}
\end{figure}

From \cite[Lemma 3.12]{ACK10} we already know that the first term on the right-hand side of \eqref{rewrite} satisfies, for any $\rho\in(0,\infty)$,
\begin{equation}\label{limnorming}
\lim_{N\to\infty}\frac1{|\L_N|}\nu_{\L_N}^{\ssup{\rm {bc}}}(\L_N\times \Ccal^{\ssup\circlearrowleft})=(4\pi\beta)^{-d/2}\sum_{k\in\N} k^{-1-d/2}=:\overline q,
\end{equation}
where $\L_N$ is the centred box with volume $\sim N/\rho$ (it was shown there for periodic and Dirichlet boundary condition only, but it is easy to see that it holds also for particle boundary condition). Hence, it appears natural to formulate all our main assertions about the  interacting Bose gas in terms of the  transformed probability measure 
\begin{equation}\label{transformedmeasure}
\d\widehat \LPP^{\ssup{\L,{\rm bc}}}=\frac{\e^{-\Phi_{\L,\L}}}{\widehat Z^{\ssup{{\rm bc}}}(\L)}\,\d \LPP^{\ssup{\L,{\rm bc}}},\qquad\mbox{where }\widehat Z^{\ssup{{\rm bc}}}(\L)=\LPP^{\ssup{\L,{\rm bc}}}[\e^{-\Phi_{\L,\L}}],
\end{equation}
and we do that for all three boundary conditions (periodic, zero, particle). Note that $\widehat \LPP^{\ssup{\L,{\rm bc}}}$ is a probability measure on the set of configurations with an arbitrary number of particles; there is no condition on $\mathfrak N_\L{\ssup\ell}$.

\subsection{Purpose of this paper}\label{sec-Purpose}

\noindent The main purpose of the present paper is to identify, in the thermodynamic limit, the limiting free energy with distinction between long and short loops. The most immediate such distinction would be to call all particles in loops with lengths $> L$ the $L${\it -condensate}, and to register its cardinality and to quantify its contribution to the partition function in the limit as $N\to\infty$, followed by $L\to\infty$. See Figure \ref{Pix_BoseGassub} for an illustration of a configuration with one long loop.
\begin{figure}[!htpb]
\centering
\includegraphics[scale=0.4]{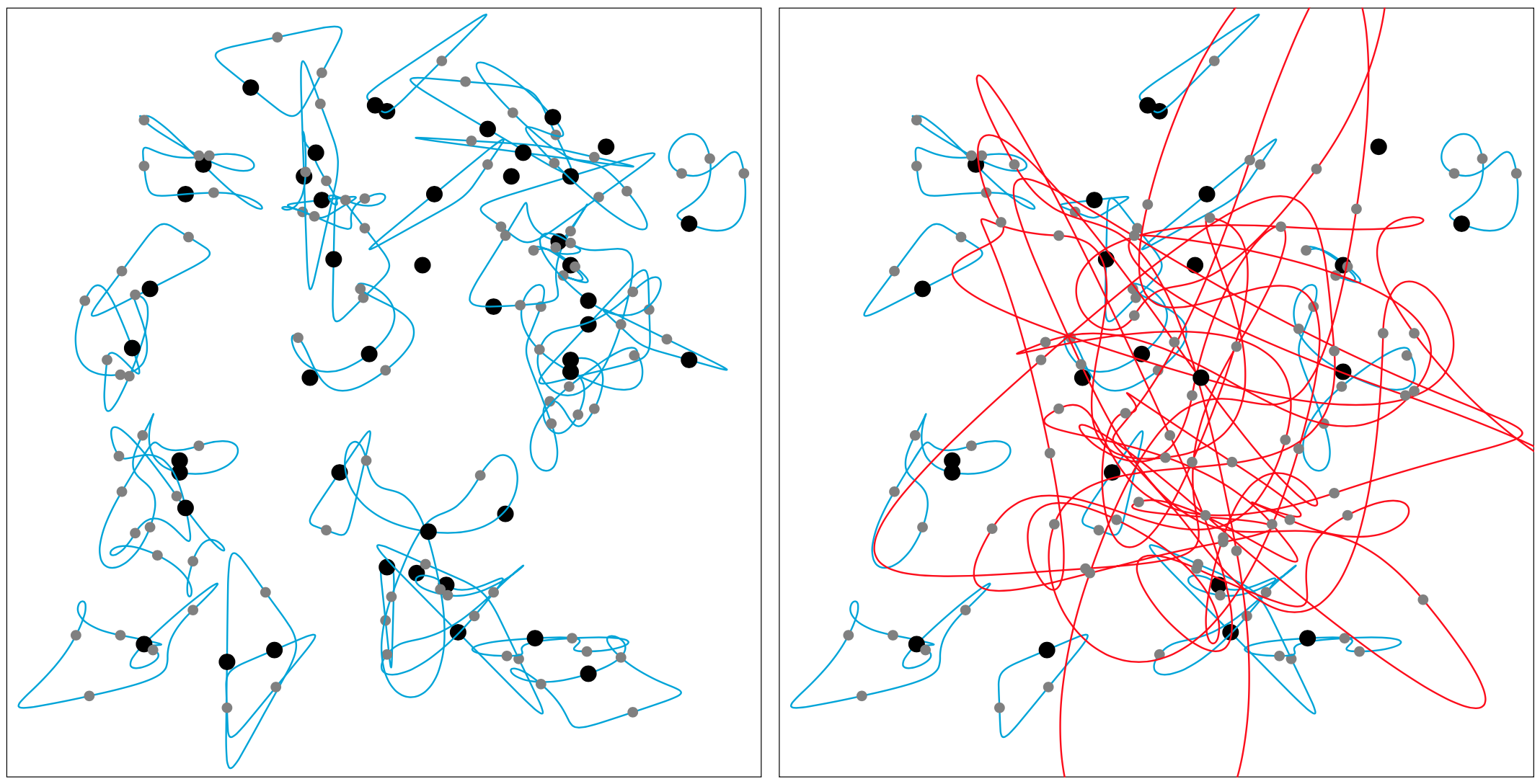}
\caption{Illustration (taken from \cite{CJK23}) of two realisations of a Bose gas: left with low density with only short loops, and right with higher density with one additional long loop (red).}
\label{Pix_BoseGassub}
\end{figure}

However, we are going to rely on a slightly different distinction, which distinguishes the loops according to space criteria rather than length criteria, which will come in handy in our proofs and is also dictated by the presence of interactions, which are spatially defined. Furthermore, this new concept of distinction between short and long loops opens up the possibility of suitable methods to cut the long loops into pieces (``shreds'') and to embed these pieces later into infinitely long trajectories.

Fix $\rho\in(0,\infty)$  and a centred box $\L_N$ in $\R^d$ with volume $|\L_N|\sim N/ \rho$. Fix $R\in\N$ and  write $W=[-R,R]^d$ and $W_z= z+W$ for $z\in 2R\Z^d$. We assume $N$ and $R$ to be such that for suitable $Z_{N,R}\subset 2R\Z^d$ we have that $\L_N=\bigcup_{z\in Z_{N,R}}W_z $ is the (up to boundaries, disjoint) union of the subboxes $W_z$ with  $z\in Z_{N,R}$. (Here, $R$ may depend on $N$ and converges to $R$ as $N\to\infty$; we will sometimes not indicate this dependence in the notation.) We will consider $N$ particles in $\L_N$.

We will call a loop $R${\it -crossing} if it has particles in  at least two different ones of these subboxes. Now we consider the number 
\begin{equation}\label{numberRcrossing}
{\mathfrak N}_\L^{\ssup {\ell,R}}(\omega)=\sum_{x\in\zeta \cap\L}\ell(f_x)\1\{f_x\mbox{ is $R$-crossing}\},\qquad \omega=\sum_{x\in\zeta}\delta_{(x,f_x)},
\end{equation} 
of particles in $R$-crossing loops and take this as a substitute of the number of particles in long loops. We call the particles in $R$-crossings the $R${\em -condensate}, and $\frac1{|\L_N|}{\mathfrak N}_\L^{\ssup {\ell,R}}(\omega)$ the $R${\em -condensate density}. The number of all the other particles is denoted by
\begin{equation}\label{numbernonRcrossing}
{\mathfrak N}_\L^{\ssup {\ell,\neg R}}(\omega)=\sum_{x\in\zeta \cap\L}\ell(f_x)\1\{f_x\mbox{ is not $R$-crossing}\},
\end{equation} 
i.e., the number of particles in loops whose particles are entirely contained in one of the subboxes $z+W$. Note that ${\mathfrak N}_\L^{\ssup {\ell,R}}+{\mathfrak N}_\L^{\ssup {\ell,\neg R}}={\mathfrak N}_\L^{\ssup {\ell}}$. 

The main object of interest is the pair $({\mathfrak N}_{\L_N}^{\ssup {\ell, R}},{\mathfrak N}_{\L_N}^{\ssup {\ell,\neg R}})$  in the limit $N\to\infty$, followed by $R\to\infty$. We will formulate our main findings in Theorem~\ref{thm-freeenergy} in terms of a large-deviation principle for this pair (properly normalised) under the measure $\widehat \LPP^{\ssup{\L,{\rm bc}}}$ defined in \eqref{transformedmeasure}.

See Section~\ref{sec-condensate} for an explanation that this concept of $R$-crossing loops is equivalent to the more familiar concept of long loops.

Here is a survey of the remainder of the paper: In Section~\ref{sec-MainRes} we present our main results, in Section~\ref{sec-Comments} we make various comments about the impact and background of our results and the history of the problem, and we survey our proof strategy, in Section~\ref{sec-Scanning} we begin our proof by introducing the necessary framework for point processes, introducing our main tool, a crucial empirical measure of the loop-shred configurations in subboxes of  a fixed radius $R$, and rewriting the free energy in terms of it. In Section~\ref{sec-distshreds}, we examine the distribution of a crucial part of our empirical measure, the configurations of the entry-exit sites of the shreds into the subboxes. The necessary large-deviation setting is developed and applied in Section~\ref{sec-LDPana}. For making the limit as $R\to\infty$ afterwards, we introduce in Section~\ref{sec-limitingentropy} the crucial specific relative entropy density. Finally, we finish the proof of our main theorem (the description of the free energy in terms of a variational formula) in Section~\ref{sec-finishproof}. As preparations, some crucial, but technical properties of the formula are proved in Section~\ref{sec-propchi}. In the Appendix, Section~\ref{sec-InterlacePPP}, we explain the connection of our description with the Brownian interlacement Poisson point process by Sznitman.

\section{Main results}\label{sec-MainRes}

\noindent In this section we describe our main results. In Section~\ref{sec-mainresult} we formulate our result on the identification of the constrained free energy with full control on the number of particles in $R$-crossing loops and in the others (i.e., in long and short loops). This formula will be a variational formula that plays in the world of stationary processes of points with loops as marks, superposed with a point process of interlacements, i.e., infinitely long double-sided locally finite trajectories. In the formula, such processes minimise the sum of an energy term plus an entropy term, subject to a fixed particle density in loops and in interlacements, respectively. The underlying loop space was already introduced in Section~\ref{sec-PPPdescription}, but we also need to introduce the interlacement space, which we do in Section~\ref{sec:SPP}. The crucial and novel specific relative entropy density term is introduced in Section~\ref{sec-twoentropies}. Then we formulate our main result on the free energy in Section~\ref{sec-mainresult}. The arising characteristic variational formula contains a lot of information about the limiting behaviour of the interacting Bose gas; this is presented in Section~\ref{sec-VP}.

\subsection{The interlacement space and the shred spaces}\label{sec:SPP}

We are going to introduce the space on which we will describe the long loops and their particles in the limit as $N\to\infty$. Define
\begin{equation}\label{Cinftydef}
\Ccal_\infty=\Big\{g\in\Ccal(\R\to\R^d)\colon \{g(k\beta)\colon k\in\Z\}\mbox{ is locally finite}\Big\},
\end{equation}
where by $\Ccal(\R\to\R^d)$ we denote the space of continuous functions from $\R$ to $\R^d$. We will always tacitly identify functions in $\Ccal_\infty$ as equal if one of them is the time-shift of the other by an integer multiple of $\beta$; hence strictly speaking we are dealing with equivalence classes modulo time shifts in $\beta\Z$.

\begin{remark}[Interlacements]\label{rem-interlacements}
The term \lq interlacements\rq\ (in plural) is often used for the range $\bigcup_{g\in \Gamma_{\rm B}}g(\R)$ of a certain Poisson point process $\varpi^{\ssup{\rm B}}=\sum_{g\in\Gamma_{\rm B}}\delta_g$ of bi-infinite Brownian motions in $\R^d$ with $d\geq 3$ (see Remark~\ref{rem-BrownianInterlacements}), but this set will play no role in this paper, as we will have to work always with the entire trajectories $g$. Therefore, we will take the freedom to call every element of $\Ccal_\infty$ an {\it interlacement}.
\hfill$\Diamond$
\end{remark}

We call the sites $g(k\beta)$  the {\em particles} of the interlacement $g\in\Ccal_\infty$ and the pieces $g_k= g|_{[\beta k,\beta(k+1)]}$ the {\em legs} of $g$. We consider the space
\begin{equation}\label{shredprocess}
\Shreds=\Big\{\varpi=\sum_{g\in\Gamma}\delta_{g}\colon \Gamma\subset\Ccal_{\infty}, \varpi\mbox{ is simple}, \{g(k\beta)\colon k\in\Z, g\in \Gamma\}\mbox{ is locally finite}\Big\}
\end{equation}
of all simple locally finite point processes with points in $\Ccal_\infty$.
The set $\{g(k\beta)\colon k\in\Z, g\in \Gamma\}$ is the set of all particles of $\varpi$. For $W\Subset \R^d$, 
\begin{equation}\label{partnumbinter}
\mathfrak N_W^{\ssup\ell}(\varpi)=\sum_{g\in \Gamma}\ell_W(g),\qquad\mbox{where }\ell_W(g)=\sum_{k\in\Z}\1_W(g(k\beta)),
\end{equation}
is the number of particles that $\varpi$ has in $W$. The maximal pieces $g|_{[\beta k_1,\beta k_2]}$ with $k_1,k_2\in\Z$ and $g(\beta k_1), g(\beta (k_1+1)), g(\beta(k_1+2)), \dots,g(\beta(k_2-1))\in W$ (and hence $g(\beta (k_1)-1))\notin W$ and $g(\beta k_2)\notin W$) are the $W${\em -shreds} of $g$; see Section~\ref{sec-chopping}, in particular \eqref{eq:shred_op}. Then  $\Ccal_W$ denotes the set of all $W$-shreds, and $\Scal_W=\Mcal_{\N_0}(\Ccal_W)$ denotes the set of locally finite point measures of all the $W$-shreds. On $\Shreds$ we consider the topology that is generated by a family of certain projection (\lq shredding\rq) operators $\Pi_W^{\ssup\Scal}\colon\Scal\to\Scal_W$ with $W\Subset\R^d$, which we spell out in \eqref{shredding}. Let us already mention that these shredding operators cut the interlacements $g\in\Ccal_\infty$ into a point process of all the $W$-shreds, but never cut legs. Let $\Mcal_1^\ssup{{\rm s}}(\Shreds)$ denote the set of shift-invariant probability measures on $\Shreds$. On this set we consider the weak topology that is generated by the shredding operators. For details, see Section~\ref{sec-chopping}.

We now introduce a crucial Brownian reference measure on the set $\Scal_W$. Let $B=(B_t)_{t\in[0,\infty)}$ be a standard Brownian motion in $\R^d$ that starts from $x\in\R^d$ under $\P_x$.  We denote by
\begin{equation}
q_{x,y}^{\ssup {W,l}}(\d g)=\P_x\big((B_t)_{t\in [0,\beta l]}\in\d g\mid B_\beta,B_{2\beta},\dots,B_{(l-1)\beta}\in W, B_{l\beta}=y\big)\in\Mcal_1(\Ccal_W)
\end{equation}
the distribution of a Brownian shred from $x\in W$ to $y\in W^{\rm c}$ with $l\in\N$ particles, and by
\begin{equation}\label{kernelK}
{\tt K}_W(\mu,\d\varpi)=\bigotimes_{i\in I}q_{x_i,y_i}^{\ssup {W,l_i}}(\d g_i)\in\Mcal_1(\Scal_W),\qquad \mu=\sum_{i\in I}\delta_{(x_i,l_i,y_i)},\varpi=\sum_{i\in I}\delta_{g_i},
\end{equation}
the conditional distribution of a Brownian $W$-shred configuration given the point process $\mu\in\Mcal_{\N_0}(W\times \N\times W^{\rm c})$ of their entry and exit sites and their lengths (certainly under the assumption that $g_i(0)=x_i$ and $g_i(l_i\beta)=y_i$ and $g_i\in \Ccal_{l_i}$ for all $i\in I$, and $I$ is a finite index set).

\begin{remark}[Brownian interlacements PPP]\label{rem-BrownianInterlacements} 
In dimension $d\geq 3$, the kernel ${\tt K}_W$  appears as a regular version of the conditional distribution of a variant of the well-known Brownian interlacement Poisson point process; see  Section~\ref{sec-InterlacePPP}. More precisely, given that the entry and exit sites and lengths of all the $W$-shreds of this PPP are given as $\mu$, $K_W(\mu,\cdot)$ is the conditional distribution  of the point process of the $W$-shreds in this PPP, see Lemma~\ref{lem-KWident}. This works only in $d\geq 3$ because of the transience; in $d\leq 2$, this PPP would be uninteresting; it would not even be locally finite.
\hfill$\Diamond$
\end{remark}

\subsection{Our first result: A new specific relative entropy density}\label{sec-twoentropies}

In our description of the free energy of the interacting Bose gas in Theorem~\ref{thm-freeenergy} below, a crucial ingredient will be a new specific relative entropy density with respect to the superposition of the reference measure  introduced in Definition~\ref{def-PPP} and the family of kernels defined in \eqref{kernelK}. Here we introduce it as our first main result. By $H(\mu\mid\nu)$ we denote the relative entropy of a finite measure $\mu$ with respect to another one, $\nu$, on a Polish space $\mathcal X$, defined by
\begin{equation}\label{Entropydef}
H(\mu\mid\nu)= \nu(\mathcal X)-\mu(\mathcal X)+ \int_{\mathcal X}\mu(\d x)\log \frac{\d\mu}{\d\nu}(x),
\end{equation}
if $\mu\ll\nu$, and otherwise $H(\mu\mid\nu)=\infty$. A characterisation that is often helpful if $\mu$ and $ \nu$ are probability measures is the formula
\begin{equation}\label{entropyformula}
H(\mu\mid\nu)=\sup_{f}\Big(\langle f,\mu\rangle - \log\langle \e^f,\nu\rangle\Big),\qquad\mu,\nu\in\Mcal_1(\Xcal),
\end{equation}
where the supremum is on measurable and bounded functions $f\colon \Xcal\to\R$ (the supremum is approximately attained in $f=\log\frac{\d\mu}{\d\nu}$).

It is well-known, see  \cite[Ch.~15]{G88} and \cite{GZ93},
that there is a notion of a {\em specific relative entropy density} with respect to the loop reference process $\LPP$, defined as
\begin{equation}\label{entropydensitydef}
\h(P\mid\LPP)=\lim_{R\to\infty}\frac 1{|W_R|}H\big(\Pi_{W_R}(P)\mid \Pi_{W_R}(\LPP)\big),\qquad P\in\Mcal_1^{\ssup{\rm s}}(\Lcal),W_R=[-R,R]^d,
\end{equation}
for all shift-invariant probability measures $P$ on the set $\Lcal$ of simple marked point measures with loops as marks; here $\Pi_W$ denotes the usual projection to the loops with starting site in $W\Subset \R^d$. The function $\h$ is introduced in the general frame of point processes with marks; and the shape or any properties of the marks play hardly any role in this theory.

In the present paper, we are going to work on the extended space $\Lcal\times \Shreds$, which we conceive as the set $\Mcal_{\N_0}([\R^d\times\Ccal^{\ssup\circlearrowleft}]\times \Ccal_\infty)$ of simple point measures on the set of loops and interlacements in $\R^d$.
Let $\Mcal_1^{\ssup{\rm s}}(\Lcal\times \Shreds)$ denote the set of all shift-invariant probability measures on $\Lcal\times \Shreds$. We write the elements of $\Lcal$ as $\omega=\sum_{x\in\zeta}\delta_{(x,f_x)}$ and the elements of $\Scal$ as $\varpi=\sum_{g\in\Gamma}\delta_g$. 

We make a distinction between loops whose particles are contained entirely in one of the subboxes $z+W_R$ with $z\in 2R \Z^d$ and those whose particles are not; the latter loops are called $R$-crossings, see Section~\ref{sec-Purpose}. Hence, adapted notions of specific relative entropies need to be introduced.
For the first (the loops whose particles are contained in some $z+W$ with $z\in 2R\Z^d$), we need to introduce the {\em restricted} projection  operator of loops,
\begin{equation}\label{RestrictedLoopsProj}
\Pi_W^{\ssup\Lcal}(\omega)=\sum_{x\in \zeta\colon f_x(k\beta)\in W\,\forall k\in[\ell(f_x)]}\delta_{(x,f_x)}\in\Lcal_W=\Mcal_{\N_0}(W\times \Ccal_W^{\ssup{\circlearrowleft}}),
\end{equation}
which restricts to all the loops that have all their particles in $W$ (we used the notation $[n ]=\{1,\dots,n\}$). We wrote $\Mcal_{\N_0}(\Xcal)$ for the set of all point measures on points in some space $\Xcal$ and $\Ccal_W^{\ssup{\circlearrowleft}}$ for the set of loops with all particles in $W$. (See Section~\ref{sec-PPPs} for more explanations.) Note that the legs $f_{x,i}$ start and  end in $W$; but do not have to stay in $W$; they can travel arbitrarily far away from $W$. 

For describing the second kind of loops (the $R$-crossing ones), we introduce the shredding operator 
\begin{equation}\label{shredding}
\Pi_W^{\ssup\Scal}(\omega)=\sum_{x\in\zeta\colon \{f_x(k\beta)\colon k\in[\ell(f_x)]\}\not\subset W}\sum_{g\in\Gamma_W(f_x)}\delta_{g}\in \Scal_W=\Mcal_{\N_0}(\Ccal_W),
\end{equation}
where $\Gamma_W(f_x)$ is the set of $W$-shreds of the loop $f_x$, and $\Ccal_W$ is the set of all $W$-shreds. (See Section~\ref{sec-chopping} for more details and explanations.)  Analogously, the operator $\Pi_W^{\ssup \Scal}(\varpi)$ is defined also for point measures $\varpi=\sum_{g\in\Gamma}\delta_g$ of interlacements.  By $\Pi_W(\omega,\varpi)$ we denote the sum (i.e., the superposition) $\Pi_W^{\ssup\Lcal}(\omega)+\Pi_W^{\ssup\Scal}(\varpi)$.  Another important operator is the boundary-shred operator
\begin{equation}\label{boundaryshred}
\partial\Pi_W^{\ssup\Scal}(\varpi)=\sum_{g\in\Gamma}\sum_{g' \in \Gamma_W(g)}\delta_{(g'(\beta),\ell(g'),g'(\ell(g')\beta))}\in \Tcal_W=\Mcal_{\N_0}(W\times\N\times W^{\rm c}),
\end{equation}
which registers initial and terminal sites and particle numbers of the $W$-shreds $(g'(\beta),\dots,g'(\ell(g')\beta))$ with $g'(0)\in W^{{\rm c}}$ and $g'(\beta),\dots,g((\ell(g')-1)\beta)\in W$ and $g'(\ell(g')\beta)\in W^{{\rm c}}$. For probability measures $P$ on loop or shred space and any of these operators $\Pi$, we denote by $\Pi(P)$ the image measure $P\circ\Pi ^{-1}$ of $P$ under $ \Pi$. Hence, for $ P\in\Mcal^{\ssup{ \rm s}}_1(\Lcal\times\Scal)$, the measure $\Pi_W(P)$ is the distribution of all the loops of $P$ that have all particles in $W$, superposed with all the $W$-shreds that come from interlacements or from the other loops; {\em a priori} the $W$-shreds cannot be distinguished as to whether they come from a loop or from a shred.

As it is standard in measure theory, the product of a measure $\mu$ on a measurable space $\Xcal$ with a kernel $K$ from $\Xcal$ into another measurable space $\Ycal$ is defined by
\begin{equation}\label{productmeasure}
\mu \otimes K(\d (x,y))= \mu(\d x)\,K(x,\d y).
\end{equation}
We introduce the relative entropy
\begin{equation}\label{JWdeffirst}
J_W(\xi)=\frac 1{|{W}|}H_{\Lcal_{W}\times\Scal_{W}}\big(\xi\,\big|\,\Pi_{W}^{\ssup\Lcal}(\LPP)\otimes[ \partial\Pi_{W}^{\ssup\Shreds}(\xi)\otimes {\tt K}_{W}]\big)\in[0,\infty],\qquad W\Subset \R^d,\xi\in\Mcal_1(\Lcal_W\times\Scal_W).
\end{equation}
In our next result, we  introduce the main entropy function of this paper,
\begin{equation}\label{jointentropydef}
\h^{\ssup{\Lcal,\Scal}}(P)
=\lim_{R\to\infty} J_{W_R}(\Pi_{W_R}(P)),\qquad P\in\Mcal_1^{\ssup {\rm s}}(\Lcal\times\Scal), \mbox{with }W_R=[-R,R]^d.
\end{equation}
This can be seen as a kind of specific relative entropy density with respect to the $\Pi_W^{\ssup\Lcal}$-projections of the reference loop process ${\tt Q}$ and the family of Brownian $W$-shred reference kernels ${\tt K}_W$ defined in \eqref{kernelK}; observe that the shred-part of the reference measure comes from the boundary-shred distribution of $P$. This means that the loop-part reference measure has a particle density, but the shred-part inherits it from $P$.

We need to introduce also some particle-counting operators. Let ${\mathfrak N}_W ^{\ssup{\ell,\Lcal}}$ and ${\mathfrak N}_W ^{\ssup{\ell,\Shreds}}$ denote the functionals that count the particles in $W $ of the loop point process and of the shred point process, respectively (recall \eqref{partnumbloop} and \eqref{partnumbinter}):
\begin{equation}\label{def_particlenumbers}
	{\mathfrak N}_W ^{\ssup{\ell,\Lcal}}(\omega,\varpi)={\mathfrak N}_W ^{\ssup{\ell}}(\omega)=\sum_{x\in\zeta\cap W}\ell(f_x)\qquad\mbox{and}\qquad {\mathfrak N}_W ^{\ssup{\ell,\Shreds}}(\omega,\varpi)={\mathfrak N}^{\ssup{\ell}}_W (\varpi)=\sum_{g\in\Gamma} \ell_W (g),
\end{equation}
and put $\mathfrak N_W^{\ssup {\ell}}(\omega,\varpi)= {\mathfrak N}_W ^{\ssup{\ell,\Lcal}}(\omega,\varpi)+{\mathfrak N}_W ^{\ssup{\ell,\Shreds}}(\omega,\varpi)$ for the total particle number in $W$.
Note that ${\mathfrak N}_W ^{\ssup{\ell,\Lcal}}$ counts particles in all loops that start from $W$ (regardless whether the particle lies in $W$ or not), while ${\mathfrak N}_W ^{\ssup{\ell,\Shreds}}$ counts only particles that lie in $W$ coming from any interlacement. (See Remark~\ref{rem-counting} for an alternate way of counting particles.)

\begin{theorem}[Joint specific relative entropy density]\label{thm-specrelent} If $P\in\Mcal_1^{\ssup {\rm s}}(\Lcal\times\Scal)$ satisfies $\langle P,\mathfrak N_U^{\ssup\ell}\rangle <\infty$ (with $U=[-\frac 12,\frac 12]^d$ the unit box), then the limit in \eqref{jointentropydef} exists in $[0,\infty]$, and it satisfies
\begin{equation}\label{hbound}
J_{W_R}(\Pi_{W_R}(P))\leq  \h^{\ssup{\Lcal,\Scal}}(P),\qquad R\in\N.
\end{equation}
Furthermore, the map $P\mapsto \h^{\ssup{\Lcal,\Scal}}(P)$ is lower semi-continuous and convex, and it is also concave (i.e., even affine). Moreover, for  any sequence $(K_{R_N})_{N\in\N}$ of compact sets $K_{R_N}\subset\Tcal_{W_{R_N}}$ such that $R_N\to\infty$ as $N\to\infty$, the restricted level set 
\begin{equation}\label{levelset}
\bigcap_{N\in\N}\big\{P\in\Mcal_1^{\ssup {\rm s}}(\Lcal\times\Scal)\colon \partial\Pi_{W_{R_N}}^{\ssup\Scal}(P)\in K_{R_N}, \h^{\ssup{\Lcal,\Scal}}(P)\leq c, \langle P,\mathfrak N_U^{\ssup\ell}\rangle \leq c\big\}
\end{equation}
is compact for any $c\in\R$.
\end{theorem}

The proof is in Section~\ref{sec-limitingentropy}. While establishing $\h(P\mid\LPP)$ relies on superadditivity, our proof of \eqref{jointentropydef} relies on probabilistic arguments in connection with large-deviation estimates for  a crucial empirical measure.

We call the function $\h^{\ssup{\Lcal,\Scal}}$ the {\em specific relative entropy density} with respect to the superposition of the Brownian loop soup and the family of Brownian shred kernels ${\tt K}_W$. 
The reference measure in \eqref{jointentropydef} samples with $\Pi_W^{\ssup\Lcal}({\tt Q})$ the loops and then, independently, with $\partial\Pi_W^{\ssup\Scal}(P)$ the boundary shred configuration and then with ${\tt K}_W$ the shred interior configuration in $W$.

\subsection{Our main result: an extended variational description}\label{sec-mainresult}
 We are heading towards a formulation of our main result.
We work on the space $\Lcal\times \Shreds$ of pairs of point processes of points with loops as marks and point processes of interlacements. Recall that we identify this set as the set of all superpositions of point processes of loops and interlacements.

We need to introduce an interaction functional, decomposed into its different loop-interlacement components. For this, we need notation for interactions between any two measurable sets $W,\widetilde W\subset \R^d$. For  $(\omega,\varpi)=(\sum_{x\in \zeta}\delta_{(x,f_x)},\sum_{g\in\Gamma}\delta_g)\in\Loops\times\Shreds$ we put
\begin{eqnarray}
	F_{W,\widetilde W}^{\ssup {\Lcal\Lcal}}(\omega,\varpi)&=& \sum_{x\in\zeta\cap W}\sum_{y\in\zeta\cap \widetilde W}\sum_{i=1}^{\ell(f_x)}\sum_{j=1}^{\ell(f_y)}\1_{(x,i)\not=(y,j)} V(f_{x,i},f_{y,j}),\label{FLL}\\
	F_{W,\widetilde W}^{\ssup {\Lcal\Shreds}}(\omega,\varpi)&=& \sum_{x\in\zeta\cap W}\sum_{g\in\Gamma}\sum_{i=1}^{\ell(f_x)}\sum_{k \in \Z}\1_{\widetilde W}(g(k\beta)) V(f_{x,i},g_{k+1}),\label{FLS}\\
	F_{W,\widetilde W}^{\ssup {\Shreds\Shreds}}(\omega,\varpi)&=&\sum_{g,g'\in\Gamma}\sum_{k,k'\in \Z}\1_{(g,k)\not=(g',k')}\1_W(g(k\beta))\1_{\widetilde W}(g'(k'\beta)) V(g_{k+1}, g'_{k'+1}),\label{FSS}
\end{eqnarray}
where we write the legs of $g$ here as $g_k(\cdot)=g((k-1)\beta+\cdot)|_{[0,\beta]}$, i.e., the $(k+1)$-st leg starts at the particle $g(k\beta)=g_{k+1}(0)$. In words, $F_{W,\widetilde W}^{\ssup {\Lcal\Lcal}}$ is the interaction between any two different legs of loops that start in $W$ and $\widetilde W$, respectively, (regardless where the leg starts), $F_{W,\widetilde W}^{\ssup {\Lcal\Shreds}}$ is the one between legs of loops that start in $W$ and legs starting in $\widetilde W$ of any interlacement, and $F_{W,\widetilde W}^{\ssup {\Shreds\Shreds}}$ is the one between any two legs that start in $W$, respectively in $\widetilde W$, of any (one of two) interlacement(s). Then we define the total interaction within $W$ as
\begin{equation} \label{FWWdef}
F_{W,W}= \frac 12 F_{W,W}^{\ssup {\Lcal\Lcal}}+F_{W,W}^{\ssup {\Lcal\Scal}}+\frac 12 F_{W,W}^{\ssup {\Scal\Scal}},
\end{equation}
(since every pair of  legs should appear precisely once and because of the possibility of double-counting, we need to be a bit cautious).

For the formulation of Theorem~\ref{thm-freeenergy}, we need the total interaction between the unit box $U=[-\frac 12,\frac 12]^d$ and $\R^d$:
\begin{equation}\label{FUdef}
F_U=\frac 12 \big[F_{U,U}^{\ssup{\Lcal\Lcal}}+ 2F_{U,U}^{\ssup {\Lcal\Shreds}}+ F_{U,U}^{\ssup {\Shreds\Shreds}}\big]
+F_{U,U^{\rm c}}^{\ssup {\Lcal\Lcal}}+ F_{U,U^{\rm c}}^{\ssup {\Lcal\Shreds}}+ F_{U^{\rm c},U}^{\ssup {\Lcal\Shreds}}+F_{U,U^{\rm c}}^{\ssup {\Shreds\Shreds}}.
\end{equation}
In words, this is the interaction between all legs of loops that start in $U$ and all the other legs plus the interaction between all interlacement-legs that start in $U$ and all the others (without double-counting). See Remark~\ref{rem-counting} for an alternative way of counting the interaction between $U$ and the rest.
 
We are going to work on the following assumption on the interaction potential:

\medskip

\noindent{\bf Assumption (V).} {\em The interaction potential $v\colon \R^d\to[0,\infty)$ is  bounded and continuous and has compact support. Furthermore, there exists $C>0$ such that, for any box $W\Subset\R^d$ containing $[-1,1]^d$,}
 \begin{equation}\label{LowBoundAss}
	\sum_{1\leq i<j\leq k} v(x_i-x_j)\geq C\left(\frac {k^2}{|W|} - k\right),\qquad k\in\N, x_1,\dots,x_k\in W.
 \end{equation}
 
 The last part of Assumption (V) is called \emph{superstability} of $v$~\cite[\textsection 3.2.9]{Rue69}.

\begin{remark}[Assumptions on the interaction]\label{rem-interaction}
As far as we are aware of, there is no canonical choice of an interaction function $v$ from a physicist's point of view. Details of $v$ might be motivated by physical or chemical properties of the matter that is to be modelled, but we have no particular choice in mind. We tried to keep $v$ general and put the assumptions of measurability, boundedness of $v$ and of the support just for convenience to keep the technicalities in the proofs low. In fact, several lengthy technicalities in the proofs in \cite{ACK10} are due to our decision to allow $v$ and its support  to be unbounded there, but do not appear in the present paper. However, in our proof we need that the effect of $v$ is sufficiently repulsive, which motivates \eqref{LowBoundAss}. 
\hfill$\Diamond$
\end{remark} 

The following is our main result; it identifies the limiting free energy in the thermodynamic limit with fixed mass in the $R$-crossings in the limit $R\to\infty$. Recall the definition of the number $\mathfrak N_\L^{\ssup{\ell,R}}(\omega)$ of particles in $R$-crossing loops from \eqref{numberRcrossing} and the number $ \mathfrak N_{\L_N}^{\ssup{\ell, \neg R}}$ of the remaining particles from \eqref{numbernonRcrossing}.

\begin{theorem}[Explicit formula for constrained free energy]\label{thm-freeenergy} Assume $\beta, \rho\in(0,\infty)$ to be fixed and suppose that $v$ satisfies Assumption (V). Fix a sequence of centred boxes $\L_N\subset\R^d$ such that $\lim_{N\to\infty}|\L_N|/N=1/\rho$. We partition (up to boundary elements) $\L_N=\bigcup_{z\in Z_{N,R}}W_z $ into the subboxes $W_z=z+W=z+[-R,R]^d$ with  $Z_{N,R}\subset 2R\Z^d$. Then, in the limit as $N\to\infty$, followed by $R\to\infty$, the pair $\frac 1{|\L_N|}(\mathfrak N_{\L_N}^{\ssup{\ell, R}},\mathfrak N_{\L_N}^{\ssup{\ell, \neg R}})$ satisfies under the transformed measure $\widehat{\tt Q }^{\ssup{\L_N,{\rm bc}}}$ defined in \eqref{transformedmeasure} a large-deviations principle on the set $\{(\rho_1,\rho_2)\in[0,\infty)^2\colon \rho_1+\rho_2=\rho\}$ with continuous and convex rate function $(\rho_1,\rho_2)\mapsto \chi(\rho_1,\rho_2)-\overline \chi(\rho)$, where $\overline \chi(\rho)=\inf_{\rho_1+\rho_2=\rho}\chi(\rho_1,\rho_2)$ and
\begin{equation}\label{chidefneu}
	\chi(\rho_1,\rho_2)\defeq \inf\Big\{\h^{\ssup{\Lcal,\Scal}}(P)+\langle F_U,P\rangle\colon 
	P\in \Mcal_1^{\ssup{\rm s}}(\Lcal\times \Shreds), \langle {\mathfrak N}_U^{\ssup {\ell,\Lcal}},P\rangle=\rho_1 , \langle {\mathfrak N}_U^{\ssup {\ell,\Shreds}},P\rangle=\rho_2\Big\}.
\end{equation}
\end{theorem}

The proof of Theorem~\ref{thm-freeenergy} comes in three steps in Sections~\ref{sec-uppboundN} (preparation), \ref{sec-LDPupperbound} (making $N\to\infty$) and \ref{sec-finishuppbound} (making $R\to\infty$) for the upper bound, and  Sections~\ref{sec-lowboundN},  \ref{sec-Nlowbound} and \ref{sec-finishlowerbound} for the lower bound. A survey on our proof strategy is in Section~\ref{sec-strategy}. The variational formula in \eqref{chidefneu} has good compactness properties. Indeed, as we formulate in Lemma~\ref{lem-propertieschi} below, the infimum in \eqref{chidefneu} may be restricted to some compact set, and $\chi$ is continuous.

We also obtain the same LDP for the numbers of particles in long and short loops, which are more natural quantities from the viewpoint of  Feynman's proposal to look at loop lengths as order parameters:

\begin{cor}[LDP for particle numbers in long and short loops]\label{cor-LDPlongshortloops}
 In the situation of Theorem~\ref{thm-freeenergy}, the pair $\frac 1{|\L_N|}(\mathfrak N_{\L_N,> L}^{\ssup\ell},\mathfrak N_{\L_N,\leq L}^{\ssup\ell})$ satisfies the same LDP as the pair $\frac 1{|\L_N|}(\mathfrak N_{\L_N}^{\ssup{\ell, R}},\mathfrak N_{\L_N}^{\ssup{\ell, \neg R}})$, where $\mathfrak N_{\L_N,\leq L}^{\ssup\ell}$ and $\mathfrak N_{\L_N,> L}^{\ssup\ell}$ are the numbers of particles in loops of length $\leq L$ and $>L$, respectively.
\end{cor}

This follows from  Lemma~\ref{lem-condensatenotion} below, where we show that these two pairs of counting variables are exponentially equivalent to each other in the limit $N\to\infty$, followed by $R\to\infty$, respectively $L\to\infty$.

For the general theory of (the probabilities of) large deviations, see \cite{DZ98}. The way how the large-deviations principle (LDP) of Theorem~\ref{thm-freeenergy} is to be understood (and how it will be proved) is that, for any  $\rho_1,\rho_2\in[0,\infty)$, 
\begin{equation}\label{freeenergyident}
\begin{aligned}
\lim_{\delta\to 0}\limsup_{R\to\infty}\limsup_{N\to\infty}&\frac1{|\Lambda_N|}\log \Big[\widehat Z^{\ssup{\rm bc}}(\L_N)\widehat{\tt Q }^{\ssup{\L_N,{\rm bc}}}\Big(\Big|\frac1{|\L_N|}\mathfrak N_{\L_N}^{\ssup{\ell, R}}-\rho_1\Big|\leq \delta, \Big|\frac1{|\L_N|}\mathfrak N_{\L_N}^{\ssup{\ell,\neg R}}-\rho_2\Big|\leq \delta\Big)\Big]\\
&\leq-\chi(\rho_1,\rho_2),
\end{aligned}
\end{equation}
and a corresponding lower bound with $\limsup$ twice replaced by $\liminf$. As a consequence, the minimiser(s) $(\rho_1,\rho_2)$ play the decisive role when further analysing the limiting Bose gas. As a first partial result, we can identify the free energy defined in \eqref{freeenergy} in the following way (recall \eqref{limnorming}).

\begin{cor}[Variational description of the free energy]\label{cor-freeenergyVP}
In the situation of Theorem~\ref{thm-freeenergy},
\begin{equation}
\begin{aligned}
\beta \,{\rm f}(\beta, \rho)=-\overline q+\inf_{\rho_1,\rho_2\geq 0\colon \rho_1+\rho_2=\rho}\chi(\rho_1,\rho_2)=-\overline q+\overline\chi(\rho),
\end{aligned}
\end{equation}
where
\begin{equation}\label{chifreeenergy}
\begin{aligned}
\overline\chi(\rho)=\inf\Big\{\h^{\ssup{\Lcal,\Scal}}(P)+\langle F_U,P\rangle\colon P\in \Mcal_1^{\ssup{\rm s}}(\Lcal\times \Shreds), \langle P,{\mathfrak N}_U^{\ssup{\ell,\Lcal}}+{\mathfrak N}_U^{\ssup{\ell,\Scal}}\rangle=\rho\Big\}.
\end{aligned}
\end{equation}
\end{cor}

We leave the proof of this corollary to the reader. 
In \eqref{chibarloopsformula} below it turns out that in \eqref{chifreeenergy}, the interlacements are not necessary, i.e., the value does not change if we restrict $P$ to having only loops.

\begin{remark}[Interpretation of Theorem~\ref{thm-freeenergy}]\label{rem-interpretation}
The limiting  Bose gas with density parameters $\rho_1$ and $\rho_2$ is characterised as a random superposition of a loop point process and an interlacement point process. The particles of the latter ones represent the condensate with condensate density $\rho_2$ per unit volume, and the loop part contains the remaining particles with density $\rho_1$. This random loop-interlacement process is stationary and has strong correlations between the loop- and the interlacement part via the interaction.  The process is characterised by minimality of the sum of energy (the interaction) and entropy with respect to the reference PPP-process ${\tt Q}$ and the family of kernels ${\tt K}_W$ under the constraint of having  these two densities.  Any minimiser $P$ in \eqref{chifreeenergy} has the interpretation of the effective spatial distribution of all the bosons in space, organised in loops and interlacements, of the interacting Bose gas in the thermodynamic limit with total particle density $\rho$.
\hfill$\Diamond$
\end{remark}

\begin{remark}[More refined assertions]\label{rem-moredetails}
Using Lemma~\ref{lem-condensatenotion} below, it is also no problem to derive an LDP for more quantities, like the joint distribution over $k\in\N$ of the number of particles in loops of length $k$ (without any reference to the decomposition of $\L_N$ into $R$-subboxes)  or the number of such loops. The corresponding rate function that arises is clear; one needs to restrict the variational problem to those $P$ that have on an average the requested density of $k$-loops. Examples of proofs for such statements in a different context can be found in \cite{AKP19} and \cite{AKLP24}.
\hfill$\Diamond$
\end{remark}

\begin{remark}[Approximations of $\chi(\rho_1,\rho_2)$.]\label{rem-chiappr}
In the course of the proof of Theorem~\ref{thm-freeenergy}, we are forced to establish some technical approximations of the minimizers $P$ in the variational formula $\chi(\rho_1,\rho_2)$ in \eqref{chidefneu}. In particular, in Lemma~\ref{lem-ergappr}, we approximate $P$ with ergodic measures that are concentrated on uniformly bounded loop/interlacement configurations with some nice additional properties, and in Lemma~\ref{lem-propertieschi}, we compactify the set of $P$s on which we minimize.
\hfill$\Diamond$
\end{remark}

\begin{remark}[Counting particle numbers and interaction in $U$]\label{rem-counting}
Instead of counting particle numbers in loops that start from $U$, one can also count particles in $U$ in loops starting at any place. More precisely, if  we define 
$\widetilde{\mathfrak N}_U^{\ssup{\ell,\Lcal}}(\omega)=\sum_{x\in\zeta}\sum_{i=1}^{\ell(f_x)}\1\{f_x((i-1)\beta)\in U\}$, then one can easily see that $\widetilde{\mathfrak N}_U^{\ssup{\ell,\Lcal}}$ has the same expectation under any $P\in\Mcal_1^{\ssup{\rm s}}(\Lcal\times\Scal)$ as ${\mathfrak N}_U^{\ssup{\ell,\Lcal}}$. Indeed,
$$
\begin{aligned}
\langle P,\mathfrak N_U^{\ssup{\ell,\Lcal}}\rangle
&=\int P(\d\omega)\,\sum_{x\in\zeta}\sum_{i=1}^{\ell(f_x)}\1\{x\in U\}\sum_{z\in\Z^d}\1\{f_x((i-1)\beta)\in z+U\} \\
&
=\sum_{z\in\Z^d}\int P(\d\omega)\, \sum_{x\in\zeta}\sum_{i=1}^{\ell(f_x)}\1\{f_x((i-1)\beta)\in U\}\1\{x\in -z+U\}
= \langle P,\widetilde{\mathfrak N}_U^{\ssup{\ell,\Lcal}}\rangle.
\end{aligned}
$$

 Analogously, one can see that the expected interaction between legs whose loop starts in $U$ and all other legs in loops is the same as the expected interaction between legs that start in $U$ (regardless where the loop starts) and all other legs in loops, with an analogous statement about interaction between loops and interlacements. More precisely, defining 
 $$
 \widetilde F_{W,\widetilde W}^{\ssup{\Lcal\Lcal}}(\omega,\varpi)
=\sum_{x,y\in\zeta}\sum_{i=1}^{\ell(f_x)}\sum_{j=1}^{\ell(f_y)}\1\{f_x((i-1)\beta))\in W\}\1\{f_y((j-1)\beta)\in \widetilde W\} \1_{(x,i)\not=(y,j)}V(f_{x,i},f_{y,j}),
$$
one sees in the same way that $\langle P,\frac 12 F_{U,U}^{\ssup{\Lcal\Lcal}}+F_{U,U^{\rm c}}^{\ssup{\Lcal\Lcal}}\rangle =\langle P, \frac 12 \widetilde F_{U,U}^{\ssup{\Lcal\Lcal}}+\widetilde F_{U,U^{\rm c}}^{\ssup{\Lcal\Lcal}}\rangle$ for any $P\in\Mcal_1^{\ssup{\rm s}}(\Lcal\times\Scal)$. Consequently, $F_U$ defined in \eqref{FUdef} has the same expectation under $P$ as the total interaction between any leg that starts in $U$ (regardless where its loop starts, if it is in a loop) and any other leg.
\hfill$\Diamond$
\end{remark}

\subsection{The variational formula and BEC}\label{sec-VP}

In this section, inspired by the results of \cite[Section 2]{CJK23}, we analyse the variational formula $\chi$ in \eqref{chidefneu} with respect to its analytic properties (existence and properties of minimisers), depending on the parameters $\rho_1$ and $\rho_2$. On the technical side, it will be nice to know that the infimum may be restricted to the set
\begin{equation}\label{Kdef}
K=\big\{P\in \Mcal_1^{\ssup{\rm s}}(\Lcal\times \Shreds)
\colon  \h^{\ssup{\Lcal,\Scal}}(P)\leq L,\langle P, ({\mathfrak N}^{\ssup\ell}_{U})^{5/4}\rangle \leq L\big\},
\end{equation}
with a sufficiently large $L\in(0,\infty)$ (recall that $U=[-\frac 12,\frac 12]^d=W_{1/2}$). This, together with the compactness property of Theorem~\ref{thm-specrelent}, makes the domain of the variational formula compact, and it gives us a good control on the two particle densities.

\begin{lemma}[Regularity properties of $\chi$]\label{lem-propertieschi}
Suppose that the interaction potential $v$ satisfies Assumption (V), in particular \eqref{LowBoundAss}. Then the following holds.
\begin{enumerate}
\item The set $K$ in \eqref{Kdef} is compact, and for any $\rho_1,\rho_2$, the infimum in the definition \eqref{chidefneu} of $\chi(\rho_1,\rho_2)$ may be restricted to $K$, not changing its value for any sufficiently large $L$. Furthermore, the maps $P\mapsto \langle P,\mathfrak N_U^{\ssup{\ell,\Lcal}}\rangle$ and $P\mapsto\langle P,\mathfrak N_U^{\ssup{\ell,\Scal}}\rangle$ are continuous on $K$.

\item The maps  $[0,\infty)^2 \ni (\rho_1,\rho_2)\mapsto \chi(\rho_1,\rho_2)$ and $[0,\infty)\ni\rho \mapsto \overline\chi(\rho)$ are convex. Furthermore, for any $\rho_1$, respectively $\rho_2\in[0,\infty)$, $\lim_{\rho_2\to\infty}\chi(\rho_1,\rho_2)=\infty$ and $\lim_{\rho_1\to\infty}\chi(\rho_1,\rho_2)=\infty$.

\item The map $\rho_1\mapsto \chi(\rho_1,0)$ is  continuous at $\rho_1=0$ with $\chi(0,0)=\chi^{\ssup{v=0}}(0,0)=\overline q=(4\pi\beta)^{-d/2}\zeta(1+d/2)$ and $\partial_{ \rho_1}\chi(\rho_1,0)|_{\rho_1=0}=-\infty$.

\item For any $\rho_1, \rho_2 \in[0,\infty)$, there is at least one minimiser $P$ in the formula for $\chi(\rho_1,\rho_2)$ on the right-hand side of \eqref{chidefneu}.

\item For any $\rho\in[0,\infty)$, there is at least one minimiser $P$ in the formula for $\overline\chi(\rho)$ on the right-hand side of \eqref{chifreeenergy}.

\end{enumerate}

\end{lemma}

The proof of Lemma~\ref{lem-propertieschi} is in Section~\ref{sec-lemmaproof}. As easy consequences of Lemma~\ref{lem-propertieschi}, the set of minimising $P$'s is convex and compact, and $\chi$ and $\overline \chi$ are continuous and have minimisers $(\rho_1,\rho_2)$ respectively $\rho$.

Now we prove that every value of energy plus entropy that can be achieved with loops and interlacements can also be achieved with only loops, with the same total particle density. In particular, the free energy can be expressed in terms of configurations that have no interlacements, for any arbitrarily large particle density.

\begin{lemma}[Interlacements are not necessary]\label{lem-noshreds}
Suppose that the interaction potential $v$ satisfies Assumption (V). Then $\chi(\rho_1,\rho_2)\geq \chi(\rho_1+\rho_2,0)$ for any $\rho_1,\rho_2\in[0,\infty)$.
\end{lemma}

The proof of Lemma~\ref{lem-noshreds} is in Section~\ref{sec-noshreds}. In particular, in the variational formula $\overline\chi(\rho)$ describing the free energy according to Corollary~\ref{cor-freeenergyVP}, we  see now that
\begin{equation}\label{chibarloopsformula}
\overline\chi(\rho)=\inf\Big\{\h^{\ssup{\Lcal,\Scal}}(P)+\langle F_U,P\rangle\colon P\in \Mcal_1^{\ssup{\rm s}}(\Lcal), \langle P,{\mathfrak N}_U^{\ssup{\ell,\Lcal}}\rangle=\rho\Big\},\qquad \rho\in[0,\infty),
\end{equation}
where we conceive $P$ as an element of $\Mcal_1^{\ssup{\rm s}}(\Lcal\times \Scal)$ with empty interlacement part. Note that its entropy $\h^{\ssup{\Lcal,\Scal}}(P)$ has a shred-part that is {\it a priori} non-trivial, but Lemma~\ref{lem-loopshredsnegl} says that the expected particle number in shreds is negligible.

Hence, on the level of the value of the free energy, it is not possible to distinguish between optimal configuration with loops {\em and} interlacements or just with loops alone. In \cite{ACK10}, using totally different tools, a variant of \eqref{chibarloopsformula} was proved for a range of $\rho$'s that is bounded from above; instead of $\h^{\ssup\Lcal}$, the function $\h$ from \eqref{entropydensitydef} appeared.

\begin{remark}[Existence of minimizers?]\label{rem-minimizers} For the variational problem $\chi(\rho_1,\rho_2)$ in \eqref{chidefneu} there is a relatively simple argument that it has always minimizer(s) (see Lemma~\ref{lem-propertieschi}(1)), which is based on 
 the set $K$ and Assumption (V) and the compactness property of Theorem ~\ref{thm-specrelent}.

However, there is no argument like that for the formula for $\overline \chi(\rho)$ in \eqref{chibarloopsformula}. Indeed, if we use the topology on $\Mcal_1^{\ssup{\rm s}}(\Lcal\times\Scal)$ that is induced by the weak topologies defined by all the $\Pi_W^{\ssup\Lcal}$ and $\Pi_W^{\ssup \Scal}$ with $W\Subset \R^d$ (as we do in the proofs), then $\Mcal_1^{\ssup{\rm s}}(\Lcal)$ is not a closed subset of $\Mcal_1^{\ssup{\rm s}}(\Lcal\times\Scal)$, since it is possible to approach loop/interlacement configurations having non-trivial interlacement part with pure loop configurations, which is a kind of condensation effect. On the other hand, if one is based only on the topology defined by all the $\Pi_W^{\ssup\Lcal}$ with $W\Subset \R^d$, then the particle density in loops is not a continuous function, since this quantity needs the shreds of the loops (see the proof of Lemma~\ref{lem-lowboundinteract}).
\hfill$\Diamond$
\end{remark}

Let us therefore introduce  the critical threshold  between existence and non-existence of interlacements in minimizing configurations:
\begin{equation}\label{rhocdef}
\begin{aligned}
\rho_{\rm c}&=\sup \big\{\rho\in[0,\infty)\colon \exists \mbox{ minimiser }P\in\Mcal_1^{\ssup{\rm s}}(\Lcal\times\Scal)\mbox{ in \eqref{chifreeenergy} with }\Pi^{\ssup\Scal}(P)=0\big\}\\
&=\sup \big\{\rho\in[0,\infty)\colon \eqref{chibarloopsformula}\mbox{ has a minimizer}\big\}
\in[0,\infty].
\end{aligned}
\end{equation}

We conjecture the following:
\begin{itemize}
\item For $\rho\leq \rho_{\rm c}$, \eqref{chibarloopsformula} has a minimizer, and for $\rho>\rho_{\rm c}$ it has not. 

\item $\rho\mapsto \overline\chi(\rho)$ is real-analytic in $(0,\rho_{\rm c})$ and in $(\rho_{\rm c},\infty)$, but not in any larger interval.

\item $\rho_{\rm c}<\infty$ for $d\geq 3$, but $\rho_{\rm c}=\infty$ in $d\leq 2$.
\end{itemize}

Note that these statements are true in the free case (i.e., for $v=0$); see Remark~\ref{rem-freegas}. Proofs of these conjectures for some choices of $v$ (for example, those that satisfy our Assumption (V)) may be within reach using the new variational framework that was introduced in this work. In some technical proofs in this paper (see the proofs of Lemmas~\ref{lem-noshreds} and \ref{lem-ergappr} in Sections~\ref{sec-noshreds} and \ref{sec-ergappr}), we were able to develop strategies that make it possible to compare different loop/interlacement configurations with each other by rewiring shreds in large boxes outside the box. See the short descriptions of the two proofs at the beginnings of Sections~\ref{sec-noshreds} and \ref{sec-ergappr}. The variational description in terms of entropy plus energy in large boxes made it possible to show that every configuration can be approximated by loop-configurations, by just rewiring the interlacement-shreds outside a large box in such a way that only loops appear.

However, variants that could lead to a proof of finiteness of $\rho_{\rm c}$ in $d\geq 3$ seem to require much deeper ideas. One needs to show that, if $\rho$ is sufficiently large, any configuration that has only loops is not optimal.  One idea is to remove a non-vanishing loop part from the large box and replace it by an interlacement part of the same particle density and to show that the entropy goes down by more than the energy, but Lemma~\ref{lem-loopshredsnegl} says that this cannot be done just be rewiring all the shred ends; they concern too few particles. Another idea would be to execute some few manipulations on the loop configuration in the interior of the large box and a rewiring outside the box, such that interlacements appear and the energy is hardly changed, but the entropy drops down significantly, but it is not clear that such a strategy exists. Therefore more extensive manipulations have to be carried out, which we leave to future work.

 However, our conjectures extend much further: we think that $\rho_{\rm c}$ will turn out to be the threshold for the occurrence of BEC, i.e., for the validity of off-diagonal long-range order (ODLRO) for the symmetrised and normalized exponential of the Hamiltonian in \eqref{Hamiltonian}, but proving this seems to require much more work than we can currently indicate. Nevertheless, the framework introduced here is likely to serve as a platform for attacking proofs for such a statement.

\begin{remark}[The free gas]\label{rem-freegas}
Let us briefly discuss BEC in the {\em free Bose gas}, where no interaction is present, i.e., $v=0$. Here it is known since long that $\rho_{\rm c} =(4\pi\beta)^{-d/2}\zeta(d/2)$ with $\zeta$ the Riemann zeta function. In this case, we easily see that a minimiser of the formula in \eqref{chifreeenergy} is given as 
\begin{equation}\label{minimizerfree}
P_\rho=\begin{cases}{\tt Q}^{\ssup\rho}\otimes \delta_{\underline 0},&\mbox{if }\rho\leq \rho_{\rm c},\\
{\tt Q}\otimes {\tt R}^{\ssup{u_\rho,\beta}},&\mbox{if }\rho > \rho_{\rm c},
\end{cases}
\end{equation}
where $\underline 0$ is the empty interlacement point process, and ${\tt Q}^{\ssup\rho}$ is the marked  Poisson process as in Definition~\ref{def-PPP} with $q_k=\frac 1k(4\pi\beta k)^{-d/2}$ for $k\in\N$ replaced by $m_k^{\ssup\rho}=q_k\e^{\alpha_\rho k} $ with the unique suitable $\alpha_\rho\in(-\infty,0]$ such that $\sum_{\in\N}k m_k^{\ssup{\rho}}=\rho$. Furthermore, in $d\geq 3$ (only in this case the second case is not empty), ${\tt R}^{\ssup{u_\rho,\beta}}$ is the Brownian interlacement PPP introduced in Section~\ref{sec-InterlacePPP}, and the density parameter  $u_\rho$ is picked in such a way that the expected number of particles of ${\tt R}^{\ssup{u_\rho,\beta}}$ in $U$ is equal to $\rho-\rho_{\rm c}$.  One can write \eqref{minimizerfree} in one line by writing ${\tt Q}^{\ssup 0}={\tt Q}$ and noting that $\alpha_\rho=0$ and $m^{\ssup \rho}=q$ for $\rho\geq \rho_{\rm c }$, and putting  ${\tt R}^{\ssup{u_0,\beta}}=\delta_{\underline 0}$. Indeed, using the notation in \eqref{JWdef}, we see that
$$
\begin{aligned}
J_W(\Pi_W(P_\rho))&=\frac 1{|W|}H_{\Lcal_W\times\Scal_W}\big(\Pi_W^{\ssup\Lcal}({\tt Q}^{\ssup\rho})\otimes \Pi_W^{\ssup\Scal}({\tt R}^{\ssup{u_\rho,\beta}})\,\big|\,
\Pi_W^{\ssup\Lcal}({\tt Q}^{\ssup{\rho \wedge \rho_{\rm c}}})\otimes \Pi_W^{\ssup\Scal}({\tt R}^{\ssup{u_\rho,\beta}})\big)\\
&\to H(m^{\ssup\rho}\mid q)\qquad \mbox{ as }W\uparrow \R^d,
\end{aligned}
$$
and can deduce that
 \eqref{chifreeenergy} drastically reduces to the formula
\begin{equation}\label{entropylengthstat}
\overline\chi^{\ssup{v=0}}(\rho)=\inf\Big\{H(m|q)\colon m\in [0,\infty)^{\N}, \sum_{k\in\N}k m_k=\rho\Big\},
\end{equation}
where $H(m|q)=\sum_k(q_k-m_k+m_k\log \frac {m_k}{q_k})$ is the relative entropy of the sequence $m$ with respect to $q$ and to find the unique minimiser $m^{\ssup\rho}$ via a standard Lagrange minimization.  See \cite{BKM24} for a more direct way to show that $-\beta {\rm f}(\beta)$ is equal to the right-hand side of \eqref{entropylengthstat}.

In words, the limiting non-interacting Bose gas puts Brownian loops of length $k$ independently with density $m_k^{\ssup\rho}$ and adds, if their total particle mass has not reached $\rho_{\rm c}$, additionally and independently a Brownian interlacement process with the appropriate density. The formula in \eqref{entropylengthstat} appeared in several earlier works on the free Bose gas, e.g., in \cite{A07}. See also \cite{V23} for a proof of the emergence of the above minimiser (i.e., the one on point-process level) in the thermodynamic limit as a limit distribution.
\hfill$\Diamond$
\end{remark}

\section{Discussion}\label{sec-Comments}

\noindent In this section, we discuss some aspects and explain some background. In Section~\ref{sec-condensate} we show that our concept of $R$-crossing loops is equivalent to the usual idea of long loops. In Section \ref{sec-strategy} we concisely describe our strategy to prove Theorem~\ref{thm-freeenergy},  and in Section~\ref{sec-literature} we give an account on the literature on the thermodynamic limit of the Bose gas in connection with path-integral analysis.

\subsection{$R$-crossings versus long loops}\label{sec-condensate}

\noindent In the partition function, we want to distinguish between condensate particles and the others. In the loop soup representation of the Bose gas, the former should be the ones in \lq long loops\rq. The separation line between long and short requires an additional parameter, $L\in\N$: one can call the $L$-condensate all the particles in loops with lengths $>L$. This requires taking another limiting procedure, namely $L\to\infty$, after having taken $N\to\infty$. In our work, we decided to stick to another definition of the condensate, namely we call the $R$-condensate the totality of all the particles which form part of loops that are an $R${\it -crossing}, i.e., that have particles in more than one of the subboxes $z+ [-R,R]^d$ with $z\in 2R\Z^d$. Again, after taking the limit as $N\to\infty$, one needs to take the limit as $R\to\infty$.

In this section, we point out that these two notions are asymptotically the same under the measure $\widehat {\tt Q}^{\ssup{\L_N,{\rm bc}}}$ defined in \eqref{transformedmeasure}. That is, we will show that the number of particles in $R$-crossings and the number of particles in loops of lengths $>L$ are exponentially equivalent under $\widehat  \LPP^{\ssup{\L_N,{\rm bc}}}$ in the limit $N\to\infty$, followed by $R\to\infty$, respectively $L\to\infty$. For this purpose, let, for $\omega=\sum_{x\in\zeta}\delta_{(x,f_x)}\in \Lcal$,
\begin{equation}\label{numberRcrossing<L}
{\mathfrak N}_{\L,\leq L}^{\ssup {\ell,R}}(\omega)=\sum_{x\in\zeta \cap\L}\ell(f_x)\1\{f_x\mbox{ is $R$-crossing}\}\1\{\ell(f_x)\leq L\}
\end{equation} 
denote the number of particles in loops that are $R$-crossing and have length $\leq L$, and let 
\begin{equation}\label{numbernotRcrossing>L}
{\mathfrak N}_{\L,>L}^{\ssup {\ell,\neg R}}(\omega)=\sum_{x\in\zeta \cap\L}\ell(f_x)\1\{f_x\mbox{ is not $R$-crossing}\}\1\{\ell(f_x) > L\}
\end{equation}  
be the number of particles in loops that are not $R$-crossings, but longer than $L$. The following lemma shows that these two particle numbers are exponentially negligible if $L=L_R^{\ssup{\leq}}$ respectively $L=L_R^{\ssup  >}$ are chosen properly, tending to infinity as $R\to\infty$. This implies that our notion of condensate via $R$-crossings  is  equivalent to the usual notion via the length for the interacting Bose gas.

We remark that one half of this property (namely, the negligibility of $ {\mathfrak N}_{\L,>L}^{\ssup {\ell,\neg R}}(\omega)$) highly depends on the interaction. Indeed, the fact that long loops that stay within one of the boxes $z+W_R$ with $z\in 2R\Z^d$ are suppressed does not come from combinatorics of numbers of loops, but only from the fact that their self-interaction is large, using the superstability of the potential $v$. Therefore, as a pre-step, we first estimate the expectation of $\e^{-\Phi_{\L,\L}}$ from above for any Brownian loop that lies in some subbox $z+W_R$ with $z\in Z_{N,R}$ and has a large length. For this, we are going to use \eqref{LowBoundAss} in Assumption (V) in order to obtain a good lower bound for the self-interaction of this loop. Using this is not so immediate, since such a loop has only all its particles in $W_R$ (without loss of generality, we put $z=0$), but each leg may travel far away and evade the interaction. For future reference, we isolate the necessary part of the statement as follows.

\begin{lemma}[Upper bound for expected path interaction]\label{lem-pathinteraction}
Suppose that Assumption (V) holds. Then there is a $C\in(0,\infty)$ (only depending on $d$, $v$ or $\beta$) such that, for any box $W$ of diameter $R\in\N$ that is contained in the box $\L\subset \R^d$, and for any $k\in\N$ that is sufficiently large (depending on $R$),
\begin{equation}
\E\big[\e^{-\Phi_{\L,\L}}\big]\leq \e^{-Ck^{3/2}R/\sqrt{|W|}},\qquad  x_1,\dots,x_k\in W,
\end{equation}
where $\E=\bigotimes_{i=1}^k \P_{x_i}$ is with respect to $k$ independent Brownian motions with time interval $[0,\beta]$, starting from $x_1,\dots,x_k\in W$ and having any of the three boundary conditions in $ \L$, and $\Phi_{\L,\L}$ is the interaction between all the $k$ motions.
\end{lemma}

\begin{proof}
Without loss of generality, we put $W=W_R=[-R,R]^d$, as the boundary conditions will not be used anyway in our proof. We use $C$ as a generic constant that depends only on $d$, $v$ or $\beta$ and may change its value from appearance to appearance.

Even though all starting sites of the $k$ motions lie in $W$, we get a useful bound only for the pairs of those legs that stay, say, in the larger box $W_{2R}$, at least for some part of the time interval $[0,\beta]$. For a vector $f$ of $k$ legs $f_1,\dots,f_k\in\Ccal_1$, let $L_\gamma(f)=\#\{i\in[k] \colon f_i([0,\gamma])\subset W_{2R}\}$ denote the number of those legs who stay for the first $\gamma$ time units in $W_{2R}$. We will be using \eqref{LowBoundAss} only in the way that the self-intersection of $f$ satisfies the lower bound
\begin{equation}\label{lowboundinteraction}
\sum_{1\leq i<j\leq k} \int_0^\beta v\big(|f_i(s)-f_j(s)|\big)\,\d s\geq C\gamma\Big(\frac {m^2}{|W_R|}-m\Big),\qquad m,R\in\N\mbox{ if }m=L_\gamma(f).
\end{equation}
For a vector $B$ of $k$ independent motions $B_1,\dots,B_k$ with starting sites $x_1,\dots,x_k$, we lower bound $\e^{-\Phi_{\L,\L}}$ against the exponential  of the negative left-hand side of \eqref{lowboundinteraction} with $f_i=B_i$ and hence against the exponential  of the negative right-hand side, and we split the expectation according to the values of $L_\gamma(B)\in\{0,1,\dots,k\}$. Note that the probability of $\{L_\gamma(B)=m\}$ is not larger than $\e^{-(k-m)CR^2/\gamma}$, since $k-m$ of them travel a distance $\geq R$ during the time interval $[0,\gamma]$. Using \eqref{lowboundinteraction} on this event, we obtain  
\begin{equation}\label{calc1}
\begin{aligned}
\E\big[\e^{-\Phi_{\L,\L}}\big]\leq \sum_{m=0}^k \e^{-C\gamma (\frac {m^2}{|W_R|}-m)-(k-m)CR^2/\gamma}
&=\e^{-CkR^2/\gamma}\sum_{m=0}^k \e^{Cm(\gamma+R^2/\gamma)}\e^{-Cm^2\gamma/|W_R|}\\
&\leq \e^{-CkR^2/\gamma}(k+1) \e^{\frac C4\gamma |W_R|(1+R^2/\gamma^2)^2},
\end{aligned}
\end{equation}
where in the last step we estimated against the maximum over $m$, namely taking $m=\frac{|W_R|}2(1+R^2/\gamma^2)$. Now, assuming that $k$ is much larger than $|W_R|R^2$, we pick  $\gamma=R \sqrt{|W_R|/k}\in[0,\beta]$ and estimate, for sufficiently large $k$,  the second exponent as follows.
$$
\frac C4\gamma |W_R|(1+R^2/\gamma^2)^2
\leq \frac C3 \gamma\,|W_R|\frac{R^4}{\gamma^4}=\frac C3   k^{3/2}\frac R{\sqrt{|W_R|}}.
$$
Hence, adapting the value of $C$ and assuming that $k$ is large enough (only depending on $R$) we obtain that the right-hand side of \eqref{calc1} is not larger than $\exp\{-Ck^{3/2}R/\sqrt{|W_R|}\}$.
\end{proof}

Now we prove that the concepts of $R$-crossing loops and long loops are equivalent. Recall the definition of $\widehat\LPP^{\ssup{\L_N,\rm bc}}$ in \eqref{transformedmeasure}.

\begin{lemma}[Equivalence of condensate notions]\label{lem-condensatenotion} Suppose that $v$ satisfies Assumption (V).
Fix $\rho\in(0,\infty)$ and let $\L_N$ be the centered box in $\R^d$ with volume $|\L_N|/N\to 1/\rho$ as $N\to\infty$. Then there are diverging sequences $(L_R^{\ssup{>}})_{R\in\N}$ and $(L_R^{\ssup{\leq }})_{R\in\N}$ such that, for any $\eps\in(0,1)$,
\begin{eqnarray}
\limsup_{R\to\infty}\limsup_{N\to\infty}\frac 1{|\L_N|}\log \widehat\LPP^{\ssup{\L_N,\rm bc}}\big(\mathfrak N^{\ssup{\ell,R}}_{\L_N,\leq L_R^{\ssup{\leq}}}>\eps |\L_N|\big)&=& - \infty,\label{condensatenotionnegligibleleq}\\
\limsup_{R\to\infty}\limsup_{N\to\infty}\frac 1{|\L_N|}\log \widehat\LPP^{\ssup{\L_N,\rm bc}}\big(\mathfrak N^{\ssup{\ell,\neg R}}_{\L_N, > L_R^{\ssup{> }}}>\eps |\L_N|\big)&=& - \infty.\label{condensatenotionnegligible>}
\end{eqnarray}

\end{lemma}

\begin{proof} We put $W_R=[-R,R]^d$ and assume that $\L_N$ is equal to the union of $z+W_R$ over $z\in Z_{N,R}=\{z\in 2R\Z^d\colon z+W_R\subset\L_N\}$. Possibly we need to admit that $R=R_N$ depends on $N$ with $R_N\to R$, but we suppress this from notation.

We first prove \eqref{condensatenotionnegligibleleq}. This will not use the interaction, but only combinatorics and large deviations for the Poisson distribution. Note that $\mathfrak N^{\ssup{\ell,R}}_{\L_N,\leq L}=\sum_{k=1}^L k X_{k}$ where the $X_{k}$ are independent and Poisson-distributed with parameter
$$
\frac 1k \int_{\L_N}\d x\, \mu_{x,x}^{\ssup{k,\L_N,{\rm bc}}}(\tau_R\leq k);
$$
here, $\tau_R=\inf\{k\in\N\colon B_{k\beta}\notin z_{B_0}+W_R\}$, and $z_{B_0}\in 2R\Z^d$ is picked such that $B_0\in z_{B_0}+W_R$. The random variable $\sum_{k=1}^L X_k$ is Poisson-distributed with parameter $|\L_N| a_N(R)$, where $a_N(R)=\frac1{|\L_N|}\sum_ {k=1}^L \frac 1k\int_{\L_N}\d x\, \mu_{x,x}^{\ssup{k,\L_N,{\rm bc}}}(\tau_R\leq k)$.   Hence, we can estimate, as $N\to\infty$,
\begin{equation}\label{estiequiv1}
\begin{aligned}
\widehat\LPP^{\ssup{\L_N,\rm bc}}\big(\mathfrak N^{\ssup{\ell,R}}_{\L_N,\leq L}>\eps |\L_N|\big)
&\leq \widehat\LPP^{\ssup{\L_N,\rm bc}}\Big(\sum_{k=1}^L X_k>\frac\eps L |\L_N|\Big)\\
&\leq \frac1{\widehat Z_N^{\ssup{\rm bc}}(\L_N)}\exp\Big(-|\L_N|(1+o(1)) \big[ a_N(R)-\smfrac \eps L+\smfrac \eps L\log\frac\eps{L a_N(R)}\big]\Big),
\end{aligned}
\end{equation}
according to Cram\'er's theorem for the sum of $\asymp |\L_N|$ i.i.d.~Poisson-distributed random variables with parameter $a_N(R)$ (we dropped the interaction). Let us estimate $a_N(R)$. We are going to show that $\limsup_{N\to\infty}a_N(R)\leq O(1/R)$ as $R\to\infty$. We are going to use $C$ as a generic constant $\in(0,\infty)$, depending only on $d$ or on $\beta$ and possibly changing its value from appearance to appearance. We will use the estimate $\mu_{x,x}^{\ssup{k,\L_N,{\rm bc}}}(\tau_R\leq k)\leq C k^{-d/2}$ for any $x\in \L_N$ and any $k\in \N$, and we use that $\mu_{x,x}^{\ssup{k,\L_N,{\rm bc}}}(\tau_R\leq k)\leq C\e^{-Cm^2/k}$ for any $x\in\L_N$ that is further away from $\bigcup_{z\in Z_{N,R}} \partial (z+W_R)$ than $m$ (use that for the Brownian motion with time interval $[0,k\beta]$ it has probability $\leq C\e^{-Cm^2/k}$ to terminate at distance $\geq m$ away from its starting site). This gives, for any choice of $m_1,\dots,m_L \in (0,\infty)$,
$$
\begin{aligned}
a_N(R)&=\frac1{|\L_N|}\sum_{z\in Z_{N,R}}\sum_{k=1}^L \frac 1k \int_{z+W_R}\d x\, \mu_{x,x}^{\ssup{k,\L_N,{\rm bc}}}(\tau_R\leq k)\\
&\leq C \frac{1}{|W_R|} \sum_{k=1}^L \frac 1k \Big[\int_{W_{R-m_k}} \e^{-C m_k^2/k}\,\d x+k^{-d/2}\big(|W_R|-|W_{R-m_k}|\big)\Big]\\
&\leq C\sum_{k=1}^L \frac 1k \big[\e^{-C m_k^2/k}+k^{-d/2} \frac {m_k}R\Big].
\end{aligned}
$$
We pick now $m_k=\sqrt{\frac kC\log (k R)}$ and obtain the upper bound (recall that $d\geq 3$)
$$
a_N(R)\leq C\Big[\frac 1R+\sum_{k=1}^L k^{-(d+1)/2}\frac{\sqrt{\log(kR)}}R\Big]\leq \frac CR \sqrt{\log R}.
$$
We now put $L=L_R^{\ssup\leq}=\log\log R$. Substituting this in the large-deviation rate on the right-hand side of \eqref{estiequiv1}, we see that, as $R\to\infty$, 
$$
\begin{aligned}
 a_N(R)-\smfrac \eps L+\smfrac \eps L\log\frac\eps{L a_N(R)}
 &\geq o(1)+\frac \eps{\log\log R}\log \Big(\frac{\eps R}{ (\log\log R)\,C \sqrt{\log R}}\Big)
 \to\infty,
 \end{aligned}
$$
which implies \eqref{condensatenotionnegligibleleq}.

Now we prove \eqref{condensatenotionnegligible>}. Here we use the superstability of $v$, according to Assumption (V). Observe that $\mathfrak N^{\ssup{\ell,\neg R}}_{\L_N, > L}
=\sum_{k= L+1}^N k \widetilde X_{k},$
where  the $\widetilde X_{k}$ are independent Poisson-distributed random variables with parameter
$$
\frac 1k \int_{\L_N}\d x \mu_{x,x}^{\ssup{k,\L_N,{\rm bc}}}(\tau_R> k). 
$$
(Using the fact alone that this is $\leq  \frac 1k |\L_N| \e^{-Ck/R^2}$ for some $C\in(0,\infty)$ and all $k,R$ does not help in this proof, as one sees when trying it. Therefore, we employ the self-interaction of each such loop.)

 Lemma~\ref{lem-pathinteraction} implies that the expected exponential interaction of each of the $\widetilde X_k$ Brownian loops of length $k$ is not larger than $\e^{ -C_R k^{3/2}}$ for some $C_R$, if $k$ is large, only depending on $R$. Pick $L_R$ so large that $L_R ^{1/2} C_R\to\infty$ as $R\to\infty$.  Therefore, for any $a\in(0,\infty)$ and any $L\in\N$ that is larger than $L_R$, using the exponential Chebyshev inequality,
\begin{equation}\label{loopinteractionbound}
\begin{aligned}
\widehat Z_N^{\ssup{\rm bc}}(\L_N)\widehat\LPP^{\ssup{\L,\rm bc}}\big(\mathfrak N^{\ssup{\ell,\neg R}}_{\L_N, > L}>\eps |\L_N|\big)
&\leq \e^{-a \eps|\L_N|} \prod_{k>L}\E\big[\e^{a k \widetilde X_{k}-C_R k^{3/2}\widetilde X_{k}}\big] \\
&\leq \e^{-a \eps|\L_N|} \prod_{k>L}\E\big[\e^{k \widetilde X_{k}[a-C_R L^{1/2}]}\big] . 
 \end{aligned}
\end{equation}
We pick now  $a=a_R$ such that $1\ll a_R\ll L_R ^{1/2} C_R$, then the exponent in the expectation is negative for all large enough $R$, and the product is $\leq 1$. This implies \eqref{condensatenotionnegligible>}.
\end{proof}

\subsection{Our strategy}\label{sec-strategy}

\noindent We are going to describe our strategy to prove the LDP in Theorem~\ref{thm-freeenergy}. This strategy is very comprehensive and in particular provides a number of results that are also vital for the proof of Theorem~\ref{thm-specrelent}. Recall that we decompose the centred box $\L_N$ with volume $\sim N/\rho$ regularly into all the subboxes $W_z=z+W_R$ with $W_R=[-R,R]^d=W$ and $z\in Z_{N,R}\subset 2R\Z^d$. We have to find the large-$N$ exponential asymptotics of  the restricted partition function
$$
\widehat Z_{N,R,\delta}^{\ssup{\rm bc}}(\L_N,\rho_1,\rho_2)
=\LPP^{\ssup{\L_N,\rm bc}}\Big[{\rm e}^{-\Phi_ {\L_N,\L_N}}\1\{|\smfrac 1{|\L_N|} {\mathfrak N}^{\ssup{\ell,\neg R}}_{\L_N}-\rho_1|\leq\delta\}\1\{|\smfrac 1{|\L_N|} {\mathfrak N}^{\ssup{\ell,R}}_{\L_N}-\rho_2|\leq\delta\}\Big],
$$
followed by making $R\to\infty$ and then $\delta\downarrow0$.

For each $z\in Z_{N,R}$, we distinguish the  loops that are entirely contained in $W_z$, and the $W_z$-shreds of all the other loops (the $R$-crossing loops). Hence, we have a loop-shred configuration in the subboxes $W_z$.  We are going to neglect all the interaction between different subboxes. Therefore, with the shift-operator $\theta_z$ that satisfies $\theta_z(z)=0$, we can describe the entire system in terms of the crucial empirical measure
$$
\Xi_{N,R}(\omega)=\frac 1{\# Z_{N,R}}\sum_{z\in Z_{N,R}}
\delta_{(\theta_{z}(\Pi_{W_z}^{\ssup\Lcal}(\omega)),\theta_{z}(\Pi_{W_z}^{\ssup\Scal}(\omega)))},
$$
which is a probability measure on the set $\Lcal_W\times \Scal_W$ of loop-shred configurations in $W$; see the notation in Section~\ref{sec-twoentropies}. (We identify $\Lcal_W\times \Scal_W$ with the set $\Lcal_W\oplus \Scal_W$ of superpositions of loop configurations and shred configurations.) Neglecting the interaction between any two different ones of the subboxes, we will show in Section~\ref{sec-uppboundN} (upper bound) and \ref{sec-lowboundN} (lower bound), respectively, that the partition function can be roughly be written as
$$
\begin{aligned}
\widehat Z_{N,R,\delta}^{\ssup{\rm bc}}(\L_N,\rho_1,\rho_2)
&\approx \LPP^{\ssup{\L_N,\rm bc}}\Big[{\rm e}^{-|\L_N|\frac 1{|W_R|}\langle F_{W_R,W_R},\Xi_{N,R}\rangle}\\
&\qquad\1\{\smfrac 1{|W_R|}\langle \mathfrak N^{\ssup{\Lcal,\ell}}_{W_R},\Xi_{N,R}\rangle\approx \rho_1\}\1\{\smfrac 1{|W_R|}\langle \mathfrak N^{\ssup{\Scal,\ell}}_{W_R},\Xi_{N,R}\rangle\approx \rho_2\}\Big],
\end{aligned}
$$
where $F_{W,W}$ is the interaction between $W$-loops and $W$-shreds as defined in \eqref{FLL} -- \eqref{FSS} and \eqref{FWWdef}.

We need to understand the large-$N$ properties of the distribution of $\Xi_{N,R}$ under $\LPP^{\ssup{\L_N,\rm bc}}$. 
Observe that, by the virtue of the Poisson point process, the loop part is independent of the shred part, and the loop parts are essentially (modulo boundary effects from the boundary of the large box $\L_N$) i.i.d.~over $z$. However, the shred parts are by far not i.i.d. But luckily, we have the following conditional independence property: given the empirical measure of boundary configurations, 
$$
\partial \Xi^{\ssup\Scal}_{N,R}=\frac 1{\# Z_{N,R}}\sum_{z\in Z_{N,R}}
\delta_{(\theta_{z}(\partial\Pi_{W_z}^{\ssup\Scal}(\omega)))},
$$
these shred configurations $\Pi_{W_z}^{\ssup\Scal}(\omega)$, $z\in Z_{N,R}$, are independent, and the boundary configuration contains all information about their distribution that is necessary to establish an LDP for $\Xi_{N,R}$ with an explicit rate function (see below). In Section~\ref{sec-boundaryshreds} we derive the lower bound
$$
\liminf_{N\to\infty}\frac 1{|\L_N|} \log{\tt Q}^{\ssup{\L_N,\rm bc}}(\partial\Xi_{N,R}\approx\psi)\geq -\gamma_R(\psi),\qquad \psi\in\Tcal_W,
$$
where $\gamma_R(\psi)$ will turn out in Section~\ref{sec-lowboundlogpint}   to be small in the limit $R\to\infty$ if $\psi$ is picked carefully (for the upper bound we can essentially take $\gamma_R(\psi)=0$). Indeed, we will need certain ergodic properties of $\psi$ to derive the asymptotics and certain concentration properties of it in order to show that $\gamma_R(\psi)$ is small. The crucial point is that the question whether or not the event $\{ \partial\Xi_{N,R}\approx\psi\}$ has positive probability for all large $N$ or not is highly difficult to decide, since this requires knowledge about very long-range properties of the loop configuration; the long loops have global properties; and by far not any collection of loop-shred configurations in many subboxes can come from suitable global loop ensembles. This tricky problem is taken care of by a clever argument involving an ergodicity property and the ergodic theorem.

 Moreover, the contribution to the large-deviation probabilities coming from $\{\partial \Xi^{\ssup\Scal}_{N,R}=\psi\}$ can in Section~\ref{sec-LDPabstract} be explicitly be expressed using a large-deviation principle of the form
$$
\frac 1{|\L_N|} \log{\tt Q}(\Xi_{N,R}\approx \xi|\partial \Xi_{N,R}=\psi)\approx -J^{\ssup\psi}_{W_R}(\xi),\qquad N\to\infty.
$$ 
Here $J_W^{\ssup\psi}$ is the entropy term that appears in the definition of $\h^{\ssup{\Lcal,\Scal}}$ with $\partial \Pi_W^{\ssup\Scal}(\xi)$ replaced by $\psi$, 
$$
J^{\ssup\psi}_W(\xi)
=\frac 1{|W|} H_{\Lcal_W\times\Scal_W}\big(\xi\,\big|\,\Pi_W^{\ssup\Lcal}({\tt Q})\otimes [\psi\otimes {\tt K}_W]\big).
$$
For the upper bound in the proof of Theorem~\ref{thm-freeenergy}, we just optimize over all $\psi$ that have the required expected particle number  in the shreds, $\rho_2 |W_R|$, and for the lower bound we restrict to a $\psi$ of the form $\psi=\partial\Pi_{W_R}^{\ssup\Scal}(\xi)$ for a $\xi$, whose properties we still can refine. Using the above large-deviation principle and Varadhan's lemma, we will in Sections~\ref{sec-LDPupperbound} (upper bound) and \ref{sec-Nlowbound} (lower bound) arrive at 
$$
\begin{aligned}
\frac 1{|\L_N|} \log\widehat Z_{N,R,\delta}^{\ssup{\rm bc}}(\L_N,\rho_1,\rho_2)
&\approx - \inf\Big\{J_{W_R}(\xi)+\frac 1{|W_R|}\langle F_{W_R,W_R},\xi\rangle
\colon \xi\in\Mcal_1(\Lcal_{W_R}\times \Scal_{W_R}), \\
&\quad\smfrac 1{|W_R|}\langle \mathfrak N^{\ssup{\Lcal,\ell}}_{W_R},\xi\rangle\approx \rho_1, \smfrac 1{|W_R|}\langle \mathfrak N^{\ssup{\Lcal,\ell}}_{W_R},\xi\rangle\approx \rho_2\Big\}.
\end{aligned}
$$

On the way to this, a number of technical problems are to be handled, as is usual in large-deviation theory. Crucial ones of these obstacles come from topological technicalities like missing compactness and missing upper semi-continuity in the proof of the upper bound and missing openness and lower  semi-continuity in the proof of the lower bound; therefore both proofs have to separately be approached and require many separate steps. Another general issue that renders the proof lengthy and technical is the fact that Brownian motions are infinitely flexible with positive probability and need to be controlled and enclosed in bounded areas in various senses, in order to gain a clear control on the various influences on the interaction energy and particle numbers, in particular in view of the decomposition into the boxes $z+W$.

The continuities of the energy functional and of the two particle number functionals is not clear {\it a priori} and needs to be achieved by restrictions to suitable events, both for the proof of the upper and the lower bound. In Section~\ref{sec-Compactness}, using Assumption (V), we introduce two events that we need to insert such that their complement is negligible and on these sets, the two particle numbers are continuous. In Sections~\ref{sec-lowboundN} and \ref{sec-lowboundlogpint}, we restrict to an event on which only loop and shred configurations with several particular properties appear, namely only legs of bounded spread, only configuration with locally bounded particle numbers everywhere and only shreds with bounded stretches. This requires the usage of several cut-off parameters, which will be sent to $\infty$ at the very end of the proof. In Section~\ref{sec-projentr} (after a preparation in Section~\ref{sec-entropy}, where we have a closer look at the various terms in the entropy $J_{W_R}$) we control the most involved part, the entropy $J_{W_R}$, when replacing an arbitrary $\xi$ by some measure that satisfies all these cut-off properties, and additionally the ergodicity property mentioned above. This is used in Section~\ref{sec-ergappr} to control the value of the variational problem above when doing this change.

In order to carry out the limit as $R\to\infty$ in the above variational formula, the most important point is of course the existence of the limit of $J_{W_R}$ (Theorem~\ref{thm-specrelent}), which we prove in Section~\ref{sec-limitingentropy}, after some preparations in the other sections.
In the finish of the proof of the upper bound in Theorem~\ref{thm-freeenergy}, we need to embed the above  variational formula on $R$-subboxes into a limiting variational formula on gobal stationary loop-interlacement configurations in the entire $\R^d$. Here we first use some tool from large-deviation theory (see Proposition~\ref{lem-uppboundJ_W}(4)) to extend a minimizing $\xi$ to arbitrarily large boxes without noticeably enlarging the entropy. Afterwards, we use an extension theorem in the spirit of Carath\'eodory's theorem to extend it to a global configuration on the entire $\R^d$; see Proposition~\ref{Prop-construction_of_P}. In Sections~\ref{sec-finishuppbound} (upper bound) and \ref{sec-finishlowerbound} (lower bound) all these arguments are put together and the proof of Theorem~\ref{thm-freeenergy} is finished.

The proof of  Theorem~\ref{thm-specrelent} in Section~\ref{sec-limitingentropy} and several crucial arguments in other sections  relies on various properties and estimations for the entropy functional $J_{W_R}$, which we derive in Section~\ref{sec-PropertiesJW}. We remark that we carry out the necessary estimates with the help of large-deviation arguments, see Proposition~\ref{lem-uppboundJ_W}. In particular, we derive a crucial estimate of the form 
$$
J_{W_R}(\Pi_{W_R}(P))\leq J_{W_{mR}}(\Pi_{W_{mR}}(P))+o_{mR}(1), \qquad m\to\infty.
$$
This replaces the ubiquitous super-additivity argument that is usually employed for establishing the existence of a specific relative entropy (actually we did not succeed in proving sub-additivity and believe that this property does not hold).

\subsection{Literature survey}\label{sec-literature}

The study of quantum gases, in particular the Bose gas and its statistical mechanics and condensation, is a huge fascinating subject that provides many challenging questions and involves a lot of mathematical ansatzes and toolboxes, see \cite{PS01,PS03} for extensive summaries.

Unlike many investigations on the interacting Bose gas in the mathematical-physics literature, the present paper follows the ansatz of a {\em path-integral analysis}. This goes back to the vague suggestion by Feynman in \cite{F53} that Brownian bridges and the statistics of their lengths might be taken as an order parameter. To our best knowledge, there is no result in the mathematical physics literature on deeper properties of the free energy of the interacting Bose gas in the thermodynamic limit, hence we will restrict here to reviewing only ansatzes with path-integral analysis.

Technical foundations where laid by Ginibre \cite{G70}. There are not many rigorous works yet that follow this advice, starting with phenomenological
discussions in \cite{U06} and discussions of the relation between long loops and condensate in \cite{S93,S02} for some simplified models, including the free Bose gas. More recently, \cite{FKSS20} conceived the rescaled interaction of the Brownian loops in $d = 4$ as a regularisation as the intersection local time as a possible ansatz for deriving $\Phi^4$-theories.

The Brownian-bridge ansatz is greatly strengthened by combining it with a spatial point-process ansatz and a Poisson point process description and additionally with a large deviation approach that is sometimes called a level-3 approach. An early version of such a Poisson point process, the Brownian loop soup was introduced by Lawler and Werner \cite{LW04}, however with a different purpose (not for the Bose gas), namely for making conformal invariance rigorously provable in models from two-dimensional statistical mechanics. The idea of an LDP theory for marked point processes like the Bose gas was implicitly present in papers by Georgii/Zessin in the 1990s  \cite{GZ93,G94}, and it was combined with the desire to find criteria for the equivalence of ensembles for interacting marked Poisson point processes. Further works (see, e.g, \cite{R09, NPZ13, Z22, BDM24}) concentrated on the construction of such processes in the entire space $\R^d$ and to discuss them from the view point of classical Gibbs point theory as was laid down in the famous monograph \cite{G88}. 

A recent achievement in the tradition of constructing Gibbs measures of marked point processes is \cite{BDM24}, where the interacting Bose gas is achieved as such a process with the legs as marks: {\it nota bene}, the legs as members of $\Ccal_1$, not the loops. The formation of an ensemble of interacting loops is part of the interaction term, which renders this interaction term highly complicated and discontinuous. In \cite{BDM24}, an infinite-space Gibbs measure is constructed for any value of $\rho$ and $\beta$; however it is rather difficult to decide whether infinite long loops appear in this measure. 

For the problem of analysing the limiting free energy of the Bose gas using Brownian bridges, the crucial Poisson description of Proposition~\ref{lem-rewrite} was introduced in \cite{ACK10}. It was used a number of times in the last years for answering various partial questions. In the non-interacting case, an LDP   with identification of the phase transition in terms of the rate function was derived in \cite{A07}. A proof for the phase transition of Bose--Einstein condensation (that is, for off-diagonal long range order (ODLRO), the generally acknowledged criterion for BEC) in terms of Brownian bridges was given  recently in \cite{KVZ23}, and for a (still non-interacting) mean-field version in \cite{BKV24}. In \cite{BKM24} interactions only within the same loop were admitted in the gas and a related kind of condensation phase transition was proved in connection with the famous self-avoiding walk problem.

As it comes now to interacting Bose gases, the most immediate predecessor paper is \cite{ACK10}, which is based on the PPP-description of Proposition~\ref{lem-rewrite} and applies large-deviation arguments based on the ansatz from  \cite{GZ93} for the {\em empirical stationary process}, $\mathcal R_\L(\omega_{\rm P})=\frac 1{|\L|}\int_{\L}\delta_{\theta_x(\omega_{\rm P})}\,\d x$, a kind of continuous variant of the main tool of the present paper, $\Xi_{N,R}$. A weaker variant of Corollary~\ref{cor-freeenergyVP} was derived for all small enough particle densities $\rho$, but without any attempt to describe long loops, since the topological difficulties could not be resolved there and any kind of concept was missing for long loops.
Another important predecessor of the current work is \cite{CJK23}, which develops and demonstrates the strategy that we use in the proof of Theorem~\ref{thm-freeenergy} for a simpler model, where the loops are replaced by boxes with random and unbounded sizes. 

A different ansatz was followed in \cite{QT23}, which gives the -- to the best of my knowledge, first -- proof of the existence of a loop of length $\asymp N$ in an interacting spatial ensemble of random loops of the type of an interacting Bose gas. More precisely, it is shown that the expectation of the length of the loop that contains the origin is not smaller than a constant times the volume of the box. This formidable result has been achieved by transferring the well-known reflection-positive ansatz to this setting; indeed, it was found a way to employ some clever  correlation inequalities about the behaviour of the system under reflections. However, there are a number of caveats: the class of interactions is restricted by some unnatural conditions that come from some steps in the Fourier setting, so far the technique could be done only in the time- and space-discrete setting, and only for periodic boundary condition. Removing each of these three shortcomings seems to require serious new input.

A Poisson point process of interlacements was introduced by Sznitman \cite{Sz10} for simple random walks and in  \cite{Sz13} for Brownian motions, however without making any connection to Bose gases. A version of this process can be seen as lurking in the background of our description of the Bose gas in Theorem~\ref{thm-freeenergy} for $d\geq 3$, see Section~\ref{sec-InterlacePPP}. The idea of using such a process for the description of the condensate of the Bose gas was first introduced to the literature in \cite{AFY21}, however without any relation to any limits or interactions or free energies. There, the version for Gaussian random walks was constructed (alternatively to what we do in Section~\ref{sec-InterlacePPP}), and some properties were examined. Also the independent superposition of the corresponding loop soup and the Gaussian interlacement PPP was introduced and studied. In \cite{V23}, it is shown that the simple-random-walk loop soup converges, in the thermodynamic limit with free boundary condition, towards the independent superposition of the same process (with critical density) and the simple-random walk interlacement PPP, see Remark~\ref{rem-freegas}. In an interacting Bose gas (however, with interactions only involving loop lengths, no spatial details) in the spatially continuous setting, in \cite{DV24} convergence of the gas towards the Brownian loop soup, superposed with the Brownian interlacement PPP, was proved.

\begin{remark}[LDP for empirical processes]\label{Rem-LDPempproc}
As we explained, an important ingredient of the proof in the present paper is a kind of large-deviation principle for the (discrete version of the) empirical stationary field of the Poisson loop process. This is remarkably close to the main object of \cite{Sz23}, but has also decisive differences. Let us comment on that here.

Indeed, \cite{Sz23} studies  the stationary empirical field $\mathcal L_\L(\varpi_{\rm B})=\frac 1{|\L|}\int_{\L}\d x\,\delta_{\theta_x(\varpi_{\rm B})}$ of  the Brownian interlacement PPP introduced in \cite{Sz13}, more precisely, the probability of large deviations of certain linear test functions $\langle F, \mathcal L_\L(\varpi_{\rm B})\rangle$. This process $\Lcal_\L$ is the interlacement counterpart of $\mathcal R_\L$ above and it is the continuous version of the shred part of our empirical process $\Xi_{N,R}$ (which can be analogously defined for interlacements in place of $R$-crossing loops). Under certain conditions on the test function $F$, \cite{Sz23} proves the upper bound of the large-$N$ exponential rate of the probabilities of upward deviations of $\langle F, \mathcal L_\L(\varpi_{\rm B})\rangle$ in terms of an explicit variational formula that shows some reminiscences of the famous Donsker--Varadhan LDP rate function. (See \cite{CN23} for the corresponding lower bound.) In principle, this kind of result seems to be close to what is also needed for the description of the Bose gas (if we neglect the fact that Sznitman's process is based on the interlacement process rather than the long-loop part of the Brownian loop soup). But however, the precise assumptions made in \cite{Sz23} make it impossible to apply the results of \cite{Sz23} to the interacting Bose gas.

The surprise is that the LDP in \cite{Sz23} is on capacity scale $|\L_N|^{1-\frac 2d}$ rather than on volume scale $|\L_N|$, the scale on which we are working in this paper. This seems to be due to the fact that the description of the most likely behaviour of the interlacements to meet the event $\{\langle F, \mathcal L_\L(\varpi_{\rm B})\rangle>C \}$ for some large $C$ comes from an upscaling of an event in a box of a diverging radius (i.e., from some mesoscopic scale), while we have to describe all the details of the interaction on the microscopic scale. Equivalently, one can say that the description of the Bose gas needs a finer scale of details and is therefore adequately described on volume scale. 
\hfill$\Diamond$
\end{remark}

\section{Scanning the configuration with finite windows}\label{sec-Scanning}

\noindent In this section we introduce a framework for describing the loop configuration in the box $\L_N$ in terms of an empirical mixture of the local loop and shred configurations that we see in all the shifts of the window $[-R,R]^d$  that scans $\L_N$ in regular steps of size $2R$. 

In Sections~\ref{sec-PPPs} and \ref{sec-chopping} we introduce the loop space consisting of marked point processes, and the shred space consisting of point measures on shreds. In Section~\ref{sec-Caratheodory} we combine the two spaces and give a Carath\'eodory-like result on the existence of stationary probability measures on the space of loops and interlacements with given marginals on projections on boxes. Our main tool, the mentioned empirical measure $\Xi_{N,R}$ of the loop and shred configurations in the subboxes, is introduced in Section~\ref{sec-empmeasconf}. In Sections~\ref{sec-uppboundN} and \ref{sec-lowboundN} we rewrite the constrained partition function in terms of $\Xi_{N,R}$ and  derive preliminary upper and lower bounds towards a proof of Theorem~\ref{thm-freeenergy}.

Let us recall and introduce some basic notation. We denote by $\Ccal_k$ the set of all continuous functions $[0,\beta k]\to\R^d$ and write $\Ccal=\bigcup_{k\in\N}\Ccal_k$. Each $\Ccal_k$ is equipped with the topology induced by the supremum norm. On $\Ccal$ we use the topology within each $\Ccal_k$, i.e., a sequence $(f_n)_{n\in\N}$ in $\Ccal$ converges towards some $f\in\Ccal$ if there is some $k$ such that $f \in \Ccal_k$ and $f_n \in \Ccal_k$ for all $n$ large enough, and $f_n\to f$ in the supremum norm on $\Ccal_k$. Hence, we do not assign a distance to function pairs with different lengths. We introduce the loop space 
\begin{equation}\label{loopspace}
\Ccal^{\ssup\circlearrowleft}=\bigcup_{k\in\N }\Ccal^{\ssup\circlearrowleft}_k,\qquad\mbox{where}\qquad \Ccal^{\ssup\circlearrowleft}_k=\{f\in\Ccal_k\colon f(0)=f(k\beta)\}.
\end{equation}
Since loops are cyclic, two elements of $\Ccalcirc$ are considered equal if they are time-shifts by a time in $\beta \Z$ of each other; hence we conceive $\Ccalcirc_k$ actually as a quotient space. By default, if nothing else is said, we write an $f\in \Ccalcirc_k$ as $f\colon[0,k\beta]\to\R^d$. We  call the pieces $f|_{[j\beta,(j+1)\beta]}$ with $j\in\N_0$ the {\em legs} of $f$ and the sites $f(0),f(\beta),f(2\beta), \dots, f((k-1)\beta)$ the {\em particles} of $f$. The number of particles of $f$ is denoted by $\ell(f)=k$. 
Furthermore, we denote by $\Ccal_\infty$ the set of all continuous functions $f\colon \R\to\R^d$ satisfying $\lim_{t\to-\infty}|f(t)|=\lim_{t\to\infty}|f(t)|=\infty$. Such functions have only finitely many particles in $W$ for any compact set $W\subset \R^d$.

\subsection{Point processes with loops as  marks}\label{sec-PPPs}

\noindent We introduce notation for point processes in $\R^d$ with marks in $\Ccalcirc$ and provide some remarks on their topologies.

We consider the space $\Loops=\Mcal_{\N_0}(\R^d\times \Ccalcirc)$ of simple point processes on $\R^d\times \Ccalcirc$, more precisely, the set of marked point processes in $\R^d$ with marks in $\Ccalcirc$. That is, an element $\omega\in\Loops$ can be represented in terms of a point process $\zeta$  on $\R^d$ as
\begin{equation}\label{omegaloops}
	\omega = \sum_{x\in\zeta}\delta_{(x,f_{x})},\qquad\mbox{where } f_{x}\in\Ccalcirc\mbox{ with }f_{x}(0)=f_{x}(\beta \ell(f_{x})) = x,
\end{equation}
and such that $\omega(A\times \Ccal^{\ssup\circlearrowleft})<+\infty$, for any bounded measurable $A\subset\R^d$. We see $\Lcal$ as the set of superpositions $\sum_{k\in \N}\sum_{x\in\zeta_k}\delta_{(x,f_x)}$, where $\zeta_k=\{x\in \zeta\colon \ell(f_x)=k\}$. We equip $\Loops$ with the usual evaluation $\sigma$-algebra, i.e., the one that is generated by all the maps $\omega\mapsto \Pi_W(\omega)=\sum_{x\in \zeta\cap  W}\delta_{(x,f_x)}$ with $W\Subset\R^d$. Let $\Mcal_1(\Loops)$ denote the set of probability measures on $\Loops$, and $\Mcal_1^{\ssup{\rm s}}(\Loops)$ the subset of shift-invariant probability measures on $\Loops$, i.e., the set of those $P\in\Mcal_1(\Lcal)$ such that $P\circ \theta_x^{-1}=P$ for any $x\in \R^d$; by $\theta_x$ we denote the shift operator by $x\in\R^d$ satisfying $\theta_x(x)=0$, acting on vectors, subsets of $\R^d$, functions in $\R^d$, measures on $\R^d$ and on point processes in $\R^d$ with and without marks.

For any $W\Subset\R^d$, consider the projection on all the loops that start in $W$ and have all their particles in $W$:
\begin{equation}\label{PiWLdef}
\Pi_{W}^{\ssup{\mathcal L}}(\omega) = \sum_{x\in\zeta\cap W}\delta_{(x,f_{x})}\1\{f_x(k\beta )\in W\,\forall k\in [\ell(f)]\}.
\end{equation}
We denote by $\Loops_W$ the image of $\Loops$ under this projection, i.e., the set $\Lcal_W=\Mcal_{\N_0}(W\times\Ccal_W^{\ssup\circlearrowleft})$ of all point measures in $\R^d$ with marks that are in the set 
\begin{equation}\label{LoopWdef}
\Ccal_W^{\ssup\circlearrowleft}=\bigcup_{k\in\N}\big\{ f\in \Ccal_k^{\ssup\circlearrowleft}\colon f(i\beta)\in W\ \forall i\in [k]\big\}
\end{equation}
of loops whose particles are entirely contained in $W$.
 Observe the difference to the usual projection $\Pi_W$, which restricts to all loops that start in $W$ but may have particles outside $W$. Furthermore, observe that the loops in $\Ccal_W^{\ssup\circlearrowleft}$ may also leave $W$ on the way from particle to particle. It is clear that $\Pi_W^{\ssup\Lcal}$ is indeed a projection, that is, $\Pi_{\widetilde W}^{\ssup\Lcal}\circ \Pi_{W}^{\ssup\Lcal}=\Pi_W^{\ssup\Lcal}\circ \Pi_{\widetilde W}^{\ssup\Lcal}=\Pi_W^{\ssup\Lcal}$ if $W\subset \widetilde W\Subset \R^d$. A significant difference to $\Pi_W$ is that, for two disjoint  sets $W,\widetilde  W\Subset \R^d$, the image $\Lcal_{W\cup\widetilde  W}$ consists of {\em three} parts: $\Lcal_W$, $\Lcal_{\widetilde W}$ and the set of simple point measures with loops that have particles both in $W$ and in $\widetilde W$, but nowhere elsek.

The topology on $\Lcal_W$ that we use is the $w^\#$-topology, which is induced by test integrals against continuous and bounded functions $\Mcal_{\N_0}(W\times\Ccalcirc_W)\to\R$ that vanish outside some ball. Here we define a ball of radius $r$ around some $(x,f_x)\in W\times \Ccalcirc_k$ as the set of $(y,g)\in W\times \Ccalcirc_k$ such that $|x-y|+\|f_x-g\|_\infty<r$ and $g(0)=x$, and a ball around some configuration $\sum_{x\in\zeta}\delta_{(x,f_x)}$ is the union of these balls around the points of the configuration.
This turns $\Lcal_W$ itself into a Polish space, according to \cite[Prop.~9.1.IV]{DVJ08}. According to \cite[Exercise 11.1.7]{DVJ08}, convergence of point measures $\sum_{x\in \zeta_n\cap W}\delta_{(x,f^{\ssup n }_x)}$ towards a point measure $\sum_{x\in \zeta\cap W}\delta_{(x,f_x)}$ as $n\to\infty$ is equivalent to having $\#\xi_n=\#\xi$ for all large $n$ and pointwise convergence of the points $(x,f^{\ssup n})$ of $\xi_n$ in $W\times\Ccalcirc$ towards the ones of $\xi$ as $n\to\infty$ in some assignment. We denote this topology on $\Lcal_W$ by $\Gcal_W^{\ssup\Lcal}$ and the induced topology on $\Lcal$ by $\Fcal_W^{\ssup\Lcal}=(\Pi_W^{\ssup\Lcal})^{-1}(\Gcal_W^{\ssup\Lcal})$.

On the set $\Mcal_1(\Lcal_W)$ of probability measures on $\Lcal_W$, we will be using the weak topology, i.e., the one that is induced by test integral against bounded and continuous functions $\Lcal_W\to\R$. According to \cite[Theorem 11.1.VII]{DVJ08}, for probability measures $\xi,\xi_1,\xi_2,\dots $ on $\Lcal_W$, the weak convergence $\xi_n\Longrightarrow \xi$ is true if and only if the finite-dimensional distributions of $\xi_n$ converge towards the one of $ \xi$, i.e., if for any $k\in\N$ and any bounded measurable sets $A_1,\dots,A_k\subset \Lcal_W$, we have  that $(\xi_n(A_1),\dots,\xi_n(A_k))$ converges weakly towards $(\xi(A_1),\dots,\xi(A_k))$.

We write $f(P)$ for the image measure $P\circ f^{-1}$. 
On $\Mcal_1(\Lcal)$, we will consider the weak topology induced by all the projections $\Pi_W^{\ssup\Lcal}$ with $W\Subset\R^d$, i.e., $P_n\Longrightarrow P$ if and only if $\Pi_W^{\ssup\Lcal}(P_n)\Longrightarrow \Pi_W^{\ssup\Lcal}(P)$ for any $W\Subset\R^d$.

Let us also mention that the weak topology on the set $\Mcal( \Xcal)$ of all finite measures on a Polish space $\Xcal$ is induced by the Prokhorov metric
\begin{equation}\label{Prokhorovmetric}
\d(\mu,\nu)=\inf\big\{\eps\geq 0\colon  \mu(F)\leq \nu(F^\eps),\nu(F)\leq  \mu(F^\eps)\mbox{ for all closed }F\subset \Xcal\big\},\qquad \nu,\mu\in\Mcal(\Xcal),
\end{equation}
where $F^\eps=\{y\in\Xcal\colon \d(y,F)<\eps\}$ denotes the open $\eps$-neighbourhood of $F$ \cite[Sect.~A.2.5]{DVJ03}. This turns $\Mcal(\Xcal)$ into a Polish space as well \cite[A2.5.III]{DVJ03}.

\subsection{Shredding paths}\label{sec-chopping}

This section is the heart of our ansatz that we introduce in this paper. It is a chopping procedure on path level, which cuts, for each compact set $W\subset\R^d$, any loop that has not all its particles in $W$ or any infinitely long path into the pieces that run within $W$ and collects all these pieces in terms of a point measure. In this way, we introduce a kind of projection operator $\Pi_W^{\ssup\Scal}$ that has many properties similarly to $\Pi_W$ or to $\Pi_W^{\ssup\Lcal}$ and can be used to induce a natural topological structure on the set $\Scal$ of point measures on $\Ccal_\infty$.

 For any $f\in\Ccalcirc$, we introduce the stopping times
\begin{eqnarray}
\tau_W(f)&=&\inf\{k\in\N_0\colon f(k\beta)\in W\}\in\N_0\cup\{\infty\},\\
\sigma_W(f)&=&\inf\{k\in\N\colon k>\tau_W(f), f(k\beta)\in W^{\rm c}\}.
\end{eqnarray}
An analogous definition is used for $f \in\Ccal_\infty$. For $f\in\Ccalcirc_k$ or $f \in\Ccal_\infty$, we write the times of the successive entries of $(f(i\beta))_{i\in\{0,1,\dots,k\}}$, respectively $(f(i\beta))_{i\in \Z}$ to $W$ and exits from $W$ as
$$
-\infty<\tau_W^{\ssup1}(f)<\sigma_W^{\ssup1}(f)<\tau_W^{\ssup2}(f)<\sigma_W^{\ssup2}(f)<...<\infty;
$$
a more formal definition of these stopping times may be given in terms of appropriate time-shifts of $\tau_W(f)$ and $\sigma_W(f)$. Then each restriction $f|_{[(\tau_W^{\ssup i}-1)\beta ,\sigma_W^{\ssup i}\beta]}$, shifted in time to obtain a function $f^{\ssup i}\colon[0,(\sigma_W^{\ssup i}-\tau_W^{\ssup i})\beta]\to\R^d$, is called a $W${\em -shred} of $f$; it is a member of the {\em $W$-shred space}
\begin{equation}\label{SWdef}
\Ccal_W=\bigcup_{k\in\N}\big\{g\in\Ccal_{k}\colon   g(0),g(\beta),g(2\beta),\dots,g((k-1)\beta)\in W,g(k\beta)\in W^{\rm c}\big\}.
\end{equation}
(We ignore at this point the fact that part of the characterisation of $g\in\Ccal_W$ is that $g(-\beta)\in W^{\rm c}$.)
We denote by $\Mcal_{\N_0}(\Xcal)$ the set of all simple point measures in a measurable set $\Xcal$. Then we define 
\begin{equation}\label{eq:shred_op}
\Scal_W=\Mcal_{\N_0}(\Ccal_W)\qquad\mbox {and}\qquad 	\Pi_W^{\ssup\Scal}\colon \Ccalcirc\cup\Ccal_\infty\to \Scal_W,\qquad \Pi_W^{\ssup\Scal}(f)=\sum_i \delta_{f^{\ssup i}}.	
\end{equation}

All topological assertions that we spelled out in Section~\ref{sec-PPPs} for $ \Lcal_W$ hold also for $\Scal_W$ analogously. In particular, a sequence $(\Gamma_n)_{n\in\N}$ of point measures in $\Shreds_W$ converges to some $\Gamma\in \Shreds_W$ if and only if, for all sufficiently large $n$, $|\Gamma_n|=|\Gamma|$, and the $W$-shreds of $\Gamma_n$ converge to the ones of $\Gamma$ in some assignment. If $ \Gcal^{\ssup\Scal}_W$ denotes the corresponding  Borel $\sigma$-field, then $(\Shreds_W,  \Gcal^{\ssup\Scal}_W)$ is a measurable space.

Introduce $\ell_W(g)=\sum_{i}\1_W(g((i-1)\beta))$, the number of particles  of $g\in\Ccalcirc\cup\Ccal_\infty$ in $W$. 
 Observe that $g\in\Ccal_ \infty$ is characterised by the requirement that $\ell_W(g)<\infty$ for any compact $W\subset\R^d$. We also introduce the space 
\begin{equation}\label{Scaldef}
\Scal=\Big\{\varpi=\sum_{g\in\Gamma}\delta_g\in \Mcal_{\N_0}(\Ccal_\infty)\colon \sum_{g\in\Gamma}\ell_W(g)<\infty\mbox{ for all compact }W\subset\R^d\Big\}
\end{equation} 
of simple point measures of doubly-infinite continuous functions, such that the cloud has only finitely many particles in any compact subset of $\R^d$.

We also use the symbol $\Pi_{W}^{\ssup{\mathcal S}}$ for the shredding operation on point process level: for  $W$, $ \widetilde W$ compact such that $W\subset\widetilde W$, define the point process of $W$-shreds derived from $\omega$ by
\begin{equation} \label{eq:projW}
\Pi_{W}^{\ssup{\mathcal S}}\colon\Scal\cup\Scal_{\widetilde W}\to\Scal_W,\qquad \Pi_{W}^{\ssup{\mathcal S}}(\omega)=\sum_{f\in\omega} \sum_{i}\delta_{f^{\ssup i}},\qquad\mbox{where }f^{\ssup i} \mbox{ are the $W$-shreds of }f.
\end{equation} 
We sometimes write $\Pi^{\ssup\Scal}_{\widetilde W\to W}\colon \Scal_{\widetilde W}\to\Scal_W$ for this map in order to indicate the domain.
It is easy to see that the shredding operators are consistent in the sense that
\begin{equation}\label{Piconsistent}
W\subset\widetilde W\qquad\Longrightarrow\qquad\Pi_{W}^{\ssup\Shreds}\circ \Pi_{\widetilde W}^{\ssup\Shreds}=\Pi_W^{\ssup\Shreds}.
\end{equation}

On the set $\Shreds$, we  introduce a measurable structure by taking $\Fcal^{\ssup\Scal}$ as the $\sigma$-field such that all the shredding operators $\Pi_W^{\ssup\Shreds}$ are measurable, i.e., 
$$
\Fcal^{\ssup\Scal}=\sigma\big( \Fcal^{\ssup\Scal}_W\colon W\subset\R^d\mbox{ compact}\big),\qquad\mbox{
where}\quad  \Fcal^{\ssup\Scal}_W=(\Pi_W^{\ssup\Shreds})^{-1}( \Gcal^{\ssup\Scal}_W).
$$
On the set $\Mcal_1(\Shreds)$ we introduce the weak topology induced by all the $\Pi_W^{\ssup\Shreds}$ with $W\Subset \R^d$: A sequence $(P_n)_{n\in\N}$ of probability measures on $\Shreds$ converges towards some $P\in \Mcal_1(\Shreds)$ if and only if $\Pi_W^{\ssup\Shreds}(P_n)$ (which is by definition equal to the image measure $P_n\circ(\Pi_W^{\ssup\Shreds})^{-1}$) converges weakly to $\Pi_W^{\ssup\Shreds}(P)$ for any $W\Subset \R^d$. Then $(\Shreds,\Fcal)$ is also a Polish space with metric
\begin{equation}\label{metriconS}
\d(\Gamma_1,\Gamma_2)=\sum_{R\in\N}2^{-R}\frac{\d_{W_R}\big(\Pi_{W_R}^{\ssup\Shreds}(\Gamma_1),\Pi_{W_R}^{\ssup\Shreds}(\Gamma_2)\big)}{1+\d_{W_R}\big(\Pi_{W_R}^{\ssup\Shreds}(\Gamma_1),\Pi_{W_R}^{\ssup\Shreds}(\Gamma_2)\big)},\qquad W_R=[-R,R]^d,
\end{equation}
where we lifted $\Pi_{W_R}^{\ssup\Shreds}$ to $\Shreds$ by putting $\Pi_{W_R}^{\ssup\Shreds}(\Gamma)=\sum_{g\in\Gamma}\Pi_{W_R}^{\ssup\Shreds}(g)$. (Here ${\rm d}_{W}$ denotes some metric on $\Scal_W$ that induces the topology on $\Scal_W$.)

Now we introduce a kind of boundary-shredding operation as follows. Define the operator that maps each $W$-shred $g$ satisfying $g(0),g(\beta),g(2\beta),\dots,g((k-1)\beta)\in W$ and $g(k\beta)\in W^{\rm c}$ on the triple of its initial and terminal condition and its length:
\begin{equation}\label{boundaryshredoper}
\partial\Pi_W^{\ssup\Shreds}\colon \Ccal_W\to W\times\N\times W^{\rm c},\qquad  \partial\Pi_W^{\ssup\Shreds}(g)=(g(0),l,g(k\beta)).
\end{equation}
We write
$$
\begin{aligned}
\Tcal_W&=\Mcal_{\N_0}(W\times \N \times W^{\rm c})\\
&=\Big\{\mu=\sum_{i\in I}\delta_{(x_i,l_i,y_i)}\colon I\mbox{ finite},  (x_i,l_i,y_i)\in W\times \N\times W^{\rm c}\,\mbox{pairwise distinct }\forall i\in I\Big\}
\end{aligned}
$$
for the set of all finite simple counting measures on $W\times \N \times W^{\rm c}$. We use the same symbol for the projection operation on shred configurations:
\begin{equation}\label{partialPidef}
\partial\Pi_W^{\ssup\Shreds}\colon \Shreds_W\to \Tcal_W,\qquad \partial\Pi_W^{\ssup\Shreds}\Big(\sum_g\delta_g\Big)= \sum_g\delta_{\partial\Pi_W^{\ssup\Shreds}(g)},
\end{equation}
and we extend the domain of this operator to $\Scal_{\widetilde W}\cup\Scal$ by putting $\partial\Pi_W^{\ssup\Scal}(\varpi)=\partial\Pi_W^{\ssup\Scal}(\Pi_W^{\ssup\Scal}(\varpi))$ for $\varpi\in \Scal_{\widetilde W}\cup\Scal$ for compact sets $W\subset\widetilde W\subset\R^d$. Again, all topological assertions of Section~\ref{sec-PPPs} about $\Lcal_W$ hold analogously for $ \Tcal_W$.

Note that $\mu=\partial \Pi_W^{\ssup\Scal}(\varpi)$ contains all information about the number of particles of the shred configuration $\varpi=\sum_{g\in\Gamma}\delta_g$ in $W$, when we define 
\begin{equation}\label{particlenumber}
{\mathfrak N}_{\partial W}^{\ssup{\ell,\Shreds}}( \mu)=\sum_{i\in I}l_i\qquad \mbox{for }\mu=\sum_{i\in I}\delta_{(x_i,l_i,y_i)};
\end{equation}
then ${\mathfrak N}_W^{\ssup{\ell,\Shreds}}(\varpi)={\mathfrak N}_{\partial W}^{\ssup{\ell,\Shreds}}( \mu)$ and $\langle \xi, {\mathfrak N}_W^{\ssup{\ell,\Shreds}}\rangle = \langle \psi,{\mathfrak N}_{\partial W}^{\ssup{\ell,\Shreds}}\rangle$ for $ \xi\in\Mcal_1(\Scal_W)$ and $\psi=\partial\Pi_W ^{\ssup\Scal}(\xi)\in\Mcal_1(\Tcal_W)$.

Let us also mention the following tightness criterion \cite[Proposition 11.1.VI]{DVJ08}: A subset $K$ of $\Mcal_1(\Tcal_W)$ is tight if for any ball $B \subset W\times \N\times W^{\rm c}$ and for any $\eps>0$, there is $M>0$  such that $\psi(\{\mu \in\Tcal_W\colon \mu(\overline B) >M\})<\eps$  for any $\psi\in K$. In particular, the set $\{\psi\in\Mcal_1(\Tcal_W)\colon \langle  \psi, {\mathfrak N}_{\partial W}^{\ssup{\ell,\Shreds}}\rangle \leq \rho_2\}$ is relative compact for any $\rho_2\in(0,\infty)$, since $\mu(\overline B)\leq |I|\leq \sum_{i\in I}l_i={\mathfrak N}_{\partial W}^{\ssup{\ell,\Shreds}}(\mu) $. (However, closedness of this set does not follow, since the map $\psi\mapsto \langle  \psi, {\mathfrak N}_{\partial W}^{\ssup{\ell,\Shreds}}\rangle$ is {\it a priori} not continuous.)

\subsection{An extension result for loop-interlacement configurations}\label{sec-Caratheodory}

We need to introduce also some specifics of the space of the joint loop-shred configurations in boxes, respectively loop-interlacement configurations in $\R^d$. In Proposition~\ref{Prop-construction_of_P} we prove an existence theorem in the spirit of Kolmogorov's theorem in this setting.

We conceive the loop-interlacement configuration space 
\begin{equation}\label{LcaltimesScal}
\Lcal\times \Scal= \Mcal_{\N_0}\big((\R^d\times \Ccal^{\ssup\circlearrowleft})\cup \Ccal_\infty\big)
\end{equation}
as the superposition point measure space with loops and interlacements.  We conceive the shredding operator $\Pi_W^{\ssup\Scal}$ defined in \eqref{eq:shred_op} as an operator $\Pi_W^{\ssup\Scal}\colon \Lcal\times\Scal\to\Scal_W$ and the loop-restriction operator $\Pi_W^{\ssup\Lcal}$ defined in \eqref{PiWLdef} as an operator $\Pi_W^{\ssup\Lcal}\colon \Lcal\times\Scal\to\Lcal_W$. Note that $\Pi_W^{\ssup\Lcal}$ can produce something non-trivial only from $\Lcal$ (not from shreds or from interlacements), while $\Pi_W^{\ssup\Scal}$ can cut both loops and shreds into shreds. Then the operator 
\begin{equation}\label{PiWdef}
\Pi_W\colon\Lcal\times\Scal\to\Lcal_W\times\Scal_W,
\qquad\mbox{defined by }\Pi_W(\omega,\varpi)=\Pi_W^{\ssup\Lcal}(\omega)+\Pi_W^{\ssup\Scal}(\varpi),
\end{equation}
is the superposition operator of $\Pi_W^{\ssup\Scal}$  and $\Pi_W^{\ssup\Lcal}$, when we conceive $\Lcal_W\times\Scal_W$ as $\Mcal_{\N_0}((W\times\Ccal_W^{\ssup\circlearrowleft})\cup \Ccal_W)$.  On $\Scal_W$ and $\Lcal_W$, respectively, we have the natural sigma fields $\Gcal_W^{\ssup\Scal}$ and $\Gcal_W^{\ssup\Lcal}$, respectively, that are generated by the vague topology. On $\Lcal\times \Scal$ we consider the sigma field $\Fcal$ that is generated by $\Fcal_W=\sigma((\Pi_W^{\ssup\Scal})^{-1}(\Gcal_W^{\ssup\Scal})\cup (\Pi_W^{\ssup\Lcal})^{-1}(\Gcal_W^{\ssup\Lcal}))$ with $W=W_R=[-R,R]^d$ for all $R\in\N$, i.e., $\Fcal=\sigma(\bigcup_{R\in\N}\Fcal_{W_R})$. The definition of the  projection operator $\Pi_{\widetilde W\to W}\colon \Lcal_{\widetilde W}\times \Scal_{\widetilde W}\to\Lcal_W \times\Scal_W$ for $W\subset \widetilde W\Subset \R^d$ is clear.

The following proposition will be enormously helpful in the proof of properties of the specific relative entropy density in the proof of Theorem~\ref{thm-specrelent} and in the proof of the upper bound for the free energy of the interacting Bose gas (Theorem~\ref{thm-freeenergy}). It is an extension theorem for consistent families of probability measures on loop/shred configurations in boxes. It is in the spirit of the famous extension theorem by Kolmogorov, but this theorem seems not applicable, since we are not working with product spaces, and the objects on which the extension is a measure (interlacements) are global and have a long range. Hence, we need to redo its proof (actually, the proof of Ionescu-Tulcea's theorem, the \lq discrete\rq\ variant) with the help of Carath\'eodory's theorem.

\begin{prop}[Loop-interlacements point process extension]\label{Prop-construction_of_P}
Let $(\xi^{\ssup{W_R}})_{R\in\N}$ be a family of probability measures $\xi^{\ssup{W_R}}\in\Mcal_1(\Lcal_{W_R}\times\Shreds_{W_R})$, such that $\Pi_{W_{R+1}\to W_R}(\xi^{\ssup{W_{R+1}}}) = \xi^{\ssup{W_R}}$ for any $R\in\N$, i.e., $(\xi^{\ssup{W_R}})_{R\in\N}$ is consistent. Then there exists a unique $P\in\Mcal_1(\Lcal\times \Shreds)$ such that $\Pi_{W_R}(P) = \xi^{\ssup{W_R}}$ for all $R\in\N$.
\end{prop}

\begin{proof} We want to apply Carath\'eodory's theorem to the algebra $\Fcal^{\ssup 0}=\bigcup_{R\in\N}\Fcal_{W_R}$ and to the map $P_0\colon \Fcal^{\ssup 0}\to[0,1]$ that is defined by
\begin{equation}\label{P0def}
P_0(A)=\xi^{\ssup{W_R}}(B),\qquad A\in\Pi_{W_R}^{-1}(\Gcal_{W_R}),
\end{equation}
where $B\in \Gcal_{W_R}$ is the unique set such that $A=\Pi_{W_R}^{-1}(B)$. The fact that $P_0\colon \Fcal^{\ssup 0}\to[0,1]$ is well-defined follows from the consistency in the same way as in the proof of Ionescu-Tulcea's theorem. It is also not difficult to show that $P_0$ is a content, that is, finitely additive and satisfies $P_0(\emptyset)=0$. It also obviously satisfies $\Pi_W(P_0)=\xi^{\ssup W }$ for any $W\Subset\R^d$.

We would like to extend $P_0$ from $\Fcal^{\ssup 0}$ to a measure $P$ on $\Fcal$, which finishes the proof. According to the uniqueness part of Carath\'eodory's theorem, this $P$ is unique, and according to its existence part, for showing the existence it is sufficient to show that $P_0$ is continuous in the empty set. Hence, assume that $(A_m)_{m\in\N}$ is a sequence in $\Fcal^{\ssup 0}$ that decreases, as $m\to\infty$, to $A=\bigcap_{m\in\N} A_m\in\Fcal^{\ssup 0}$. We assume that $\liminf_{m\to\infty}P_0(A_m)>0$ and need to show that $A$ is not empty.

Without loss of generality, we can assume that $A_m=\Pi_{W_m}^{-1}(B_m)$ for some $B_m\in \Gcal_{W_m}$.

We can easily extend the consistency relation $\Pi_{W_{R+1}\to W_R}(\xi^{\ssup{W_{R+1}}})=\xi^{\ssup{W_R}}$ for all $R\in\N$ to the  consistency relation $\Pi_{ W_{n+m}\to W_{n}}(\xi^{\ssup{W_{n+m}}}) = \xi^{\ssup{W_{n}}}$ for any $n\in\N $ and $m\in\N_0$. For $W\subset \widetilde W\Subset \R^d$, we can write $\Lcal_{\widetilde W}\times \Scal_{\widetilde W}
=(\Lcal_W\times \Scal_W)\times (\Lcal\times \Scal)_{\widetilde W\setminus W}$, where 
$$
\big(\Lcal\times \Scal\big)_{\widetilde W\setminus W}=\Mcal_{\N_0}\big(\big[(\widetilde W\times \Ccal_{\widetilde W}^{\ssup\circlearrowleft})\cup \Ccal_{\widetilde W}\big]\setminus\big[(W\times  \Ccal_{W}^{\ssup\circlearrowleft})\cup \Ccal_{W}\big]\big).
$$
Since $\Mcal_1(\Lcal_W\times \Scal_W)$ is a Polish space for any $W\Subset\R^d$, there are Markov kernels $K_{n,m}$ from $\Lcal_{W_n}\times \Scal_{W_n}$ into $(\Lcal\times \Scal)_{W_{n+m}\setminus W_n}$ such that  $\xi^{\ssup{W_{n+m}}}=\xi^{\ssup{W_n}}\otimes K_{n,m}$  (with $K_{n,0}=\1_{B_n}$). 

We are going to construct now $(\omega^{\ssup m},\varpi^{\ssup m})\in B_m$ for $m\in\N$  such that $( \omega^{\ssup {m-1}},\varpi^{\ssup{m-1}})=\Pi_{W_m\to W_{m-1}}(\omega^{\ssup m},\varpi^{\ssup m})$. For this, define $h_{n,m}\colon \Lcal_{W_n}\times \Scal_{W_n}\to [0,1]$ by putting
$$
h_{n,m}(\omega_{W_n},\varpi_{W_n})= K_{n,m}\big((\omega_{W_n},\varpi_{W_n}),(B_{n+m})^{\ssup 1}_{(\omega_{W_n},\varpi_{W_n})}\big),
$$
where
$$
\begin{aligned}
(B_{n+m})^{\ssup 1}_{(\omega_{W_n},\varpi_{W_n})}
&=\big\{(\omega^{\ssup n}_{W_{n+m}},\varpi^{\ssup n}_{W_{n+m}})\in (\Lcal\times \Scal)_{W_{n+m}\setminus W_n}\colon\\
&\qquad  (\omega_{W_n}+\omega^{\ssup n}_{W_{n+m}},\varpi_{W_n }+\varpi^{\ssup n}_{W_{n+m}})\in B_{n+m}\big\}
\end{aligned}
$$
is the $W_n$-cut of $B_{n+m}$. It is relatively easy to show that, for any $n\in\N$ and any $(\omega_{W_n},\varpi_{W_n})$, the sequence $(h_{n,m}(\omega_{W_n},\varpi_{W_n}))_{m\in\N_0}$ is non-increasing, since  $B_m\subset \Pi_{W_{m}\to W_{m-1}}^{-1}(B_{m-1})$ and hence $(B_{n+m+1})^{\ssup 1}_{(\omega_{W_n},\varpi_{W_n})}\subset (B_{n+m})^{\ssup 1}_{(\omega_{W_n},\varpi_{W_n})}$. 

We construct now recursively a sequence of pairs $(\omega_{W_n},\varpi_{W_n})\in B_n\subset  \Lcal_{W_n}\times \Scal_{W_n}$ such that $\lim_{m\to\infty} h_{n,m}(\omega_{W_n},\varpi_{W_n})>0$. Because of monotonicity in $m$ and $h_{n,0}=\1_{B_n}$, this finishes the proof of the existence of  a sequence $(\omega^{\ssup m},\varpi^{\ssup m})\in B_m$ for $m\in\N$  such that $( \omega^{\ssup {m-1}},\varpi^{\ssup{m-1}})=\Pi_{W_m\to W_{m-1}}(\omega^{\ssup m},\varpi^{\ssup m})$.

First, we consider $n=1$. For $m\in \N_0$, we have
$$
\int h_{1,m}(\omega_{W_1},\varpi_{W_1})\,\xi^{\ssup{{W_1}}}(\d (\omega_{W_1},\varpi_{W_1}))
=\xi^{\ssup{{W_m}}}(B_m)=P_0(A_m).
$$
By our assumption that $P_0(A_m)$ has a non-zero limit, there is a $(\omega_{W_1},\varpi_{W_1})\in  B_1$. Now we make the induction step from $n$ to $n+1$. A calculation shows that, for any $m\in\N_0$, 
$$
\begin{aligned}
\int_{(\Lcal\times \Scal)_{W_{n+m}\setminus W_n}} &h_{n+1,m}\big((\omega_{W_n},\varpi_{W_n}),(\omega^{\ssup n}_{W_{n+m}},\varpi^{\ssup n}_{W_{n+m}})\big)\,
K_{n,1}\big((\omega_{W_n},\varpi_{W_n}), \d (\omega^{\ssup n}_{W_{n+m}},\varpi^{\ssup n}_{W_{n+m}})\big)\\
&=h_{n, m+1}(\omega_{W_n},\varpi_{W_n}).
\end{aligned}
$$
Hence, the sequence $(h_{n+1,m}((\omega_{W_n},\varpi_{W_n}),(\omega^{\ssup n}_{W_{n+m}},\varpi^{\ssup n}_{W_{n+m}}))_{m\in\N}$ is a non-increasing sequence of functions in the second argument that does not vanish as $m\to\infty$. Therefore, there is $(\omega_{W_{n+1}},\varpi_{W_{n+1}})\in B_{n+1}$ whose projection on $W_n$ is equal to $(\omega_{W_n},\varpi_{W_n})$.

Now that we have a sequence of $(\omega^{\ssup n},\varpi^{\ssup n})\in B_n\subset\Lcal_{W_n}\times\Scal_{W_n}$ such that $(\omega^{\ssup {n-1}},\varpi^{\ssup{n-1}})=\Pi_{W_n\to W_{n-1}}(\omega^{\ssup n },\varpi^{\ssup n})$ for any $n\in\N$; hence we can construct an element $(\omega,\varpi)\in\Lcal\times \Scal$ such that $\Pi_{W_n}(\omega,\varpi)=(\omega^{\ssup n},\varpi^{\ssup  n})$ for any $n\in\N$. Therefore, $(\omega,\varpi)\in A_n$ for any $n$, i.e., it lies in $A$, which ends the proof. First we construct $\widetilde\omega$ by simply extending all the $\omega^{\ssup n}$ from point measures in  $W_n\times \Ccal_{W_n}^{\ssup\circlearrowleft}$ to point measures in  $\R^d\times \Ccal^{\ssup\circlearrowleft}$. The construction of $\varpi$ and of $\omega-\widetilde \omega$ goes as follows. Start with some of the points, $f_1\in\Ccal_{W_1 }$ of $\varpi^{\ssup 1}$ with some time interval $[T_1^- \beta ,T_1^+\beta]$ and integers $T_1^-, T_1^+$. In $\varpi^{\ssup 2}$, there must be a point $f_2\in (W_2\times \Ccal_{W_2}^{\ssup\circlearrowleft})\cup\Ccal_{W_2}$ that extends $f_1$ (possibly also involving some other points of $\varpi^{\ssup 1}$)  to a continuous function with time interval $[T_2^-\beta,T_2^+\beta]$. If $f_2\in W_2\times \Ccal_{W_2}^{\ssup\circlearrowleft}$, then we add $\delta_{(f_2(0),f_2)}$ to $\widetilde \omega$. Otherwise we keep $f_2$ as a $W_2$-shred and proceed. Recall that $f_1$ satisfies $f_1(T_1^{-}\beta)\in W_1^{\rm c}$, since it is a $W_1$-shred. Inductively, one can go on and find, for any $m\in\N$, either a $W_m$-loop $f_m\in W_m\times \Ccal_{W_m}^{\ssup\circlearrowleft}$, which we add to $\widetilde \omega$, or a $W_m$-shred $f_m\in \Ccal_{W_m}$ in $\varpi^{\ssup m}$ that extends $f_{m-1}$ and has a time interval $[T_m^-\beta,T_m^+\beta]$ with integers $T_m^-, T_m^+$, and $T_m^-\leq T_{m-1}^-$ and $T_m^+\geq T_{m-1}^+$ and $f_{m}((T_m^--1)\beta)\in W_m^{\rm c}$ and $f_{m}(T_m^+\beta)\in W_m^{\rm c}$. Since $\lim_{m\to\infty}T_m^-=-\infty$ and $\lim_{m\to\infty}T_m^+=\infty$ and $\lim_{m\to\infty} f_m((T_m^--1)\beta)=\infty$ and $\lim_{m\to\infty} f_m(T_m^+\beta)=\infty$, we have constructed a continuous function $f\in\Ccal_\infty$, the extension of all the $f_m$ that are $W_m$-shreds. This we take as a point of $\varpi$. Repeat the same procedure with all the other points of $\varpi^{\ssup 1}$ (if there are some left), otherwise with $\varpi^{\ssup 2}$ and so on. The sum of all the delta measures on $W_m$ loops that we obtained during the entire procedure is defined as $\omega-\widetilde\omega$, and this defines $\omega$.
\end{proof}

\subsection{The empirical measure of the loop/shred configurations}
\label{sec-empmeasconf}

\noindent We now introduce the main tool in our proof of Theorem~\ref{thm-freeenergy} and explain the first steps of our strategy. The main object is the empirical measure of all the loop/shred configurations in all the $R$-subboxes within the macrobox $\L_N$. We are in the situation of Section~\ref{sec-Purpose}: we fix $R\in\N$ and write $W=[-R,R]^d$ and $W_z=z+W$ for $z\in 2R\Z^d$ and fix $\rho\in(0,\infty)$ and a centered box $\L_N$ with volume $ \sim N/\rho$ as $N\to\infty$. We decompose $\L_N=\bigcup_{z\in Z_{N,R}}W_z$ for some suitable  $Z_{N,R}\subset 2R\Z^d$ (possibly $R=R_N$ depends on $N$ and converges to $R$; we suppress this in the notation). Note that $\# Z_{N,R}\sim |\L_N|/|W_R|\sim\frac N{\rho (2R)^d}$ as $N\to\infty$.

We introduce now our main tool in the proof of Theorem~\ref{thm-freeenergy}: for any $\omega\in \Lcal$, we define
\begin{equation}\label{Xidef}
\Xi_{N,R}^{\ssup\omega}=\frac 1{\# Z_{N,R}}\sum_{z\in Z_{N,R}}
\delta_{(\theta_{z}(\Pi_{W_z}^{\ssup\Lcal}(\omega)),\theta_{z}(\Pi_{W_z}^{\ssup\Scal}(\omega))}\in \Mcal_1(\Lcal_W\times\Scal_W),
\end{equation}
where $\theta_z$ is the shift-operator such that $\theta_z(W_z)=W$. (Recall the operators $\Pi_W^{\ssup \Lcal}$ and $\Pi_W^{\ssup \Scal}$ from  \eqref{PiWLdef} and \eqref{eq:projW}.) This is the empirical measure of all the loop/shred configurations of $\omega$ in the subboxes, shifted to the origin.

Let us briefly recall our plan from Section~\ref{sec-strategy} to prove Theorem~\ref{thm-freeenergy}: for $\omega=\omega_{\rm P}$, the marked reference PPP of Definition~\ref{def-PPP}, 
\begin{enumerate}
 \item rewrite the partition function $\widehat Z_N$ in terms of an integral over $\Xi_{N,R}^{\ssup{\omega_{\rm P}}}$,
 
 \item find a large deviation principle for $\Xi_{N,R}^{\ssup{{\omega_{\rm P}}}}$ as $N\to\infty$,
 
\item use Varadhan's lemma to express the large-$N$ exponential rate of the partition function as a variational formula on the space $\Mcal_1(\Lcal_W\times \Scal_W)$ with $W=[-R,R]^d$,

\item make $R\to\infty$ in that formula to arrive at $\chi$ defined in \eqref{chidefneu}. 

\end{enumerate}

For making the first step, we will present upper and lower bounds in Sections~\ref{sec-uppboundN} and \ref{sec-lowboundN}, respectively. The second step will be done in Section~\ref{sec-LDPabstract}, and the third in Sections~\ref{sec-LDPupperbound} and \ref{sec-Nlowbound}. In the fourth step in Section~\ref{sec-finishproof}, we are employing a new relative entropy density that we introduce in Section~\ref{sec-limitingentropy}. For various reasons, the proof of the upper and the lower bound will have substantial differences. 

It will be convenient to distinguish the the loop part and the shred part of $\Xi_{N,R}$, 
\begin{eqnarray}\label{XiLdef}
\Xi_{N,R}^{\ssup{\omega_{\rm P},\Lcal}}&=&\Pi_W^{\ssup\Lcal}(\Xi_{N,R}^{\ssup{\omega_{\rm P}}})=\frac 1{\# Z_{N,R}}\sum_{z\in Z_{N,R}}
\delta_{\theta_{z}(\Pi_{W_z}^{\ssup\Lcal}(\omega_{\rm P}))}\in \Mcal_1(\Lcal_W),\\
\Xi_{N,R}^{\ssup{\omega_{\rm P},\Scal}}&=&\Pi_W^{\ssup\Scal}(\Xi_{N,R}^{\ssup{\omega_{\rm P}}})=\frac 1{\# Z_{N,R}}\sum_{z\in Z_{N,R}}
\delta_{\theta_{z}(\Pi_{W_z}^{\ssup\Scal}(\omega_{\rm P}))}\in \Mcal_1(\Scal_W).
\end{eqnarray}
These two are independent under the reference PPP, since the first one comes from the part of the PPP that has no $R$-crossings, and the second from the $R$-crossing part of the PPP, which are independent. The second observation is that the loop part is an empirical measure of independent objects,  whose distributions are basically identical, modulo the boundary conditions, whose influence will be shown to vanish in the limit as $N\to\infty$. 

Now, as for the shred part, it is obvious that the subconfigurations in the $R$-subboxes are very far from independent, which presents a major obstacle in the analysis of its distribution. However, the third observation is that the configurations of the shreds in the subboxes are independent if the statistics of their starting sites, lengths and terminal sites are fixed, i.e., conditional on the boundary-shred empirical measure
\begin{equation}\label{partialXiSdef}
 \partial \Xi_{N,R}^{\ssup{\omega_{\rm P},\Scal}}=\partial\Pi_W^{\ssup\Scal}(\Xi_{N,R}^{\ssup{\omega_{\rm P}}})
 =\frac1{\# Z_{N,R}}\sum_{z\in Z_{N,R}}\delta_{\theta_{z}(\partial\Pi_{W_z}^{\ssup\Scal}(\omega_{\rm P}))}\in\Mcal_1(\Tcal_W).
\end{equation}

\subsection{Upper bound}\label{sec-uppboundN}

We are going to derive an upper bound for the partition function in terms of a expectation with respect to the empirical measure $\Xi_{N,R}$. We need to estimate it in such a way that it is decomposed into a product of the contributions over all the disjoint subboxes. For this, we restrict the interaction to the sum of (slightly modified) interactions within the subboxes. In the proof of the upper bound, since the interaction potential $v$ is nonnegative, we are in the lucky situation that we can simply drop all the interaction between different subboxes and that we can drop also the probabilistic cost of the behaviour of the boundary configurations of the shreds. 

Assume the situation and notation described in Section~\ref{sec-empmeasconf}.  As one sees from \eqref{freeenergyident}, we want to find an upper bound for 
\begin{equation}\label{hatZ}
\widehat Z_{N,R,\delta}^{\ssup{\rm bc}}(\L_N,\rho_1,\rho_2)
=\LPP^{\ssup{\L_N,\rm bc}}\big[{\rm e}^{-\Phi_ {\L_N,\L_N}}\1\{|\smfrac 1{|\L_N|} {\mathfrak N}^{\ssup{\ell,\neg R}}_{\L_N}-\rho_1|\leq\delta\}\1\{|\smfrac 1{|\L_N|} {\mathfrak N}^{\ssup{\ell,R}}_{\L_N}-\rho_2|\leq\delta\}\big].
\end{equation}
The main result of this section is the following.

\begin{lemma}[Upper bound]\label{lem-UpperBound}
For any $N,R\in\N$, $\delta\in(0,1)$, $\rho_1,\rho_2\in[0,\infty)$,
\begin{equation}\label{hatZ2}
\begin{aligned}
\widehat Z_{N,R,\delta}^{\ssup{\rm bc}}(\L_N,\rho_1,\rho_2)
&\leq \LPP^{\ssup{\L_N,\rm bc}}\Big[\exp\Big\{- N_R |W|\big\langle \Xi_{N,R}^{\ssup{\omega_{\rm P}}},\smfrac 1{|W|}F_{W,W}\big\rangle\Big\}\\
&\quad\times \1\Big\{\langle\Xi_{N,R}^{\ssup{\omega_{\rm P}}},\smfrac 1{|W|}{\mathfrak N}^{\ssup{\ell,\Lcal}}_W\rangle  \in\overline\Bcal_{2\delta}(\rho_1)\Big\}\1\Big\{\langle \partial \Xi_{N,R}^{\ssup{\omega_{\rm P},\Scal}},\smfrac 1{|W|}{\mathfrak N}^{\ssup{\ell,\Scal}}_{\partial W}\rangle \in\overline\Bcal_{2\delta}(\rho_2)\Big\}
\Big],
\end{aligned}
\end{equation}
where $N_R=\# Z_{N,R}$, and $F_{W,W}$ is defined in \eqref{FWWdef}, and the particle-counting operators are defined in  \eqref{def_particlenumbers}.
 \end{lemma}
 
 \begin{proof}We are going to express $\widehat Z_{N,R,\delta}^{\ssup{\rm bc}}(\L_N,\rho_1,\rho_2) $ in terms of the subconfigurations
\begin{equation}\label{subconfigurations}
\omega_{W_z}=\Pi_{W_z}^{\ssup\Lcal}(\omega)=\sum_{x\in W_z\colon f_x(i\beta)\in W_z\forall i}\delta_{(x,f_x)}
\qquad\mbox{and}\qquad \varpi_{W_z}=\Pi_{W_z}^{\ssup\Scal}(\omega)=\sum_{g\in\Gamma_z}\delta_{g}\qquad \mbox{with }z\in Z_{N,R},
\end{equation}
where we wrote $\Gamma_z=\Gamma_z(\omega)$ for the set of $W_z$-shreds of $\omega$. Recall that we write $g\colon[0,\ell(g)\beta]\to\R^d$ for a $W$-shred $g$ with $ g(0),g(\beta),g(2\beta),\dots,g((\ell(g)-1)\beta)\in W$ and $g(\ell(g)\beta)\in W^{\rm c}$, and its $i$-th leg is $g_i=g|_{[(i-1)\beta,i\beta]}$ for $i\in[\ell(g)]$. Then we estimate, for  all sufficiently large $N$,
\begin{equation}\label{particlenumberesti}
\begin{aligned}
\1\{|\smfrac 1{|\L_N|} &{\mathfrak N}^{\ssup{\ell,\neg R}}_{\L_N}-\rho_1|\leq\delta\}\1\{|\smfrac 1{|\L_N|} {\mathfrak N}^{\ssup{\ell,R}}_{\L_N}-\rho_2|\leq\delta\}\\
&\leq \1\Big\{\frac 1{N_R}\sum_{z\in Z_{N,R}} \frac 1{|W|}{\mathfrak N}^{\ssup{\ell,\Lcal}}_{W_z}(\omega_{W_z})\in \overline\Bcal_{2\delta}(\rho_1)\Big\}
\1\Big\{\frac 1{N_R}\sum_{z\in Z_{N,R}} \frac 1{|W|}{\mathfrak N}^{\ssup{\ell,\Scal}}_{W_z}(\varpi_{W_z})\in\overline\Bcal_{2\delta}(\rho_2)\Big\},
\end{aligned}
\end{equation}
where we write $\Bcal_\delta(x)=(x-\delta,x+\delta)$ and $\overline \Bcal_\delta(x)$ for its closure, for any $x\in\R$ and $\delta>0$.

Recall the interaction defined in \eqref{interaction}.  We drop all interactions between subboxes $W_z$ and $W_{\widetilde z}$ for different $z,\widetilde z \in Z_{N,R}$ and obtain a lower bound for the entire interaction. Recall the notation from \eqref{Vdef} and from \eqref{FLL}--\eqref{FSS}. We extend \eqref{FLS} and \eqref{FSS} from $\varpi\in \Scal$ to $\varpi_W=\sum_{g\in\Gamma}\delta_g\in \Scal_W$, i.e., from $\Gamma\subset\Ccal_\infty$ to $\Gamma\subset \Ccal_W$, in an obvious way. Then we estimate
\begin{equation}\label{interactionlowbound}
\Phi_{\L_N,\L_N}(\omega)\geq \sum_{z\in Z_{N,R}} F_{W,W}(\theta_z(\omega_{W_z}),\theta_z(\varpi_{W_z})),
\end{equation}
where $F_{W,W}(\omega_{W},\varpi_{W})=\frac 12 F_{W,W}^{\ssup {\Lcal\Lcal}}(\omega_{W})+F_{W,W}^{\ssup {\Lcal\Scal}}(\omega_{W},\varpi_{W})+\frac 12 F_{W,W}^{\ssup {\Scal\Scal}}(\varpi_{W})$.

This implies the following for all sufficiently large $N$: 
\begin{equation}\label{hatZ1}
\begin{aligned}
\widehat Z_{N,R,\delta}^{\ssup{\rm bc}}(\L_N,\rho_1,\rho_2)
&\leq \LPP^{\ssup{\L_N,\rm bc}}\Big[\1\Big\{\frac 1{N_R}\sum_{z\in Z_{N,R}} \frac 1{|W|} {\mathfrak N}^{\ssup{\ell,\Lcal}}_{W_z}\in\overline\Bcal_{2\delta}(\rho_1)\Big\}\\
&\qquad\times \1\Big\{\frac 1{N_R }\sum_{z\in Z_{N,R}}\frac 1{|W|} {\mathfrak N}^{\ssup{\ell,\Scal}}_{W_z}\in\overline\Bcal_{2\delta}(\rho_2)\Big\}\prod_{z\in Z_{N,R}}\e^{-F_{W,W}(\theta_z(\omega_{W_z}),\theta_z(\varpi_{W_z}))}\Big].
\end{aligned}
\end{equation}

We now write the right-hand side of \eqref{hatZ1} in terms of the empirical measure $\Xi_{N,R}^{\ssup{\omega_{\rm P}}}$ from \eqref{Xidef} and its boundary-shred projection $\partial\Xi_{N,R}^{\ssup{\omega_{\rm P},\Scal}}$ from \eqref{partialXiSdef} as follows. Recall \eqref{particlenumber} for writing the particle number in shreds as a functional of $\partial \Xi_{N,R}^{\ssup{\omega_{\rm P},\Scal}}$. From this, the assertion follows.
\end{proof}

For later developments in Section~\ref{sec-LDPana}, when we will take the limit as $N\to\infty$, it will be convenient to work exclusively with particle boundary condition, which was introduced below \eqref{nnBBM}. Therefore, let us further estimate from above, until we arrive at a similar expectation with ${\rm bc}={\rm par}$. For \lq bc\rq\ being Dirichlet zero boundary conditions, this is trivial, since zero boundary conditions in particular satisfy particle boundary conditions. For periodic boundary conditions, however, we have to work harder for this point; this will be done in Sections~\ref{sec-distshreds} and at the beginning of Section \ref{sec-LDPabstract}.

At the beginning of Section~\ref{sec-LDPabstract} we will make some technical preparations for applying a suitable large-deviation principle in the limit as $N\to\infty$ to the right-hand side of \eqref{hatZ2}. A suitable formulation of this LDP will be in Corollary~\ref{cor-LDP}. After inserting in Section~\ref{sec-Compactness} some events that induce some necessary compactness and control on the particle number and interaction, we will apply the LDP in Section~\ref{sec-LDPupperbound}, resulting in an $R$- and $\delta$-depending variational formula. This will be used in Section~\ref{sec-finishuppbound} to make $R\to\infty$ and $\delta\downarrow0$, and to finish the proof of the upper bound in Theorem~\ref{thm-freeenergy}.

\subsection{Lower bound}\label{sec-lowboundN}

We are going to derive a lower bound for $\widehat Z_{N,R,\delta}^{\ssup{\rm bc}}(\L_N,\rho_1,\rho_2)$ defined in \eqref{hatZ} for any sufficiently large $N$.  The main result of this section is Lemma~\ref{lem-LowerBound} below. Our lower bound there will be analogous to the upper bound in Lemma~\ref{lem-UpperBound}, but will require much more restrictions in the integration area. 

In general, we assume the situation and notation described in Section~\ref{sec-empmeasconf}. We will have to neglect all the interactions between loops and shreds in different subboxes $W_z$, such that the contributions will be independent over the subboxes. For this, we need to restrict the amount of particles in the boundary area of  $W_z$ for any $z\in Z_{N,R}$. We will do that in a much more comprehensive way than what is needed here, in order to prepare for future steps in Section~\ref{sec-Nlowbound}, where we will take the limit as $N\to\infty$. That is, we will impose here that, in each subbox $W_z$, all the legs (of loops and of shreds) have a {\it spread} $\|f\|_{\rm sp}=\max_{t\in[0,\beta]}|f(t)-f(0)|$ that is $< M$, and in each ball of radius $\mathfrak r$ (to be introduced), the particle number is $\leq T$ for some large $M, T$. Furthermore, we need to restrict the global configuration close to the boundary of $\L_N$ in order to achieve particle boundary condition.

We are going to use $C\in(0,\infty)$ as a generic constant that depends only on $d$, $\beta$ or $v$ (which we neglect) and may change its value from appearance to appearance. If it may depend on additional parameters, then these are appended as indices.

We need to restrict to configurations in
\begin{equation}\label{LoopsMvarthetaS}
\Lcal^{\ssup {M}}=\Big\{\sum_{x\in\zeta}\delta_{(x,f_x)}\in\Lcal\colon  \|f_{x,i}\|_{\rm sp}< M\ \forall x\in\zeta,\forall i\Big\},\qquad M\in(1, \infty),
\end{equation}
where we recall that $f_{i}$ is the $i$-th leg of the loop $f$.  We will write $\Lcal_W^{\ssup  {M}}=\Lcal_W\cap \Lcal^{\ssup  {M}}$. We will make a change of measure from the original PPP ${\tt Q}$ (see Definition~\ref{def-PPP}) to the PPP ${\tt Q}_{M}$ with intensity measure 
\begin{equation}\label{nuMdef}
\nu_{\L_N, {M}}^{\ssup{\rm par}}(\d x,\d f)=\sum_{k\in\N}\frac 1k \d x\otimes\mu_{x,x}^{\ssup{k, {M},\L_N,{\rm par}}}(\d f),
\end{equation}
where $\mu_{x,x}^{\ssup{k, {M},\L_N,{\rm par}}}$ is the restriction of $\mu_{x,x}^{\ssup{k,\L_N,{\rm par}}}$ to the set of loops of length $k$ whose legs have spread $<M$. The $W_R$-shred-configurations under ${\tt Q}_{M}$ lie in the set 
\begin{equation}\label{SMdef}
\Scal^{\ssup  {M}}=\Big\{\sum_{g\in\Gamma}\delta_g\in \Scal\colon \| g_i\|_{\rm sp}< M\ \forall g\in \Gamma,i\in \Z\Big\}.
\end{equation}

We also need a notation for the number of all particles in $W\Subset \R^d$, regardless where their loop or shred comes from:
\begin{equation}\label{particlenumberinW}
\widetilde {\mathfrak N}^{\ssup\ell}_W(\omega,\varpi)
=\sum_{x\in\zeta}\sum_i\1\{f_{x,i}(0)\in W\}+\sum_{g\in\Gamma}\sum_i \1\{g_i(0)\in W\},\qquad \omega\in\Lcal,\varpi\in \Scal.
\end{equation}
With a parameter $\mathfrak r\in(0,1)$, we decompose (up to the boundaries) $W$ into the boxes $z+W_{\mathfrak r}$ with $z\in W\cap 2\mathfrak r\Z^d$. With positive parameters $ L,K,T$, define the events
\begin{eqnarray}
A_{W;M,L,K}^{\ssup{\Lcal}}&=&\bigcap_{z\in W\cap 2\mathfrak r\Z^d}\Big\{\omega\in\Lcal_W^{\ssup{M ,\leq L}}\colon \forall l\in [L]\colon |\zeta_{z,l}|\leq K\Big\},\label{ALcaldef}\\
A_{W;M,\mathfrak r, T}^{\ssup{\Scal}}&=&\bigcap_{z\in W\cap 2\mathfrak r\Z^d}\Big\{\varpi\in \Scal_W^{\ssup {M}}\colon \widetilde{ \mathfrak N}_{z+W_\mathfrak r}^{\ssup{ \ell}}(\varpi)\leq T\Big\},\label{AScaldef}
\end{eqnarray}
where  $\zeta_{z,l}$ is the set of initial points of $l$-length loops of $\omega$ in $z+W_\mathfrak r$. That is, the loop configurations in $A_{W;M,L,K}^{\ssup{\Lcal}}$ have only legs with spread $< M$, only loop lengths $ \leq L$, and there are, in any of the subboxes $z+W_\mathfrak r$, no more $l$-length loops than $K$, while in $ A_{W;M,\mathfrak r, T}^{\ssup{\Scal}} $ we impose a uniform upper bound for the number of particles in any of the subboxes $z+W_\mathfrak r$.

We will later restrict the integration with respect to the PPP $\omega_{\rm P}$  to the event 
\begin{equation}\label{Adef}
A_{W;\Theta}= A_{W;M,L,K}^{\ssup{\Lcal}}\times A_{W;M,\mathfrak r, T}^{\ssup{\Scal}},\qquad\mbox{with }\Theta=(M,L,K,\mathfrak r, T).
\end{equation}
More precisely, we will restrict to  the event $\bigcap_{z\in Z_{N,R}^\circ} \{(\omega_{W_z},\varpi_{W_z})\in A_{W_z;\Theta}\}$, where we introduce the slightly smaller (than $\L_N$) box $\L_N'=\bigcup_{z\in Z_{N,R}^\circ}W_z$, and $Z_{ N,R}^\circ$ is the set of all $z\in Z_{N,R}$ that have no neighbour in $Z_{N,R}^{\rm c}$ (where we adopt the usual neighbourhood relation in $\Z^d$, upscaled by a factor of $2R$). Let us introduce the conditional version of $\LPP_{M}^{\ssup{\L_N',\rm par}}[\,\cdot\, |\partial \Xi_{N,R}^{\ssup {\circ,\omega,\Scal}}=\psi]$ given this event, i.e., given that $(\omega_{W_z},\varpi_{W_z}) \in A_{W_z;\Theta}$ for all $z\in Z_{N,R}^\circ $, which we denote by $\LPP^{\ssup{\L_N',{\rm par}}}_{W;\Theta}[\,\cdot\, |\partial \Xi_{N,R}^{\ssup {\circ,\omega,\Scal}}=\psi]$. The slightly modified empirical measure is denoted
\begin{equation} \label{eq:PNdefcirc}
\Xi_{N,R}^{\ssup{\circ,\omega}} = \frac1{\# Z_{N,R}^\circ}\sum_{z\in Z_{N,R}^\circ} \delta_{ \big(\theta_{z}(\omega_{W_z}^{\ssup{\Lcal} }),\theta_{z}(\varpi_{W_z}^{\ssup{\Scal}})\big)}\in\Mcal_1(\Lcal_W\times \Scal_W)
\end{equation}
and $\partial \Xi_{N,R}^{\ssup {\circ,\omega,\Scal}}$ is the boundary-shred projection of $\Xi_{N,R}^{\ssup{\circ,\omega}}$ analogously to \eqref{partialXiSdef}. 
We also need the function
\begin{equation}\label{pMTdef}
p_{W;M,\mathfrak r,T}(\mu)= {\tt K}_W\big(\mu; A_{W;M,\mathfrak r,T}^{\ssup\Scal}\big),\qquad \mu\in \Tcal_W.
\end{equation}

As usual, we write $W=W_R=[-R,R]^d$ for short.

\begin{lemma}[Lower bound]\label{lem-LowerBound}
Fix  $\delta\in(0,1)$ and $\rho_1,\rho_2\in[0,\infty)$. Then, for any $M,L,K,T,\mathfrak r\in(0,\infty)$, there are constants $c_M\in(0,1)$ satisfying $ \lim_{M\to\infty}c_M=1$ and  $C_{M,L,K,T}\in(0,\infty)$ (depending only on $M$ respectively only on $M,L,K,T$) such that, for any $N,R\in\N$ such that $R>M$, and for any $\psi\in\Mcal_1(\Tcal_W)$ (possibly depending on $N$) such that the following conditional expectation is well-defined,
\begin{equation}\label{hatZ4}
\begin{aligned}
\widehat Z_{N,R,\delta}^{\ssup{\rm bc}}&(\L_N,\rho_1,\rho_2)
\geq
\e^{-\frac 1R{C_{M,L,K,T}}|\L_N |} \e^{-o(|\L_N|)}c_{M}^{|\L_N|}\big[\Pi_W^{\ssup\Lcal}({\tt Q})(A^{\ssup\Lcal}_{W;M,L,K})\big]^{|\L_N'|/|W|}\e^{\frac{|\L_N'|}{|W|}\langle \log p_{W;M,\mathfrak r,T},\psi\rangle }\\
&\quad \times
\LPP_{M}^{\ssup{\L_N',\rm par}}\big(\partial \Xi_{N,R}^{\ssup {\circ,\omega,\Scal}}=\psi \big) \\
&\quad \times
\LPP^{\ssup{\L_N',{\rm par}}}_{W;\Theta}\Big[\exp\Big\{-|\L_N'|\big\langle \Xi_{N,R}^{\ssup{\circ,\omega_{\rm P}}},\smfrac 1{|W|}F_{W,W}\big\rangle\Big\}\\
&\quad\quad\times \1\Big\{\langle\Xi_{N,R}^{\ssup{\circ,\omega_{\rm P}}},\smfrac 1{|W|}{\mathfrak N}^{\ssup{\ell,\Lcal}}_W\rangle \in \Bcal_{\delta/3}(\rho_1)\Big\}
\1\Big\{\langle \psi,\smfrac 1{|W|}{\mathfrak N}^{\ssup{\ell,\Scal}}_{\partial W}\rangle \in \Bcal_{\delta/3}(\rho_2)\Big\}
\Big|\partial \Xi_{N,R}^{\ssup {\circ,\omega,\Scal}}=\psi\Big] .
\end{aligned}
\end{equation}
\end{lemma}

\begin{proof} In the expectation in \eqref{hatZ} that defines $\widehat Z_{N,R,\delta}^{\ssup{\rm bc}}(\L_N,\rho_1,\rho_2)$ we insert the indicator on the event that the configuration lies in $\Lcal^{\ssup  {M}}$. Furthermore, we now require that the configuration $\omega_{\L_N}$ has no particle in the boundary area $\L_N\setminus \L_N'$. This implies that the restricted configuration $\omega_{\L_N'}$ satisfies  particle-boundary conditions in $\L_N'$ (which was introduced below \eqref{nnBBM}), since no loop reaches $\partial \L_N$ (because of $M<R$) and this holds for both, Dirichlet zero and periodic boundary condition in $\L_N$. The event of having no particles in $\L_N\setminus\L_N'$ is the intersection of the two independent events that no loops start in $\L_N\setminus\L_N'$ (whose probability is $\geq \e^{-o(|\L_N|)}$, since the volume of $\L_N\setminus\L_N'$ is $o(|\L_N|)$) and the event that no shred terminates in $\L_N\setminus\L_N'$, which is absorbed in the expectation with respect to particle-boundary conditions in $\L_N'$. Hence, we will restrict to an expectation on configurations in $\L_N'$ instead of $\L_N$, which will be asymptotically lead to no changes in the limit as $N\to\infty$. 

Note that 
$$
\lim_{N\to\infty}\frac 1{|\L_N|}\big([\nu_{\L_N}^{\ssup{\rm par}}-\nu_{\L_N, {M}}^{\ssup{\rm par}}](\L_N\times \Ccal^{\ssup{\circlearrowleft}})\big)
=\sum_{k\in\N}\frac 1k \mu_{0,0}^{\ssup k}(\Ccal^{\ssup{\circlearrowleft}}_k\setminus \Ccal^{\ssup{\circlearrowleft}}_{k,M})\to 0\qquad \mbox{as }M\to\infty,
$$
where $\Ccal^{\ssup{\circlearrowleft}}_{k,M}$ denotes the set of loops of length $k$ whose legs have spread $<M$. Furthermore, the number of loops in the event in \eqref{hatZ} is $\leq O(|\L_N|)$;  hence the change of measure from ${\tt Q}^{\ssup{\L_N',{\rm par}}}$ to  ${\tt Q}_{M}^{\ssup{\L_N',{\rm par}}}$ yields, as a lower bound a factor of $c_{M}^{|\L_N|}$ with some $c_M\in(0,1)$ satisfying $c_M\to 1$ as $M\to\infty$.

Summarizing, this transition to the $M$-restricted configurations in $\L_N'$ instead of $\L_N$ leads to the bound
\begin{equation}\label{lowbound1}
\widehat Z_{N,R,\delta}^{\ssup{\rm bc}}(\L_N,\rho_1,\rho_2)
\geq \e^{o(|\L_N|)} c_{M}^{|\L_N|}
\LPP_{M}^{\ssup{\L_N',\rm par}}\Big[{\rm e}^{-\Phi_ {\L_N',\L_N'}}\1\{|\smfrac 1{|\L_N'|} {\mathfrak N}^{\ssup{\ell,\neg R}}_{\L_N'}-\rho_1|\leq\smfrac\delta 2\}\1\{|\smfrac 1{|\L_N'|} {\mathfrak N}^{\ssup{\ell,R}}_{\L_N'}-\rho_2|\leq\smfrac\delta2\}\Big].
\end{equation}

Our next step is to restrict to the event $\{\partial \Xi_{N,R}^{\ssup {\circ,\omega,\Scal}}=\psi\}$ for some (later particularly chosen) $\psi\in \Mcal_1(\Tcal_W)$, where $\Xi_{N,R}^{\ssup{\circ,\omega}}$ is defined in \eqref{eq:PNdefcirc}.

Here we will pick $\psi$ (possibly depending on $N$) such that this event has positive probability, and hence we write \eqref{lowbound1} as
\begin{equation}\label{lowbound2}
\begin{aligned}
\widehat Z_{N,R,\delta}&^{\ssup{\rm bc}}(\L_N,\rho_1,\rho_2)
\geq \e^{o(|\L_N|)} c_{M}^{|\L_N|}\LPP_{M}^{\ssup{\L_N',\rm par}}\big(\partial \Xi_{N,R}^{\ssup {\circ,\omega,\Scal}}=\psi \big)\\
&\times\LPP_{M}^{\ssup{\L_N',\rm par}}\Big[{\rm e}^{-\Phi_ {\L_N',\L_N'}}\1\{|\smfrac 1{|\L_N'|} {\mathfrak N}^{\ssup{\ell,\neg R}}_{\L_N'}-\rho_1|<\smfrac\delta 2\}\1\{|\smfrac 1{|\L_N'|} {\mathfrak N}^{\ssup{\ell,R}}_{\L_N'}-\rho_2|<\smfrac\delta2\}\,\Big|\,\partial \Xi_{N,R}^{\ssup {\circ,\omega,\Scal}}=\psi\Big].
\end{aligned}
\end{equation}

Our next task is to write the expectation in the second line of \eqref{lowbound2} as an expectation with respect to $\Xi_{N,R}^{\ssup{\circ,\omega_{\rm P}}}$, which is not able to express interaction between any two subboxes $z+W$ and $z'+W$  with distinct $z.z'\in 2R\Z^d$. For this, we need to split all the interaction into the sum of the interactions in all the subboxes $z+W_R$ with $z\in 2R\Z^d$, and we need to upper estimate the interaction between any two of such subboxes. This will be based on the following estimate of the interaction in terms of an integral of the square of particle numbers. Recall the notation from \eqref{Vdef} and from \eqref{FLL}--\eqref{FSS} and \eqref{subconfigurations} and that $F_{W,W}=\frac 12 F_{W,W}^{\ssup {\Lcal\Lcal}}+F_{W,W}^{\ssup {\Lcal\Scal}}+\frac 12 F_{W,W}^{\ssup {\Scal\Scal}}$. For two  (not necessarily disjoint) subboxes $W,\widetilde W$, we  write $F_{W,\widetilde W}((\omega_W,\varpi_W),(\omega_{\widetilde W},\varpi_{\widetilde W}))$ for all the interaction of loops and shreds between $W$ and $\widetilde W$ (i.e., the loops with all their particles in $W$ respectively in $\widetilde W$).

\begin{lemma}[Estimating interaction]\label{lem-Interactesti}
Assume that the interaction potential $v$ satisfies Assumption (V). Then, for any $M \in(0,\infty)$, with $\mathfrak r_{M }^{\ssup v}=M+\sup\supp(v)$,  for  any configuration $\omega\in \Lcal^{\ssup {M }} $, and for any $R\in(1,\infty)$ and any $z,z'\in 2R\Z^d$,
\begin{equation}\label{crossinteractionesti}
\begin{aligned}
F_{W_z,W_{z'}}\big((\omega_{W_z},\varpi_{W_z}),(\omega_{W_{z'}},\varpi_{W_{z'}})\big)
&\leq \beta \|v\|_\infty \Big(\int_{\Bcal_{\mathfrak r_{M }^{\ssup v}}(W_z\cap W_{z'})}\d x\,\widetilde{ \mathfrak N}_{\Bcal_{\mathfrak r_{M }^{\ssup v}}(x)}^{\ssup{ \ell}}(\omega_{W_z},\varpi_{W_z})^2\Big)^{1/2}\\
&\qquad \times\Big(\int_{\Bcal_{\mathfrak r_{M }^{\ssup v}}(W_z\cap W_{z'})}\d x\,\widetilde {\mathfrak N}_{\Bcal_{\mathfrak r_{M }^{\ssup v}}(x)}^{\ssup{ \ell}}(\omega_{W_{z'}},\varpi_{W_{z'}})^2\Big)^{1/2},
\end{aligned}
\end{equation}
where $\Bcal_r(A)$ denotes the $r$-ball around a compact set $A$.
\end{lemma}

\begin{proof}We write $\mathfrak r$ instead of $\mathfrak r_{M }^{\ssup v}$. The main observation is that any leg of $\omega$ in $W$ (i.e., any leg that starts from a particle in $W$) stays in the $M$-neighbourhood of $W$ and has therefore interaction only in the $\mathfrak r$-neighbourhood of $W$. Hence,  any $\omega\in\Lcal^{\ssup M}$ has the property that, for any two legs $g,g'$ of a loop or shred in $W_z$ and of a loop or shred in $W_{z'}$ (meaning that the initial site lies in $W_z$ respectively in $W_{z'}$), respectively, that have a non-vanishing interaction with each other, there is a site $x\in \Bcal_{\mathfrak r}(W_z\cap W_{z'})$ such that $|g(0)-x|\leq \mathfrak r$ and $|g'(0)-x|\leq \mathfrak r$. Denote by $\mathfrak L(\omega_W,\varpi_W)=\{f_{x,i}\colon x\in\zeta\cap W, i\in[\ell(f_x)]\}\cup\{g_i\colon i\in \Gamma_W(g), g\in\varpi_{W}\}$ the set of all the legs of loops or of shreds of $\omega_W=\sum_{x\in\zeta}\delta_{(x,f_x)}$ or of $\varpi_W$. (Recall that all particles of all the loops of $\omega_W$ are in $W$, according to \eqref{subconfigurations}.) Then, for any $z,z'\in 2R\Z^d$,
$$
\begin{aligned}
F_{W_z,W_{z'}}&\big((\omega_{W_z},\varpi_{W_z}),(\omega_{W_{z'}},\varpi_{W_{z'}})\big)\\
&=\sum_{g\in \mathfrak L(\omega_{W_z},\varpi_{W_z})}\sum_{g' \in \mathfrak L(\omega_{W_{z'}},\varpi_{W_{z'}})} V(g,g')\\
& \leq \beta \|v\|_\infty \sum_{g,g'} \int_{\Bcal_\mathfrak r(W_z\cap W_{z'})}\d x\, \1\{|g(0)-x|\leq \mathfrak r\}\1\{|g'(0)-x|\leq \mathfrak r\}\\
&\leq \beta \|v\|_\infty \int_{\Bcal_\mathfrak r(W_z\cap W_{z'})}\d x\,
\widetilde{ \mathfrak N}_{\Bcal_\mathfrak r(x)}^{\ssup{ \ell}}(\omega_{W_z},\varpi_{W_z})\widetilde{ \mathfrak N}_{\Bcal_\mathfrak r(x)}^{\ssup{ \ell}}(\omega_{W_{z'}},\varpi_{W_{z'}})\\
&\leq \beta \|v\|_\infty \Big(\int_{\Bcal_\mathfrak r(W_z\cap W_{z'})}\d x\,\widetilde{ \mathfrak N}_{\Bcal_\mathfrak r(x)}^{\ssup{ \ell}}(\omega_{W_z},\varpi_{W_z})^2\Big)^{1/2}
\Big(\int_{\Bcal_\mathfrak r(W_z\cap W_{z'})}\d x\,\widetilde{ \mathfrak N}_{\Bcal_\mathfrak r(x)}^{\ssup{ \ell}}(\omega_{W_{z'}},\varpi_{W_{z'}})^2\Big)^{1/2}.
\end{aligned}
$$
\end{proof}

In words, this estimates the amount of cross-interaction between two neighbouring  subboxes against the product of the integrals  over $x$ in a neighbourhood of the intersection of the subboxes over the square of particle numbers in a certain box around $x$, and the amount of interaction within one of such  subboxes in terms of the integral of the squares of the particle numbers in a neighbourhood of the subbox.

Let us proceed with the proof of Lemma~\ref{lem-LowerBound}, restarting from \eqref{lowbound2} and showing what Lemma~\ref{lem-Interactesti} gives us for the interaction.
We will, as usual, always write $\omega=\sum_{x\in\zeta}\delta_{(x,f_x)}$ for any $\omega\in\Lcal$. We write $\Lcal_W^{\ssup{M,\leq L}}=\{\omega\in\Lcal_W^{\ssup {M}}\colon \ell(f_x)\leq L\, \forall x\in\zeta\}$. We pick now an arbitrary $\mathfrak r\in(0,\infty)$. Without loss of generality, we assume that the boxes $z+W_\mathfrak r$ with $z\in W\cap 2\mathfrak r\Z^d$ cover $W$, i.e.,  $W=\bigcup_{z\in W\cap 2\mathfrak r\Z^d} z+W_\mathfrak r$. Recall from Remark~\ref{rem-counting} that $\widetilde {\mathfrak N}^{\ssup\ell}_V$ counts particles only within $V$ for any $V\subset W$; hence for $(\omega,\varpi)\in \Lcal_W\times\Scal_W$, we have that $\widetilde{\mathfrak N}_W^{\ssup\ell}(\omega,\varpi)=\sum_{z\in W\cap 2\mathfrak r\Z^d} \widetilde{\mathfrak N}^{\ssup\ell}_{z+W_\mathfrak r}(\omega,\varpi)$. 

 For $\omega\in A_{W;M,L,K}^{\ssup{\Lcal}}$ we easily see that $\widetilde{ \mathfrak N}_{\Bcal_ {\mathfrak r_M^{\ssup v}}(x)}^{\ssup{ \ell}}(\omega)\leq C_{M,L,K} $ for every $x\in \Bcal_\mathfrak r(W)$ with some $C_{M,L,K}$ that depends only on $M,L,$ and $K$. Analogously, for $ \varpi\in A_{W;M,\mathfrak r, T}^{\ssup{\Scal}}$ we easily see that $\widetilde{ \mathfrak N}_{\Bcal_{\mathfrak r_M^{\ssup v}}(x)}^{\ssup{ \ell}}(\varpi)\leq C_MT $ for every $x\in \Bcal_\mathfrak r(W)$ with some $C_M\in(0,\infty)$ that depends only on $M$. 

  On the event $A_{W;\Theta}$ defined in \eqref{Adef}, since $R$ is much larger than $M$, for each $z\in Z^{\ssup\circ}_{N,R}$, the number of $z' \not= z$ such that there is some interaction between $W_z$ and $W_{z'}$ is not larger than $3^d$. Furthermore, the volume of $\Bcal_{\mathfrak r}(\partial W_z)$ is not larger than $C M R^{d-1}$. Hence,  we can estimate, with the help of Lemma~\ref{lem-Interactesti},
\begin{equation}\label{Phiestidecomp}
\begin{aligned}
\Phi_{\L_N',\L_N'}(\omega)
&=\sum_{z,z'\in Z^{\ssup\circ}_{N,R}}F_{W_z,W_{z'}}\big((\omega_{W_z},\varpi_{W_z}),(\omega_{W_{z'}},\varpi_{W_{z'}})\big)\\
&\leq \sum_{z\in Z^{\ssup\circ}_{N,R}}\Big( F_{W_z,W_z}(\omega_{W_z},\varpi_{W_z})+C_{M,L,K,T} R^{d-1}\Big)
\leq \frac {|\L_N|}{|W|}\langle F_{W,W},\Xi_{N,R}^{\ssup{\circ,\omega}}\rangle+ \frac {C_{M,L,K,T}}R |\L_N |,
\end{aligned}
\end{equation}
with some $C_{M,L,K,T}\in(0,\infty)$ that does not depend on $R$ nor on $\omega$ nor on $N$. (We choose the event $A_{W;\Theta}$ much smaller than what is needed for \eqref{Phiestidecomp}; this will come in handy in our further developments in Section~\ref{sec-Nlowbound}.)

So far, we have argued that we may estimate, for any sufficiently large $N$,
\begin{equation}\label{particlenumberestilower}
\begin{aligned}
\e^{-\Phi_{\L_N',\L_N'}}&\1\{|\smfrac 1{|\L_N'|} {\mathfrak N}^{\ssup{\ell,\neg R}}_{\L_N'}-\rho_1|\leq\smfrac\delta2\}\1\{|\smfrac 1{|\L_N'|} {\mathfrak N}^{\ssup{\ell,R}}_{\L_N'}-\rho_2|\leq\smfrac \delta2\}\\
&\geq \e^{-|\L_N|\frac 1{|W|}\langle F_{W,W},\Xi^{\ssup\circ}_{N,R}\rangle} \e^{-\frac {C_{M,L,K,T}}R |\L_N |}\prod_{z\in Z^{\ssup\circ }_{N,R}}\1_{A_{W_z;\Theta}}(\omega_{W_z},\varpi_{W_z})\\
&\quad\times\1\Big\{\frac 1{N_R}\sum_{z\in Z^{\ssup\circ }_{N,R}} \frac 1{|W|}{\mathfrak N}^{\ssup{\ell,\Lcal}}_{W_z}\in \Bcal_{\delta/2}(\rho_1)\Big\}
\1\Big\{\frac 1{N_R}\sum_{z\in Z_{N,R}^{\ssup{\circ}}} \frac 1{|W|}{\mathfrak N}^{\ssup{\ell,\Scal}}_{W_z}\in\Bcal_{\delta/2}(\rho_2)\Big\}.
\end{aligned}
\end{equation}

Under $\LPP_{M}^{\ssup{\L_N',\rm par}}[\,\cdot\, |\partial \Xi_{N,R}^{\ssup {\circ,\omega,\Scal}}=\psi]$, the indicators $\1_{A_{W_z;\Theta}}(\omega_{W_z},\varpi_{W_z})$ are independent over $z\in Z^{\ssup\circ }_{N,R}$ and their product has the probability 
$$
\big[\Pi_W^{\ssup\Lcal}({\tt Q})(A^{\ssup\Lcal}_{W;M,L,K})\big]^{|Z_{N,R}^{\circ}|}\e^{|Z_{N,R}^{\circ}|\langle \log p_{W;M,\mathfrak r,T},\psi\rangle },
$$
where $p_{W;M,\mathfrak r,T}$ is defined in \eqref{pMTdef}.
 Using the notation for the conditional measure $\LPP^{\ssup{\L_N',{\rm par}}}_{W;\Theta}[\,\cdot\, |\partial \Xi_{N,R}^{\ssup {\circ,\omega,\Scal}}=\psi]$ given just before the lemma, we obtain the assertion of the lemma, by substituting \eqref{particlenumberestilower} into \eqref{lowbound2}.
\end{proof}

This lower bound is analogous to the upper bound in \eqref{hatZ2} with $Z_{N,R}$ replaced by $Z_{N,R}^\circ$ and the corresponding amendments and with $\overline\Bcal_{2\delta}(\rho_i)$ replaced by $\Bcal_{\delta/3}(\rho_i)$ and with additional auxiliary terms whose exponential rates turn out to be small when we later make $R, M,T,L\to\infty$. Note that we have now particle boundary condition in the (slightly smaller) box $\L_N'$. Therefore, all the $\theta_z(\omega_{W_z})$ with $z\in Z_{N,R}^\circ$ have now the same distribution; in particular the $\theta_z(\omega_{W_z})$ are i.i.d. 

In Lemma~\ref{lem-UpperLDPbound} below, we are going to take $\psi\in\Mcal_1(\Tcal_W)$ in a particular way and derive a suitable large-deviation lower bound for the term in the second line in Section~\ref{sec-boundaryshreds}. In Corollary~\ref{cor-LDP} in Section~\ref{sec-LDPabstract} we explain abstractly how to find the large-$N$ exponential rate of the expectation in the last two lines, and in Section~\ref{sec-Nlowbound} we proceed with making $N\to\infty$ in \eqref{hatZ4}. In Section~\ref{sec-finishlowerbound} we put everything together and finish the proof of the lower bound in Theorem~\ref{thm-freeenergy}.

\section{The empirical measure of boundary-shred configurations}\label{sec-distshreds}

\noindent In this section we provide crucial information about the distribution of the empirical measure of the boundary-shred configuration, $\partial \Xi_{N,R}^{\ssup{\omega_{\rm P},\Scal}}$, defined in \eqref{partialXiSdef}, for the reference marked PPP $\omega_{\rm P}$ under ${\tt Q}^{\ssup{\L_N,{\rm bc}}}$. The main result of this section will be a lower bound for the large-$N$ deviations in Lemma~\ref{lem-UpperLDPbound}. This will be crucial for handling the term in the second line of \eqref{hatZ4}.  Finally, this term should and will turn out to be negligible in the limit $R\to\infty$ since it is ``only'' a boundary term after all, but the way to proving this will be long. A step towards this is done in Lemma~\ref{lem-lowerboundlogpint}, where the lower bound of Lemma~\ref{lem-UpperLDPbound} is further lower bounded for particular  configurations. More work in Section~\ref{sec-ergappr} will show that a general configuration can be approximated with such particular ones. In contrast, for proving the upper bound, we will just drop the influence coming from $\partial \Xi_{N,R}^{\ssup{\omega_{\rm P},\Scal}}$, i.e., we estimate probabilities involving this just by one, which is of course much easier.

In Lemma~\ref{lem-UpperLDPbound} and its proof we make a connection between long loops and interlacements; we need to construct large-loop configurations with a given empirical measure $\psi$ of boundary-shred configurations. By far not every $\psi$ admits the existence of loop-configurations in arbitrarily large boxes $\L_N$ that is consistent with $\psi$, the criterion for this property is rather subtle and has a global character. In Lemma~\ref{lem-UpperLDPbound} we first circumvent this problem by coining the term ``admissible'' for $\psi$'s having this extension property and prove the lemma under this assumption. But in Lemma~\ref{lem-Psi_admissible} we show that the requirement of ergodicity of shift-invariant random interlacement configurations implies admissibility; this is the reason that in  Section~\ref{sec-ergappr} we need to approximate arbitrary loop-interlacement configurations with ergodic ones.

As a preparation, in Section~\ref{sec-distshredoneloop} we look at the $W$-shreds emanating from a single loop and describe, given the boundary-shred configuration, the distribution of the resulting shred configuration via the strong Markov property of the Brownian motion. This is used in Section~\ref{sec-boundaryshreds} in Lemma~\ref{lem-UpperLDPbound} for deriving the crucial LDP lower bound for the probability of the event $\{\partial \Xi_{N,R}^{\ssup{\omega_{\rm P},\Scal}}\approx \psi\}$ for any admissible $\psi$. In Section~\ref{sec-lowboundlogpint} we derive further estimates for this lower bound for particular shred distributions. Furthermore, as an application and also for independent interest and further use later, we show that the expected number of shreds in a large box is negligible with respect to the volume of the box for any shift-invariant interlacement configuration measure with bounded entropy and bounded expected particle density; a property that is very suggesting, but surprisingly difficult to prove.

In the entire section, assume that the situation and the notation described in Section~\ref{sec-Purpose} is given. That is, the box $\L_N$ with volume $\sim N/\rho$ is decomposed into  $N_R\sim |\L_N|/|W_R|$ regular boxes $W_z=z+W$ of diameter $2R$, where $R$ may depend on $N$ and converges to a given $R\in\N$.

\subsection{The shreds of one loop}\label{sec-distshredoneloop}

In this section, we are examining the joint distribution of all the  $W$-shreds of a given Brownian loop in a large box $\L$ containing the box $W$. Using the strong Markov property, we will collect all the pieces between entries and exits to and from $W$ and drop the remaining pieces. That is, we will consider one $W$-shred $g\in \Ccal_{k}$ as in \eqref{SWdef}, ignoring  the last particle before $g(0)$, which is in $W^{\rm c}$. Recall that $\P_x^{\ssup{\L,{\rm bc}}}$ denotes the distribution of a Brownian motion $B=(B_t)_{t\in[0,\infty)}$ in $\L$ with boundary condition \lq bc\rq\, starting from $x$.  We assume that \lq{\rm bc}\rq\ is one of the three boundary conditions \lq per\rq, \lq 0\rq\ or \lq par\rq.

Define
\begin{equation}\label{pdef}
p_x^{\ssup{W,\L,{\rm bc}}}(l,y)=\P_x^{\ssup{\L,{\rm bc}}}(\tau=l,B_{\tau \beta}\in\d y)/\d y,\qquad (x,l,y)\in \L\times \N\times \L,
\end{equation}
where 
\begin{equation}\label{taudef}
\tau=\tau_W=\inf\{l\in\N\colon B_{l \beta}\notin W^{\ssup {B_0}}\}
\end{equation}
is the first of the times in $\beta\N$ at which the motion changes the box, where we write  $W^{\ssup x}=W_{z_x}$ with $x\in W_{z_x}$ for that subbox that contains $x$. Then $p_x$ is a sub-probability density (with respect to the counting measure on $\N$ times the Lebesgue measure) on $\N\times (W^{\ssup x})^{\rm c}$ (for periodic boundary condition even a probability density). Furthermore, for a loop $B\in\Ccalcirc$, we define $\tau_1=\tau$ and $\tau_{j+1}=\inf\{l>\tau_j\colon l\in\N, B_{l\beta}\notin W^{\ssup{B_{\tau_j}}}\}$ (the $(j+1)$-st time that the motion $B$ changes the box at some time in $\beta\Z$) for $j\in\N$. Since $B$ is a cycle, there is a unique smallest $\mathfrak m_R=\mathfrak m_R(B)\in\N$ such that $B_{\tau_{\mathfrak m_R+1}\beta}=B_{\tau_1 \beta}$. The number $\mathfrak m_R(B)$ is the number of times that $B$ changes the subbox at times in $\beta\N$.

In order to describe the pieces of the loop between the changes between the subboxes (i.e., the shreds), we introduce  the following probability measure on $\Ccal_l$: by
\begin{equation}\label{qdef}
q^{\ssup{l,W,\L,{\rm bc}}}_{x,y}(\d f)
=\frac{\P_x^{\ssup{l,\L,{\rm bc}}}(\tau=l,B\in \d f, B_{\tau \beta}\in\d y)}{p_x^{\ssup {W,\L,{\rm bc}}}(l,y)\, \d y}
=\P_x^{\ssup{l,\L,{\rm bc}}}\big((B_s)_{s\in [0, l\beta]}\in \d f\mid \tau=l, B_{l\beta}=y\big) 
\end{equation}
we denote the  distribution of a Brownian $W$-shred with $l$ legs from $x\in W$ to $y\in W^{\rm c}$. 

We also need the restricted versions of $p_x$ and $q_{x,y}$ for shreds whose legs have spreads $< M$ for some $M\in(0, \infty)$. For this, we denote 
\begin{equation}\label{Cboundedspread}
\Ccal_1^{\ssup{M}}=\{f\in\Ccal_1\colon \|f\|_{\rm sp}< M\},\qquad\mbox{where }\|f\|_{\rm sp}=\max_{t\in[0,\beta]}|f(t)-f(0)|.
\end{equation}
Denote 
\begin{equation}\label{pMdef}
p_x^{\ssup {W,M}}(l,y)=\P_x\big(\tau=l,\forall j\in[l]\colon B_j\in \Ccal_1^{\ssup{M}}, B_{\tau \beta}\in\d y\big)/\d y,\qquad (x,l,y)\in W\times \N\times W^{\rm c},
\end{equation}
and
\begin{equation}\label{qMdef}
\begin{aligned}
q^{\ssup{l,W,M}}_{x,y}(\d f)
&=\frac{\P_x(\tau=l,\forall j\in[l]\colon B_j\in \Ccal_1^{\ssup{M}},B\in \d f, B(\tau \beta)\in\d y)}{p_x^{\ssup {W,M}}(l,y)\,\d y}\\
&=\P_x\big(B\in \d f\mid \tau=l, B(l\beta)=y, \forall j\in[l]\colon B_j\in \Ccal_1^{\ssup{M}}\big).
\end{aligned}
\end{equation}
Note that $B$ does not leave the $M$-neighbourhood of the subbox in which it started. Hence, if this box is in the interior of $\L$, then it satisfies particle-boundary condition in $\L$.

We are going to identify the distribution of the lengths of the time lags between subsequent exit times from subboxes and the location of the Brownian cycle at these times. 

For $k\in\N$ we consider the (non-normalised) measure 
\begin{equation}\label{Phatkdef}
\widehat\P^{\ssup{k,\L,{\rm bc}}}(\cdot)=\int_\L\d x\, \mu_{x,x}^{\ssup{k,\L,{\rm bc}}}(\{\mathfrak m_R\geq 1\}\cap\cdot).
\end{equation}
Recall the decomposition of $\L_N$ into the subboxes $z+W$ with $z\in Z_{N,R}$ and the smaller box $\L_N'$, which is the union of the $z+W$ over $z\in Z_{N,R}^\circ$, the (discrete) interior of $Z_{N,R}$, as introduced before Lemma ~\ref{lem-LowerBound}.

\begin{lemma}[Distribution of boundary shred configuration of one loop]\label{lem-distboundshred} 
Fix $R\in(0,\infty)$ and $W=W_R=[-R,R]^d$ and $M\in(0,R)$.  Assume that $\L$ is a large centred box.

\begin{enumerate}
\item Fix $m\in\N$ and $l_1,l_2,\dots,l_m\in\N$ such that $\sum_{j=1}^m l_j=k$. Fix $w_1,\dots,w_m\in W$ and $z_1,\dots,z_m\in Z_{N,R}$ such that $z_j\not=z_{j-1}$ for any $j$, where $z_0=z_m$. Then
\begin{equation}\label{intrance-exit-distribution}
\begin{aligned}
\widehat \P^{\ssup{k,\L,{\rm bc}}}&\Big(\mathfrak  m(B)=m,\forall j=1,\dots,m \colon \tau_{j}-\tau_{j-1}=l_j, B_{l_j \beta}\in \d ( z_j+w_j)\Big)\\
&=\bigotimes_{j=1}^m p_{z_{j-1}+w_{j-1}}^{\ssup{W,\L,{\rm bc}}}(l_i,z_j+w_j)\,\d w_j, \qquad z_0=z_m, w_0=w_m.
\end{aligned}
\end{equation}
Furthermore, under  conditioning on the event 
\begin{equation}\label{event}
A_{m}\big((l_j,w_j)_{j\in[m]}\big)=\{\mathfrak m_R(B)=m,\forall j=1,\dots,m \colon \tau_{j+1}-\tau_j=l_j, B_{l_j \beta}=  z_j+w_j\},
\end{equation}
the shreds $(B_{s +\tau_j\beta})_{s\in [0,l_j\beta]}$ are independent over $j\in\{1,\dots,m \}$ and have distribution $q^\ssup{l_j,W,\L,{\rm bc}}_{ z_{j-1}+w_{j-1}, z_j+w_j}$.

\item For ${\rm  bc}={\rm par}$ and $\widehat \P^{\ssup{k,\L,{\rm par}}}$ restricted to all the legs of $B$ lying in $\Ccal_1^{\ssup{M}}$, we can replace $p^{\ssup{W,\L,{\rm par}}}$ by $p^{\ssup{W,M}}$ and $q^\ssup{l_j,W,\L,{\rm par}}$ by $q^{\ssup{l_j,W,M}}$ in these two statements.

\item For any $z\in Z_{N,R}^\circ$ such that all legs start start in $W_z$ have spread $< M$ and for any $j$ satisfying $z_{j-1}=z$, we may replace $p_{ z_{j-1}+w_{j-1}}^{\ssup{W,\L,{\rm per}}}$ by $p_{ z_{j-1}+w_{j-1}}^{\ssup{W,M}}$ and $q^\ssup{l_j,W,\L,{\rm per}}_{ z_{j-1}+w_{j-1}, z_j+w_j}$ by $q^\ssup{l_j,W,M}_{ z_{j-1}+w_{j-1}, z_j+w_j}$ in these statements.
\end{enumerate}
\end{lemma}

\begin{proof} (1) For any $m\in\N$ and $\widetilde l_1,l_2,\dots,l_{m}\in \N$ such that $\widetilde l_1+\sum_{j=2}^{m-1} l_j\leq k$, and $z_1,\dots,z_m\in Z_{N,R}$  satisfying $z_j\not=z_{j-1}$ for all $j$ and  $w_1,\dots,w_m\in W$, we have, with $\tau_0=0$,
\begin{equation}
\begin{aligned}
\widehat \P^{\ssup{k,\L,{\rm bc}}}\big(&\tau_1=\widetilde l_1,B_{\widetilde l_1 \beta}\in \d( z_1+w_1), \forall j=1,\dots,m-1\colon \tau_j-\tau_{j-1}=l_j, B_{l_j\beta}\in\d\big(  z_j+w_j\big)\big)\\
&= \int_\L\,\d x_0 \, p^{\ssup{W,\L,{\rm bc}}}_{x_0}\big(\widetilde l_1,\d (  z_1+w_1)\big)
\Big[\prod_{j=2}^{m} p^{\ssup{W,\L,{\rm bc}}}_{ z_{j-1}+w_{j-1}}\big(l_j,\d (  z_j+w_j)\big)\Big]
\\
&\qquad\P^{\ssup{\L,{\rm bc}}}_{ z_m+w_m}\Big(\tau>k-\widetilde l_1-\sum_{j=2}^m l_j,B_{(k-\widetilde l_1-\sum_{j=2}^m l_j)\beta}\in\d x_0\Big)\\
&= \prod_{j=1}^{m} p^{\ssup{W,\L,{\rm bc}}}_{ z_{j-1}+w_{j-1}}\big(l_j,z_j+w_j\big)\,\d w_j,\qquad l_1=k-\sum_{j=2}^{m} l_j, z_0=z_m, w_0=w_m,
\end{aligned}
\end{equation}
where we used the Markov property in the second step to concatenate the last piece of the Brownian cycle from the last change of a subbox at time $\tau_m \beta$ until the end of the time interval, $k\beta$.

The additional assertion stems also directly from the strong Markov property, applied successively at the times $\tau_j \beta$. 

(2) follows from the definition of $p^{ \ssup{W,M}}$, and (3) from the fact that, by $M<R$, the legs in such subboxes to not reach the boundary of $\L$ but their particles stay in the subbox, and hence periodic boundary condition in $\L$ can be replaced by particle boundary condition in $W$.
\end{proof}

\subsection{Large-deviation lower bound for the boundary-shred configurations}\label{sec-boundaryshreds} 

In this section we show how to lower bound the term in the second line on the right-hand side of \eqref{hatZ4}. That is, we now use Lemma~\ref{lem-distboundshred} to derive a large-deviations lower bound for  the distribution of the boundary configuration $\partial \Xi_{N,R}^{\ssup{\omega_{\rm P},\Scal}}$ defined in \eqref{partialXiSdef} of all the $W$-shreds, i.e., their  lengths and initial and starting sites, for shreds coming from the $R$-crossing part of marked Poisson process $\omega_{\rm P}$.

The distribution of $ \partial \Xi_{N,R}^{\ssup\Scal}$ is not easy to describe.  As it concerns large-deviation upper bounds, we will in Section~\ref{sec-finishuppbound} be satisfied with estimating the probability of  $ \partial \Xi_{N,R}^{\ssup\Scal}$ being in a compact set simply against one. However, for the lower bound we have to work.  Because of our preparations in Lemma~\ref{lem-LowerBound}, we need the lower bound only  for particle boundary conditions. We will be consent with a large-deviation lower bound for the probability of the event $\{ \partial \Xi_{N,R}^{\ssup\Scal}=\psi\}$ for some given $\psi\in\Mcal_1(\Tcal_W)$ (possibly $N$-dependent) with certain properties.  For a general $\psi$, it is not {\it a priori} clear whether or not this event has positive probability at all under any loop configuration measure in a large box $\L_N$, i.e., whether or not it is possible to find a loop configuration in $\L_N$ such that the empirical measure of its boundary-shred triples is close to the given $\psi$. For the time being we are therefore working under the following assumption on $\psi$.

\begin{defn}[Admissible boundary-shred configurations]\label{def-admisssible} Fix $R>0$ and $W=W_R=[-R,R]^d$. We call $\psi\in\Mcal_1(\Tcal_W)$ {\em admissible} (more precisely, $W$-admissible) if, for any open neighbourhood $\Bcal$ of $\psi$, for any sufficiently large $N$, there is a loop-configuration $\omega\in\Lcal_{\L_N}$ such that $\partial \Xi_{N,R}^{\ssup{\omega,\Scal}}\in \Bcal$.
\end{defn}

Actually, we will later need to apply the lower bound only to $\psi$'s of the form $\partial\Pi_W^{\ssup\Scal}(P)$ with $P\in \Mcal_1^{\ssup{\rm s}}(\Lcal\times \Scal)$.  An important ingredient will then be Lemma~\ref{lem-Psi_admissible}, which says that such a $\psi$ is $W$-admissible if $P$ is $W$-ergodic.

Recall the particle-number operator $\mathfrak N_{\partial W}^{\ssup{\ell,\Scal}}$ from  \eqref{particlenumber} and observe that the expected shred particle number of a configuration $\xi\in \Mcal_1(\Lcal_W\times \Scal_W)$  may be written, if $\psi=\partial\Pi_W^{\ssup\Scal}(\xi)$,
 $$
 \langle  {\mathfrak N}^{\ssup{\ell,\Scal}}_{W},\xi\rangle=\langle  {\mathfrak N}^{\ssup{\ell,\Scal}}_{\partial W},\psi\rangle=\int\psi(\d\mu)\,\sum_{i\in I}l_i,\qquad \mu=\sum_{i\in I}\delta_{(x_i,l_i,y_i)}.
 $$
By $\Bcal_\delta(x)$ we always denote the open $\delta$-ball around $x$ for any $x$ in a metric space. If $x=\psi\in\Mcal_1(\Tcal_W)$, then we use the Prokhorov metric defined in \eqref{Prokhorovmetric}, which induces the weak topology.

The following lemma will later be applied for $\L_N'$ and $Z_{N,R}^\circ$ instead of $\L_N$ and $Z_{N,R}$. Recall that ${\tt Q}^{\ssup{\L_N,{\rm par}}}_M$ is the PPP with intensity measure  given in \eqref{nuMdef}, i.e., with all leg spreads $<M$. Analogously, $\mu_{x,x}^{\ssup{k,M}}$ is the canonical Brownian loop measure starting at $x$ with $k$ particles and all leg spreads $<M$. We recall from Section~\ref{sec-distshredoneloop} that $\mathfrak m_R(B)$ is the number of times that the loop $B$ changes the subbox of radius $R$. By $C_M$ we denote some constant $\in(0,\infty)$ that depends only on $M$ ( and possibly on $d$ and $\beta$).

\begin{lemma}[Lower LDP bound for boundary shred configurations]\label{lem-UpperLDPbound}
Assume the situation and the notation described in Section~\ref{sec-Purpose}. That is, $\L_N$ is a centred box with volume $|\L_N|\sim N/\rho$ for some $\rho\in(0,\infty)$, and is decomposed into the subboxes $z+W_R=z+[-R,R]^d$ for some $R\in(0,\infty)$ with $z\in 2R\Z^d\cap\L_N$. We consider particle boundary condition and indicate it by \lq par\rq. Fix a large parameter $M\in(0,R/2)$ and fix  $\xi\in \Mcal_1(\Lcal_W^{\ssup {M}}\times\Scal_W^{\ssup {M}})$ such that $ \langle  {\mathfrak N}^{\ssup{\ell,\Scal}}_{\partial W},\psi\rangle<\infty$, where $\psi=\partial\Pi_W^{\ssup \Scal}(\xi)\in  \Mcal_1(\Tcal_W)$, and assume that $\psi$ is $W$-admissible.  Then, for any sufficiently large $R$, there is $\psi_N\in \Mcal_1(\Tcal_W)$ that converges to $\psi$ as $N\to\infty$ such that
\begin{equation}\label{lowboundboundaryshreds}
\begin{aligned}
\liminf_{N\to\infty}\frac 1{|\L_N|}&\log{\tt Q}^{\ssup{\L_N,{\rm par}}}_{M}\Big(\partial\Xi^{\ssup\Scal}_{N,R} =\psi_N\Big)\geq -\frac{C_M}{R^{d/2}}+\frac 1{|W|}\langle\psi,\mathfrak p_{W,M}\rangle,
\end{aligned}
\end{equation} 
where we recall \eqref{pMdef} and introduce 
\begin{equation}\label{pmathfrakdef}
\mathfrak p_{W,M}(\mu)=\sum_{i\in I}\log p_{x_i}^{\ssup{W,M}}(l_i,y_i),\qquad \mu=\sum_{i\in I}\delta_{(x_i,l_i,y_i)}.
\end{equation}
\end{lemma}

This lower bound is no surprise, since $\langle\psi,\mathfrak p_{W,M}\rangle$  is nothing but the rate of the product of all the probability densities for the shreds, given they start from some $x_i\in W$, having length $l_i$ and terminating at $y_i\in W^{\rm c}$. The $W$-admissibility of $\psi$ guarantees that there is at least one way to arrange the shred-triple configurations in all the subboxes (each configuration $\mu$ appears in $\approx\frac {|\L_N|}{|W|}\psi(\mu)$ subboxes) in such a way that a global loop configuration arises. A similar upper bound is generally suggesting, but here one would have to upper bound the number of such arrangements, which appears a very difficult task. Since $\langle\psi,\mathfrak p_{W,M}\rangle$ should be roughly of $W$-surface order and therefore negligible in the limit $R\to\infty$, we decided to control only the lower bound so explicitly. Lemma~\ref{lem-lowerboundlogpint} in Section~\ref{sec-lowboundlogpint} is devoted to deriving a lower bound for the last term  on the right-hand side of \eqref{lowboundboundaryshreds}.

\begin{proof}[Proof of Lemma~\ref{lem-UpperLDPbound}.]
Note that the admissibility of $\psi$ guarantees the existence of some $\psi_N\to\psi$ such that $\{\partial\Xi^{\ssup\Scal}_{N,R}=\psi_N\}$ is not empty for all sufficiently large $N$. Picking $M$ large enough (not depending on $N$), we also have that this event has a positive probability under ${\tt Q}^{\ssup{\L_N,{\rm par}}}_{M}$, and we assume that $N$ is so large. We are going to carefully decompose this event and to identify a lower bound for its probability.

We are considering only the shred-configuration part of $\Xi_{N,R}$ under the  PPP in $\L_N$, that is, we neglect all loops that have particles only in one of the subboxes. In other words, we can assume that the intensity measure is equal to $\sum_{k\in\N}\frac 1k\int_{\L_N}\d x \,\mu_{x,x}^{\ssup{k,M,\L_N,{\rm par}}}(\{\mathfrak m_R\geq 1\}\cap\cdot)$, where  $\mu_{x,x}^{\ssup{k,M,\L_N,{\rm par}}}$ is the Brownian loop measure for $k$-particle loops starting at $x$ and having all leg spreads $<M$ with particle-boundary condition in $\L_N$.

From now on we drop the superscript $\Scal$ at $\Xi_{N,R}$ and write $\Xi_{N,R}^{\ssup B}$ for the empirical shred configuration measure coming from just one single path $B \in\Ccal$, then $\Xi_{N,R}=\sum_{k\in\N}\sum_{i=1}^{r_k}\Xi_{N,R}^{\ssup{B^{\ssup {i,k}}}}$ if $B^{\ssup {i,k}}\in\Ccal_k$ is the $i$-th Brownian loop of length $k$, and $r_k$ is the number of loops of length $k$ in the Brownian loop soup. We are going to use that the PPP is realised by a bunch of $r_k$ independent Brownian loops $B_1^{\ssup k},\dots,B_{r_k}^{\ssup k}$ with $k\in\N$ and normalised distribution, where the $r_k$ are picked as Poisson-distributed random variables with parameter $\frac 1k \int_\L\d x \,\mu_{x,x}^{\ssup{k,M,\L_N,{\rm par}}}( \mathfrak m_R\geq 1)$. This gives the formula
\begin{equation}\label{boundarydistr1}
\begin{aligned}
{\tt Q}^{\ssup{\L_N,{\rm par}}}_{M}\big(&\partial\Xi_{N,R}=\psi_N\big)
=\e^{-\sum_k\frac 1k \int_\L\d x \,\mu_{x,x}^{\ssup{k,M,\L_N,{\rm par}}}(\mathfrak m_R\geq 1)}\\
&\quad \sum_{(r_k)_{k\in\N}\in\N_0^\N}\Big(\prod_k \frac{k^{-r_k}}{r_k!}\Big)\Big(\bigotimes_k\big(\widehat \P_{M}^{\ssup {k,\L_N,{\rm par}}}\big)^{\otimes r_k}\Big)\Big(\sum_{k\in\N}\sum_{i=1}^{r_k}\partial\Xi_{N,R}^{\ssup{B^{\ssup {i,k}}}}=\psi_N\Big),
\end{aligned}
\end{equation}
where $\widehat\P^{\ssup{k,\L_N,{\rm par}}}$ is defined in \eqref{Phatkdef} (note that the Poisson parameter $\frac 1k \int_{\L_N}\d x \,\mu_{x,x}^{\ssup{k,M,\L_N,{\rm par}}}( \mathfrak m_R\geq 1)$ is, up to the factor of $\frac 1k$, equal to its total mass) and $\widehat\P_{M}^{\ssup {k,\L_N,{\rm par}}}$ is its restriction to $ \bigcap_{z\in Z_{N,R}}\{\varpi_{W_z}^{\ssup\Scal}\in \Scal_{W_z}^{\ssup{M}}\} $.  The sum on $(r_k)_k$ may be restricted to those that satisfy $|\frac 1{|\L_N|}\sum_k k r_k-\rho_2|\leq C \delta$ for some $C$ (where $\rho_2=\frac 1{|W|}\langle \psi,\mathfrak N^{\ssup\ell}_{\partial W}\rangle$), but this restriction is implicit in the event $\{\sum_{k\in\N}\sum_{i=1}^{r_k}\partial\Xi_{N,R}^{\ssup{B^{\ssup {i,k}}}}\in U_\delta\}$.

It is easy to prove that
$$
\lim_{N\to\infty}\frac 1{|\L_N|}\sum_{k\in\N}\frac 1k \int_{\L_N}\d x \,\mu_{x,x}^{\ssup{k,M,\L_N,{\rm par}}}(\mathfrak m_R\geq 1)=\sum_{k\in\N}\frac 1k  \mu_{0,0}^{\ssup{ k,M}}(\mathfrak m_R\geq 1).
$$
Observe that 
\begin{equation}\label{cRMvanishes}
\sum_{k\in\N}\frac 1k  \mu_{0,0}^{\ssup{ k,M}}(\mathfrak m_R\geq 1)\leq \frac{C_M}{ R^{d/2}},\qquad M,R\in(1,\infty),
\end{equation} 
for some $C_M\in (0,\infty)$, only depending on $M$ (and on $d$ and $\beta$). Indeed, from  the definition of $\mathfrak m_R$ in Section~\ref{sec-distshredoneloop} and the one of $\mu_{0,0}^{\ssup {k,M}}$ in \eqref{nnBBM} we see that the event $\{\mathfrak m_R\geq 1\}$ is empty if $k\leq R/M$, since the spreads $|B_{i\beta}-B_{(i-1)\beta}|$ are $\leq M$ under $\mu_{0,0}^{\ssup {k,M}}$ and hence only the sum on $k>R/M$ remains, and the total mass of $\mu_{0,0}^{\ssup {k,M}}$ is $\leq C k^{-d/2}$; we leave the details to the reader. This explains the first term on the right-hand side of \eqref{lowboundboundaryshreds}. 

Since $R>2M$,  every step from a box $z+W$ with size $\leq M$ can terminate only in that box or its $\boxplus$-neighbouring boxes. By $\boxplus$-neighbours of the origin we mean the members of the set $\{-1,0,1\}^d\setminus\{0\}$; analogously we talk of $\boxplus$-neighbouring boxes in the sense of an $2R$-upscaling of this notion.

Each $B^{\ssup {i,k}}$ contributes to $\partial\Xi_{N,R}$ a vector of triples $(x_j^{\ssup{i,k}},l_j^{\ssup{i,k}},y_j^{\ssup{i,k}})\in W\times\N\times W^{\rm c}$, $j\in[m^{\ssup{i,k}}]$, with a distribution given by the product measure of the $p_{x_j}^{\ssup{k,M,\L,{\rm par}}}(l_j^{\ssup{i,k}},\d y_j)$, where the positive integers $l_j^{\ssup{i,k}}$ sum up over $j$ to $k$, and $x_j= z_{j-1}^{\ssup{i,k}}+w_{j-1}^{\ssup{i,k}}$ and $y_j= z_j^{\ssup{i,k}}+w_j^{\ssup{i,k}}= z_{j-1}^{\ssup{i,k}}+\widetilde w_j^{\ssup{i,k}}$, and $\widetilde w_j^{\ssup{i,k}}= \e_j^{\ssup{i,k}}+w_j^{\ssup{i,k}}\in W^{\rm c}$ and $\e_j^{\ssup{i,k}}= z_j^{\ssup{i,k}}-z_{j-1}^{\ssup{i,k}}$. We have to sum over all $m^{\ssup{i,k}}\in\N$ (since ${\mathfrak m_R}\geq 1$) and all $z_1^{\ssup{i,k}},\dots,z_{m^{\ssup{i,k}}}^{\ssup{i,k}}\in Z_{N,R}$ and integrate over all $w_1^{\ssup{i,k}},\dots,w_{m^{\ssup{i,k}}}^{\ssup{i,k}}\in W$ such that the event $\{\partial \Xi_{N,R} =\psi_N\}$ is met. Each $(z^{\ssup{i,k}}_0, z^{\ssup{i,k}}_1,\dots,z_{m^{\ssup{i,k}}}^{\ssup{i,k}})$ is a cycle, i.e., $z^{\ssup{i,k}}_0=z_{m^{\ssup{i,k}}}^{\ssup{i,k}}$.  We may and shall assume that all the $w_j^{\ssup{i,k}}$ are distinct. For any $j$ such that $z_ {j-1}^{\ssup{i,k}}\in Z_{N,R}$, we know that $\e_j^{\ssup{i,k}}$ is a $\boxplus$-neighbour of the origin; more precisely $z_j^{\ssup{i,k}}$ and $z_{j-1}^{\ssup{i,k}}$ are $\boxplus$-neighbours. Each $w_j^{\ssup{i,k}}$ has a distance $\leq M$ to the boundary of $W^{\rm c}$.

Without loss of generality (see \cite[Example 8.1.6(i)]{Bogachev2}), we may and shall restrict to a $\psi\in \Mcal_1(\Tcal_W)$ (the limit point of the $\psi_N$)  that consists of finitely many atoms only, i.e., it is of the form $\psi=\sum_{\alpha\in I}s_\alpha\delta_{\mu_\alpha}$ with $I$ a finite set and $\sum_\alpha s_\alpha=1$ and $s_\alpha>0$ for all $\alpha \in I$, and $\mu_\alpha=\sum_{\gamma\in J_\alpha }\delta_{(x^{\ssup\alpha }_\gamma, l^{\ssup\alpha }_\gamma,y^{\ssup\alpha }_\gamma)}\in \Tcal_W$ a simple point measure. We may assume that all the $x_\gamma^{\ssup \alpha}$  with $\alpha\in I$ and $\gamma\in J_\alpha$ are distinct; the same for the $y_\gamma^{\ssup \alpha}$.

Let us identify $\partial \Xi_{N,R}^{\ssup{B^{\ssup {i,k}}}}$ for one single loop $B^{\ssup {i,k}}$. For $z\in Z_{N,R}$, we see that $\theta_{z}(\Pi_{W_z}(\delta_{B^{\ssup {i,k}}}))=\sum_{j\in[m^{\ssup{i,k}}]\colon z_{j-1}^{\ssup{i,k}}=z}\delta_{(l_j^{\ssup{i,k}},w_{j-1}^{\ssup{i,k}},\widetilde w_j^{\ssup{i,k}})}$ is equal to the shred configuration that $B^{\ssup {i,k}}$ induces in the subbox $W_z$. Hence, writing ${\bf w}=(w_j^{\ssup{i,k}})_{k,i,j}$ and ${\bf z}=(z _j^{\ssup{i,k}})_{k,i,j}$ and ${\bf l}=(l_j^{\ssup{i,k}})_{k,i,j}$, the superposition
$$
C_z( {\bf z},{\bf w},{\bf l})=\sum_{k\in\N}\sum_{i\in[r_k]}\sum_{j\in[m^{\ssup{i,k}}]\colon z_{j-1}^{\ssup{i,k}}=z}\delta_{(l_j^{\ssup{i,k}},z+w_{j-1}^{\ssup{i,k}},z+\widetilde w_j^{\ssup{i,k}})}\in \Tcal_{W_z}
$$
is the configuration that all the  loops $B^{\ssup {i,k}}$ induce in the subbox $W_z$.  Then
$$
\partial\Xi_{N,R}
=\sum_{k,i}\partial\Xi_{N,R}^{\ssup{B^{\ssup {i,k}}}}
=\frac1{\# Z_{N,R}}
\sum_{z\in Z_{N,R}}\delta_{\theta_{z}(C_z({\bf z},{\bf w},{\bf l}))}\in\Mcal_1(\Tcal_W).
$$
We call every such collection 
\begin{equation}\label{admissible}
(\mu_z )_{z\in Z_{N,R}}=
(C_z ({\bf z},{\bf w},{\bf l}))_{z\in Z_{N,R}}\in \bigtimes_{z\in Z_{N,R}} \Tcal_{W_z}
\end{equation}
{\em admissible}. In other words, a collection $(\mu_z )_{z\in Z_{N,R}}\in \bigtimes_{z\in Z_{N,R}} \Tcal_{W_z}$ of configurations is called admissible if the arrangement of all these $\mu_z$ connects up precisely at all the boundaries of all the $W_z$ and produces a bunch of continuous loops inside $\L_N$.  With this notion, we can write the probability term on the right-hand side of \eqref{boundarydistr1} as follows.
\begin{equation}\label{probreprloops}
\begin{aligned}
\Big(\bigotimes_k&\big(\widehat \P_{M}^{\ssup {k,\L,{\rm par}}}\big)^{\otimes r_k}\Big)\Big(\sum_{k\in\N}\sum_{i=1}^{r_k}\partial\Xi_{N,R}^{\ssup{B^{\ssup {i,k}}}}= \psi_N\Big)
=\int_{\bigtimes_{z\in Z_{N,R}} \Tcal_{W_z}} \1\big\{(\mu_z )_{z\in Z_{N,R}}\mbox{ admissible}\big\}\\
&\quad \1\Big\{\frac 1{N_R}\sum_{z\in Z_{N,R}}\delta_{\theta_z(\mu_z) }= \psi_N\Big\}\,\Big(\bigotimes_k\big(\widehat \P_{M}^{\ssup {k,\L,{\rm par}}}\big)^{\otimes r_k}\Big)\big(\d((\mu_z )_{z\in Z_{N,R}})\big).
\end{aligned}
\end{equation}
Strictly speaking, the indicator on admissibility is superfluous, since the probability measure is concentrated on collections of $\mu_z$ that admit only global configurations of continuous paths. But we are going to calculate now the probability by means of a product measure over $z\in Z_{N,R}$, which ignores admissibility. Only under admissibility, it gives a proper lower bound. For a upper bound, we have to upper estimate the number of admissible collections $(\mu_z)_z$; the remainder is the same. For the lower bound, we rely here on the assumption of admissibility of $\psi_N$, which guarantees the existence of at least one admissible $(\mu_z)_z$.

Let us express what $\Delta=\psi_N$ means, where $\Delta=\frac 1{N_R}\sum_{z\in Z_{N,R}}\delta_{\theta_z(\mu_z) }$. Since $\psi_N\to\psi$, and $\psi$ is a simple point measure at the $\mu_\alpha$'s, this means that, for any small $\delta\in(0,1)$ and all sufficiently large $N$ and all $\alpha\in I$, we have $\Delta(B_\delta(\mu_\alpha))=\psi_N(B_\delta(\mu_\alpha))$, and the latter is equal to some $s_\alpha^{\ssup N}$ that converges to $s_\alpha$ as $N\to\infty$. We pick $\delta$ so small that all the $\delta$-balls around the $x_\gamma^{\ssup\alpha}$'s and the $\delta$-balls  around the $y_\gamma^{\ssup\alpha}$'s are pairwise disjoint. Now observe that, for any $\alpha$,
$$
\begin{aligned}
N_R \Delta(B_\delta(\mu_\alpha))
&=\#\{z\colon \theta_{ z}(\mu_z )\in \Bcal_\delta(\mu_\alpha)\}\\
&=\#\{z\colon \forall \gamma\in J_\alpha  \exists (k,i,j)\colon z_{j-1}^{\ssup{i,k}}=z, w_{j-1}^{\ssup{i,k}}\in \Bcal_\delta(x_\gamma^{\ssup\alpha}), l_j^{\ssup{i,k}}=l_\gamma^{\ssup\alpha},\widetilde w_j^{\ssup{i,k}}\in \Bcal_\delta(y_\gamma^{\ssup\alpha})\}.
\end{aligned}
$$ 
Note that, for any $\alpha\in I$ and $\gamma\in J_\alpha$ and all $(k,i)$, the index $j$ in the description of the set is unique. We denote the set of all $({\bf z},{\bf l}, {\bf w})$ having the property on the right-hand side by $A$.

Recall the notation for certain densities $p_x(l,\cdot)$ from \eqref{pdef} and \eqref{pMdef} (with various super-indices). Now using Lemma~\ref{lem-distboundshred} and the independence of all the $B^{\ssup {i,k}}$, we can write (conceiving $p_x(l,\cdot)$ as a measure instead of a density)
\begin{equation}\label{probrepre}
\begin{aligned}
\Big(\bigotimes_k&\big(\widehat \P_{M}^{\ssup {k,\L,{\rm par}}}\big)^{\otimes r_k}\Big)\Big(\frac 1{N_R}\sum_{z\in Z_{N,R}}\delta_{\mu_z }=\psi_N\,\Big|\, (\mu_z )_{z\in Z_{N,R}}\mbox{ admissible}\Big)\\
&=\sum_{\bf z,l}\int \d{\bf w} \1_{A}({\bf z,l,w})\,\prod_{z\in Z_{N,R}}\prod_{k\in\N}\prod_{i\in[r_k]}\prod_{j\in[m^{\ssup{i,k}}]\colon z=z_{j-1}^{\ssup{i,k}}} p^{\ssup{W,M,\L_N,{\rm par}}}_{z+w_{j-1}^{\ssup{i,k}}}(l_j^{\ssup{i,k}},\d (z+\widetilde w_j^{\ssup{i,k}})).
\end{aligned}
\end{equation}
We may replace the $p$-term  by $p^{\ssup{W,M}}_{w_{j-1}^{\ssup{i,k}}}(l_j^{\ssup{i,k}},\d\widetilde w_j^{\ssup{i,k}})$, since every leg that starts from $W_z$ has spread $\leq M$ and hence this term is invariant under shift by $z$. For these $z$, we obtain a lower bound by estimating
\begin{equation}\label{upperboundp}
p^{\ssup {W,M}}_{w_{j-1}^{\ssup{i,k}}}(l_j^{\ssup{i,k}},\d\widetilde w_j^{\ssup{i,k}}))
=p^{\ssup {W,M}}_{w_{j-1}^{\ssup{i,k}}}(l_\gamma^{\ssup \alpha},\d (\e_j^{\ssup{i,k}}+w_j^{\ssup{i,k}}))
\geq \inf_{x\in \Bcal_\delta(x_\gamma^{\ssup\alpha})}p^{\ssup {W,M}}_x(l_\gamma^{\ssup\alpha},\d (\e_j^{\ssup{i,k}}+w_j^{\ssup{i,k}})),
\end{equation}
if $(k,i,j)$ is such that $w_{j-1}^{\ssup{i,k}}\in \Bcal_\delta(x_\gamma^{\ssup\alpha}), l_j^{\ssup{i,k}}=l_\gamma^{\ssup\alpha}, w_j^{\ssup{i,k}}\in \Bcal_\delta(y_\gamma^{\ssup\alpha})$.

We further have 
\begin{equation}\label{lowboundloopprob}
\begin{aligned}
\mbox{l.h.s.~of }\eqref{probrepre}
&=\sum_{\bf z,l}\int \d{\bf w} \1_{A}({\bf z,l,w})\,\prod_{k\in\N}\prod_{i\in[r_k]}\prod_{j\in[m^{\ssup{i,k}}]} p^{\ssup{W,M}}_{w_{j-1}^{\ssup{i,k}}}(l_j^{\ssup{i,k}},\d \widetilde w_j^{\ssup{i,k}})\\
&\geq \prod_{\alpha\in I} \Big(\prod_{\gamma\in J_\alpha }\inf_{x\in \Bcal_\delta(x_\gamma^{\ssup\alpha})}p^{\ssup  {W,M }}_x(l_\gamma^{\ssup\alpha}, \Bcal_\delta(y_ \gamma^{\ssup\alpha}))\Big)^{s_\alpha^{\ssup N}N_R}\\
&=\e^{-N_R \Ocal(\delta)}\exp\Big(N_R \int \psi(\d \mu)\,\log \P^{\ssup {{\partial W,M}}}_{\inf}(\Bcal_\delta(\mu))\Big),
\end{aligned}
\end{equation}
where $\P^{\ssup {{\partial W,M}}}_{\inf}$ is defined by
\begin{equation}\label{Pmaxdef}
\P^{\ssup {\partial W,M}}_{\inf}(\Bcal_\delta(\mu))=\prod_{i\in I} \inf_{x\in \Bcal_\delta(x_i)}p_x^{\ssup{W,M}}(l_i,\Bcal_\delta(y_i)),\qquad \Bcal_\delta(\mu)= \bigtimes_{i\in I}\big[\Bcal_\delta(x_i)\times \{l_i\}\times \Bcal_\delta(y_i)\big], 
\end{equation}
where $\mu=\sum_{i\in I}\delta_{(x_i,l_i,y_i)}$ and $\Bcal_\delta(z)$ is the ball with radius $\delta$ around $z\in\R^d$.

Note that this lower bound does not depend on $(r_k)_k$. For the lower bound, one can pick one choice of $(r_k)_{k\in\N}$ and lower estimate the sum over $(r_k)_k$ against this single summand. This implies, for some $C\in(0,\infty)$,
$$
\liminf_{N\to\infty}\frac 1{|\L_N|}\log{\tt Q}^{\ssup{\L_N,{\rm par}}}_{M}\big(\partial\Xi^{\ssup\Scal}_{N,R}= \psi_N\big)
\geq -\frac {C_M}{R^{d/2}}-\frac {C \delta}{|W|}+\frac 1{|W|}\int \psi(\d \mu)\,\log \P^{\ssup {{\partial W,M}}}_{\inf}(\Bcal_\delta(\mu)).
$$
Since the left-hand side does not depend on $\delta$, we can make $\delta\downarrow 0$ on the right-hand side and obtain the assertion in \eqref{lowboundboundaryshreds}.
\end{proof}

We now give an existence assertion for admissible boundary-shred configurations $\psi$ in a setting that will be relevant in Section~\ref{sec-finishproof} when we finish the proof of the lower bound in Theorem~\ref{thm-freeenergy}.  We keep using the notation that we reminded on at the beginning of Section~\ref{sec-distshreds}. We call a measure $P\in\Mcal_1(\Lcal\times \Scal)$ {\em ergodic with respect to} $W_R$ or simply $W_R${\em -ergodic}, if it is $W_R$-shift invariant (i.e., $\theta_x(P)=P$ for $x\in 2R\Z^d$) and ergodic under shifts by vectors in $2R\Z^d$.

\begin{lemma}[$W_R$-admissible $\psi$'s]\label{lem-Psi_admissible} Fix $R\in(0,\infty)$  and put $W=W_R=[-R,R]^d$. Pick  $P\in\Mcal_1(\Lcal\times \Scal)$ and assume that $P$ is $W$-ergodic. Then $\psi=\partial\Pi_W^{\ssup \Scal}(P)$ is $W$-admissible.
\end{lemma}

\begin{proof}Put $\xi=\Pi_W(P)\in \Mcal_1(\Lcal_W)\times \Scal_W)$. Extending the definition \eqref{Xidef} to loops and interlacements, we conceive $\Xi_{N,R}=\Xi_{N,R}^{\ssup{\omega,\varpi}}$ as the empirical measure coming from loops and interlacements in $\L_N$ (and $\omega\in\Lcal$ is the loop configuration and $\varpi\in \Scal$ the interlacement configuration). Indeed, it is an empirical measure (discrete mixture) of $W$-shifted copies of $\Pi_W(\omega,\varpi)$, which have distribution $\xi$ under $P(\d(\omega,\varpi))$. Here we use the shift-invariance of $P$, more precisely only the invariance under shifts by vectors $\in 2R\Z^d$. 

According to the ergodic theorem for $P$, for any open neighbourhood $U(\xi)\subset \Mcal_1(\Lcal_W\times \Scal_W)$ of $\xi$,
$$
\lim_{N\to\infty} P(\Xi_{N,R}\in U(\xi))=1.
$$
In particular, the event $\{(\omega,\varpi)\in\Lcal\times\Scal\colon \Xi^{\ssup{\omega,\varpi}}_{N,R}\in U(\xi) \}$ is non-empty for all sufficiently large $N$. This implies that, for any $N$, there is a configuration $(\omega_N,\varpi_N)\in \Lcal_{\L_N}\times \Scal_{\L_N}$ (namely $\omega_N=\Pi_{\L_N}^{\ssup\Lcal}(\omega+\varpi)
$ and $\varpi_N=\Pi_{\L_N}^{\ssup\Scal}(\omega+\varpi)$) such that $\Xi_{N,R}^{\ssup{\omega_N,\varpi_N}}$ (which is equal to $\Xi^{\ssup{\omega,\varpi}}_{N,R}$) lies in $\Ucal(\xi)$. From $(\omega_N,\varpi_N)$, we can construct a loop configuration $ \widetilde \omega_N\in \Lcal_{\L_N}$ as follows. We keep the loop configuration $\omega_N$ unchanged and modify all the shreds in $\varpi_N$ by concatenating all $\L_N$-shreds to a loop by removing the last leg (the only one that leaves $\L_N$) and adding a piece from its termination site in $\L_N$ to its initial site in $\L_N$ that runs within $\L_N$. It is easily seen that the above modification can be done in such a way that the distance between $\Xi_{N,R}^{\ssup {\omega_N,\varpi_N}}$ and $\Xi_{\widetilde N,R}^{\ssup{\widetilde \omega_N}}$ vanishes as $N\to\infty$, so $\Xi_{\widetilde N,R}^{\ssup{\widetilde \omega_N}}$ lies also in $U(\xi)$ for all large $N$.

Now let $\Bcal(\psi)$ be some open neighbourhood of $\psi$, and pick the neighbourhood $\Bcal(\xi)$ in such a way that $\{\partial\Pi_{W}^{\ssup\Scal}(\xi')\colon \xi'\in \Bcal(\xi)\}\subset \Bcal(\psi)$. Therefore, for any large $N$, $\partial\Xi_{\widetilde N,R}^{\ssup{\widetilde \omega_N,\Scal}}$ lies in $ \Bcal(\psi)$. Hence $\psi$ is admissible.
\end{proof}

\subsection{Estimating the LDP lower bound}\label{sec-lowboundlogpint}

In this section, we give in Lemma~\ref{lem-lowerboundlogpint} a further lower bound for the LDP lower bound of Lemma~\ref{lem-UpperLDPbound}, which shows that it is negligible for all the (approximations of) $\psi$'s that we will later apply it to. Furthermore, we use this in order to prove that the expected shred number in a large box under a shift-invariant interlacement-configuration distribution is negligible with respect to the volume of the box, see Corollary~\ref{cor-shrednumberesti}, which needs the preparation in Lemma~\ref{lem-expnumbershreds}.

The idea in the proof of Lemma~\ref{lem-lowerboundlogpint} is that we will be able to control $\langle \psi,\mathfrak p_{W,M}\rangle$ (see \eqref{pmathfrakdef}) only for $\psi$'s that  are concentrated on shred configurations in $W_R$ with legs that have a bounded spread, whose shreds have a bounded spread-out behaviour on longer scales, and that produce a bounded number of particles in any subbox of a given size. More precisely, with large parameters $M\in(1,\infty)$ and $\vartheta,S\in\N$, introduce the set of configurations of shreds whose legs have a spread $<M$ and which satisfy an upper bound for subpieces of length $S$,
\begin{equation}\label{SMvarthetaSdef}
\begin{aligned}
\Scal^{\ssup  {M,\vartheta, S}}
&=\Big\{\sum_{g\in\Gamma}\delta_g\in \Scal\colon \forall g\in \Gamma, \forall i\in[\ell(g)], \| g_i\|_{\rm sp}< M,\\
&\quad \forall g\in \Gamma,\forall j\in [\lfloor \smfrac 1S \ell(g)\rfloor], g(((jS)\wedge \ell(g))\beta)-g((j-1)S\beta)|\leq \vartheta \sqrt S\Big\},
\end{aligned}
\end{equation}
where we remind that $g\in\Ccal_{\ell(g)}$ is a shred of length $\ell(g)$. Write $\Scal_W^{\ssup  {M,\vartheta, S}}=\Pi_W^{\ssup\Scal}(\Scal^{\ssup  {M,\vartheta, S}})$ and $A^{\ssup\Scal}_{W;\Theta}=A^{\ssup\Scal}_{W;M,\mathfrak r,T}\cap \Scal_W^{\ssup  {M,\vartheta, S}}$, where $\Theta=(M,\mathfrak r,T,\vartheta,S)$, then we will assume that $\psi$ is concentrated on $A^{\ssup\Scal}_{W;\Theta}$.

\begin{lemma}[Lower estimates for expectations of $\mathfrak p_{W,M}$]\label{lem-lowerboundlogpint} For any $R,M\in (1,\infty)$, fix a shred configuration measure $\xi\in\Mcal_1(\Scal_{W_R}^{\ssup M})$.

\begin{enumerate}
\item There is a $C_M$ (only depending on $M$, $d$ and $\beta$) such that,  for any $R,M\in (1,\infty)$ and any shred configuration measure $\xi\in\Mcal_1(\Scal_{W_R}^{\ssup M})$,
\begin{equation}\label{finallog_p_integralallgemein}
-\langle\partial\Pi_{W_R}^{\ssup{\Scal}}(\xi),\mathfrak p_{W,M}\rangle \leq C_M\langle \xi,\mathfrak N_{W_R}^{\ssup{\ell,\Scal}}\rangle.
\end{equation}

\item There are constants $C_{M,\mathfrak r,T}$ (only depending on $M,\mathfrak r, T$ and $d$, $\beta$) and $C$ (depending only on $d$ and $\beta$) such that, for any $R,M,\mathfrak r, T, \vartheta\in(0,\infty)$ and $S\in\N$ and for any $\xi\in\Mcal_1(\Scal_{W_R})$ that is concentrated on $A^{\ssup\Scal}_{W_R;\Theta}$, where $\Theta=(M,\mathfrak r,T,\vartheta,S)$,
\begin{equation}\label{finallog_p_integral}
-\langle\partial\Pi_{W_R}^{\ssup{\Scal}}(\xi),\mathfrak p_{W,M}\rangle \leq \Big[\frac{C_{M,\mathfrak r,T}}R+\frac{C\vartheta^2}{S}\Big]\langle \xi,\mathfrak N_{W_R}^{\ssup{\ell,\Scal}}\rangle.
\end{equation}
\end{enumerate}
\end{lemma}

\begin{proof} Let us first derive the lower bound in \eqref{pMlowbound} for $p_{x}^{\ssup{W,M}}(l,i)$, also using that $|x-y|\leq l M$ and that $y\in W_{R+M}\setminus W_R$ and $x\in W_R\setminus W_{R-M}$ for any $(x,l,y)$ in the support of any $\mu=\sum_{i\in I}\delta_{(x_i,l_i,y_i)}$ in the support of $\partial\Pi_{W_R}^{\ssup{\Scal}}(\xi)$. 

Recall that $\tau \beta\in\beta \N$ is the first time $\in\beta\N$ at which the shred leaves $W$. The steps of the underlying random walk $(B_{k\beta})_{k\in\N_0}$ are Gaussian, conditioned on being $<M$, and have therefore a density that is bounded from below (by some $M$-depending constant). Therefore, for any $(x,l,y)\in ( W_R\setminus W_{R-M})\times\N\times (W_{R+M}\setminus W_R)$, using the Markov property at time $\beta(l-1)$, writing $g_M$ for the Gaussian density restricted to $B_M(0)$,
$$
\begin{aligned}
p_{x}^{\ssup{W,M}}(l,y)
&\geq \int_{W\cap B_M(y)} \P_x^{\ssup M}(\tau>l-1, B_{\beta(l-1)}\in \d w)g_M(w-y)\\
&\geq C_M \P_x^{\ssup M}(\tau>l-1, B_{\beta(l-1)}\in W\cap B_M(y)).
\end{aligned}
$$
Now we distinguish the cases where $l\gg R^2$ (where the long stay of the walk in $W$ is a late-escape deviation), $\eps |x-y|^2\leq l\leq C R^2$ (where the walk has so much time to arrive in $W\cap B_M(y)$ that the CLT can be applied) and $\frac {|x-y|}M\leq l\ll |x-y|^2$ (where the sprint to $W\cap B_M(y)$ is a large or moderate deviation, since it makes very long steps).

Indeed, if $l\geq C R^2$ for some $C\in (0,\infty)$, then we can estimate $\P_x^{\ssup M}(\tau>l-1, B_{\beta(l-1)}\in W\cap B_M(y))$ from below against $\e^{-c l/R^2} C_M R^{1-d}$ since the probability for a long path to stay in a given box decays exponentially in the length of the path divided by the square of the radius of the box, according to Brownian scaling; and since the endpoint of such a path has a positive density in $W_R\setminus W_{R-M}$, which has volume $\geq C_M R^{d-1}$. (This asymptotics may be made much more precise with the help of the principal eigenvalue of the Laplace operator in $W_R$, which is $\sim C/R^2$, and the corresponding eigenfunction, but this would be cumbersome.)

Second, if $\eps |x-y|^2\leq l\leq C R^2$ for some small $\eps\in(0,1)$, then, on the event $\{\tau>l-1\}$, the path $(B_{\beta k})_{k\in\{0,1, \dots,l-1\}}$ behaves like a Brownian motion  with diffusive constant $\asymp l$  and arrives in $B_M(y)$ with probability $\geq \e^{-c|x-y|^2/l} C_M$, according to Donsker's invariance principle, where the factor $C_M$ expresses the volume of $B_M(y)$. (Another explanation of the term $\e^{-c|x-y|^2/l}$ is that the Brownian motion path can be divided into $l$ steps of length $\approx |x-y|/l$ each.)

Furthermore, if we consider the case $\frac {|x-y|}M\leq l\leq  \eps |x-y|^2$, then we estimate $\P_x^{\ssup M}(\tau>l-1, B_{\beta(l-1)}\in W\cap B_M(y))\geq \e^{-c |x-y|^2/l}$, according to a large-deviation principle (if $l\asymp |x-y|$), respectively a moderate-deviation principle (if $|x-y|\ll l\leq \eps |x-y|^2$). Indeed, the walk has to make $\asymp l$ steps of sizes $\asymp \frac{|x-y|}l$, and the probability for  making such large steps is approximately Gaussian.

Summarizing the second and third case, we obtain
\begin{equation}\label{pMlowbound}
\log p_{x}^{\ssup{W,M}}(l,y)
\geq 
\begin{cases}
-c \frac l{R^2}+\log (C_M R^{1-d}),&\mbox{if  }l\geq C R^2,\\
-c\frac{|x-y|^2}l-C_M,&\mbox{if  }\frac {|x-y|}M\leq l\leq C R^2.
\end{cases}
\end{equation}

For an arbitrary $\xi\in\Mcal_1(\Scal_{W_R}^{\ssup M})$, abbreviate $\psi=\partial \Pi_{W_R}^{\ssup{\Scal}}(\xi)$. Then the left-hand side of  \eqref{finallog_p_integralallgemein} or  \eqref{finallog_p_integral} is an integral with respect to $\psi(\d\mu)$, where $\mu=\sum_{i\in I}\delta_{(x_i,l_i,y_i)}$. 
Hence, it  can be estimated from above against
\begin{equation}\label{mathfrak_p_esti}
\int\psi(\d\mu)\,\sum_{i\in  I}\Big[\big(c\smfrac {l_i}{R^2}-\log(C_M R^{1-d})\big)\1\{l\geq C R^2\}+\big(c \smfrac{|x_i-y_i|^2}{l_i}+ C_M\big)\1\{  \smfrac{|x_i-y_i|}M\leq l_i\leq C R^2\}\Big].
\end{equation} 
Estimating $c\frac {l_i}{R^2}\1\{l\geq C R^2\}\leq c\frac {l_i}{R^2}$ and using in the next term that $\1\{l\geq C R^2\}\leq \frac l{C R^2}$, the first part can be estimated, for any large $R$, against
$$
\frac {C_M}{R}\int\psi(\d\mu)\,\sum_{i\in  I} l_i=\frac {C_M}{R}\langle \xi,\mathfrak N_{W_R}^{\ssup{\ell,\Scal}}\rangle.
$$

Furthermore, we can simply estimate the second part against $C_M l_i$ and obtain the bound in \eqref{finallog_p_integralallgemein}.

Now assume that $\xi$  is concentrated on $A^{\ssup\Scal}_{W;\Theta}$. We recall that $\sum_{i\in I}\frac{|x_i-y_i|^2}{l_i}$ is indeed integrated with respect to $\xi$, under which $|f_i(jS\beta)-f_i((j+1)S\beta)|\leq \vartheta \sqrt S$ holds for any $j\in\Z$ and for any $i\in I$, where we recall that we wrote $f_i\in\Ccal_{l_i}$ for the $i$-th shred (with $f_i(0)=x_i$ and $f_i(\ell(f_i)\beta)=y_i$), such that we can estimate $|x_i-y_i|\leq \sum_{j=1}^{\lceil l_i/S\rceil}|f_i(jS\beta)-f_i((j-1)S\beta)|\leq \frac{\vartheta l_i}{\sqrt S}$. Therefore, the second part is upper bounded against 
$$
\begin{aligned}
\int\xi(\d\varpi)\,&\sum_{i\in  I} \Big[c\frac{|x_i-y_i|^2}{l_i}+C_M\Big]
\leq \int\xi(\d\varpi)\,\sum_{i\in  I}\Big[ C\frac {\vartheta^2 l_i}{S}+C_M\Big]=\frac{C\vartheta^2}{S}\langle \xi,\mathfrak N_{W_R}^{\ssup{\ell,\Scal}}\rangle+C_M \int\xi( \d\varpi)\,|I|.
\end{aligned}
$$
Recall that $\xi$ is concentrated on $A^{\ssup\Scal}_{W_R;\Theta}$, which in particular means that in any of the boxes $z+W_{\mathfrak r}$ there are no more than $T$ particles in the shreds and each leg has spread $<M$ with $\xi$-probability one. This implies that, with $\xi$-probability one, there are no more than of surface order (i.e., $O(R^{d-1})$) of shreds, and the proportionality factor depends on $M,\mathfrak r$ and $T$. This shows that $C_M \int\xi( \d\varpi)\,|I|\leq C_{M,\mathfrak r,T}R^{d-1}$ for some constant that depends only on $M$, $\mathfrak r$, and $T$ (and on $d$ and on $\beta$). Finally, we absorb the $C_M$ from above in $C_{M,\mathfrak r,T}$, which finishes the proof of \eqref{finallog_p_integral}.
\end{proof}

\begin{remark}[Simpler estimates for Brownian bridges]\label{rem-mathfrakp_esti_for_BM}
It is clear that, under shred-configuration measures $\xi$ whose shreds are independent Brownian bridges (for example $\Pi_W^{\ssup\Scal}({\tt Q})$), there are much simpler upper bounds for $-\langle \xi,\mathfrak p_{W_R,M}\rangle$, but the main point in Lemma~\ref{lem-lowerboundlogpint} is that we need to control expectations under pretty arbitrary $\xi$s and therefore need to put the condition that {\it all} of them lie in $\Scal_W^{\ssup  {M,\vartheta, S}}$ under $\xi$. Indeed, if all the shreds under $\xi$ are Brownian bridges, then the only term in the above proof that needs and offers a simpler argument is $\int \xi(\d\varpi)\,\sum_{i\in I}\frac{|x_i-y_i|^2}{l_i}$, which is easily estimated  by estimating $\frac{|x_i-y_i|^2}{l_i}
\leq M\sum_{j=1}^{l_i/S}|f_i(jS\beta)-f_i((j+1)S\beta)|$, which gives, by the Gaussian nature of $f_i(jS\beta)-f_i((j+1)S\beta)$,
$$
\int \xi(\d\varpi)\,\sum_{i\in I}\frac{|x_i-y_i|^2}{l_i}
\leq M C\sqrt S \int \xi(\d\varpi)\,\sum_{i\in I}\frac {l_i}S
\leq \frac{C_M}{\sqrt S} \langle \xi,\mathfrak N_{W_R}^{\ssup{\ell,\Scal}}\rangle.
$$
\hfill$\Diamond$
\end{remark}

Let us append now some useful estimate for the expected number of interlacement-shreds of shift-invariant interlacement configuration measures in a large box. Together with Lemma~\ref{lem-lowerboundlogpint}, it will lead to a flexible  and useful bound in Corollary~\ref{cor-shrednumberesti} below.

\begin{lemma}[Estimating the expected number of interlacement-shreds]
\label{lem-expnumbershreds} 
There is a $C\in(0,\infty)$ such that, for any $P\in\Mcal_1^{\ssup {\rm s}}(\Scal)$, and for any $R\in\N$,
\begin{equation}\label{expshrednumberesti}
\int \partial\Pi_{W_{R}}^{\ssup\Scal}(P)(\d\mu)\,|\mu|
\leq \frac {C|W_R|}R \Big(J_{W_R}(\Pi_{W_{R}}(P))+\langle P,\mathfrak N_U^{\ssup\ell}\rangle-\frac1{|W_R|}\langle \partial \Pi_{W_R}^{\ssup{\Scal}}(P),\mathfrak p_{W_R}\rangle\Big),
\end{equation}
where we recall the definition of $\mathfrak p$ from \eqref{pmathfrakdef} (with $M=\infty$) and \eqref{pMdef}.

Additionally, if $P\in\Mcal_1^{\ssup {\rm s}}(\Scal^{\ssup M})$ for some $M\in(1,\infty)$, then \eqref{expshrednumberesti} holds with $\mathfrak p_{W_R}$ replaced by $\mathfrak p_{W_R,M}$.
\end{lemma}

\begin{proof} We write $W$ instead of $W_R$. We claim that there is a constant $C\in(0,\infty)$ such that
\begin{equation}\label{shrednumberbound}
\int \partial\Pi_{W}^{\ssup\Scal}(P)(\d\mu)\,|\mu|
\leq \frac CR \int_{\Scal} P(\d \varpi)\,\sum_g\1\{g(0)\in W\}|g(\beta)-g(0)|,\qquad R\in\N,
\end{equation}
where $\sum_g$ is the sum on all legs in the entire interlacement configuration (there are no loops).
In other words, the expected number of $W$-shreds under $\Pi_{W}(P)$ is not larger than $C/R$ times the expected sum of all jump lengths of all the legs with starting sites in $W$. Indeed, our argument for this is as follows. Fix one of the $2d$ hyperplanes that define the boundaries of $W$, say $H=\{x\in\R^d\colon x_1=R\}$. Consider all the $W$-shreds that make a final jump over this hyperplane (i.e., with a starting site with first index $\in[-R,R)$ and a target site with first index in $(R,\infty)$). Put $I_x=[x,x+1)\times [-R,R]^{d-1}$ and call $M(x,y)$ the number of jumps of any leg in the configuration from $I_x$ into $I_y$. For $z\in\Z$, consider $M_z=\sum_{x=z-R}^{z-1}\sum_{y=z}^\infty M(x,y)$, then the the number of shreds with final jump over $H$ is not larger than $M_R$. By shift-invariance of $P$, the expectation of $M_z$ does not depend on $z$. Hence, the expected number of $W$-shreds that make a final jump over $H$ is not larger than $C/R$ times the expectation of 
$$
\sum_{z=-R}^R M_z=\sum_{x,y,z\in\Z\colon -R\leq x\leq z\leq y, z\leq  R} M(x,y)\leq \sum_{x=-R}^R \sum_{y\geq x}|y-x| M(x,y)
\leq \sum_g\1\{g(0)\in W\}|g(\beta)-g(0)|,
$$
where $\sum_g$ is the sum over all legs $\in \Ccal_1$ in the entire configuration. Note that the expectation of $M_z$ (i.e., the one of $M_R$) is not smaller than the left-hand side of \eqref{shrednumberbound}. Summing over all the $2d$ hyperplanes, our argument implies \eqref{shrednumberbound}.

Now we estimate the expectation on the right-hand side of \eqref{shrednumberbound} by using \eqref{entropyformula}. We write the sum over the legs as $\sum_{i,k}\1\{f_i((k-1)\beta)\in W\} |f_i(k\beta)-f_i((k-1)\beta)|$, where $\varpi=\sum_{i\in I}\delta_{f_i}$ is the shred configuration in $W$. We abbreviate the kernel $(\Pi_{W}(P))_{\Tcal_{W}\to\Scal_{W}}(\mu,\cdot)$ by $P(\mu,\cdot)$ and see (by first considering a cut-off version of $\sum_{i,k}|f_i(k\beta)-f_i((k-1)\beta)|$) that
\begin{equation}\label{expectedleglengthsum}
\begin{aligned}
\int & \Pi_{W}^{\ssup\Scal}(P)(\d \varpi)\,\sum_{i,k}\1\{f_i((k-1)\beta)\in W\}|f_i(k\beta)-f_i((k-1)\beta)|\\
&=\int_{\Tcal_{W}} \partial\Pi_{W}^{\ssup\Scal}(P)(\d \mu)\,\int_{\Scal_{W}} P(\mu,\d\varpi)\sum_{i,k}|f_i(k\beta)-f_i((k-1)\beta)|\\
&\leq\int \partial\Pi_{W}^{\ssup\Scal}(P)(\d \mu)\,\Big[H_{\Scal_{W}}(P(\mu,\cdot)\mid {\tt K}_{W}(\mu, \cdot))+\log \int {\tt K}_{W}(\mu,\d \varpi)\Big(\prod_{i,k}\e^{|f_i(k\beta)-f_i((k-1)\beta)|}\Big)\Big]\\
&\leq |W| J_{W}(\Pi_{W}(P))+ \int \partial\Pi_{W}^{\ssup\Scal}(P)(\d \mu)\, \sum_i \log \int q_{x_i,y_i}^{\ssup{W,l_i}}(\d f)\prod_{k=1}^{l_i}\e^{|f(k\beta)-f((k-1)\beta)|},
\end{aligned}
\end{equation}
where we recall the definition of ${\tt K}_W$ in \eqref{kernelK}, and we wrote $\mu=\sum_i \delta_{(x_i,l_i,y_i)}$, as usual. Recalling the definition of $p_{x}^{\ssup W}(l,y)$ from \eqref{pMdef} and \eqref{qMdef} (both with $M=\infty$) and writing $g_\beta$ for the Gaussian density with variance $2\beta$, we further estimate
$$
\begin{aligned}
\int q_{x_i,y_i}^{\ssup{W,l_i}}(\d f)\prod_{k=1}^{l_i}\e^{|f(k\beta)-f((k-1)\beta)|}
&=
\frac 1{p_{x_i}^{\ssup W}(l_i, y_i)}\int_{W^{l_i-1}}\d w_1\dots \d w_{l_i-1}\, \prod_{k=1}^{l_i-1} \big[\e^{|w_{k}-w_{k-1}|}g_\beta (w_k-w_{k-1})\big]\\
&\leq \frac 1{p_{x_i}^{\ssup W}(l_i, y_i)}
C^{l_i-1} C',
\end{aligned}
$$
with $C=\log \E_0[\e^{|B_\beta|}]$ and $C'=\sup_{x\in\R^d}[\e^{|x|}g_\beta(x)]$, and we wrote $w_0=x_i$ and $w_{l_i}=y_i$. We estimate the $C$-terms against $C^{l_i}$ (using our convention that $C$ is a generic constant) and see that the last term on the right-hand side of \eqref{expectedleglengthsum} is not larger than $\int \partial\Pi_{W}^{\ssup\Scal}(P)(\d \mu)\, \sum_i [Cl_i-\log p_{x_i}^{\ssup W}(l_i, y_i)]=C\langle P, \mathfrak N_{W}^{\ssup{\ell,\Scal}}\rangle -\int \partial\Pi_{W}^{\ssup\Scal}(P)(\d \mu)\, \sum_{i\in I} \log p_{x_i}^{\ssup W}(l_i, y_i)$. Now use the shift-invariance of $P$ to see that $\frac 1{|W_R|}\langle P, \mathfrak N_{W}^{\ssup{\ell,\Scal}}\rangle=\langle P, \mathfrak N_{U}^{\ssup{\ell,\Scal}}\rangle$, which finishes the proof (recalling \eqref{pmathfrakdef} for the second term).
\end{proof}

As a consequence, here  is an upper bound for the expected number of interlacement-shreds in a large box $W_R$ that vanishes as $R\to\infty$  for every shift-invariant measure $P$ that has finite particle density and finite entropy $\h^{\ssup{\Lcal,\Scal}}(P)$ (thanks to \eqref{hbound}).

\begin{cor}[Estimating the expected number of interlacement-shreds]\label{cor-shrednumberesti} For any $M\in(1,\infty)$, there is a $C_M\in(1,\infty)$ such that, for any $P\in \Mcal_1^{\ssup{\rm s}}(\Scal^{\ssup M})$ and any sufficiently large $R$, 
$$
\frac 1{|W_R|}\int \partial \Pi_{W_R}^{\ssup\Scal}(P)(\d\mu)\, |\mu|\leq\frac  {1}R \Big(J_{W_R}(\Pi_{W_R}(P))+C_M\langle P,\mathfrak N_{U}^{\ssup{\Scal,\ell}}\rangle\Big).
$$
\end{cor}

\begin{proof} In \eqref{expshrednumberesti} for $M<\infty$, use \eqref{finallog_p_integralallgemein} with $\xi=\Pi_{W_R}(P)$ for estimating the last term and use the shift-invariance of $P$.
\end{proof}

\section{Large deviation analysis: making $N\to\infty$}\label{sec-LDPana}

\noindent In this section, we derive the necessary LDP-framework for the thermodynamic limit as $N\to\infty$ for the empirical measure $\Xi^{\ssup{\omega_{\rm P}}}_{N,R}$ defined at the beginning of Section~\ref{sec-empmeasconf}, where $\omega_{\rm P}$ is the marked PPP under ${\tt Q}^{\ssup{\L_N,{\rm bc}}}$.  In Section~\ref{sec-LDPabstract} we give first some technical preparation and then the basic large deviation principle that we need (see Corollary \ref{cor-LDP}) and in Section~\ref{sec-PropertiesJW} some crucial properties of the rate function, $J_W$, which will be important for the proofs of the main results, Theorems~\ref{thm-freeenergy} and \eqref{thm-specrelent}. In order to apply later the upper bound of that LDP, we need to derive some crucial assertions about compactness, which we do in Section~\ref{sec-Compactness}. In Section~\ref{sec-LDPupperbound} we apply it to derive an upper bound for the limit superior as $N\to\infty$ of $\frac 1{|\L_N|}$ times the logarithm of the right-hand side of \eqref{hatZ2}; in Section~\ref{sec-Nlowbound} we give a lower bound for the limit inferior of $\frac 1{|\L_N|}$ times  the logarithm of  the right-hand side of \eqref{hatZ4}.

As always, we assume that  the situation of Section~\ref{sec-Purpose} is given, i.e., the regular decomposition of the centred box $\L_N$ into subboxes of radius $R$, which possibly depends on $N$ and converges to a given $R\in\N$. Recall the notion of a relative entropy from \eqref{Entropydef}.

\subsection{The conditional large-deviation principle}\label{sec-LDPabstract}

\noindent The empirical measure $\Xi^{\ssup{\omega_{\rm P}}}_{N,R}$ defined in \eqref{Xidef} is a mixture over $z\in Z_{N,R}$ of the   $\Lcal_W\times \Scal_W$-valued random variable pairs $(\theta_z(\omega_{W_z}^{\ssup\Lcal}), \theta_z(\omega_{W_z}^{\ssup\Scal}))$ defined in \eqref{subconfigurations}, and we consider the distribution under the PPP ${\tt Q}^{\ssup{\L_N,{\rm bc}}}$. The first member of these pairs comes from those loops that are entirely contained in some of these  subboxes, while the second one, the shred part, comes from those loops  that are $R$-crossing. If these pairs would be i.i.d.~over $z$ and their distribution would not depend on $N$, then Sanov's theorem (see \cite[Section ]{DZ98}) would immediately give the validity of the LDP with an explicit rate function. But the shred-parts are far from being i.i.d., and the distribution of the pairs depends on $N$ (namely via the boundary condition).

Hence, we need to take care of these two obstacles. The second one is only technical: because of the boundary condition (periodic or particle) in $\L_N$, the distribution in the subboxes depend slightly on $N$ and on $z$. We will handle this problem with the help of the concept of exponential equivalence (which says that two exponentially equivalent sequences satisfy the same LDP, if any, see \cite[Section 4.2.2]{DZ98}) and will be using the following technical lemma, whose proof is straight-forward and is left to the reader.

\begin{lemma}[Perturbed empirical measures are exponentially equivalent]\label{lem-LDPtechnical}
On a Polish space $\Xcal$, assume that an i.i.d.~sequence $(X_z)_{z\in[N]}$ is given  and another sequence $(X_z^{\ssup N})_{z\in[N]}$ that satisfies
\begin{equation}\label{SanovCriterion}
\lim_{N \to\infty}\frac 1N\sum_{z=1}^N \log \E\big[\e^{f(X_z^{\ssup N})}\big] =\log\E[\e^{f(X_1)}],\qquad f\in\Ccal_{\rm b}(\Xcal).
\end{equation}
Then the empirical measures $\frac 1N\sum_{z=1}^N \delta_{X_z^{\ssup N}}$ and $\frac 1N\sum_{z=1}^N \delta_{X_z}$ are exponentially equivalent.
\end{lemma}

In order to apply this, we need to prove that the distributions of the pairs satisfy the criterion in \eqref{SanovCriterion}. 
Recall the Brownian kernel ${\tt K}_W$ defined in \eqref{kernelK} and introduce also the version ${\tt K}_W^{\ssup {\L_N,{\rm bc}}}(\mu,\cdot)\in\Mcal_1(\Scal_W)$, where $q^{\ssup{l_i, W}}$ is replaced by $q^{\ssup{l_i, W,\L_N,{\rm bc }}}$; see \eqref{qdef}. This measure is concentrated on the set of those collections $(f_i)_i$ of $W$-shreds such that $f_i$ is a length-$l_i$ shred with $f_i(0)=x_i$ and $f_i(l_i\beta)=y_i$ for any $i\in I$. 

For any $\xi\in\Mcal_1(\Shreds_W)$ and any $W\Subset \R^d$, we denote by $\xi_{\Tcal_W\to  \Scal_W}$ the kernel from $\Tcal_W$ into $\Shreds_W$ that is defined by letting $\xi_{\Tcal_W\to  \Scal_W}(\mu,\cdot)$ be a regular conditional version of $\xi$ given $\partial\Pi_W^{\ssup\Shreds}(\xi)=\mu$ for $\mu\in\Tcal_W$. Then $\xi$ is equal to $\partial\Pi_W^{\ssup\Shreds}(\xi)\otimes \xi_{\Tcal_W\to  \Scal_W}$.

\begin{lemma}[The shred-configuration kernel]\label{lem-Sanovcritsatisfied} We consider either periodic or particle boundary condition.

\begin{enumerate}

\item For any $N\in\N$ and any $z\in Z_{N,R}$, the map $\Tcal_{W_z}\times  \Gcal^{\ssup\Scal}_{W_z}\ni(\mu,\d\xi)\mapsto {\tt K}_{W_z}^{\ssup {\L_N,{\rm bc}}}(\mu,\d \xi)$ is equal to the kernel $\Pi_{W_z}^{\ssup\Scal}(\LPP^{\ssup{\L_N,{\rm bc}}})_{\Tcal_{W_z}\to  \Scal_{W_z}}$, where we recall that $W_z=z+W$.

\item The empirical measure of the kernels from (1) converges to the kernel ${\tt K}_W$ in the following sense. Assume that $\frac 1{\# Z_{N,R}}\sum_{z\in Z_{N,R}}\delta_{\mu_z}\to \delta_\mu$ in $\Mcal_1(\Tcal_W)$ as $N\to\infty$, then 
\begin{equation}\label{LDPconditionkernel}
\lim_{N\to\infty}\frac 1{\# Z_{N,R}}\sum_{z\in Z_{N,R}}
\log \E_{\mu_z}\big[\e^{f(\theta_z(\Pi_{W_z}^{\ssup\Scal}))}\big]
=\log {\tt K}_W\Big(\mu,\e^{f(\Pi_W^{\ssup\Scal})}\Big),\qquad f\in \Ccal_{\rm b}(\Scal_W),
\end{equation}
where $\E_{\mu_z}[\cdot]={\tt Q}^{\ssup{\L_N,{\rm bc}}}[\cdot\mid \theta_z(\partial\Pi_{W_z}^{\ssup\Scal})=\mu_z]$.

\item   Both under ${\tt Q}^{\ssup{\L_N,{\rm par}}}$ and ${\tt Q}^{\ssup{\L_N,{\rm per}}}$, under the assumption that $\frac 1{\# Z_{N,R}}\sum_{z\in Z_{N,R}}\delta_{\mu_z}\to \delta_\mu$ in $\Mcal_1(\Tcal_W)$ as $N\to\infty$, the
pairs $(\theta_z(\omega_{W_z}^{\ssup\Lcal}), \theta_z(\omega_{W_z}^{\ssup\Scal}))$, $z\in Z_{N,R}$ satisfy \eqref{SanovCriterion}.
\end{enumerate}

\end{lemma} 

\begin{proof}(1) Lemma~\ref{lem-distboundshred}(1) implies that $\LPP^{\ssup{\L_N,\rm bc}}$, conditional on the family $(\mu_z)_{z\in Z_{N,R}}=(\theta_{z}(\partial\Pi^{\ssup{\Scal}}_{W_z}(\omega_{\rm P} )))_{z\in Z_{N,R}}$ of the boundary shred configurations, the shred-configuration family $(\theta_{z}(\Pi^{\ssup{\Scal}}_{W_z}(\omega_{\rm P})))_{z\in Z_{N,R}}$ has the distribution $\bigotimes_{z\in Z_{N,R}}{\tt K}_{W_z}^{\ssup{\L_N,\rm bc}}(\mu_z,\cdot)$. By superposition of all the independent loops of the loop soup, it implies also that, for any $z$, the conditional distribution of the family of shred configurations in $W_z$ given the boundary shred configuration $\mu_z$ is given by the product measure of the $q_{x_i,y_i}^{\ssup{l_i,W_z,\L_N,{\rm bc}}}$.

(2) For \lq$\rm bc$\rq\ equal to particle boundary condition, according to  Lemma~\ref{lem-distboundshred}(2), the boundary condition has no effect, since  all particles involved in the definition of $q$ lie in $W$ anyway, i.e., $q^{\ssup{W,\L_N,{\rm bc}}}=q^{\ssup W}$. Hence, by Assertion (1), \eqref{LDPconditionkernel} is trivially true without taking the limit. However, for periodic boundary condition, one needs to prove that ${\tt K}_W^{\ssup{\L_N,{\rm bc}}}(\mu,\cdot)$ is close to ${\tt K}_W(\mu,\cdot)$, i.e., $q_{x_i,y_i}^{\ssup {l_i, W,\L_N,{\rm bc}}}$ is close to $q_{x_i,y_i}^{\ssup {l_i, W}}$ in the sense of \eqref{LDPconditionkernel}. We leave this to the reader, see the hints in the proof of (3).

(3)  Observe that, since  the empirical measures of the $\mu_z$ are given and converge, the loop part and the shred part are independent, and the necessary assertion for the shred part is in (2). The proof for the corresponding loop part is similar, hence we restrict ourselves to giving just some hints and leave the details to the reader.

Note that we can estimate, for some $C\in(0,\infty)$, only depending on $R$ (and on $\beta$ and on $d$), for any sufficiently large $N$,
\begin{equation}\label{periodic_p_esti}
p_x^{\ssup{W,\L_N,{\rm per}}}(l,y)\leq p_x^{\ssup{W}}(l,y)\big[1+\e^{-C N^{2/d}/l}\big],\qquad x\in W,l\in\N,y\in W^{\rm c}.
\end{equation}
The reason for this is the following. According to \cite[Theorem 6.3.8]{BR97}, there are, for any $K\in(0,\infty)$, constants $C_1,C_2, C_3$ such that the solution $g^{\ssup{\L,\rm per}}_t(x,y)$ to the heat equation with the Laplace operator with periodic boundary condition in the box $\L$ satisfies
\begin{equation}\label{kernelperiodicesti}
0\leq g^{\ssup{\L,\rm per}}_t(x,y)-g_t(x,y)\leq C_1 g_{t/C_2}(x,y) \,\e^{-C_3 \dist(y,\partial\L)^2/t},\qquad x,y\in \L,t\in (0, K],
\end{equation}
where $g_t$ denotes the free fundamental solution, i.e., the standard Gaussian kernel with variance $2t$. We leave the remainder of the details to the reader.
\end{proof}

Let us turn to the first of the two obstacles. This is more serious. Indeed, the first members of the pairs are independent over $z$ and (ignoring the small dependence on $z$ and $N$) also identically distributed. But the second ones are not at all, since neighbouring subboxes have a strong dependence, and this dependence has very long correlations.  Nevertheless, there is a way out of that, and this is a conditioning on the boundary configurations. The idea is that, under this conditioning, the configurations inside the boxes are independent with a distribution that depends only on the boundary configuration (by the virtue of the Markov property of the Brownian motion). More precisely, we condition on the boundary empirical measure of $\Xi_{N,R}$, namely
\begin{equation}\label{partialXidef}
\partial \Xi^{\ssup\Scal}_{N,R}=\partial \Pi_W^{\ssup\Scal}\big(\Xi_{N,R}\big)=\frac1{\# Z_{N,R}}\sum_{z\in Z_{N,R}}\delta_{\partial\Pi_W^{\ssup{\Scal}}(\theta_{z}\varpi_{W_z}^{\ssup{\Scal}})}\in\Mcal_1(\Tcal_W)
\end{equation}
and obtain that $\Xi_{N,R}$ is indeed an empirical measure of independent random variables, and we also have sufficient information about their distributions.

 In order to catch the large deviation behaviour of the conditional distribution of $\Xi_{N,R}$, we need to employ an LDP of the following type.  The following LDP is taken from \cite[Thm.~17]{PRV13}; it goes back to an unpublished manuscript by C.~L\'eonard; see \cite[Thm.~1]{ADPZ11} for a conditional version of this LDP.

\begin{lemma}[Sanov-type LDP for type-dependent independent random variables]\label{lem-AbstractLDP}
Let $\Xcal, \Ycal$ be two Polish spaces and pick $\rho\in\Mcal_1(\Xcal)$ and let $(x_i^{\ssup n})_{i\in\N}$ be a sequence in $\Xcal$ such that $\frac 1n \sum_{i=1}^n\delta_{x_i^{\ssup n}}$ converges weakly towards $\rho$ as $n\to\infty$. Let $\zeta\colon \Xcal\times \sigma(\Ycal)\to[0,1]$ be a continuous Markov kernel from $\Xcal$ to $ \Ycal$, where $\sigma(\Ycal)$ denotes the Borel sigma field on $\Ycal$. Furthermore, let $(Y_i^{\ssup n})_{i\in[n]}$ have the distribution $\bigotimes_{i\in [n]} \zeta(x_i^{\ssup n},\cdot)$. Then the empirical pair measure $\frac 1n\sum_{i=1}^n \delta_{(x_i^{\ssup n},Y_i^{\ssup n})}$ satisfies an LDP with rate function
\begin{equation}\label{RFabstract}
\Mcal_1(\Xcal\times\Ycal)\ni q\mapsto \begin{cases} H_{\Xcal\times\Ycal}(q\,|\,\rho\otimes \zeta)&\mbox{if }\pi_1(q)=\rho,\\
+\infty&\mbox{otherwise,}
\end{cases}
\end{equation}
where $\pi_1$ is the canonical projection $\Xcal\times\Ycal\to\Xcal$, and $\pi_1(q)$ is the corresponding image measure.
\end{lemma}

We recall that $\rho\otimes\zeta$ is defined by $\rho\otimes \zeta(\d(x,y))=\rho(\d x)\zeta(x,\d y)$.

Based on Lemma~\ref{lem-AbstractLDP}, we can give a conditional LDP for $\Xi_{N,R}$ for a fixed box $W=W_R=[-R,R]^d$. We need to take care of  a consistency property (or marginal property) that  $\Xi_{N,R}$ asymptotically has: the rate function is equal to $\infty$ in any configuration measures $\xi\in\Mcal_1(\Lcal_W\times\Scal_W)$ that cannot be extended to the larger box $3^k W$ for some $k\in\N$ such that the extension is $W$-shift invariant (meaning that projection of the extension to any subbox $W_z=z+W$ of $3^kW$, $z\in 2R\Z^d$, is equal to $\xi$ after shifting back to $W$. (This property is analogous to the marginal  property of the rate function for the empirical measures $\frac 1N\sum_{k=1}^N \delta_{(M_{k-1},M_k)}$ of a Markov chain $(M_n)_{n\in\N}$, whose two marginal measures are equal in the limit $N\to\infty$, and therefore the rate function is equal to $\infty$ if this property is not satisfied; see \cite[Section 3.1.3]{DZ98}.) Here is a precise definition. We introduce the set $\Rcal_{3^k,W}$ of rectangles $\subset 3^kW$ of the form $\bigcup_{z\in 2R Z}z+W$ for some $Z\subset\Z^d$. Then we introduce the set of $W$-shift invariant configurations in $3^k W$,
\begin{equation}\label{MvcalkdefWshift}
\begin{aligned}
\widetilde \Mcal_1^{\ssup k}&=\Big\{\xi^{\ssup k}\in\Mcal_1(\Lcal_{3^kW}\times\Scal_{3^kW})\colon\forall \widetilde W\in\Rcal_{3^k,W}\forall z\in 2R\Z^d\colon\\
&\qquad z+\widetilde W\in\Rcal_{3^k,W}\Longrightarrow \theta_{z}(\Pi_{z+\widetilde W}(\xi^{\ssup k}))=\Pi_{\widetilde W}(\xi^{\ssup k})\Big\}
\end{aligned}
\end{equation}
and the set 
\begin{equation}\label{Mvcalkdef}
\begin{aligned}
\Mcal_1^{\ssup k}(\Lcal_W\times \Scal_W)
&=\big\{\xi\in\Mcal_1(\Lcal_W\times\Scal_W)\colon\exists \xi^{\ssup k}\in\widetilde \Mcal_1^{\ssup k}\colon \Pi_{3^k W \to W}(\xi^{\ssup k})=\xi \big\}=\Pi_{3^kW\to W}(\widetilde \Mcal_1^{\ssup k}),
\end{aligned}
\end{equation}
where we remind that $\Pi_{3^kW\to W}\colon \Mcal_1(\Lcal_{3^kW}\times \Scal_{3^kW})\to \Mcal_1(\Lcal_{W}\times \Scal_{W})$ is the well-known projection operator from \eqref{PiWdef}, here with domain $\Mcal_1(\Lcal_{3^kW}\times \Scal_{3^kW})$.

\begin{lemma}[LDP for $\Xi_{N,R}$, conditional on $\partial \Pi_W^{\ssup\Scal}(\Xi_{N,R})$]\label{lem-LDPforXi}
Fix a window $W=[-R,R]^d$ and any $\psi\in\Mcal_1(\Tcal_W)$.  Then, under $\LPP^{\ssup{\L_N,{\rm bc}}}$ conditioned on $\{\partial\Pi_W^{\ssup \Scal}(\Xi_{N,R})=\psi_N\}$ for some $\psi_N\to\psi$, the sequence $(\Xi_{N,R})_{N\in\N}$ satisfies an LDP on $\Mcal_1(\Lcal_W\times\Scal_W)$ with speed $|\L_N|$ with rate function
\begin{equation}\label{RFinW}
J_{W}^{\ssup {\psi}}(\xi) =\begin{cases}\frac 1{|W|} H_{\Lcal_W\times\Scal_W}\big(\xi\,\big|\,\Pi_W^{\ssup\Lcal}({\tt Q})\otimes [\psi\otimes {\tt K}_{W}]\big)&\mbox{if }\partial\Pi_W^{\ssup\Scal}(\xi)=\psi\mbox{ and }\xi\in\bigcap_{k\in\N}\Mcal_1^{\ssup k}(\Lcal_W\times\Scal_W),\\
+\infty&\mbox{otherwise.}
\end{cases}
\end{equation}
\end{lemma}

\begin{proof} First  we show that the rate function is $\infty$ outside $\Mcal_1^{\ssup k}(\Lcal_W\times\Scal_W)$, for any $k\in\N$.  Clearly, $\Mcal_1^{\ssup k}(\Lcal_W\times\Scal_W)$ is closed, since all the projection operators appearing in the definition of $\Mcal_1^{\ssup k}(\Lcal_W\times\Scal_W)$ are continuous. 

Consider the measure
\begin{equation}\label{XiNkdef}
\Xi_{N,R}^{\ssup k}(\omega)=\frac1{\#Z_{N,R}}\sum_{z\in Z_{N,R}}\delta_{\theta_z(\Pi_{z+3^kW}(\omega))}\in\Mcal_1(\Lcal_{3^kW}\times\Scal_{3^kW}),\qquad \omega\in \Lcal\times\Scal.
\end{equation}
This is the empirical measure of the loop/shred configuration in the boxes $z+3^kW$ with $z\in Z_{N,R}=\{z\in 2R\Z^d\colon z+W\subset\Lambda_N\}$ (these boxes overlap each other significantly by disjoint unions of shifted copies of $W$), shifted to the origin by $z$. (It is of no importance for this proof if we replace $Z_{N,R}$ by $\{z\in 2R\Z^d\colon z+3^kW\subset\Lambda_N\}$ or not, since we only rely on the limit as $N\to\infty$.)

Up to boundary effects, its $\Pi_{3^kW\to W}$-projection is equal to $\Xi_{N,R}(\omega)$, by which we mean that
 \begin{equation}\label{TVconvergence}
\sup_\omega{\d}_{\rm TV}\big(\Pi_{3^kW\to W}(\Xi_{N,R}^{\ssup k}(\omega)),\Xi_{N,R}(\omega)\big)\to 0,\qquad N \to\infty,
\end{equation}
where ${\d}_{\rm TV}$ denotes the total-variation distance on $\Mcal_1(\Lcal_W\times \Scal_W)$. Similarly,
 \begin{equation}\label{TVconvergence2}
\sup_\omega{\d}_{\rm TV}\big(\theta_{z}(\Pi_{z+\widetilde W}(\Xi_{N,R}^{\ssup k}(\omega))),\Pi_{\widetilde W}(\Xi_{N,R}^{\ssup k}(\omega))\big)\to 0,\qquad N \to\infty,z\in \Z^d,\widetilde W,z+\widetilde W\in\Rcal_{3^k,W}.
\end{equation}
This means that the distance between $\Xi_{N,R}^{\ssup k}$ and $\widetilde \Mcal_1^{\ssup k}$ vanishes as $N\to\infty$, uniformly in $\omega$. By continuity of $\Pi_{3^kW\to W}$, also the distance between $\Pi_{3^kW\to W}(\Xi_{N,R}^{\ssup k})$ and $\Pi_{3^kW\to W}(\widetilde \Mcal_1^{\ssup k})$ vanishes as $N\to\infty$. Combining with \eqref{TVconvergence}, we see that the distance between $\Xi_{N,W}$ and $\Pi_{3^kW\to W}(\widetilde \Mcal_1^{\ssup k})$ vanishes as $N\to\infty$, uniformly in $\omega$.

As a consequence, for any open set $ \Ucal\subset \Mcal_1(\Lcal_W\times \Scal_W)$ that entirely lies in the complement of $\Mcal_1^{\ssup k}(\Lcal_W\times \Scal_W)=\Pi_{3^kW\to W}(\widetilde \Mcal_1^{\ssup k})$ and has a positive distance to it, the probability of $\{\Xi_{N,R}\in \Ucal\}$ under ${\tt Q}^{\ssup {\L_N,{\rm bc}}}$ is zero for all large $N$, since the event is empty. Indeed, the distance  of $\Ucal$ to $\Pi_{3^kW\to W}(\widetilde\Mcal_1^{\ssup k})$ is positive, but $\Xi_{N,R}$ comes arbitrarily close to it. This implies that the rate function is equal to $+\infty$ outside $\Mcal_1^{\ssup k}(\Lcal_W\times \Scal_W)$ for any $k\in\N$.

Now we use Lemma~\ref{lem-AbstractLDP} to prove the LDP and to identify the rate function. We conceive the state space $\Lcal_W\times \Scal_W$ as the product $\Tcal_W\times (\Lcal_W \times \Scal_W)$, which is meant as the set of those pairs $(\mu,(\omega,\varpi))$ that are compatible in the sense that $\partial\Pi_W^{\ssup\Scal}(\omega,\varpi)=\mu$. That is, the $\Lcal$-part of $\xi$ does not depend on $\mu$, and $\mu$ is the configuration of all the triples of initial sites, lengths and terminal sites of the shreds of $\varpi$. Hence, in Lemma~\ref{lem-AbstractLDP}, $\Xcal=\Tcal_W$ and $\Ycal=\Lcal_W \times \Scal_W$, and the projection $\pi_1$ there is called $\partial\Pi_W^{\ssup\Scal}$ in our situation. 

Now consider $\omega \in \Lcal_{\L_N}$, distributed under $ \Pi_W^{\ssup \Lcal}({\tt Q}^{\ssup{\L_N,{\rm bc}}})$. Conditionally given $(\theta_z(\partial\Pi_{W_z}^{\ssup\Scal}(\omega)))_{z\in Z_{N,R}}$, the family of pairs $(\theta_z(\omega_{W_z}^{\ssup\Lcal}), \theta_z(\omega_{W_z}^{\ssup\Scal}))_{z\in Z_{N,R}}$ has the distribution $\bigotimes_{z\in Z_{N,R}} \zeta_z(\partial\Pi_W^{\ssup\Scal}(\omega),\cdot)$, where $\zeta_z(\mu,\cdot)$ is, for general $\mu$, the product measure of the distribution $\Pi_W^{\ssup \Lcal}({\tt Q}^{\ssup{\L_N,{\rm bc}}})$ (not depending on $ \mu$) and $\P_{\mu}^{\ssup {W_z,\L_N,{\rm bc}}}$ (see Lemma~\ref{lem-Sanovcritsatisfied}(3)).  We are under the assumption that the empirical measure of the $\theta_z(\partial\Pi_{W_z}^{\ssup\Scal}(\omega))$ (namely $\partial\Xi^{\ssup\omega}_{N,R}$) converges towards $\psi$.  If $\zeta_z(\partial\Pi_W^{\ssup\Scal}(\omega),\cdot)$ would not depend on $z$ nor on $N$ (it does through our boundary condition in $ \Pi_W^{\ssup \Lcal}({\tt Q}^{\ssup{\L_N,{\rm bc}}})$), i.e., if it would be just equal to $\zeta(\mu, \d(\omega,\varpi))= \Pi_W^{\ssup \Lcal}({\tt Q})(\d\omega){\tt K}_W(\mu,\d\varpi)$ (see \eqref{kernelK}), then Lemma~\ref{lem-AbstractLDP} would be directly applicable, and we would obtain that the empirical measure of these pairs (and this is $\Xi^{\ssup\omega}_{N,R}$) satisfies an LDP with the rate function in \eqref{RFinW}. Indeed, the reference measure in \eqref{RFinW} is equal to $\psi\otimes \zeta(\d\mu,\d(\omega,\varpi))=\psi(\d \mu)\Pi_W^{\ssup \Lcal}({\tt Q})(\d\omega){\tt K}_W(\mu,\d\varpi)=\Pi_W^{\ssup \Lcal}({\tt Q})\otimes[\psi\otimes {\tt K}_W](\d(\omega,\varpi))$. Now apply Lemmas~\ref{lem-Sanovcritsatisfied} and \ref{lem-AbstractLDP} to finish the proof.
\end{proof}

Hence, it is natural to restrict to the set of $\xi$'s and $\psi$'s to those that satisfy $\partial\Pi_W^{\ssup\Scal}(\xi)=\psi$ and to use only the rate function
\begin{equation}\label{JWdef}
J_{W}(\xi)=J_{W}^{\ssup{\partial \Pi_W^{\ssup\Scal}(\xi)}}(\xi)
=\frac 1{|W|} H_{\Lcal_W\times\Scal_W}\big(\xi\,\big|\,\Pi_W^{\ssup\Lcal}({\tt Q})\otimes [\partial \Pi_W^{\ssup\Scal}(\xi)\otimes {\tt K}_{W}]\big)
,\qquad \xi\in \bigcap_{k\in\N}\Mcal_1^{\ssup k}(\Lcal_W\times \Scal_W).
\end{equation}
Note that $\Pi_W(P)$ lies automatically in $\bigcap_{k\in\N}\Mcal_1^{\ssup k}(\Lcal_W\times \Scal_W)$ for $P\in \Mcal_1^{\ssup{\rm s}}(\Lcal\times\Scal)$, hence for considering $J_W(\Pi_W(P))$ later we will not need to take extra care of the restriction.

Let us formulate Lemma \ref{lem-LDPforXi} in a way in which we will actually apply it later. For the lower bound, because of our approach in Lemmas~\ref{lem-LowerBound} and \eqref{lem-lowerboundlogpint}, we need to condition the integration to the event $\{\Xi_{N,R}^{\ssup\circ}(A_{W;\Theta})=1\}$, where we collect all the cutting parameters in the tuple $\Theta=(M,L,K,\mathfrak r,T, S, \vartheta)$ and summarize the definitions in \eqref{ALcaldef}, \eqref{AScaldef} and \eqref{Adef} and \eqref{SMvarthetaSdef} by putting $A_{W;\Theta}=A_{W;M,\mathfrak r, K,L,T}^{\ssup\Lcal}\times (A_{W;M,\mathfrak r, T}^{\ssup\Scal}\cap\Scal^{\ssup{M,\vartheta,S}})$. By ${\tt Q}^{\ssup{\L_N,{\rm par}}}_{W;\Theta}$ we denote probability with respect to the PPP with intensity measure $\nu^{\ssup{\rm par}}_{\L_N,W;\Theta}$, which is defined as the restriction of $\nu_{\L_N}^{\ssup{{\rm par}}}$ defined in \eqref{nudef} to the event $\bigcap_{z\in Z_{N,R}}\{(\omega_{W_z},\varpi_{W_z})\in A_{W_z;\Theta}\}$.

\begin{cor}[Upper and lower LDP bounds for $\Xi_{N,R}$]\label{cor-LDP}
Fix $R\in\N$ and $W=W_R=[-R,R]^d$ and a function $F\colon\Mcal_1(\Lcal_W\times\Scal_W)\to\R$ and a measurable set $B\subset\Mcal_1(\Lcal_W\times\Scal_W)$.

\begin{enumerate}
\item If $F$ is lower semi-continuous and bounded from below and $B$ closed and $B'\subset \Mcal_1(\Tcal_W)$ compact, then
$$
\begin{aligned}
\limsup_{N\to\infty}\frac 1{|\L_N|}\log {\tt Q}^{\ssup{\L_N,{ \rm bc }}}&\Big[\e^{-|\L_N| F(\Xi_{N,R})}\1\{\Xi_{N,R}\in B\}\{\partial\Xi_{N,R}^{\ssup\Scal}\in B'\}\Big]\\
&\leq -\inf_{\xi\in B\cap\bigcap_{k\in \N}\Mcal_1^{\ssup k}(\Lcal_W\times  \Scal_W)\colon \partial\Pi_W^{\ssup\Scal}(\xi)\in B'}\big(F(\xi)+J_{W}(\xi)\big).
\end{aligned}
$$
\item Fix all the parameters $M,L,K,\mathfrak r,T,\vartheta,S\in(0,\infty]$ and put $\Theta=(M,L,K,\mathfrak r,T,\vartheta, S)$. Then, for any $W$-admissible $\psi\in\Mcal_1(\Tcal_W)$ and for any sequence $\psi_N\to\psi$ such that ${\tt Q}^{\ssup{\L_N,{ \rm bc }}}_{W;\Theta}(\partial\Pi_W^{\ssup\Scal}(\Xi_{N,R})=\psi_N)>0$ for any $N$, if $F$ is upper semi-continuous and $B$ is subset of and open in $A_{W;\Theta}$,
$$
\begin{aligned}
\liminf_{N\to\infty}\frac 1{|\L_N|}&\log {\tt Q}^{\ssup{\L_N,{ \rm bc }}}_{W;\Theta}\Big[\e^{-|\L_N| F(\Xi_{N,R})}\1\{\Xi_{N,R}\in B\}\,\Big|\,\partial\Pi_W^{\ssup\Scal}(\Xi_{N,R})=\psi_N\Big]\\
&\geq -\inf_{\xi\in B\colon \partial\Pi_W^{\ssup\Scal}(\xi)=\psi}\big(F(\xi)+J_{W;\Theta}(\xi)\big),
\end{aligned}
$$
where  
\begin{equation}\label{JWMTdef}
J_{W;\Theta}(\xi)=\frac 1{|W|} H_{\Lcal_W\times\Scal_W}\big(\xi\,\big|\,\Pi_W^{\ssup\Lcal}({\tt Q}_{M})_{L,K}\otimes [\partial \Pi_W^{\ssup\Scal}(\xi)\otimes {\tt K}^{\ssup{M,\mathfrak r, T,\vartheta,S}}_{W}]\big)
,\qquad \xi\in\Mcal_1(\Lcal_W\times \Scal_W),
\end{equation}
where $\Pi_W^{\ssup\Lcal}({\tt Q}_{M})_{L,K}$ is the conditional version of $\Pi_W^{\ssup\Lcal}({\tt Q}_M)$ given $A_{W; M,L,K}^{\ssup\Lcal}$, and  ${\tt K}^{\ssup{M,\mathfrak r, T,\vartheta,S}}_{W}(\mu,\cdot)$ is the conditional version of ${\tt K}_{W}(\mu,\cdot)$ given $A_{W;M,\mathfrak r, T}^{\ssup\Scal}\cap\Scal_W^{\ssup{M,\vartheta,S}}$. (In particular, $J_{W;\Theta}(\xi)=\infty$ if $\xi(A_{W;\Theta})<1$.)
\end{enumerate}
\end{cor}

The upper bound also motivates the following Section~\ref{sec-Compactness}: we will need good compactness criteria on $\Mcal_1(\Tcal_W)$. The lower bound will later be applied for $\psi=\partial \Pi_{W_R}^{\ssup\Scal}(P)$ with an ergodic $P\in\Mcal_1^{\ssup{\rm s}}(\Lcal\times \Scal)$, evoking Lemma~\ref{lem-Psi_admissible}.

\begin{proof}
 (1) follows, by the virtue of the arguments in the proof of  Varadhan's lemma (see \cite[Section 4.3]{DZ98}) from Lemma~\ref{lem-LDPforXi} by dropping all probability terms that concern $\partial\Pi_W^{\ssup\Scal}(\Xi_{N,R})$. However, for the application of Varadhan's lemma, we need to ensure that the rate function $J_W$ has compact level sets, and this is not true since its reference measure  has no autonomous $\Mcal_1(\Tcal_W)$-part. But, by the proof of Lemma~\ref{lem-exponentiallytight}(1), the intersection of the level set with $\{\xi\colon \partial \Pi_W^{\ssup\Scal}(\xi)\in B'\}$ is compact, which suffices for the upper-bound part of Varadhan's lemma.
 
 \medskip
 
To prove (2), again by the virtue of the arguments in the proof of  Varadhan's lemma, we only need to observe that the proof of Lemma~\ref{lem-LDPforXi} applies also for the restricted and conditional version ${\tt Q}^{\ssup{\L_N,{ \rm bc }}}_{W;\Theta}$ instead of ${\tt Q}^{\ssup{\L_N,{ \rm bc }}}$ and that $J_{W;\Theta}(\xi)=\infty$ if $\partial\Pi_W^{\ssup\Scal}(\xi)\not=\psi$. 
\end{proof}

\subsection{Properties of $J_W$}\label{sec-PropertiesJW}

We are going to collect some crucial properties of $J_W$ defined  in \eqref{JWdef}, in particular of $J_{W_R}(\Pi_{W_{R}}(P))$ for shift-invariant distributions $P$ of loops and interlacements. These properties will be essential for all the future developments. Our main result here is Proposition~\ref{lem-uppboundJ_W}; it has four assertions.

In the first two assertions, we show that the $J$-value in a box $mW$ (with $m\in  \N$) of a given configuration measure $\xi$ is not smaller than the $J$-value in $W$ of a kind of convex combination of parts of $\xi$. This will be used in the proof of Theorem~\ref{thm-specrelent} when we show that the limit inferior of $J_{W_R}(\Pi_{W_R}(P))$ as $R\to\infty$ is large enough for any $P\in\Mcal_1(\Lcal\times\Scal)$. The third assertion is a technical help here in form of a certain lower-semicontinuity property. The fourth assertion is in a sense dual to the first one, as it solves the question how to construct from a configuration measure $\xi$ in a box $W$ a configuration measure $\xi^{\ssup k}$ in the box $3^k W$ (in fact, an extension of $\xi$) such that its $J$-values do not significantly increase if $k$ gets large. (This will be crucial in the proof of the upper bound in Theorem~\ref{thm-freeenergy}, when we have to make a connection from the variational problem in $W_R$ to the $\R^d$-version of it, $\chi(\rho_1,\rho_2)$.) Here we construct rather comfortable extensions $\xi^{\ssup k}$ of $\xi$ which lie in the nice set $\widetilde \Mcal_1^{\ssup k}$ of $W$-shift invariant measures defined in \eqref{MvcalkdefWshift}; in particular they are extensions of each other.

The main body of our arguments relies on LDP-arguments instead of the explicit entropy formula in \eqref{JWdef}. One of the biggest difficulties comes from the fact that extensions of point measures on shreds are much more complicated than extensions of point measures of loops, and also the latter is not at all clear in our setting, since our $W$-loops have all their particles in $W$: an extension to a larger box $\widetilde W$ should also produce loops that have particles both in $W$ and $\widetilde W$. Our LDP-arguments guarantee also that.

For $m\in\N$, denote by $Z_{m,R}$ the set of those $y\in 2R \Z^d$ such that $W_{mR}=mW_R$ is equal to the pairwise disjoint (up to boundaries)  union of all the $y+W_R$ with $y\in Z_{m,R}$. Furthermore, consider the map
\begin{equation}\label{fmdef}
f_{m,R}\colon \Mcal_1(\Lcal_{W_{mR}}\times \Scal_{W_{mR}}) \to \Mcal_1(\Lcal_{W_{R}}\times \Scal_{W_{R}}),\qquad f_m(\xi)= \frac 1{\# Z_{m,R}}\sum_{z\in Z_{m,R}}\theta_z(\Pi_{W_{mR}\to z+W_R}(\xi)).
\end{equation}
Recall the notion of admissibility from Definition~\ref{def-admisssible} and the definitions of $\Lcal^{\ssup M}$ and $\Scal^{\ssup M}$ from \eqref{LoopsMvarthetaS} and \eqref{SMdef}, respectively. Recall the definition of the set $\widetilde\Mcal_1^{\ssup k}$ from  \eqref{MvcalkdefWshift}.

\begin{prop}[Estimates for $J_{W}$]
\label{lem-uppboundJ_W}
Fix $M,R\in(1,\infty)$ and any odd $m\in\N$. Recall $\mathfrak p_{W_{R},M}$  defined in \eqref{pMdef}. The constants $C$ and $C_{M,\mathfrak r,T}$ stem from Lemma~\ref{lem-lowerboundlogpint}, while $C_M$ appears first in Lemma~\ref{lem-UpperLDPbound}. Then the following hold.

\begin{enumerate} 
\item For any $\xi\in\Mcal_1(\Lcal_{W_{mR}}^{\ssup M}\times\Scal^{\ssup M}_{W_{mR}})$ such that $\psi=\partial \Pi_{W_{mR}}^{\ssup \Scal}(\xi)$ is $W_{mR}$-admissible,
\begin{equation}\label{JWupperbound}
J_{W_R;M}(f_{m,R}(\xi))
 \leq J_{W_{mR};M}(\xi)+ \frac {C_M}{(mR)^{d/2}}-\frac 1{|W_{mR}|}\langle\psi,\mathfrak p_{W_{mR},M}\rangle,
\end{equation}
and, for any $x\in \R^d$ such that $x+W_R\subset W_{mR}$,
\begin{equation}\label{JWupperboundx}
J_{W_R;M}(\theta_x(\Pi_{x+W_R}(\xi)))\leq J_{W_{mR};M}(\xi)+ \frac {C_M}{(mR)^{d/2}}-\frac 1{|W_{mR}|}\langle\psi,\mathfrak p_{W_{mR},M}\rangle.
\end{equation}
where $J_{W;M}$ is defined in \eqref{JWMTdef} with all other parameters $K,L,T,\vartheta$ equal to $\infty$.

\item For any  $P\in \Mcal_1^{\ssup{\rm s}}(\Lcal^{\ssup M}\times \Scal^{\ssup M})$ such that  $P$ is $W_{mR}$-ergodic and $\Pi^{\ssup\Scal}_{W_R}(P)$ is concentrated on $A^{\ssup\Scal}_{W_R;M,\mathfrak r, T}\cap\Scal_{W_R}^{\ssup{\vartheta, S}}$,
\begin{equation}
\begin{aligned}
J_{W_R;M}(\Pi_{W_R}(P))&\leq J_{W_{mR};M}(\Pi_{W_{mR}}(P))+ \frac {C_M}{(mR)^{d/2}}+\Big[\frac {C_{M,\mathfrak r,T}}{mR}+\frac{C\vartheta^2}S\Big]\langle P,\mathfrak N_U^{\ssup\ell}\rangle.
\end{aligned}
\end{equation}

\item For any  $P,P_1,P_2,\dots\in \Mcal_1^{\ssup{\rm s}}(\Lcal\times \Scal)$ such that $P_n$ converges weakly to $P$, we have $J_{W_R}(\Pi_{W_R}(P))\leq \liminf_{n\to\infty} J_{W_{R_n}}(\Pi_{W_{R_n}}(P_n))$  for any sequence $R_n\downarrow R$.

\item For any $\xi\in\Mcal_1(\Lcal^{\ssup M}_{W_{R}}\times\Scal^{\ssup M}_{W_{R}})$ such that $\psi=\partial \Pi_{W_{R}}^{\ssup \Scal}(\xi)$ is $W_{R}$-admissible,  and for  any $k\in\N$,
\begin{equation}\label{JWgeqJ_3W}
\begin{aligned}
J_{W_R;M}(\xi)&\geq  \inf\big\{J_{3^k W_{R};M}(\xi^{\ssup k})\colon \xi^{\ssup k}\in \widetilde\Mcal_1^{\ssup k},\Pi_{3^k W_R\to W_R}(\xi^{\ssup k})=\xi\big\}
 -\frac {C_M}{R^{d/2}}+\frac 1{|W_{R}|}\langle\psi,\mathfrak p_{W_{R},M}\rangle.
\end{aligned}
\end{equation}
\end{enumerate}
\end{prop}

\begin{proof}Recall that ${\tt Q}^{\ssup{\L_N,{\rm par}}}_{M}$ denotes the PPP with intensity measure given in \eqref{nuMdef}. All asymptotic assertions as $N\to\infty$ in the following are meant under this measure.

(1) Note that $f_{m,R}(\Xi_{N,mR})=\Xi_{N,R}$ for any large $N\in\N$ that is an odd multiple of $m R$. Hence, for any $\delta$-neighbourhood $ \Bcal_\delta$ of $\xi$, there is an $\delta'$-neighbourhood $\Bcal'_{\delta'}$ of $f_{m,R}(\xi)$ such that
\begin{equation}\label{Q-Qcomparison}
{\tt Q}_M^{\ssup{\L_N,{\rm par}}}\big(\Xi_{N,mR }\in \Bcal_{\delta}\big)
\leq {\tt Q}_M^{\ssup{\L_N,{\rm par}}}\big(\Xi_{N,R }\in \Bcal'_{\delta'}\big).
\end{equation}
We take $\Bcal_\delta$ such that $\widetilde U=\partial\Pi_{W_R}^{\ssup \Scal}(\overline \Bcal'_{\delta'})=\{\partial\Pi_{W_R}^{\ssup \Scal}(\xi')\in\Mcal_1(\Tcal_W)\colon \xi'\in \overline \Bcal'_{\delta'}\}$ is compact. Then, according to Lemma~\ref{lem-LDPforXi} (or Corollary~\ref{cor-LDP}(1)), the right-hand side can be further estimated as follows.
$$
\limsup_{N\to\infty} \frac1{|\L_N|}\log {\tt Q}_M^{\ssup{\L_N,{\rm par}}}\big(\Xi_{N,R }\in \Bcal'_{\delta'}\big)
\leq -\inf_{\psi\in \widetilde U} \inf_{\xi'\in  \overline \Bcal'_{\delta'}}J_{W_R;M}^{\ssup\psi}(\xi')=-\inf_{\xi'\in \overline \Bcal'_{\delta'}}J_{W_R;M}(\xi'),
$$
where the last step derives from the fact that $J_{W_R;M}^{\ssup\psi}(\xi')=\infty$ if $\psi\not=\partial\Pi_{W_R}^{\ssup\Scal}(\xi')$. Because of the lower semicontinuity of $ J_{W_R;M}$, we can estimate $-\inf_{\xi'\in \overline \Bcal'_{\delta'}}J_{W_R;M}(\xi')\leq - J_{W_R;M}(f_{m,R}(\xi))+C_\delta$ for some $C_\delta$ that vanishes as $\delta\downarrow 0$; here we use that $\overline \Bcal'_{\delta'}$ is an $ \delta'$-neighbourhood of $f_{m,R}(\xi)$, and $\delta'$ vanishes as $\delta\downarrow0$.

The left-hand side of \eqref{Q-Qcomparison} is lower estimated using a similar strategy to what we will present in Section~\ref{sec-Nlowbound}. Indeed, we can lower bound
$$
{\tt Q}_M^{\ssup{\L_N,{\rm par}}}\big(\Xi_{N,mR }\in \Bcal_{\delta}\big)
\geq {\tt Q}_{M}^{\ssup{\L_N,{\rm par}}}\big(\Xi_{N,mR }\in \Bcal'_{\delta'}\,\big|\,\partial \Xi_{N,mR}^{\ssup\Scal} =\psi_N\big){\tt Q}_{M}^{\ssup{\L_N,{\rm par}}}\big(\partial\Xi_{N,mR }=\psi_N),
$$
where $\psi_N\in\Mcal_1(\Tcal_{W_R})$ are picked such that they converge towards $\psi$ as $N\to\infty$. Now Lemma ~\ref{lem-UpperLDPbound} and Corollary~\ref{cor-LDP}(2) give that 
$$
\begin{aligned}
\liminf_{N\to\infty} \frac1{|\L_N|}&\log {\tt Q}_M^{\ssup{\L_N,{\rm par}}}\big(\Xi_{N,mR }\in \Bcal_{\delta}\big)
\geq -\inf_{\widetilde \xi\in  \Bcal'_{\delta'}\colon \partial\Pi_{W_{mR}}^{\ssup\Scal}(\widetilde\xi)=\psi}J_{W_{mR};M}(\widetilde \xi)\\
&-\frac{C_M}{(mR)^{d/2}}+\frac 1{|W_{mR}|}\langle\psi,\mathfrak p_{W_{mR},M}\rangle,
\end{aligned}
$$
recalling that $\psi=\partial \Pi_{W_{mR}}^{\ssup \Scal}(f_{m,R}(\xi))$ is $W_{mR}$-admissible. 
Clearly, the infimum on the right-hand side is not larger than $J_{W_{mR};M}(\xi)$. This yields
\begin{equation}\label{JestiLDP}
\begin{aligned}
- J_{W_R;M}(f_{m,R}(\xi))&+C_\delta
 \geq -J_{W_{mR};M}(\xi)-\frac{C_M}{(mR)^{d/2}}+\frac 1{|W_{mR}|}\langle\psi,\mathfrak p_{W_{mR},M}\rangle.
 \end{aligned}
\end{equation}
Making $\delta\downarrow 0$, this implies \eqref{JWupperbound}.

For proving \eqref{JWupperboundx}, fix   $x\in \R^d$ such that $x+W_R\subset W_{mR}$ and an odd $m\in\N\setminus \{1\}$ and define
\begin{equation}\label{fxdef}
f^{\ssup x}\colon \Mcal_1(\Lcal_{W_{mR}}\times \Scal_{W_{mR}}) \to \Mcal_1(\Lcal_{W_{R}}\times \Scal_{W_{R}}),\qquad f^{\ssup x}(\xi^{\ssup m})= \theta_x\big(\Pi_{x+W_R}(\xi^{\ssup m})\big).
\end{equation}
Then $f^{\ssup x}(\Xi_{N,mR}(\omega))=\Xi_{N,mR}^{\ssup x}(\omega)=\frac 1{\# Z_{N,mR}}\sum_{z\in Z_{N,mR}}\delta_{\theta_{z+x}(\Pi_{z+x+W_{R}}(\omega))}$. We proceed now as above with $f_{m,R}$ replaced by $f^{\ssup x}$. Indeed, for any   $\delta$-neighbourhood $\Bcal_\delta$ of $\Pi_{W_{mR}}(P)$, there is an $\delta'$-neighbourhood $\Bcal'_{\delta'}$ of $f^{\ssup x}(\Pi_{W_{mR}}(P))=\theta_x(\Pi_{x+W_R}(P))$ such that 
\begin{equation}\label{Q-Qcomparisonx}
{\tt Q}_M^{\ssup{\L_N,{\rm par}}}\big(\Xi_{N,mR }\in \Bcal_{\delta}\big)
\leq {\tt Q}_M^{\ssup{\L_N,{\rm par}}}\big(\Xi^{\ssup x }_{N,mR }\in \Bcal'_{\delta'}\big).
\end{equation}
For further estimating the right-hand side, observe that Lemma~\ref{lem-LDPforXi} applies to $\Xi^{\ssup x }_{N,mR }$ when conditioned on $\{\partial \Xi^{\ssup x }_{N,mR }=\psi_N\}$ for some $\psi_N\to\psi \in\Mcal_1(\Tcal_{W_R })$, where we define $\partial \Xi^{\ssup x }_{N,mR }=\frac 1{\# Z_{N,mR}}\sum_{z\in Z_{N,mR}}\delta_{\theta_{z+x}(\partial\Pi_{z+x+W_{R}}^{\ssup\Scal}(\omega))}$, since both ${\tt Q}_M$ and the kernel ${\tt K}_W$ are shift-invariant and can be equally applied in $x+W_R$ as in $W_R$. Therefore, as above, we arrive at \eqref{JWupperboundx}.

\medskip

(2) We apply (1) to   $\xi= \Pi_{W_{mR}}(P)$ and note that $f_{m,R}(\Pi_{W_{mR}}(P))=\Pi_{W_R}(P)$ by stationarity of $P$. Note that $\partial \Pi_{W_{mR}}^{\ssup\Scal}(P)$ is $W_{mR}$-admissible by Lemma~\ref{lem-Psi_admissible}. Use Lemma~\ref{lem-lowerboundlogpint} for $mR$ instead of $R$ in order to estimate the right-hand side of \eqref{JWupperbound}.

\medskip

(3) Fix $R\in(0,\infty)$ and a sequence $(R_n)_{n \in\N }$ that tends to $R$ from above. Using \eqref{entropyformula}, we see that, for any bounded measurable function $f\colon \Lcal_{W_{R_n}}\times  \Scal_{W_{R_n}}\to\R$,
$$
|W_{R_n} | J_{W_{R_n}}(\Pi_{W_{R_n}}(P))\geq \langle \Pi_{W_{R_n}}(P), f\rangle 
-\log \big[\Pi^{\ssup\Lcal}_{W_{R_n}}({\tt Q})\otimes[\partial\Pi_{W_{R_n}}^{\ssup\Scal}(P)\otimes {\tt K}_{W_{R_n}}\big](\e^f).
$$
We use this for a function $f$ that depends only on $\Lcal_{W_R}\times\Scal_{W_R}$, i.e., for an $f$ of the form $f=g\circ \Pi_{W_R}$ with some bounded measurable $g\colon \Lcal_{W_{R}}\times  \Scal_{W_{R}}\to\R$. Using that $R\leq R_n$ and hence $\Pi_{W_R}\circ \Pi_{W_{R_n}}=\Pi_{W_R}$, we see that 
$$
\langle \Pi_{W_{R_n}}(P), f\rangle =\langle P, f\circ \Pi_{W_{R_n}}\rangle =\langle P, g\circ\Pi_{W_R}\rangle= \langle \Pi_{W_{R}}(P), g\rangle.
$$
On the other hand, for such an $f$, the integral of $\e^f$ with respect to $\Pi^{\ssup\Lcal}_{W_{R_n}}({\tt Q})$ is equal to the integral of $\e^g$ with respect to $ \Pi^{\ssup\Lcal}_{W_{R}}({\tt Q})$, and the integral of $\e^g $ with respect to $\partial\Pi_{W_n}^{\ssup\Scal}(P_n)\otimes {\tt K}_{W_{R_n}}$ is asymptotic to its integral with respect to $\partial\Pi_{W_R}^{\ssup\Scal}(P)\otimes {\tt K}_{W_{R}}$ as $n\to\infty$, since the $P_n$-probability that a $W_{R_n}$-shred is not a $W_R$-shred vanishes. Picking $g$ such that the supremum over all bounded measurable functions $ \Lcal_{W_{R}}\times  \Scal_{W_{R}}\to\R$ is approached in the formula \eqref{entropyformula} for $|W_{R} | J_{W_{R}}(\Pi_{W_{R}}(P))$, we arrive at the assertion.

\medskip

(4) Recall the empirical process $\Xi_{N,R}^{\ssup k}$ defined in \eqref{XiNkdef} and recall from \eqref{TVconvergence} that $\Xi_{N,R}$ and $\Pi_{W}(\Xi_{N,R}^{\ssup k})$ have a vanishing total-variation distance as $N\to\infty$, where we abbreviate here and in the rest of the proof $\Pi_W=\Pi_{3^k W\to W}$ and $W=W_R$. We also note that
\begin{equation}\label{XiNkdecomp}
\Xi_{N,R}^{\ssup k}=\frac 1{\#(3^k W\cap\Z^d)}\sum_{y\in 3^k W\cap\Z^d}\Xi_{N,R}^{\ssup{k,y}}, \qquad\mbox{where }\Xi_{N,R}^{\ssup{k,y}}
=\frac1{\# Z_{N,3^k R}}\sum_{z\in Z_{N,3^k R}}\delta_{\theta_{y+z}(\Pi_{y+z+3^kW}(\omega))},
\end{equation}
and note, for any $y\in 3^k W\cap\Z^d$, that $(\Xi_{N,R}^{\ssup{k,y}})_{N\in\N}$ can be shown to be exponentially equivalent to $(\Xi_{N,R}^{\ssup{k,0}})_{N\in\N}=(\Xi_{N,3^k R})_{N\in\N}$, since they are only a fixed shift of each other, which corresponds to replacing $\L_N$ by $\L_N+y$. In particular, every $(\Xi_{N,R}^{\ssup{k,y}})_{N\in\N}$ satisfies, according to (a variant of) Corollary~\ref{cor-LDP}, the LDP upper bound with rate function $J_{3^k W;M}$.

We introduce the projection $\pi_y\colon \Mcal_1(\Lcal_{3^k W}\times\Scal_{3^k W})^{3^k W\cap\Z^d}$ on the $y$-th factor. We first show that $((\Xi_{N,R}^{\ssup {k,y}})_y)_{N\in\N}$ satisfies the LDP upper bound with rate function $(\xi_{y})_{y}\mapsto \max_y J_{3^k W}( \xi_y)$. Indeed, for any closed set $U$ in 
$ \Mcal_1(\Lcal_{3^k W}\times\Scal_{3^k W})^{3^k W\cap\Z^d}$, we can easily estimate, for any $y$,
$$
{\tt Q}^{\ssup{\L_N,{\rm par}}}_{M}\big((\Xi_{N,R}^{\ssup {k,y}})_y\in U\big)
\leq {\tt Q}^{\ssup{\L_N,{\rm par}}}_{M}\big(\Xi_{N,R}^{\ssup {k,y}}\in \pi_y(U)\big)
\leq \exp\Big\{-(1+o(1))|\L_N|\inf_{\pi_y(U)} J_{3^k W,M }\Big\}.
$$
We abbreviate $G((\xi_y)_{y\in 3^k W\cap\Z^d})=\frac 1{\#(3^k W\cap\Z^d)}\sum_y \xi_y$, then $\Xi_{N,R}^{\ssup k}=G((\Xi_{N,R}^{\ssup{k,y}})_y)$ satisfies, according to  the (upper-bound half of the) contraction principle from the theory of large deviations (see \cite[Section 4.2.1]{DZ98}) the upper-bound half of the LDP with rate function $\Mcal_1(\Lcal_{3^k W}\times\Scal_{3^k W})\ni \xi^{\ssup k}\mapsto \inf\{\max_y J_{3^k W,M}\circ \pi_y((\xi_{y})_{y})\colon G((\xi_{y})_{y})=\xi^{\ssup k}\}$. Since $\Xi_{N,R}^{\ssup k }$ has a vanishing distance to the set $\widetilde \Mcal_1^{\ssup k}$, the rate function is $=\infty$ in points $\xi^{\ssup k}\in (\widetilde \Mcal_1^{\ssup k})^{\rm c}$. Note that this rate function is not smaller than $J_{3^k W,M}(\xi^{\ssup k})$, since, by convexity, $\max_y J_{3^k W,M}\circ \pi_y((\xi_{y'})_{y'})\geq G((J_{3^k W,M}(\xi_{y})_{y})
\geq J_{3^k W,M}(G((\xi_{y}))_{y})=J_{3^k W,M}(\xi^{\ssup k})$. Hence, $\Xi_{N,R}^{\ssup k}$ satisfies the upper-bound half of the LDP with rate function $\xi^{\ssup k}\mapsto J_{3^k W,M}(\xi^{\ssup k})$ if $\xi^{\ssup k}\in \widetilde \Mcal_1^{\ssup k}$ and $\mapsto \infty$ otherwise. Now recall that $\Xi_{N,R}$ (and therefore $\Pi_W( \Xi_{N,R}^{\ssup k})$, since they have vanishing total-variation distance) satisfies the LDP lower bound  with rate function $\xi\mapsto J_{W,M}(\xi)+\frac {C_M}{R^{d/2}}-\frac 1{|W|}\langle \psi, \mathfrak p_{W,M}\rangle$ (for a proof, adapt the arguments of the proof of (1)). Again using the contraction principle for the function $\Pi_{W}$, we see that \eqref{JWgeqJ_3W} is satisfied.
\end{proof}

\subsection{Compactnesses}\label{sec-Compactness}

In Section~\ref{sec-LDPupperbound}  we will apply the LDP of Corollary~\ref{cor-LDP}(1) to the right-hand side of \eqref{hatZ2}. For this, it is necessary to generate compactness in the integration area beforehand. Furthermore,  we need to introduce a condition that makes the two particle densities that appear in \eqref{hatZ2}  continuous functions of $\Xi_{N,R}$, since it is crucial for our purposes to keep control on the condensate density and the total density of particles. The purpose of this section is to introduce this condition and to show how it can be incorporated in \eqref{hatZ2}. 

Part of our solution to this problem is to prove the compactness of the sub-level sets of $J_W$, which will be also used in the proof of Theorem~\ref{thm-specrelent}. But we will also need to explicitly control the expected particle number of the boundary-shred configuration $\partial\Pi_W^{\ssup\Scal}(\xi)$, since this is not controlled in the sub-level sets.

More explicitly, we will apply the (upper bound of the) LDP on the event $ \{\Xi^{\ssup\Lcal}_{N,R }\in K^{\ssup\Lcal}\}\cap\{\partial\Xi^{\ssup\Scal}_{N,R}\in K^{\ssup\Scal}\}$, where $K^{\ssup\Lcal}$ and $K^{\ssup\Scal}$ are sets that need to satisfy three crucial properties:

\begin{enumerate}
\item they are closed, and $K^{\ssup\Scal}$ induces compactness for the $\Mcal_1(\Tcal_W)$-part on the sub-level sets of $J_W$,

\item  the probability of $\Xi^{\ssup\Lcal}_{N,R }$, respectively $\partial\Xi^{\ssup\Scal}_{N,R }$, being in their complement is extremely small on the exponential scale,

\item on these sets, the number of particles is a continuous function of $\Xi^{\ssup\Lcal}_{N,R }$ respectively $\partial\Xi^{\ssup\Scal}_{N,R }$.

\end{enumerate}

The problem with (3) is that, {\em a priori}, the maps $\xi\mapsto\langle \xi,\mathfrak N_W^{\ssup{\ell,\Lcal}}\rangle$, respectively $\psi\mapsto\langle \psi,\mathfrak N_{\partial W}^{\ssup{\ell,\Scal}}\rangle$, are  not continuous in the weak topology since neither $\mathfrak N_W^{\ssup{\ell,\Lcal}}$ nor $ \mathfrak N_{\partial W}^{\ssup{\ell,\Scal}}$ are bounded. 
Therefore, we make use of Assumption (V), in particular \eqref{LowBoundAss}, which gives an efficient bound on having an extraordinarily high number of particles in the subboxes. Hence, (2) will be handled in terms of the transformed probability measure $\widehat {\tt Q}^{\ssup{\L_N,{\rm bc}}}$ defined in \eqref{transformedmeasure}.

\subsubsection{Loop part.} \label{sec-loopscompact}
Following Remark~\ref{rem-counting}, we decided to work (instead with the  number ${\mathfrak N}_{W_R}^{\ssup{\ell,\Lcal}}$ of particles in all loops in $W_R$; see \eqref{def_particlenumbers}) with the number $\widetilde{\mathfrak N}_{y+U}^{\ssup{\ell,\Lcal}}(\omega)$ of particles in $y+U$ (recall that $U=[-\frac 12,\frac 12]^d=W_{1/2}$) coming from any of the loops of $\omega\in \Lcal_W$ starting in $W$. Note that $\sum_{y\in Z_{R,1/2}}\widetilde{\mathfrak N}_{y+U}^{\ssup{\ell,\Lcal}}(\omega)=\widetilde{\mathfrak N}_{W}^{\ssup{\ell,\Lcal}}(\omega)={\mathfrak N}_{W}^{\ssup{\ell,\Lcal}}(\omega)$, where $Z_{R,1/2}$ (in a small abuse of notation) is the set of those $y \in \frac 12\Z^d$ such that $y+U \subset W_R$ (which implies $|Z_{R,1/2}|\leq |W_R|$).  Then we consider the set
\begin{equation}\label{KAloopdef}
\begin{aligned}
K_{R}^{\ssup{\Lcal}}(L)&=\Big\{\xi\in\Mcal_1(\Lcal_{W_R})\colon\sum_{y\in Z_{R,1/2}} \big\langle \xi, \big(\widetilde{\mathfrak N}_{y+U}^{\ssup{\ell,\Lcal}}\big)^{5/4}\big\rangle\leq L |W_{R}|
\Big\},\qquad L\in(0,\infty).
\end{aligned}
\end{equation}
In words,  under any $\xi\in K_{R}^{\ssup{\Lcal}}(L)$,  the number of particles in any shift of $U$ in $W_R$ is not extremely unevenly distributed, uniformly over the subbox and over $\xi$.

The set $K_{R}^{\ssup{\Lcal}}(L)$ turns out to possess the mentioned three properties:

\begin{lemma}[Restriction of the loop-part]\label{lem-compactnessLoop}Suppose that Assumption (V) holds, then the following assertions hold.

\begin{enumerate}
\item For any $L\in(0,\infty)$ and for any $R\in \N$, the set $K_{R}^{\ssup {\Lcal}}(L)$ is closed in the weak topology. 

\item 
\begin{equation}\label{exptightKLcal}
\lim_{L\to \infty}\sup_{R\in\N}\limsup_{N\to\infty}\frac 1{|\L_N|} \log \widehat\LPP^{\ssup{\L_{N},{\rm par}}}\Big(\Xi_{N,R}^{\ssup{\Loops}}\in \big(K_{R}^{\ssup\Lcal}(L)\big)^{\rm  c}\Big)=-\infty.
\end{equation}

\item  For any $L\in(0,\infty)$ and for any $R\in \N$, the  map $\xi\mapsto \langle \xi,\widetilde{\mathfrak N}_{W}^{\ssup{\ell,\Lcal}}\rangle $ is continuous on the set $K_{R}^{\ssup{\Lcal}}(L)$.

\end{enumerate} 
\end{lemma}

\begin{proof} We drop the super-indices \lq$\Lcal$\rq, and we abbreviate $W=W_R$.

(1) Since the map $\xi\mapsto \langle \xi, \widetilde{\mathfrak N}^{\ssup{\ell}}_{U}\1\{\widetilde{\mathfrak N}^{\ssup{\ell}}_{U}>L_k \} \rangle$ is lower semi-continuous for any $k\in\N$, $K_{R}(L) $ is closed.

(2) Now we show \eqref{exptightKLcal}. We adapt and  extend the proof of Lemma~\ref{lem-condensatenotion} and will also employ Lemma~\ref{lem-pathinteraction} (which uses \eqref{LowBoundAss} from Assumption (V)). As always, we assume that $\L_N$ is equal to the (up to boundaries) disjoint union of the boxes $W_z=z+W$ with $z\in Z_{N,R}$ (and $R$ may depend on $N$). We use that $\# Z_{N,R}\sim |\L_N|/|W|$ as $N\to\infty$.

First note that, for any $y\in Z_{R,1/2}$, in the limit as $N\to\infty$,
$$
\big\langle \Xi_{N,R}(\omega_{\rm P}), \big(\widetilde{\mathfrak N}^{\ssup{\ell}}_{y+U}\big)^{5/4}\big \rangle
\sim \frac {|W|}{|\L_N|}\sum_{z\in Z_{N,R}} \big(\widetilde{\mathfrak N}^{\ssup{\ell}}_{z+y+U}(\omega_{\rm P})\big)^{5/4}.
$$
Hence, for any $L\in (0,\infty)$, recalling the notation in \eqref{subconfigurations},
\begin{equation}\label{uppboundsmallprob}
\begin{aligned}
\widehat\LPP^{\ssup{\L_{N},{\rm par}}}&\big(\Xi_{N,R}\in (K_{R}(L))^{\rm  c}\big)
\leq \widehat\LPP^{\ssup{\L_{N},{\rm par}}}\Big(\sum_{z\in Z_{N,R}}\sum_{y\in Z_{R,1/2}} \big(\widetilde{\mathfrak N}^{\ssup{\ell}}_{z+y+U}(\omega_{W_z})\big)^{5/4}>L |\L_N|\Big)\\
&\leq \e^{-L|\L_N|} \widehat\LPP^{\ssup{\L_{N},{\rm par}}}\Big(\e^{ \sum_{z\in Z_{N,R}}\sum_{y\in Z_{R,1/2}} \big(\widetilde{\mathfrak N}^{\ssup{\ell}}_{z+y+U}(\omega_{W_z})\big)^{5/4}}\Big)\\
&\leq\frac{\e^{- L|\L_N|}}{\widehat Z^{\ssup{\rm bc}}_N(\L_N)}\prod_{z\in Z_{N,R}}\LPP^{\ssup{\L_{N},{\rm par}}}\Big(\e^{\sum_{y\in Z_{R,1/2}} \big(\widetilde{\mathfrak N}^{\ssup{\ell}}_{z+y+U}(\omega_{W_z})\big)^{5/4}}\e^{-\sum_{y\in Z_{R,1/2}}\Phi_{z+y+U,z+y+U}(\omega_{W_z})}\Big),
\end{aligned}
\end{equation}
where we used that, under particle boundary condition,  the configurations in the subboxes $W_z$ are i.i.d.\ over $z$, and we estimated the total interaction $\Phi_{\L_n,\L_N}$ in $\L_N$ from below against the sum of the interactions within the $z+y+U$ (i.e., the interactions between any pair of different legs that start in $z+y+U$).

We  apply Lemma~\ref{lem-pathinteraction} for any $z$ and $y$ to the expectation of $\e^{-\Phi_{z+y+U,z+y+U}}$  over all the legs starting in $z+y+U$, conditioned on their starting sites, which are the particles. Hence, there is some $C\in(0,\infty)$ such that this conditional expectation is upper bounded by $\e^{-C  (\widetilde{\mathfrak N}^{\ssup{\ell}}_{z+y+U})^{3/2} \1\{\widetilde{\mathfrak N}^{\ssup{\ell}}_{z+y+U}\geq C\}}$. Abbreviate $f(m)=m^{5/4}-C m^{3/2}\1\{m\geq C\}$. Hence, when carrying out the expectation over all the (independent!) legs in the configuration in each subbox $z+y+U$, we obtain from \eqref{uppboundsmallprob}  the following upper bound:
\begin{equation}
\begin{aligned}
\widehat\LPP^{\ssup{\L_{N},{\rm par}}}\big(\Xi_{N,R}\in (K_{R}(L))^{\rm  c}\big)
\leq \frac{\e^{- L|\L_N|}}{\widehat Z^{\ssup{\rm bc}}_N(\L_N)}\prod_{z\in Z_{N,R}}
\LPP^{\ssup{\L_{N},{\rm par}}}\Big[\e^{\sum_{y\in Z_{R,1/2}} f(\widetilde{\mathfrak N}^{\ssup{\ell}}_{z+y+U}(\omega_{W_z}))
}\Big].
\end{aligned}
\end{equation} 
Since $f$ is bounded from above (here is the point where we need the decomposition into the sum over $y$ of the $5/4$-moments of the particle numbers in the unit boxes in the definition of $K_R^{\ssup\Lcal}(L)$ instead of directly bounding only the $5/4$-moment of the particle number in $W$), and since the number of summands $(z,y)$ does not depend on $R$ and is $\sim |\L_N|$, and noting that the normalization constant $\widehat Z^{\ssup{\rm bc}}_N(\L_N)$ has a positive exponential rate, this shows that \eqref{exptightKLcal} holds.

(3) First note that the  map $\xi\mapsto \langle \xi,\widetilde{\mathfrak N}_{W}^{\ssup\ell}\rangle $ is lower semicontinuous on $K_{R}^{\ssup\Lcal}(L)$, since $\widetilde{\mathfrak N}_{W}^{\ssup\ell} $ can be monotonously approached from below with nonnegative bounded continuous functions. Furthermore, the upper semicontinuity can be easily be derived, since on $K_{R}^{\ssup\Lcal}(L)$ there is a uniform upper bound for a higher moment of $\widetilde{\mathfrak N}_{W}^{\ssup\ell}$. More precisely, for any sequence $(\xi_n)_n$ in $K_R^{\ssup\Lcal}(L)$, tending weakly towards some $\xi\in K_R(L)$, we use that $\widetilde{\mathfrak N}_W^{\ssup\ell}=\sum_{y\in Z_{R,1/2}}\widetilde{\mathfrak N}_{y+U}^{\ssup\ell}$ and estimate, for any $n\in\N$ and any $T\in(0,\infty)$,
$$
\begin{aligned}
\big|\langle \xi_n-\xi,&\widetilde{\mathfrak N}_W^{\ssup\ell}\rangle\big|
\leq \sum_{y\in Z_{R,1/2}}\Big(\big|\langle \xi_n-\xi,\widetilde{\mathfrak N}_{y+U}^{\ssup\ell}\1\{\widetilde{\mathfrak N}_{y+U}^{\ssup\ell}\leq T\}\rangle\big|
+T^{-1/4}\langle \xi_n+\xi,(\widetilde{\mathfrak N}_{y+U}^{\ssup\ell})^{5/4}\rangle\Big)\\
&\leq \sum_{y\in Z_{R,1/2}}\big|\langle \xi_n-\xi,\widetilde{\mathfrak N}_{y+U}^{\ssup\ell}\1\{\widetilde{\mathfrak N}_{y+U}^{\ssup\ell}\leq T\}\rangle\big|+T^{-1/4} L |W|.
\end{aligned}
$$
The latter term  can be made arbitrarily small by picking $T$ large, and then the former sum can be made arbitrarily small by picking $n$ large. 
\end{proof}

\subsubsection{Compactness for shred part.} Analogously to the loop part, we present now a strategy to compactify the boundary-shred-part in the variational formula in \eqref{hatZ2} such that the number of particles in the shreds, i.e., the map $\psi\mapsto  \langle \psi,{\mathfrak N}_{\partial W_R}^{\ssup{\ell,\Scal}}\rangle $, becomes continuous, where we recall from \eqref{particlenumber} that ${\mathfrak N}_{\partial W_R}^{\ssup{\ell,\Scal}}(\mu)= \sum_{i\in I}l_i$ denotes the number of particles in the configuration $\mu=\sum_{i\in I}\delta_{(x_i,l_i,y_i)}$. We also denote by ${\mathfrak N}_{y+U}^{\ssup{\ell,\Scal}}(\varpi)$ the number of particles in $y+U$ from all the shreds in $\varpi\in \Scal_W$ (compare to \eqref{def_particlenumbers}).

Then  we consider the set 
\begin{equation}\label{KAdef}
\begin{aligned}
K_{R}^{\ssup{\Scal}}(L)&=\Big\{\xi\in\Mcal_1(\Scal_{W_R})\colon\sum_{y\in Z_{R,1/2}} \big\langle \xi, \big({\mathfrak N}_{y+U}^{\ssup{\ell,\Scal}}\big)^{5/4}\big\rangle\leq L |W_{R}|
\Big\},\qquad L\in(0,\infty).
\end{aligned}
\end{equation}

The set $K_{R}^{\ssup{\Scal}}$ turns out to possess the mentioned three properties:

\begin{lemma}[Compactification of the shred-part]\label{lem-exponentiallytight}Suppose that Assumption (V) holds. Then the following holds.
\begin{enumerate}
\item For any $R,L\in \N$, the set $K_{R}^{\ssup\Scal}(L)$ is closed in the weak topology. 
Furthermore, the set
\begin{equation}\label{compactsetR}
\big\{\xi\in \Mcal_1(\Lcal_{W_R}\times\Scal_{W_R})\colon  \Pi_{W_R}^{\ssup\Scal}(\xi)\in K_R^{\ssup\Scal}(L),\langle \partial\Pi_{W_R}^{\ssup\Scal}(\xi), \mathfrak N_{\partial {W_R}}^{\ssup\Scal}\rangle =\rho_2, J_{W_R}(\xi)\leq C\big\}
\end{equation}
is compact for any $\rho_2,C\in(0,\infty)$.

\item
\begin{equation}\label{exptightK_A}
\lim_{L\to\infty}\sup_{R\in \N}\limsup_{N\to\infty}\frac 1{|\L_N|} \log \widehat\LPP^{\ssup{\L_N,{\rm par}}}\Big(\Xi_{N,R}^{\ssup{\Shreds}}\in \big(K_{R}^{\ssup\Scal}(L)\big)^{\rm  c}\Big)=-\infty.
\end{equation}

\item 
For any  $R,L\in \N$, the map $\psi\mapsto \int\psi(\d \mu)\,\sum_{i\in I} l_i$ is continuous on $K_{R}^{\ssup\Scal}(L)$.

\end{enumerate}
\end{lemma}

\begin{proof} Much of this proof is analogous to the one of Lemma~\ref{lem-compactnessLoop}, and we leave the details to the reader. Indeed, the closedness of $K_R^{\ssup \Scal}(L)$ is analogous, the proof of (2) and the one of (3) as well. However, the compactness of the set in \eqref{compactsetR} needs to be shown. Here we can restrict to showing the relative compactness of the set of $\xi$ such that $J_{W_R}(\xi)\leq C$ and $\langle \partial\Pi_{W_R}^{\ssup\Scal}(\xi), \mathfrak N_{\partial {W_R}}^{\ssup\Scal}\rangle \leq\rho_2$, since the continuity of the map $\xi\mapsto \langle \partial\Pi_{W_R}^{\ssup\Scal}(\xi), \mathfrak N_{\partial {W_R}}^{\ssup\Scal}\rangle$ follows from assertion (3).

First recall from the end of Section~\ref{sec-chopping} that the set $K_{\rho_2}$ of all $\psi\in\Mcal_1(\Tcal_{W_R})$ such that $\langle \psi, {\mathfrak N}_{\partial W_R}^{\ssup{\ell,\Scal}}\rangle\leq \rho_2$ is relatively compact. Hence, we only need to show that its intersection with the sub-level set of $J_{W_R}$ (more precisely, the set of $\xi$ such that $\partial\Pi_{W_R}^{\ssup\Scal}(\xi)$ is in that set and $J_{W_R}(\xi)\leq C$) is relative compact in $\Mcal_1(\Lcal_{W_R}\times\Scal_{W_R})$.

Pick a sequence $(\xi_n)_{n\in\N}$ in that set. Using  the standard  Lemma~\ref{lem-entropyproducts} below we see that its loop-part, i.e., the sequence $(H_{\Lcal_W}(\Pi_{W_R}^{\ssup\Lcal}(\xi_n)\mid \Pi_{W_R}^{\ssup\Lcal}({\tt Q})))_{n\in\N}$, is bounded. Therefore, as it is proved in the large-deviations literature as part of the proof of Sanov's theorem, $(\Pi_{W_R}^{\ssup\Lcal}(\xi_n))_{n\in\N}$ is tight. Therefore, we only need to show that the shred-part  $(\Pi_{W_R}^{\ssup\Shreds}(\xi_n))_{n\in\N}$ is also tight. Then $(\Pi_{W_R}(\xi_n))_{n\in\N}$ is also tight, since, for any compact sets $K^{\ssup\Lcal}\subset \Lcal_W$ and $K^{\ssup\Scal}\subset \Scal_W$, also $K^{\ssup\Lcal}\times K^{\ssup\Scal}$ is compact, and its complement is contained in $[(K^{\ssup\Lcal})^{\rm c}\times \Scal_W]\cup [ \Lcal_W\times (K^{\ssup\Scal})^{\rm c}]$.

Let us show tightness of the sequence $(\Pi_{W_R}^{\ssup\Shreds}(\xi_n))_{n\in\N}$.
Again from \eqref{hbound} and Lemma~\ref{lem-entropyproducts} we see that it lies in the set
\begin{equation}
A_R= \big\{\xi\in\Mcal_1(\Shreds_{W_R})\colon \partial \Pi_{W_R}^{\ssup\Scal}(\xi)\in K_{\rho_2}, H_{\Shreds_{W_R}}(\xi\mid \partial\Pi_{W_R}^{\ssup\Shreds}(\xi)\otimes {\tt K}_{W_R})\leq C_R\big\},
\end{equation} 
for some $C_R\in\R$.

We show now that this set is compact. The following is an adaptation of the standard proof of compactness of level sets of entropies, as it can be found in the literature on Sanov's theorem (see \cite[Section 6.2]{DZ98}). Indeed, let a small $\delta\in(0,1)$ be given. Since $K_{\rho_2}$ is compact, there is a compact set $L_1\subset\Tcal_{W_R}$ such that $\psi(L_1^{\rm c})\leq \delta$ for any $\psi\in K_R$. Note that there is an $\alpha\in \N$ such that, for any $\mu=\sum_{i}\delta_{(x_i,y_i,l_i)}\in L_1$, we have $|y_i|\leq \alpha$, $l_i\leq \alpha$ and $|\mu|\leq \alpha$. Additionally, let a small $\eps\in(0,1)$ be given. Then, for any $\mu\in \Tcal_{W_R}$, there is a compact set $L_2(\mu)\subset \Scal_{W_R }$ such that ${\tt K}_{W_R}(\mu,L_2(\mu)^{\rm c})\leq \eps$. Then the set $L=\bigcup_{\mu\in L_1}\{\mu\}\times L_2(\mu)$ is compact. 

We denote a regular version of the conditional distribution of $\xi$ given $\partial \Pi_{W_R}^{\ssup\Scal}(\xi)$ by $\xi(\mu,\cdot)$ (in earlier notation $\xi_{\Tcal_{W_R}\to \Scal_{W_R}}(\mu,\cdot)$).  Using the characterisation of the entropy from \eqref{entropyformula}, we see by taking $f=\1_{L_2(\mu)^{\rm c}}\log M$ for any large $M$ that, for any $\xi \in\Mcal_1(\Scal_{W_R})$,
$$
H_{\Scal_{W_R}}(\xi(\mu,\cdot)\mid{\tt K}_{W_R}(\mu,\cdot))
\geq \xi(\mu,L_2(\mu)^{\rm c}) \log M-\log\big(1-{\tt K}_{W_R}(\mu,L_2(\mu)^{\rm c})+M {\tt K}_{W_R}(\mu,L_2(\mu)^{\rm c})\big),
$$
and hence
$$
\xi(\mu,L_2(\mu)^{\rm c})\leq\frac1 {\log M }\Big[H_{\Scal_{W_R}}(\xi(\mu,\cdot)\mid{\tt K}_{W_R}(\mu,\cdot))+\log(1+(M-1)\eps)\Big].
$$
Now we assume that $\xi\in A_R$ and integrate over $\partial\Pi_{W_R}^{\ssup\Scal}(\xi)(\d \mu)$ to obtain
$$
\begin{aligned}
\xi\Big(\bigcup_{\mu\in \Tcal_{W_R}}\{\mu\}\times L_2(\mu)^{\rm c}\Big)
&=\int \partial\Pi_{W_R}^{\ssup\Scal}(\xi)(\d \mu)\,\xi(\mu,L_2(\mu)^{\rm c})\\
&\leq \frac1 {\log M }H_{\Scal_{W_R}}(\xi\mid \partial\Pi_{W_R}^{\ssup\Scal}(\xi)\otimes{\tt K}_{W_R})+\frac{\log(1+(M-1)\eps)} {\log M }\\
&\leq \frac{c|W_R |+\log(1+(M-1)\eps)} {\log M }.
\end{aligned}
$$
Now picking $M=1+\frac 1\eps$ and making $\eps$ small, we can press the right-hand side under the given small number $\delta\in(0,1)$.  Now we want to argue that $\xi(L^{\rm c})\leq 2\delta$ for any $\xi\in A_R$. 

Indeed, split $L^{\rm c}$ into the union of $\{\mu\}\times L_2(\mu)^{\rm c }$ with $\mu\in L_1$ and  the union of $\{\mu\}\times L_2(\mu)$ with  $\mu\in L_1^{\rm c}$ to see that the $\xi$-measure of the complement of the former is below $\delta$ by the choice of $L_1$ and the one of the latter is below $\delta$ by the preceding argument. This finishes the proof of tightness of $A_R$ and therefore the proof of the entire lemma.
\end{proof}

\subsection{Upper large deviations  bound}\label{sec-LDPupperbound}

In this section we derive an upper bound for the large-$N$ exponential rate of $\widehat Z_{N,R,\delta}^{\ssup{\rm bc}}(\L_N,\rho_1,\rho_2)$ defined in \eqref{hatZ}, on the base of the upper bound of Lemma~\ref{lem-UpperBound} in \eqref{hatZ2}.  Recall that we fixed $\rho_1,\rho_2\in(0,\infty)$ with $\rho=\rho_1+\rho_2$. Furthermore, as always, we fix $R\in\N$ and  decompose the box $\L_N$ into the subboxes $z+W$ with $z\in Z_{N,R}$, where $R$ might depend on $N$ and converges then to $R$. Furthermore, we fix a small $\delta\in(0,1)$.

We want to apply Corollary~\ref{cor-LDP}(1) to the right-hand side of \eqref{hatZ2}. For this, we insert the indicator on $\{\Xi_{N,R}^{\ssup\Lcal}\in K_{R}^{\ssup\Lcal}(L)\}\cap\{ \Xi_{N,R}^{\ssup\Scal}\in K_{R}^{\ssup\Scal}(L)\}$, plus the one on its complement, and we note that the second part (the one with the indicator on the complement) is negligible on the exponential scale in the limit $N \to\infty$ by Lemmas~\ref{lem-compactnessLoop} and \ref{lem-exponentiallytight}, if $L$ is taken large enough (not depending on $R$, according to assertion (2) there). On the first event, the map $\xi\mapsto \langle \xi,\frac 1{|W|}F_{W,W}\rangle$ is lower semicontinuous, and the two maps $\xi\mapsto\langle \xi,\mathfrak N_W^{\ssup{\ell,\Lcal}}\rangle$ and $\psi\mapsto\langle \psi, \mathfrak N_{\partial W}^{\ssup{\ell,\Scal}}\rangle$ are continuous; hence the two indicators on the right-hand side of \eqref{hatZ2} are on closed sets, and the $\Mcal_1(\Tcal_W)$-part is even continuous. Applying Corollary~\ref{cor-LDP}(1), we  arrive at
\begin{equation}\label{hatZNlimit}
\begin{aligned}
\limsup_{N\to\infty}\frac 1{|\L_N|}&\log \widehat Z_{N,R,\delta}^{\ssup{\rm bc}}(\L_N,\rho_1,\rho_2)\\
&\leq -\inf\Big\{\frac1{|W|}\langle \xi,F_{W,W}\rangle+J_W(\xi)
\colon \xi\in\bigcap_{k\in \N}\Mcal_1^{\ssup k}(\Lcal_{W}\times \Scal_{W}),\\
&\quad 
\Pi_W^{\ssup\Lcal}(\xi)\in K_{R}^{\ssup\Lcal}(L),\Pi_{W}^{\ssup\Scal}(\xi)\in K_{R}^{\ssup \Scal}(L), \\
&\quad\langle \xi,\smfrac 1{|{W}|}{\mathfrak N}_{W}^{\ssup{\ell,\Lcal}}\rangle\in\overline\Bcal_{2\delta}(\rho_1), \langle \xi,\smfrac 1{|{W}|}{\mathfrak N}_{W}^{\ssup{ \ell,\Scal}}\rangle\in\overline\Bcal_{2\delta}(\rho_2)\Big\}.
\end{aligned}
\end{equation}

In Section~\ref{sec-limitingentropy} we show how to handle the limit as $R\to\infty$ for the crucial entropy term $J_W(\xi)$ if $\xi=\Pi_W(P)$ for some $P\in\Mcal_1^{\ssup{\rm s}}(\Lcal\times\Scal)$. The finish of the proof of the upper bound in Theorem~\ref{thm-freeenergy} will be carried out in Section~\ref{sec-finishuppbound}.

\subsection{Lower large deviations bound}\label{sec-Nlowbound}

Now we apply the large deviations assertions of Corollary~\ref{cor-LDP}(2) to our preliminary lower bound in Lemma~\ref{lem-LowerBound} in \eqref{hatZ4} to obtain a lower bound for the limiting free energy. We keep $R\in\N$ fixed and put $W=[-R,R]^d$. Furthermore, we fix also all the auxiliary parameters $\mathfrak r, M ,K,L,T\in(0,\infty)$ and recall the events in \eqref{ALcaldef}, \eqref{AScaldef} and \eqref{Adef}.  We also recall that $\L_N'$ is the \lq interior\rq\ of $\L_N$, i.e., the union of the $W_z$ with $z\in Z_{N,R}^\circ$.
 Recall also other notation from Section~\ref{sec-lowboundN}, in particular that we write $\|f\|_{{\rm sp}}=\max_{t\in[0,\beta]}|f(t)-f(0)|$ for the spread of a leg $f\in\Ccal_1$. We pick any admissible $\psi^*\in \Mcal_1(\Tcal_W)$ (see Definition~\ref{def-admisssible}). We assume that $\langle \psi^*,\frac 1{|W|}\mathfrak N_{\partial W}^{\ssup{\ell,\Scal}}\rangle \in \Bcal_{\delta/3}(\rho_2)$. For some small $\delta\in(0,1)$, we insert the indicator on the event $\{\partial \Xi_{N,R}^{\ssup{\circ,\omega_{\rm P},\Scal}}\in \Bcal_\delta(\psi^*)\}$ into the expectation on the right-hand side of \eqref{hatZ4}, where $\Bcal_\delta(\psi^*)$ is a $\delta$-neighbourhood of $\psi^*$. Recall the kernel ${\tt K}_W$ from \eqref{kernelK}. 
We are going to use $C\in(0,\infty)$ as a generic constant that depends only on $d$, $\beta$, $\rho_1$, $\rho_2$ or $v$ and may change its value from appearance to appearance.

In view of an application of  Corollary~\ref{cor-LDP}(2) we switched already in \eqref{hatZ4} in Section~\ref{sec-lowboundN} to the integration with respect the conditional version given an event, namely $ 
\bigcap_{z\in Z_{N,R}^\circ}\{(\omega_{W_z},\varpi_{W_z})\in A_{W_z;M,L,K,T}\}$, on which the particle numbers in loops and in shreds, respectively,  and the interaction are continuous functionals of $\Xi_{N,R}$. This is necessary in order that the event 
$$
B=\Big\{\xi\in\Mcal_1(A_{W;\mathfrak r,M,K,L,T})\colon \langle\xi, \smfrac 1{|W|} \mathfrak N_W^{\ssup{\ell,\Lcal}}\rangle \in\Bcal_{\delta/3}(\rho_1)\mbox{ and }\langle\xi, \smfrac 1{|W|} \mathfrak N_W^{\ssup{\ell,\Scal}}\rangle \in\Bcal_{\delta/3}(\rho_2)\Big\}
$$
is open in $\Mcal_1(A_{W;\mathfrak r,M,K,L,T})$ and the map $\xi\mapsto \langle F_{W,W},\xi\rangle$ is continuous. Indeed the functionals $\mathfrak N_W^{\ssup{\ell,\Lcal}}$, $\mathfrak N_W^{\ssup{\ell,\Scal}}$ and $F_{W,W}$ need to be bounded and continuous  on our restriction, since we are working in the weak topology. Note that their continuity is not a problem, since (see Section~\ref{sec-PPPs}) convergence of loop-shred configurations is defined leg-wise, and $V$ defined in \eqref{Vdef} is continuous. Indeed, on $A_{W;\mathfrak r,M,L,K,T}$, the three functions $\mathfrak N_W^{\ssup{\ell,\Lcal}}$, $\mathfrak N_W^{\ssup{\ell,\Scal}}$ and $F_{W,W}$ are bounded. This is clear for the first two, and for $F_{W,W}$ this follows from Lemma~\ref{lem-Interactesti} with $z=z'$. Hence, we are in the setting of our crucial LDP-tool, Corollary~\ref{cor-LDP}(2). 

Now taking the limit $N\to\infty$ in \eqref{hatZ4} gives, according to Lemma~\ref{lem-UpperLDPbound} and  Corollary~\ref{cor-LDP}(2),
\begin{equation}\label{hatZNlimitlowbound}
\begin{aligned}
\liminf_{N\to\infty}&\frac 1{|\L_N|}\log \widehat Z_{N,R,\delta}^{\ssup{\rm bc}}(\L_N,\rho_1,\rho_2)
\geq -\frac {C_{M,L,K,T}}R- \frac {C_M}{R^{d/2}}+\langle \psi^*,\mathfrak p_{W_R,M}\rangle
\\
&-\inf\Big\{\frac1{|W|}\langle \xi,F_{W,W}\rangle+J_{W;\mathfrak r,M,L,K,T}(\xi)-\langle \log p_{W;\mathfrak r,M,L,K,T},\psi^*\rangle
\colon  \\ 
&
\qquad\qquad\xi\in\Mcal_1(A_{W;\mathfrak r,M,L,K,T}),\partial\Pi_W^{\ssup\Scal}(\xi)=\psi^*,\langle \xi,\smfrac 1{|W|}{\mathfrak N}_W^{\ssup{\ell,\Lcal}}\rangle\in \Bcal_{\delta/3}(\rho_1)\Big\},
\end{aligned}
\end{equation}
where $ J_{W;\mathfrak r,M,L,K,T}$ is defined in \eqref{JWMTdef} with $\vartheta=\infty$. Note that, for $\xi\in\Mcal_1(\Lcal_W\times \Scal_W)$ that is concentrated on $A_{W;\mathfrak r,M,L,K,T}$ with $\partial\Pi_W^{\ssup\Scal}(\xi)=\psi^*$, we have
$$
\begin{aligned}
J_{W;\mathfrak r,M,L,K,T}(\xi)-\langle \log p_{W;\mathfrak r,M,L,K,T},\psi^*\rangle
&=\int_{A_{W;\mathfrak r,M,L,K,T}}\xi(\d(\omega,\varpi))\,\log\frac{\d \xi}
{\d[\Pi_W^{\ssup\Lcal}({\tt Q})\otimes [\psi^ *\otimes {\tt K}_W]]}(\omega,\varpi)\\
&=J_{W}(\xi),
\end{aligned}
$$
since the integral on $A_{W;\mathfrak r,M,L,K,T}^{\rm c}$ does not contribute. We take $\psi^*$ of this form.

In order to apply Lemma~\ref{lem-lowerboundlogpint} to the third term on the right-hand side of \eqref{hatZNlimitlowbound}, we additionally require that $\xi$ is concentrated on $\Lcal_W\times \Scal_W^{\ssup{M,\vartheta,S}}$ where $\Scal^{\ssup{M,\vartheta,S}} $ is  defined in \eqref{SMvarthetaSdef}. We further confine the infimum to  those $\xi$'s of the form $\xi=\Pi_W(P)$ with an ergodic measure $P\in\Mcal_1^{\ssup{\rm s}}(\Lcal\times\Scal)$. Then $\psi^*=\partial\Pi_W^{\ssup\Scal}(P)$ is admissible according to Lemma~\ref{lem-Psi_admissible}. This implies that, for any $R,M, L,K,T,\vartheta, S,\delta,\delta'\in(0,\infty)$,
\begin{equation}\label{hatZNlimitlowbound3}
\begin{aligned}
\liminf_{N\to\infty}&\frac 1{|\L_N|}\log \widehat Z_{N,R,\delta}^{\ssup{\rm bc}}(\L_N,\rho_1,\rho_2)
\geq -\frac {C_{M,L,K,T}}R-\frac {C_M}{R^{d/2}}-\Big[\frac {C_{M,\vartheta,S}}R+\frac{C\vartheta^2}S\Big](\rho_2+\delta')
\\
&-\inf\Big\{\frac1{|W|}\langle \Pi_W(P),F_{W,W}\rangle+J_{W}(\Pi_W(P)) \colon P\in\Mcal_1^{\ssup{\rm s}}(\Lcal\times \Scal)\mbox{ ergodic, }\\ 
&\quad \Pi_W(P)(A_{W;\mathfrak r,M,L,K,T})=1
=\Pi_W^{\ssup\Scal}(P)(\Scal_W^{\ssup{M,\vartheta,S}}), \langle P,\mathfrak N_U^{\ssup {\ell,\Lcal}}\rangle\in\Bcal_{\delta/3}(\rho_1), \langle P,\mathfrak N_U^{\ssup {\ell,\Scal}}\rangle\in\Bcal_{\delta'}(\rho_2)\Big \}.
\end{aligned}
\end{equation}

In Lemma~\ref{lem-ergappr} we will show how to get rid of the two restrictions of $P$ being ergodic and $\Pi_W(P)$ being concentrated on $A_{W;\mathfrak r,M,L,K,T} $. In Section~\ref{sec-finishlowerbound}, we put this all together, make $R\to\infty$ and afterwards $M,T\to\infty$ and then $L,K\to\infty$ and finish the proof of the lower bound in Theorem~\ref{thm-freeenergy}.

\section{Analysis of the variational formula}\label{sec-propchi}

\noindent  This section is devoted to proofs of a number of results that evolve around the variational formula $\chi(\rho_1,\rho_2)$ in \eqref{chidefneu}.
In Section~\ref{sec-lemmaproof} we give  the proof of Lemma~\ref{lem-propertieschi} from  Section~\ref{sec-VP}, the regularity properties of $\chi(\rho_1,\rho_2)$. The other sections  give a number of technical preparations for some crucial points in the proofs of Theorems~\ref{thm-specrelent} and \ref{thm-freeenergy}. That is, in Section~\ref{sec-entropy} we explicitly decompose the entropies for loop point processes and for interlacement point processes according to the various parts of these point processes. In Section~\ref{sec-noshreds} we use this to prove Lemma~\ref{lem-noshreds}, i.e., that every loop/interlacement configuration can be approximated with a pure loop-configuration. In Section~\ref{sec-projentr} we show that entropies of  loop-shred configuration measures that are concentrated on uniformly bounded configurations approximate the entropy of arbitrary loop-shred configurations if the  bound diverges. Using this and a similar technique as in Section~\ref{sec-noshreds}, in Section~\ref{sec-ergappr} we carry out an  ergodic approximation in our main variational formula, i.e., we prove that the minimizers can be approximated by ergodic loop/interlacements configurations that are additionally uniformly bounded in various senses.

\subsection{Proof of Lemma~\ref{lem-propertieschi}}\label{sec-lemmaproof}

In this section, we prove Lemma~\ref{lem-propertieschi}.
As a preparation, we give a crucial lower bound on the expected energy $F_U$ defined in \eqref{FUdef}, which will be important for obtaining compactness properties. By $\widetilde {\mathfrak N}_W^{\ssup\ell}(\omega,\varpi)$ we denote the number of particles in $W$ in the loop/interlacement configuration $(\omega,\varpi)$, regardless where the loop starts in which is the particle is contained (if any). In Remark~\ref{rem-counting}, we argued that $\widetilde {\mathfrak N}_W^{\ssup\ell}$ has the same expectation under any $P\in\Mcal_1^{\ssup{\rm s}}(\Lcal\times\Scal)$ as ${\mathfrak N}_W^{\ssup\ell}$ defined in and below \eqref{def_particlenumbers}.

\begin{lemma}[Lower bound for expected interaction]\label{lem-lowboundinteract}
Suppose that $v$ satisfies  Assumption (V). Then, for any ergodic  $P\in\Mcal_1^{\ssup{\rm s}}(\Lcal\times\Scal)$,
\begin{equation}\label{FUmajorant}
\langle P,F_U\rangle \geq\frac{C\beta}2\big\langle P, \big(\widetilde{\mathfrak N}_{U}^{\ssup\ell}\big)^2\big\rangle -\frac {5C \beta}4,
\end{equation}
where $C$ is introduced in \eqref{LowBoundAss}. Moreover, the maps $\Mcal_1^{\ssup{\rm s}}(\Lcal\times\Scal)\ni P\mapsto  \langle P, \widetilde{\mathfrak N}_{U}^{\ssup{\ell,\Lcal}}\rangle$ and  $\Mcal_1^{\ssup{\rm s}}(\Lcal\times\Scal)\ni P\mapsto  \langle P, \widetilde{\mathfrak N}_{U}^{\ssup{\ell,\Scal}}\rangle$ are continuous on the set $\{P\colon \langle P, (\widetilde{\mathfrak N}_{U}^{\ssup\ell})^2\rangle\leq L\}$ for any $L\in(0,\infty)$.
\end{lemma}

\begin{proof} According to Remark~\ref{rem-counting}, we can replace $\langle P, F_U\rangle$ by the expectation of $\widetilde F_U(\omega,\varpi)$, the amount of interaction in the loop/interlacement configuration between any leg that starts in $U$ (regardless, whether it is in a loop or in an interlacement) and any other leg. Define $\widetilde F_W$ in the same way with $U$ replaced by $W=[-R,R]^d$. Define $\mathfrak L(\omega,\varpi)$ as the set of all legs of     $\omega$ or of $\varpi$. Then, by shift-invariance, for $R\in \N$, we can estimate
$$
\begin{aligned}
\langle P,F_U\rangle &=\frac 1{|W|}\langle P, \widetilde F_W\rangle
\geq \frac 1{|W|}\int_0^\beta \d s\,\int P(\d(\omega,\varpi))\,\sum_{g,g'\in\mathfrak L(\omega,\varpi)\colon g(s),g'(s)\in W}\1\{g\not= g'\}v(g(s)-g'(s))\\
&\geq \frac C{|W|}\int_0^\beta \d s\,\int P(\d(\omega,\varpi))\,\Big(\frac{\widetilde{\mathfrak N}_W(s)^2}{|W|}-\widetilde{\mathfrak N}_W(s)\Big),
\end{aligned}
$$ 
where we denote $\widetilde{\mathfrak N}_W(s)=\#\{ g\in\mathfrak L(\omega,\varpi)\colon g(s)\in W\}$, and we used the assumption \eqref{LowBoundAss}  of superstability in Assumption (V). Jensen's inequality gives
$$
\langle P,F_U\rangle\geq \frac {C\beta}{|W|}\int\d P\, \Big(\frac{\mathfrak G_W^2}{|W|}-\mathfrak G_W\Big)=C\beta\int \d P\,\Big(\Big(\frac{\mathfrak G_W}{|W|}\Big)^2-\frac{\mathfrak G_W}{|W|}\Big),
$$
where $\mathfrak G_W=\frac 1\beta \int_0^\beta\d s\,\widetilde{\mathfrak N}_W(s)$ denotes $\frac 1\beta$ times the total local time in $W$ of all the legs in the loop/interlacements configuration. We estimate $x^2-x\geq \frac 12 x^2 $ for $x\geq 2$ and $x^2-x\geq -\frac 14$ for $x\leq 2$; hence
$$
\langle P,F_U\rangle\geq -\frac{C\beta} 4+\frac {C\beta}2 \int \d P\,\Big(\frac{\mathfrak G_W}{|W|}\Big)^2\1\{\mathfrak G_W\geq 2 |W|\}.
$$
Denote by $\widetilde {\mathfrak N}_{\widetilde W}^{\ssup {\ell,M}}$ the number of particles in the entire configuration that lie in $\widetilde W$ (regardless where its loop starts, if it is contained in a loop) and have spread $\leq M$. If $R>M$, then  $ \mathfrak G_W\geq\widetilde{\mathfrak N}_{W_{R-M}}^{\ssup{\ell,M}}$, since all legs that start in $W_{R-M}$ and have spread $\leq M$ are entirely contained in $W=W_R$. Using this, we estimate
$$
\langle P,F_U\rangle\geq -\frac{C\beta} 4+\frac {C\beta}2 \int \d P\,\Big(\frac{\widetilde {\mathfrak N}_{W_{R-M}}^{\ssup {\ell,M}}}{|W|}\Big)^2\1\{\widetilde {\mathfrak N}_{W_{R-M}}^{\ssup {\ell,M}}\geq 2 |W|\}.
$$
Making $R\to\infty$ (recalling that $W=W_R$) and using the spatial ergodic theorem, and making afterwards $M\to\infty$, we see that
$$
\langle P,F_U\rangle\geq -\frac{C\beta} 4+\frac {C\beta}2 \int \d P\, \big(\widetilde{\mathfrak N}_{U}^{\ssup\ell}\big)^2\1\{\widetilde{\mathfrak N}_{U}^{\ssup\ell}\geq 2\}\geq \frac{C\beta}2\big\langle P, \big(\widetilde{\mathfrak N}_{U}^{\ssup\ell}\big)^2\big\rangle -\frac {5C \beta}4.
$$
More precisely, by ergodicity, $\widetilde{\mathfrak N}_{W_{R-M}}^{\ssup {\ell,M}}/|W|$ converges $P$-almost surely to $\widetilde{\mathfrak N}_{U}^{\ssup {\ell,M}}$ as $R\to\infty$, and we can use Lebesgue's theorem after lower estimating the integrand against $(\widetilde{\mathfrak N}_{W_{R-M}}^{\ssup {\ell,M}}/|W|)^2\wedge K\1\{\widetilde{\mathfrak N}_{W_{R-M}}^{\ssup {\ell,M}}\geq 2 |W|\}$ for some $K$. Afterwards, we make $K \to\infty$ and $M\to\infty$, using the monotone convergence theorem.

Now about the continuities. Recall that we are working with the topology on $\Mcal_1^{\ssup{\rm s}}(\Lcal\times\Scal) $ that is defined by the weak topologies induced by all the projections $\Pi_W^{\ssup\Lcal}$ and $\Pi_W^{\ssup\Scal}$ with $W\Subset \R^d$. Since the counting functionals are nonnegative, the lower semi-continuity is clear (for example via Fatou's lemma). As it concerns the shred-part, the upper semi-continuity follows from the estimate, for any $K \in (0,\infty)$,
$$
\widetilde{\mathfrak N}_{U}^{\ssup{\ell,\Scal}}
\leq \widetilde{\mathfrak N}_{U}^{\ssup{\ell,\Scal}}\,\1\{\widetilde{\mathfrak N}_{U}^{\ssup{\ell,\Scal}}\leq K\}
+\frac 1K \Big(\widetilde{\mathfrak N}_{U}^{\ssup{\ell,\Scal}}\Big)^2,
$$
since the $P$-integral of the last term vanishes uniformly in $P$ in that set as $K\to\infty$, and the first one is a bounded and continuous function of $\Pi_U^{\ssup\Scal}$. For the loop-part, we need to distinguish particles in loops that are entirely contained in $U$ and in the other loops. The first ones are functionals of $\Pi_U^{\ssup\Lcal}$, the latter ones are functionals of $\Pi_U^{\ssup\Scal}$. More precisely, writing $f_x\subset W$ if $f_{x,k}(0)\in W$ for all $k\in[\ell(f_x)]$, and using the notation $N_U(f)=\sum_{k\in[\ell(f)]}\1_U(f_{k}(0))$ for $f\in \Ccal^{\ssup\circlearrowleft}$, we have
$$
\begin{aligned}
\widetilde{\mathfrak N}_{U}^{\ssup{\ell,\Lcal}}(\omega)
&\leq \Big(\sum_{x\in \zeta\cap U\colon f_x\subset U}N_U(f_x)\Big)\1\Big\{\sum_{x\in \zeta\cap U\colon f_x\subset U}N_U(f_x)\leq K\Big\}+\frac 1K \Big(\sum_{x\in \zeta\cap U\colon f_x\subset U}N_U(f_x)\Big)^2\\
&\quad 
+\Big(\sum_{x\in \zeta\colon f_x\not\subset U}N_U(f_x)\Big)\1\Big\{\sum_{x\in \zeta\colon f_x\not\subset U}N_U(f_x)\leq K\Big\}+\frac 1K \Big(\sum_{x\in \zeta \colon f_x\not\subset U}N_U(f_x)\Big)^2.
\end{aligned}
$$
The first term on the right is a bounded and continuous function of $\Pi_U^{\ssup\Lcal}(\omega)$, the third is a bounded and continuous function of $\Pi_U^{\ssup\Scal}(\omega)$, and each of the two other terms are not bigger than $\frac 1K \widetilde {\mathfrak N}_U^{\ssup\ell}(\omega,\varpi)^2$, whose integral with respect to $P$ from that set is $\leq \frac LK$. This shows the upper semi-continuity of $P\mapsto \langle P, \widetilde{\mathfrak N}_{U}^{\ssup{\ell,\Lcal}}\rangle $ on the set of $P\in\Mcal_1^{\ssup{\rm s}}(\Lcal\times\Scal)$ satisfying $\langle P,( \widetilde{\mathfrak N}_{U}^{\ssup{\ell}})^2\rangle\leq L$.
\end{proof}

Now we proceed with the proof of Lemma~\ref{lem-propertieschi}.

{\em Proof of (1):} 
The compactness of $K$ follows from Theorem~\ref{thm-specrelent}, noting that, for any $P\in K$ and any $R\in\N$,
$$
\langle \partial\Pi_{W_R}^{\ssup\Scal}(P),(\mathfrak N_{\partial W_R}^{\ssup\ell})^2\rangle=\langle \Pi_{W_R}^{\ssup\Scal}(P),(\mathfrak N_{W_R}^{\ssup\ell})^2\rangle
\leq |W_R| \sum_{y\in Z_{R,1/2}}\langle \Pi_{y+U}^{\ssup\Scal}(P),(\mathfrak N_{y+U}^{\ssup\ell})^2\rangle\leq |W_R|^2 L,
$$
using that the distribution of $\mathfrak N_{\partial W_R}^{\ssup\ell}$ (under $\partial\Pi_{W_R}^{\ssup\Scal}(P)$) and of $\mathfrak N_{W_R}^{\ssup\ell}$ (under $\Pi_{W_R}^{\ssup\Scal}(P)$) are identical. Hence, $K$ is contained in a set of the form \eqref{levelset} and is therefore compact, according to Theorem~\ref{thm-specrelent}.

Furthermore, the maps $P\mapsto \langle P,\mathfrak N_U^{\ssup{\ell,\Lcal}}\rangle$ and $P\mapsto\langle P,\mathfrak N_U^{\ssup{\ell,\Scal}}\rangle$ are continuous on $K$, which can be proved similarly to Lemma~\ref{lem-lowboundinteract}.

From Lemma~\ref{lem-lowboundinteract} it directly follows that the value of the variational formula on the right-hand side of \eqref{chidefneu} with restriction to ergodic $P$s does not depend on the $L$ in the definition of the set $K$, if the variational problem in \eqref{chidefneu} is restricted to $K$ and $L$ is large enough. According to Lemma~\ref{lem-ergappr}, for every $P\in\Mcal_1^{\ssup{\rm s}}(\Lcal\times\Scal)$ with finite entropy and energy and particle density there is an ergodic measure $\in \Mcal_1^{\ssup{\rm s}}(\Lcal\times\Scal)$ with hardly differing values of all these quantities, therefore the statement is also true for \eqref{chidefneu} without restrictions.

\medskip

{\em Proof of (2):} Proving the convexities of $\chi$ and $\widetilde \chi$ are elementary exercises, which we leave to the reader; they rely on the convexity of the maps $P\mapsto \langle P, F_U\rangle$ and $P\mapsto \h^{\ssup{\Lcal,\Scal}}(P)$. Continuity  of $\overline  \chi$ and of $\chi(\rho_1,\cdot)$ and $\chi(\cdot,\rho_2)$  in $(0,\infty)$ for fixed $\rho_1$, respectively $\rho_2$, follows from convexity. The limiting statements as $\rho_1\to\infty$ and $\rho_2\to\infty$ follow directly from Lemma~\ref{lem-lowboundinteract} after using Jensen's inequality.

\medskip

{\em Proof of (3):}  For $\rho_1=0=\rho_2$, only the empty measure (the shift-invariant distribution of loops and interlacements that has none of them) is admissible in the variational problem. When examining its value in $\h^{\ssup{\Lcal,\Scal}}(\cdot)$, we first see that the interlacement part is zero, since also the reference measure $\partial\Pi_W^{\ssup\Scal}$ is empty. Furthermore, the $ {\tt Q}$-part is equal to the entropy $H(0\mid (q_k)_{k\in\N})$ (see \eqref{Entropydef}) of the zero sequence $m=(m_k)_{k\in \N}=(0)_k=0$ with respect to $(q_k)_k=(\frac 1k 4\pi\beta k)^{-d/2})_{k\in\N}$, and this entropy is equal to $q=\sum_{k\in\N} q_k=(4\pi\beta)^{-d/2 }\zeta (1+\frac d2)$ with $\zeta$ the Riemann zeta function (recall \eqref{limnorming}). That is, $\chi(0,0)=q$. For proving the second assertion, we need to upper bound $\frac 1{\rho_1}(\chi(\rho_1,0)-\chi(0,0))$ and make $\rho_1\downarrow 0$. We get an upper bound for $\chi(\rho_1,0)$ by inserting the marked Poisson point process ${\tt Q}^{\ssup m}$, defined as the reference process ${\tt Q}$, but with $(q_k)_{k\in\N}$ replaced by $m=(m_k)_{k\in\N}= 
\rho_1\delta_1=(\rho_1,0,0,\dots)$, which is admissible in the variational formula. Since the interaction potential $v$ is bounded, the expected interaction of ${\tt Q}^{\ssup m}$ is not larger than $C \rho_1^2$ for some $C\in(0,\infty)$ and all sufficiently small $\rho_1$. Its relative entropy density with respect to ${\tt Q}$ is equal to $H(m\mid (q_k)_k)=\sum_{k\in\N} q_k -\rho_1+\rho_1\log\frac {\rho_1}{q_1}$. Hence, 
$$
\frac 1{\rho_1}(\chi(\rho_1,0)-\chi(0,0))\leq 
\frac 1{\rho_1}\Big(C\rho_1^2-\rho_1+\rho_1\log\frac {\rho_1}{q_1}\Big)\leq \log\frac {\rho_1}{q_1},
$$
which tends to $-\infty$ as $\rho_1\downarrow 0$. By noting that $\chi(\rho_1,0)\geq \chi^{\ssup{v=0}}(\rho_1,0)$, which tends to $\chi(0,0)$ as $\rho_1\downarrow 0$, we also see that $\chi(\rho_1,0)$ is continuous in $\rho_1=0$. 

\medskip

{\em Proof of (4) and (5):} The two assertions (4) and (5) rely on the same standard arguments, which are based on the assertions in (1) and the lower semi-continuity of the two maps $P\mapsto \langle P, F_U\rangle $ and  $P\mapsto \h^{\ssup{\Lcal,\Scal}}(P)$, the latter according to Theorem~\ref{thm-specrelent}. We leave the details to the reader.

\subsection{Preparation: entropy decomposition}\label{sec-entropy}

Let us have a general look at the entropy terms that appear in the definition \eqref{JWdef} of $J_W(\xi)$ with respect to the reference measures, and decompose them into its pieces, first for loops. The next two lemmas will be jointly used in later sections for loop-interlacement distributions $\xi\in\Mcal_1(\Lcal_{W}\times\Scal_W)$ according to the formula
\begin{equation}\label{jointentropy}
\begin{aligned}
|W|J_W(\xi)&=H_{\Lcal_W\times\Scal_{W}}\big(\xi\,\big|\,\Pi_W^{\ssup\Lcal}({\tt Q})\otimes[ \partial\Pi_{W}^{\ssup\Scal}(\xi)\otimes{\tt K}_W]\big)\\
&=H_{\Lcal_W}\big(\Pi_W^{\ssup\Lcal}(\xi)\mid \Pi_W^{\ssup\Lcal}({\tt Q})\big) +\int_{\Lcal_W}\Pi_W^{\ssup\Lcal}(\xi)(\d\omega)\,H_{\Scal_{W}}\big(\xi^{\ssup\Scal}_\omega\,\big|\,\partial\Pi_{W}^{\ssup\Scal}(\xi)\otimes{\tt K}_W\big),
\end{aligned}
\end{equation}
writing $\xi^{\ssup\Scal}_\omega$ for the kernel $\xi_{\Lcal_W\to\Scal_W}(\omega,\cdot)$.

The entropy of a random loop configuration with respect to the Poisson loop reference measure (the Brownian loop soup) in a set $W$ is decomposed into the following partial entropies according to the random mechanisms: 
\begin{enumerate}

\item the expected number of $k$-length loops in $W$ as a sequence in $k\in\N$, 

\item the number of $k$-length loops in $W$ for any $k$, 

\item the locations of the starting sites of all these loops, and finally 

\item the loops themselves. 
\end{enumerate}

We write $\omega=\sum_{k\in\N}\sum_{x\in\zeta_k}\delta_{(x,f_x)}$ for loop configurations,  with $\zeta_k=\{x\in \zeta\colon \ell(f_x)=k\}$.  Furthermore, $\mathfrak N_W(\omega)$ denotes the number of loops in $W$ and $\mathfrak N_W^{\ssup{\delta_k}}(\omega)$ the number of loops of length $k$ in $W$. Let $m_{W,k}( \xi)=\frac1{|W|}\int \xi(\d\omega)\, \mathfrak N^{\ssup{\delta_k}}_W(\omega)$ denote the normalized expected number of loops of length $k$ in $W$ under $ \xi\in\Mcal_1(\Lcal_W)$. The Poisson parameter for $\Pi_{W,k}^{\ssup\Lcal}({\tt Q})$, the projection of the PPP ${\tt Q}$ on the set of $k$-length loops with all particles in $W$, is equal to
\begin{equation}\label{qWkdef}
q_{W,k}= \frac 1k\int_W\d x\,\mu_{x,x}^{\ssup{W,k,{\rm par}}}(\Ccal_k^{\ssup\circlearrowleft})= \frac 1k\int_W\d x\,\mu_{x,x}^{\ssup k}(\tau_W>k),
\end{equation}
where we recall that $\mu_{x,y}^{\ssup{W,k,{\rm par}}}$ is the (not normalized) Brownian bridge measure from $x$ to $y$ with all the particles in $W$, and we recall from \eqref{taudef} that $\beta \tau_W$ denotes the first time $\in \beta\N$ of an exit from $W$. We write $\Lcal_W^{\ssup k}$ for the set of $k$-length loop configurations in $W$. Introduce the projection $\Pi_k\colon \Lcal\to\Lcal^{\ssup k}$ defined by $\Pi_k(\omega)=\sum_{x\in\zeta}\1\{\ell(f_x)=k\}\delta_{(x,f_x)}=\sum_{x\in\zeta_k}\delta_{(x,f_x)}$. By $\overline \nu=\nu/\nu(\Xcal)$ we define the normalised version of a positive and finite measure $\nu$ on a Polish space $\Xcal$.

\begin{lemma}[Entropy of loop distributions]\label{lem-entrident} For any box $W\Subset \R^d$ and any $\xi\in\Mcal_1(\Lcal_W)$, 
\begin{equation}\label{entropyident}
\begin{aligned}
H&_{\Lcal_W}\big(\xi\mid \Pi_W^{\ssup\Lcal}({\tt Q})\big)\\
&=|W|\,H\big((m_k)_k\,\big|\, \smfrac 1{|W|} (q_{k})_k\big)+ H_{\Lcal_W}\Big(\xi\,\Big|\,\bigotimes_{k\in\N} \Pi_{k}(\xi)\Big)\\
&\quad +  \sum_{k\in\N} H_{\N_0}(\Pi_k(\xi)\circ  \mathfrak N_W^{-1}\mid \Poi_{m_k|W|})\\
&\quad + \sum_{k\in\N}\sum_{n\in\N} \Pi_k(\xi)(\mathfrak N_W=n)\, H\big(\Pi_{k}^{\ssup n}(\xi)\circ (\zeta_k\cap W)^{-1}\mid \overline \Leb_W^{\otimes n}\big)\\
 &\quad + \sum_{k\in\N}\sum_{n\in\N_0}\int \Pi_k^{\ssup n}(\xi)(\zeta_k\cap W \in \d A)\,\Big[ H\Big(\Pi_k^{\ssup A}(\xi)\,\Big|\, \bigotimes_{x\in A} \Pi_{k,x}^{\ssup A}(\xi)\Big)+ \sum_{x\in A} H\big(\Pi_{k,x}^{\ssup A}(\xi) \mid \overline \mu _{x,x}^{\ssup{W,k,{\rm par}}}\big)\Big],
\end{aligned}
\end{equation}
where we  abbreviate $m_k=m_{W,k}(\xi)$ and $q_k=m_{W,k}(\Pi_W^{\ssup\Lcal}({\tt Q}))$ and denote by $\Pi_k, \Pi_k^{\ssup n}, \Pi_k^{\ssup A}, \Pi_{k,x}^{\ssup A}\colon \Lcal_W\to\Lcal_W^{\ssup k}$ the projections on the set of loops of length $k$, respectively on the set of configurations in $W$ with precisely $n$ loops of length $k$ and no loops else, respectively on the set of configurations of loops of length $k$ with set of initial sites equal to $A\subset W$.
\end{lemma}

\begin{proof}
  We write $\Pi_{W,k}^{\ssup\Lcal}
=\Pi_{W}^{\ssup\Lcal}\circ\Pi_k=\Pi_k\circ \Pi_{W}^{\ssup\Lcal}$. In the following we drop the super-index $\Lcal$. Note that $\Pi_W({\tt Q})$ is equal to the convolution of the measures $\Pi_{W,k}({\tt Q})$ over $k\in\N$. Then one sees elementarily that 
\begin{equation}\label{firstsplit}
H_{\Lcal_W}(\xi\mid \Pi_W({\tt Q}))
=\sum_{k\in\N} H_{\Lcal_W^{\ssup k}} (\Pi_k(\xi)\mid \Pi_{W,k}({\tt Q})) +  H_{\Lcal_W}\Big(\xi\,\Big|\,\bigotimes_{k\in\N} \Pi_k(\xi)\Big).
\end{equation}
Let us further decompose $H_{\Lcal_W^{\ssup k}} (\Pi_k(\xi)\mid \Pi_{W,k}({\tt Q})) $.  We write a configuration in $\Lcal_W^{\ssup k}$ as $\omega_k=\sum_{x\in \zeta_k\cap W}\delta_{(x,f_x)}$. Note that $|\zeta_k\cap W|$ has the distribution $\Poi_{q_{W,k}}$ under $\Pi_{W,k}({\tt Q})$, where we recall \eqref{nnBBM}. By the properties of a Poisson point process, $\Pi_{W,k}({\tt Q})$ can be decomposed as
\begin{equation}\label{Qdecomp}
\Pi_{W,k}({\tt Q})(\d\omega_k)=\Poi_{q_{W,k}}(n)\,\overline \Leb_W^{\otimes n}(\d A)\,\bigotimes_{x\in A} \overline \mu _{x,x}^{\ssup{k,W}}(\d f_x),\qquad n=\mathfrak N_W(\omega_k)=|A|, A=\zeta_k\cap W.
\end{equation}

We now write $\Pi_k(\xi)$ in an analogous manner. We abbreviate  $ \Pi_k^{\ssup n}(\xi)$ for the conditional distribution under $\Pi_k(\xi)$ given that $|\zeta_k\cap W|=n$ and  $ \Pi_k^{\ssup A}(\xi)$ for the conditional distribution under $\Pi_k(\xi)$ given that $\zeta_k\cap W=A$, and $\Pi_{k,x }^{\ssup A}(\xi)$ for the distribution of any $f_x$ under $\Pi_k^{\ssup A}(\xi)$. Write $m_k=\frac 1{|W|}\Pi_k(\xi)(\mathfrak N^{\ssup\ell}_W)$ for the normalised expected particle number in the configuration under $\Pi_{k}(\xi)$ in $W$. Then we can write, with the same notation as above,
$$
\begin{aligned}
\Pi_k(\xi)(\d \omega_k)
&={\Poi}_{m_k |W|}(n) 
\frac{\d \Pi_k(\xi)\circ  \mathfrak N_W^{-1}}{\d{\Poi}_{m_k|W|}}(n)
\,\overline\Leb_W^{\otimes n}(\d A)\,
\frac{\d  \Pi^{\ssup n}_{k}(\xi)\circ (\zeta_k\cap W)^{-1}}
{\d \overline\Leb_W^{\otimes n}}(A)\\
&\quad \times\frac{\d \Pi_k^{\ssup A}(\xi)}{\d \bigotimes_{x\in A} \Pi_{k,x }^{\ssup A}(\xi)}\big((f_x)_{x\in A}\big)
\bigotimes_{x\in A}\frac{\d \Pi_{k,x }^{\ssup A}(\xi)}
{\d  \overline \mu _{x,x}^{\ssup{W,k,{\rm par}}}}(f_x)
\,\bigotimes_{x\in A} { \overline \mu _{x,x}^{\ssup{W,k,{\rm par}}}}(\d f_x).
\end{aligned}
$$
Hence, we see that the density is given by
$$
\begin{aligned}
\frac{\d \Pi_k(\xi)}{\d \Pi_{W,k}({\tt Q})}(\omega_k)
=&\frac{\d {\Poi}_{m_k|W|}(n)}{\d {\Poi}_{q_{W,k}}(n)}\, \frac{\d \Pi_k(\xi)\circ  \mathfrak N_W^{-1}}{\d{\Poi}_{m_k|W|}}(n)\, \frac{\d \Pi_{k}^{\ssup n}(\xi)\circ (\zeta_k\cap W)^{-1}}
{\d \overline \Leb_W^{\otimes n}}(A)\\
&\quad\times  \frac{\d \Pi_{k}^{\ssup A}(\xi)}{\d \bigotimes_{x\in A} \Pi_{k,x}^{\ssup A}(\xi)}\big((f_x)_{x\in A}\big)\,
\prod_{x\in A}\frac{\d \Pi_{k,x}^{\ssup A}(\xi)}{\d  \overline \mu _{x,x}^{\ssup{k,W}}}(f_x).
\end{aligned}
$$
Now take the logarithm and integrate with respect to $\Pi_k(\xi)(\d\omega_k)$ and use  \eqref{firstsplit}, to obtain the assertion in \eqref{entropyident}.
\end{proof}

Now we do the same for the shreds. We write $\xi_W$ for the kernel $\xi_{\Tcal_W\to\Scal_W}$. We also recall the notation
\begin{equation}\label{qfreedef}
q_{x,y}^{\ssup{n,W}}(\d f)=\P_x(B\in\d f\mid \tau_W=n,B_{\tau \beta}=y),\qquad n\in\N, x\in W,y\in W^{\rm c},
\end{equation}
for the distribution of an $n$-leg $W$-shred from $x$ to $y$ (which appears in the definition \eqref{kernelK} of ${\tt K}_W$).

\begin{lemma}[Entropy of shred distributions]\label{lem-entridentshreds} For any box $W\Subset \R^d$ and any $\xi\in\Mcal_1(\Scal_W)$, 
\begin{equation}\label{entropyidentshred}
\begin{aligned}
H_{\Scal_{W}}\big(\xi\,\big|\, \partial\Pi_{W}^{\ssup\Scal}(\xi)\otimes{\tt K}_W\big)
&=\int_{\Tcal_W}\partial\Pi_{W}^{\ssup\Scal}(\xi)(\d \mu)\, H_{\Scal_W}\Big(\xi(\mu,\cdot)\,\Big|\,\bigotimes_{i\in I}\Pi_i(\xi(\mu,\cdot))\Big)\\
&\quad +\int_{\Tcal_W}\partial\Pi_{W}^{\ssup\Scal}(\xi)(\d \mu)\, \sum_{i\in I}H_{\Ccal_{l_i}}\big( \Pi_i(\xi(\mu,\cdot))\,\big|\, q_{x_i,y_i}^{\ssup{l_i,W}}\big),
\end{aligned}
\end{equation}
where we wrote $\mu=\sum_{i\in I}\delta_{(x_i,l_i,y_i)}$ and $\Pi_i\colon \Ccal^I\to \Ccal$ for the canonical projection on the $i$-th factor and $\xi(\mu,\cdot)$ for the kernel $(\xi)_{\Tcal_W\to\Scal_W}(\mu,\cdot)$.
\end{lemma}

\begin{proof}We drop all super-indices $\Scal$.  Then we  see that
$$
 \begin{aligned}
H_{\Scal_{W}}&\big(\xi\,\big|\, \partial\Pi_{W}(\xi)\otimes{\tt K}_W\big)
=\int_{\Tcal_W}\partial\Pi_{W}(\xi)(\d \mu)\, H_{\Scal_W}\big(\xi(\mu,\cdot)\,\big|\,{\tt K}_W(\mu,\cdot)\big)\\
&=\int_{\Tcal_W}\partial\Pi_{W}(\xi)(\d \mu)\,
\Big[H_{\Scal_W}\Big(\xi(\mu,\cdot)\,\Big|\,\bigotimes_{i\in I}\Pi_i(\xi(\mu,\cdot))\Big)
+\sum_{i\in I}H_{\Ccal_{l_i}}\big( \Pi_i(\xi(\mu,\cdot))\,\big|\, q_{x_i,y_i}^{\ssup{l_i,W}}\big)\Big].
\end{aligned}
$$
\end{proof}

\subsection{Proof of Lemma~\ref{lem-noshreds}}\label{sec-noshreds}

In this section, we prove Lemma~\ref{lem-noshreds}, which says that $\chi(\rho_1+\rho_2,0)\leq \chi(\rho_1,\rho_2)$ for any $\rho_1,\rho_2\in[0,\infty)$. For this, we will replace an arbitrary loop-interlacements configuration by a loop configuration without changing much the number of particles per volume and without increasing the energy nor the entropy terms. We do this by rewiring the ends of all the shreds at the boundary of a large box $W=W_R$ in such a way that all the shreds of interlacements are turned into the shreds of one long loop just outside that box, i.e., in $W_{R+M}\setminus W_R$. The new configuration has only loops in $W_{R+M}$ and no shreds. Our manipulation does not change anything inside $W_R$, which is the main part, and  the changes in $W_{R+M}\setminus W_R$ have only a vanishing influence in the limit as $R\to\infty$, as it concerns the three main quantities: entropy, expected interaction energy per volume and particle densities. Finally, we put independent copies of the new configuration into the boxes $z+W_{R+M}$, $z\in 2(R+M)\Z^d$, and mix the configuration in $\R^d$ over $x\in W_{R+M}$ and control the entropy of this configuration.

We need some more notation. We need some large parameters $R,M,T\in(1,\infty)$. As usual, we write $W=[-R,R]^d$. Denote by $\widetilde\Pi_{W;M}\colon \Lcal_W\times\Scal_W\to\Lcal^{\ssup M}_W\times\Scal^{\ssup M}_W$ the natural projection on loops, respectively shreds, whose legs $f\in\Ccal_1$ have throughout spread $\|f\|_{\rm sp}=\max_{t\in [0,\beta]}|f(t)-f(0)| < M$. We denote by $\widetilde J_{W;M}$ the normalized entropy function as in \eqref{JWdef}, but with the reference measure $\Pi_W^{\ssup\Lcal}({\tt Q})\otimes [\partial \Pi_W^{\ssup\Scal}(P)\otimes {\tt K}_W]$ replaced by its image measure under $\widetilde \Pi_{W;M}$. Note that $\widetilde J_{W;M}(\xi\circ \widetilde \Pi_{W;M}^{-1})\leq J_W(\xi)$ for any $\xi\in\Mcal_1(\Lcal_W\times\Scal_W)$, by monotonicity of the relative entropy in the reference $\sigma$-field, see \cite[Proposition 15.5(c)]{G88}.

Furthermore, we write $\widetilde W=W_{R+M}=[-R-M,R+M]^d$ and  denote by $A_{\widetilde W;M,T}=\{(\omega,\varpi)\in \Lcal^{\ssup M}_{\widetilde W} \times\Scal^{\ssup M}_{\widetilde W}\colon \widetilde{\mathfrak N}^{\ssup\ell}_{z+U}(\omega,\varpi)\leq T\,\forall z\in {\widetilde W}\cap \Z^d\}$, where we recall that $U=[-\frac 12,\frac 12]^d$ is the unit box and $\widetilde{\mathfrak N}^{\ssup\ell}_{z+U}$ denotes the number of particles in the loop/shred configuration in $z+U$, regardless where the loop starts; see also Remark~\ref{rem-counting}. (This event is very similar to  $A_{{\widetilde W};\frac 12,M,K,L,T}$ defined in \eqref{Adef}.) Furthermore, let $\widetilde \Pi_{{\widetilde W};M,T}\colon \Lcal_{\widetilde W}\times\Scal_{\widetilde W} \to A_{{\widetilde W};M,T}$ be some projection onto $A_{{\widetilde W};M,T}$.

Let us explain the structure of our proof.
We may assume that $\rho_1,\rho_2\in(0,\infty)$. Pick $P\in\Mcal_1^{\ssup{\rm s}}(\Lcal\times\Scal)$ with $\langle P,\mathfrak N_U^{\ssup{\ell,\Lcal}}\rangle=\rho_1$ and $\langle P,\mathfrak N_U^{\ssup{\ell,\Scal}}\rangle=\rho_2$. Furthermore, fix $\eps\in(0,1)$. Pick $M$ and $T$ so large that $\langle \Pi_U(P),\widetilde{\mathfrak N}_U^{\ssup\ell}\circ \widetilde \Pi_{U;M,T}\rangle \geq \rho_1+\rho_2-\eps$, where $\widetilde{\mathfrak N}_U^{\ssup\ell}= \widetilde{\mathfrak N}_U^{\ssup{\ell, \Lcal}}+\widetilde{\mathfrak N}_U^{\ssup{\ell,\Scal}}$ is the number of all particles in loops and in shreds in $U$.

We are going to prove that, for any sufficiently large $M$ and $T$, and for any sufficiently large  $R\in(1,\infty)$, there exists some ergodic $\widetilde P\in\Mcal_1^{\ssup {\rm s}}(\Lcal)$ such that $|\langle \Pi_U(\widetilde P), \widetilde{\mathfrak N}_U^{\ssup{\Lcal,\ell}}\circ \widetilde \Pi_{U;M,T}\rangle- \langle \Pi_U(P), \widetilde{\mathfrak N}_U^{\ssup\ell}\circ \widetilde \Pi_{U;M,T}\rangle|\leq \eps$ and $\langle \frac 1{|\widetilde W|}F_{\widetilde W,\widetilde W}\circ\widetilde \Pi_{\widetilde W;M,T},\Pi_{\widetilde W}(\widetilde P)\rangle \leq \langle \frac 1{|W|}F_{W,W}\circ\widetilde \Pi_{W;M,T}, \Pi_W(P)\rangle +\eps$ and $J_{\widetilde W}(\Pi_{\widetilde W}(\widetilde P))\leq \widetilde J_{W;M}(\Pi_{W}(P)\circ \widetilde \Pi_{W;M}^{-1})+\eps$.

Indeed, this implies for any $\eps$ and all sufficiently large $M$ and $T$ and for any sufficiently large $R$ that
$$
\begin{aligned}
\inf\Big\{&\Big\langle  \frac 1{|W_{R+M}|}F_{W_{R+M},W_{R+M}}\circ \widetilde\Pi_{W_{R+M};M,T},\Pi_{W_{R+M}}(P)\Big\rangle +J_{W_{R+M}}(\Pi_{W_{R+M}}(P))\colon \\
&\qquad\qquad\qquad P\in\Mcal_1^{\ssup{\rm s}}(\Lcal)\mbox{ ergodic},\langle  \Pi_U(P),\mathfrak N_U^{\ssup{\Lcal,\ell}}\circ \widetilde \Pi_{U;M,T}\rangle \in\Bcal_\eps(\rho_1+\rho_2)\Big\}\\
&\leq 2\eps+\inf\Big\{\Big\langle \frac 1{|W|}F_{W,W}\circ \widetilde\Pi_{W;M,T},\Pi_W(P)\Big\rangle +\widetilde J_{W;M}(\Pi_{W}(P)\circ \widetilde\Pi_{W;M}^{-1})\colon \\
&\qquad\qquad\qquad  P\in\Mcal_1^{\ssup{\rm s}}(\Lcal\times\Scal),\langle P,\mathfrak N_U^{\ssup{\ell,\Lcal}}\rangle=\rho_1,\langle P,\mathfrak N_U^{\ssup{\ell,\Scal}}\rangle=\rho_2\Big\}.
\end{aligned}
$$
The right-hand side is not larger than $3\eps+\chi(\rho_1,\rho_2)$, as we show now. Indeed, clearly $F_{W,W}\circ \widetilde\Pi_{W;M,T}\leq F_{W,W} $. Furthermore, splitting all the contributions to $F_{W,W}$, one sees that, for any $\omega\in\Lcal$ and $\varpi\in\Scal$, 
\begin{equation}\label{splittingF}
\begin{aligned}
F_ {W,W}(\Pi_{W}(\omega,\varpi))&\leq \sum_{z\in W\cap\Z^d}
\Big[\frac 12 F_{z+U,z+U}^{\ssup{\Lcal\Lcal}}+F_{z+U,z+U}^{\ssup{\Lcal\Scal}}+\frac 12 F_{z+U,z+U}^{\ssup{\Scal\Scal}}\\
&\qquad +F_{z+U,z+U^{\rm c}}^{\ssup{\Lcal\Lcal}}+F_{z+U,z+U^{\rm c}}^{\ssup{\Lcal\Scal}}+F_{z+U^{\rm c},z+U}^{\ssup{\Lcal\Scal}}+
F_{z+U,z+U^{\rm c}}^{\ssup{\Scal\Scal}}\Big](\omega,\varpi).
\end{aligned}
\end{equation}
Now integrate both sides with respect to $P$ and use its shift-invariance, to see that $\langle \frac 1{|W|}F_{W,W}\circ \widetilde\Pi_{W;M,T},P\rangle\leq \langle F_U,P\rangle$, where we recall \eqref{FUdef} for the definition of $F_U$. Furthermore,  $\widetilde J_{W;M}(\Pi_{W}(P)\circ \widetilde\Pi_{W;M}^{-1})\leq \eps + J_{W}(\Pi_{W}(P))\leq \eps+ \h^{\ssup{\Lcal,\Scal}}(P)$ by \eqref{hbound}. 

The left-hand side converges, as $R,M,T\to\infty$ and $\eps\downarrow 0$ to $\chi(\rho_1+\rho_2,0)$, as we show now. Indeed, for any $\eps\in(0,1)$ and any ergodic $P\in\Mcal_1^{\ssup{\rm s}}(\Lcal)$ with $\langle  P,\widetilde {\mathfrak N}_U^{\ssup{\Lcal,\ell}}\rangle =\rho_1+\rho_2$ and finite values of $\h^{\ssup{\Lcal}}(P\mid{\tt Q})$ and $\langle F_U,P\rangle$, pick $M$ and $T$ so large that $P$ is amenable to the formula on the left and such that $ \langle F_U\circ\widetilde \Pi_{U;M,T},P\rangle\geq \langle F_U,P\rangle-\eps$. Making $R\to\infty$ and using Theorem~\ref{thm-specrelent}, we see that $J_{W_{R+M}}(\Pi_{W_{R+M}}(P))$ converges towards $\h^{\ssup\Lcal}(P\mid{\tt Q})$. Furthermore, making $R\to\infty$ we see that $\langle  \frac 1{|W_{R+M}|}F_{W_{R+M},W_{R+M}}\circ \widetilde\Pi_{W_{R+M};M,T},\Pi_ {W_{R+M}}(P)\rangle$ converges towards $ \langle F_U\circ\widetilde \Pi_{U;M,T},P\rangle$.
Since $\chi(\rho_1+\rho_2,0)$ is the infimum over all $P\in\Mcal_1^{\ssup{\rm s}}(\Lcal)$ (not only the ergodic ones), we have arrived at the desired estimate, $\chi(\rho_1+\rho_2,0)\leq \chi(\rho_1,\rho_2)$.

Here is our construction of $\widetilde P$. Recall 
$$
W=W_R=[-R,R]^d\qquad \mbox{and}\qquad \widetilde W=W_{R+M}=[-R-M,R+M]^d\qquad\mbox{and}\qquad W'=\widetilde W\setminus W.
$$
Consider $\xi_R=\Pi_{W}(P)\circ\widetilde \Pi_{W;M}^{-1}\in\Mcal_1(\Lcal_{W}^{\ssup{M}}\times \Scal_{W}^{\ssup{M}})$. Furthermore, assume that $M$ is also so large that 
 \begin{equation}\label{Mlargespreads}
 \frac{\d\nu}{\d\nu\circ\widetilde \Pi_{W;M}^{-1}}(\omega,\varpi)\leq (1+\eps)^{\widetilde{\mathfrak N}_W^{\ssup\ell}(\omega,\varpi)},\qquad R>1,(\omega,\varpi)\in \Lcal_W\times\Scal_W,
 \end{equation} 
where $\nu=\Pi^{\ssup\Lcal}_W({\tt Q})\otimes [\partial \Pi^{\ssup\Scal}_{W}(\xi_R)\otimes {\tt K}_W]$ (this is possible since all the leg spreads are i.i.d.~under $\nu$). We are going to construct from $\xi_R$ some loop-configuration measure $\xi_R^{\ssup M}\in \Mcal_1(\Lcal_{\widetilde W})$ as follows. First, we add an empty loop configuration $\delta_{\underline 0}$ in $W'=\widetilde W\setminus W$ (i.e., no loop starts in $W'$). Furthermore, we extend all the initial and terminal sites of all the $W$-shreds of $\xi_R$ by appending independent Brownian bridge pieces with particles only within $W'$ such that a single large loop in $\widetilde W$ arises. Hence, $\Pi_{W}(\xi_R^{\ssup M})=\Pi_{W}(\xi_R)=\xi_R$, and $\xi_R^{\ssup M}$ has no $\widetilde W$-shreds. Then we decompose $\R^d$ regularly into the boxes $z+\widetilde W$ with $z\in 2(R+M)\Z^d$ and put into each such subbox one independent copy of a loop-configuration with distribution  $\theta_z(\xi_R^{\ssup M})$ (where $\theta_z$ is the shift-operator such that $\theta_z(z)=0$). Call the distribution of this configuration $\widetilde P_{R,M}=\bigotimes_{z\in 2(R+M)\Z^d} \theta_z(\xi_R^{\ssup M})$ and define $P_{R,M}=\frac 1{|W_{R+M}|}\int_{W_{R+M}}\d x\,\widetilde P_{R,M}\circ\theta_x^{-1}$. Then $P_{R,M}\in \Mcal_1^{\ssup{\rm s}}(\Lcal)$ is obviously stationary and contains no interlacement. The proof of \cite[Theorem 14.12]{G88} applies and shows that $P_{R,M}$ is ergodic. 

We need to show the following: 
\begin{eqnarray}
\lim_{R\to\infty}\langle P_{R,M}, \widetilde{\mathfrak N}_U^{\ssup{\Lcal,\ell}}\circ \widetilde \Pi_{U;M,T}\rangle-\langle P, [\widetilde{\mathfrak N}_U^{\ssup{\ell,\Lcal}}+\widetilde{\mathfrak N}_U^{\ssup{\ell,\Scal}}]\circ \widetilde \Pi_{U;M,T}\rangle& =&0,\label{Goal1a}\\
\limsup_{R\to\infty}\Big[\Big\langle\frac 1{|\widetilde W|} F_{\widetilde W,\widetilde W}\circ \widetilde \Pi_{\widetilde W;M,T}, P_{R,M}\Big\rangle- \Big\langle \frac1{|W|}F_{W,W}\circ\widetilde \Pi_{W;M,T},P\Big\rangle\Big]&\leq& 0,\label{Goal2a}\\
\limsup_{R\to\infty}[J_{\widetilde W}(\Pi_{\widetilde W}(P_{R,M}))-  \widetilde J_{W;M}(\Pi_{W}(P)\circ \widetilde\Pi_{W;M}^{-1})]&\leq& \eps.\label{Goal3a}
\end{eqnarray}
Then $\widetilde P=P_{R,M}$ with sufficiently large $R$ is suitable, and the assertion follows.

Properties \eqref{Goal1a} and \eqref{Goal2a} are easy to show and do not need details of the construction:

\medskip

{\it Proof of \eqref{Goal1a}:} Since $|\widetilde W|/|W|\to 1$, it suffices to show that $\lim_{R\to\infty}\frac 1{|W|}\langle \xi_R^{\ssup M},\widetilde  {\mathfrak N}^{\ssup\ell}_{\widetilde W}\circ \Pi_{\widetilde W;M,T}\rangle=
\lim_{R\to\infty}\frac 1{|W|}\langle \xi_R, [\widetilde  {\mathfrak N}_{W}^{\ssup{\Lcal,\ell}}+\widetilde  {\mathfrak N}_{W}^{\ssup{\Lcal,\ell}}]\circ \Pi_{W;M,T}\rangle$.  Recall that $\xi_R^{\ssup M}$ has its particles from (1) the loops of $\xi_R$ in $W$, (2)  the shreds of $\xi_R$ in $W$, and (3) the added shreds in $W'$.  Therefore, it suffices to show that the expected amount of the particles of (3) is negligible in the limit $R\to\infty$. But this is clear since  $\widetilde  {\mathfrak N}_{W'}^{\ssup{\ell,\Scal}}\circ \Pi_{W';M,T}\leq  |W'|T\leq CMT|W|/R$ for some constant $C$ that depends only on $d$.

\medskip

{\it  Proof of \eqref{Goal2a}:} Again, we only have to show that the expected $F_{W',\widetilde W}\circ\Pi_{\widetilde W;M,T}$-interactions of the added Brownian shreds in $W'$ with itself and with any other part of the configuration in $\widetilde W$ are $o(|W|)$ as $R\to\infty$. But this is again clear, since $F_{W',\widetilde W}\circ\Pi_{\widetilde W;M,T}\leq C_{M,T}|W|/R$ for some constant $C_{M,T}$ that depends only on $M$ and $T$ (and possibly on $\beta$, $d$ and $v$); see Lemma~\ref{lem-Interactesti}.

\medskip
  
Now let us show that property \eqref{Goal3a} is satisfied.

Let us first make a connection between $J_{\widetilde W}(\Pi_{\widetilde W}(P_{R,M}))$ and $J_{\widetilde W}(\Pi_{\widetilde W}(\widetilde P_{R,M}))$ (the latter is $J_{\widetilde W}(\xi_R^{\ssup M})$). By convexity of $J_{\widetilde W}$ and by \eqref{JWupperboundx} in Proposition~\ref{lem-uppboundJ_W}(1) with $m=3$ (it is clear that $ \widetilde P_{R,M }$ is $\widetilde W$-ergodic and hence $3\widetilde W$-ergodic), we have, for all $R$ sufficiently large, 
\begin{equation}\label{J(P)_J(widetildeP)}
\begin{aligned}
J_{\widetilde W}(\Pi_{\widetilde W}(P_{R,M}))
&\leq \frac1{|\widetilde W|}\int_{\widetilde W}\d x\, J_{\widetilde W}(\Pi_{\widetilde W}(\widetilde P_{R,M}\circ\theta_x^{-1}))= \frac1{|\widetilde W|}\int_{\widetilde W}\d x\,J_{\widetilde W}(\theta_x(\Pi_{x+\widetilde W}(\widetilde P_{R,M})))\\
&\leq J_{W_{3(R+M)}}(\Pi_{W_{3(R+M)}}(\widetilde P_{R,M}))+\eps,
\end{aligned}
\end{equation}
where we used that, by construction, $\widetilde P_{R,M}$ has no shred-part and therefore has $\psi=0$ for $\psi=\partial\Pi_{W_{3(R+M)}}^{\ssup\Scal}(\widetilde P_{R,M})$.

Since, by construction, $\Pi_{W_{3(R+M)}}(\widetilde P_{R,M})$ and also the reference measure in the entropy $J_{W_{3 (R+M)}}$, decomposes into the product of $3^d$ shifted copies  of $\Pi_{\widetilde W}(\widetilde P_{R,M})$ (recall that it has no loops in $z+W'$ for all the $z\in 2(R+M)\Z^d$ and therefore an empty boundary-shred part at the boundaries of all these subboxes $z+\widetilde W$), the right-hand side is not larger than $J_{\widetilde W}(\Pi_{\widetilde W}(\widetilde P_{R,M}))+\eps$. Hence, it will be sufficient to prove \eqref{Goal3a} with $J_{\widetilde W}(\Pi_{\widetilde W}(\widetilde P_{R,M}))=J_{\widetilde W}(\xi_R^{\ssup M})$ instead of $J_{\widetilde W}(\Pi_{\widetilde W}(P_{R,M}))$. More explicitly, we need to show that $\limsup_{R\to\infty}[J_{\widetilde W}(\xi_R^{\ssup M})-\widetilde J_{W;M}(\xi_R)]\leq 0$. Recall that the reference measure for the entropy of $\xi_R^{\ssup M}$ is $\Pi_{\widetilde W}^{\ssup\Lcal}({\tt Q})$, while the one for $\xi_R$ is equal to $[\Pi_{W}^{\ssup\Lcal}({\tt Q})\otimes [\partial\Pi_{W}^{\ssup\Scal}(\xi_R)\otimes {\tt K}_{W}]]\circ \widetilde\Pi_{W;M}^{-1}$.

However, proving property \eqref{Goal3a} needs the details of the construction of $\xi_R^{\ssup M}$ from $\xi_R$, which we are going to describe now.  Fix a configuration $(\omega,\varpi)\in \Lcal^{\ssup{M}}_{W}\times\Scal^{\ssup{M}}_{W}$. Write $\varpi=\sum_{i\in I}\delta_{f_i}$ and  $\mu=\partial\Pi_{W}^{\ssup \Scal}(\varpi)=\sum_{i\in I}\delta_{(x_i,l_i,y_i)}$ with $l_i=\ell(f_i)$ and $f_i(0)=x_i$ and $f_i(\beta l_i)=y_i$.   Then $x_i\in W$ and $y_i\in W'$ for any $i\in I$, since each leg of a loop in $\varpi$ has spread $\leq M$ (we consider $\varpi$ under the distribution $\xi_R$, which has this property almost surely). With some bijection $\sigma\colon I\to I$ that has only one cycle, we sample, conditionally given $\mu$, a collection of random lengths $(L_i)_{i\in I}$ with distribution $\bigotimes_{i\in I}p_{y_i, x_{\sigma(i)}}^{\ssup {W'}}$, where we define 
\begin{equation}\label{pxyneudef}
p_{y,x}^{\ssup{W'}}(n)=\P_y(\tau_{W'}=n\mid B_{\tau_{ W'}\beta}=x), \qquad y\in W', x\in (W')^{\rm c}, n\in\N,
\end{equation} 
(compare to \eqref{pMdef}). Then we append, for $i\in I$, an $ L_i$-length Brownian bridge $B_i$ with particles only in $W'$ at $y_i$ terminating at $x_{\sigma(i)}$. In other words, each $B_i$ has distribution $p_{y_i, x_{\sigma(i)}}^{\ssup {W'}}\otimes q_{y_i,x_{\sigma(i)}}^{\ssup{\cdot, W'}}$, where
\begin{equation}\label{qneudef}
q^{\ssup{n,W'}}_{y,x}(\d f)
=\P_y\big((B_s)_{s\in [0, n\beta]}\in \d f\mid \tau_{W'}=n, B_{n\beta}=x\big),\qquad y\in W', x\in (W')^{\rm c}, n\in\N, f\in \Ccal_n,
\end{equation}
(compare to \eqref{qdef}, $\tau_{W'}$ was defined in \eqref{taudef}), and $(B_i)_{i\in I}$ is independent. 

In this way, all shreds of $\xi_R$ are turned into a single loop $f\in \Ccal^{\ssup\circlearrowleft}$ in $\widetilde W$ with length $\ell(f)=n=\sum_i (l_i+ L_i)$. By default, we take the starting site as $x_{i^*}\in W$, where $i^*$ is uniformly distributed over $I$.  This loop  runs subsequently through the functions $f_{{i^*}}$, $g_{{i^*}}$, $f_{{\sigma(i^*)}}$, $g_{{\sigma(i^*)}}$, $f_{{\sigma^2(i^*)}}$, $g_{{\sigma^2(i^*)}}$ and so on, more precisely, it is the concatenation of these functions in this order. The change from one $f_i$ to some $g_i$ happens at the stopping times where the next particle $\in W'$ appears, and the changes from $g_{j}$ to $f_{{\sigma(j)}}$ happen at the stopping times when the next particle $\in W$ appears.  Finally, we define $\xi_R^{\ssup M}$ as the distribution of  $\omega+\delta_f\in\Lcal_{\widetilde W}$ under $\delta_{\underline 0}\otimes \xi_R(\d(\omega,\varpi))\otimes \bigotimes_{i\in I}[p_{y_i, x_{\sigma(i)}}^{\ssup {W'}}\otimes q_{y_i,x_{\sigma(i)}}^{\ssup{\cdot, W'}}]$.

Let us first identify the reference measure $\Pi_{\widetilde W}^{\ssup\Lcal}({\tt Q})$ explicitly. Recall from Section~\ref{sec-PPPs} that $\Lcal_{\widetilde W}=\Lcal_{W\cup W'}$ consists of the three parts $\Lcal_{W}$, $\Lcal_{W'}$ and the set $\Lcal_{W,W'}$ of loops in $\widetilde W= W\cup W'$ that have particles both in $W$ and in $W'$, but nowhere else. The reference measure $\Pi_{\widetilde W}^{\ssup\Lcal}({\tt Q})$ is equal to the convolution of $\Pi_{W}^{\ssup\Lcal}({\tt Q})$, $ \Pi_{W'}^{\ssup\Lcal}({\tt Q})$ and $\Pi_{W, W'}^{\ssup\Lcal}({\tt Q})$. Observe that $\Pi_{W, W'}^{\ssup\Lcal}({\tt Q})(\d\delta_f)$, having just one loop with starting site at $x_{i^*}$ with length $n$, can be decomposed, according to the Markov property. Therefore, using also \eqref{qWkdef} and
\begin{equation}\label{QWW'kdef}
q_{W,W',k}=\frac 1k \int_{W} \d x\, \mu_{x,x}^{\ssup{k,W\cup W',{\rm par}}}(\mathfrak m_{W,W'}\geq 1),
\end{equation} 
where we recall that  $\mathfrak m_{W,W'}$ is the number of times that the Brownian loop changes between $W$ and $W'$,
and abbreviating
\begin{equation}\label{Ameasuredef}
A_\sigma\big(\d\mu\big)= \bigotimes_{i\in I} \Big[\P_{x_i}(\tau_W=l_i,B(\tau_W \beta)\in\d y_i)\P_{y_{\sigma^{-1}(i)}}(B(\tau_{W'}\beta)\in\d x_{i})\Big],
\end{equation}
we can write 
\begin{equation}\label{Qdecomposition}
\begin{aligned}
\Pi_{\widetilde W}^{\ssup\Lcal}&({\tt Q})\big(\d ( \omega+\underline 0+\delta_f)\big)
=\Pi_{W}^{\ssup\Lcal}({\tt Q})(\d\omega)\,\e^{-\sum_{k\in\N} q_{W',k}}\,\e^{-\sum_{k\in\N} [q_{W,W',k}+ q_{W',W,k}]}q_{W,W',n} B_n\\
&\quad 
\otimes A_\sigma( \d\mu)\otimes \bigotimes_{i\in I}\Big( q_{x_i,y_i}^{\ssup{l_i, W}}(\d f_i)\otimes q_{y_i,x_{\sigma(i)}}^{\ssup{L_i, W'}}(\d g_i)\Big),\qquad\mbox{with }B_n=\frac1{|I|}\sum_{i\in I}\frac 1{\mu_{x_i, x_i}^{\ssup{n, W\cup W',{\rm par}}}(\mathfrak m_{W,W'}\geq 1)}.
\end{aligned}
\end{equation}
Indeed, the exponential terms in the first line are the probability that no $W'$-loop exists and that one $n$-loop in $W\cup W'$ with particles both in $W$ and in $W'$ exists and no loop else; the product of the measures over $i$ is the probability distribution of the $W$-shreds and the $W'$-shreds of the loop $f$ (because of the cyclic structure and the uniform choice of the starting site, one does not see here the starting site explicitly); and  $B(n)$ is the normalization for a length-$n$ loop in $W\cup W'$ that has particles both in $W$ and in $W'$ and starts from some point that is uniformly distributed over $\{x_i\colon i\in I\}$.

Accordingly, also $\xi_R^{\ssup M}$ splits into these three parts, observing also the dependences on $\omega$ and $\delta_f$ and recalling that we added no loop that starts in $W'$:
\begin{equation}\label{xidecomposition}
\begin{aligned}
\xi_R^{\ssup M}&\big(\d ( \omega+\underline 0+\delta_f)\big)
=\xi_R\big(\d (\omega,\varpi)\big)\,\,\e^{-\sum_{k\in\N}[ q_{W',k}+q_{W',W,k}]}\bigotimes_{i\in I}\Big[p_{y_i,x_{\sigma(i)}}^{\ssup{W'}}(L_i)\, q_{y_i,x_{\sigma(i)}}^{\ssup{L_i, W'}}(\d g_i)\Big]\\
&=\Pi_W^{\ssup\Lcal}(\xi_R)(\d\omega)\,\e^{-\sum_{k\in\N} [ q_{W',k}+q_{W',W,k}]}\,(\xi_R)_{\Lcal_W\to \Scal_W}(\omega, \d\varpi)
\otimes \bigotimes_{i\in I}\Big[p_{y_i,x_{\sigma(i)}}^{\ssup{W'}}(L_i)\, q_{y_i,x_{\sigma(i)}}^{\ssup{L_i, W'}}(\d g_i)\Big],
\end{aligned}
\end{equation}
where we recall that $\varpi=\sum_{i\in I}\delta_{f_i}$ and $\mu=\partial\Pi^{\ssup\Scal}_W(\varpi)=\sum_{i\in I}\delta_{(x_i,l_i,y_i)}$. 
Then, writing $\partial\Pi_W^{\ssup{\Lcal,\Tcal}}$ for the projection on $\Lcal_W\times \Tcal_W$,
\begin{equation}\label{entropycalc}
\begin{aligned}
H_{\Lcal_{\widetilde W}}&(\xi_R^{\ssup M}\mid \Pi_{\widetilde W}^{\ssup\Lcal}({\tt Q}))
=H_{\Lcal_{W}}(\Pi_W^{\ssup\Lcal}(\xi_R) \mid \Pi^{\ssup\Lcal}_{W}({\tt Q}))\\
&\quad +\int \partial\Pi_W^{\ssup{\Lcal,\Tcal}}(\xi_R)(\d (\omega ,\mu ))
H_{\Scal_{W}}\big((\xi_R)_{(\Lcal_W\times \Tcal_W)\to\Scal_W}(\omega ,\mu ,\cdot)\,\big|\, {\tt K}_W(\mu,\cdot)\big)\\
&\quad  + \sum_{k\in\N}q_{W,W',k}
-\int \xi_R^{\ssup M}(\d\varpi )\, \log \big(q_{W,W',n} B(n)\big)+ \int\Pi_W^{\ssup\Lcal}(\xi_R)(\d\omega )\,H_{\Tcal_W}\big((\xi_R)_{\Lcal_W\to\Tcal_W}(\omega ,\cdot)\,\big|\, A_\sigma\big)\\
&=H_{\Lcal_{W}\times\Scal_W}\big(\xi_R\,\big|\, \Pi^{\ssup\Lcal}_W({\tt Q})\otimes [\partial \Pi^{\ssup\Scal}_{W}(\xi_R)\otimes {\tt K}_W]\big)\\
&\quad+ \sum_{k\in\N}q_{W,W',k}- \int \xi_R^{\ssup M})(\d\varpi )\, \log \big(q_{W,W',n} B(n)\big)+ \int\Pi_W^{\ssup\Lcal}(\xi_R)(\d\omega )\,H_{\Tcal_W}\big((\xi_R)_{\Lcal_W\to\Tcal_W}(\omega ,\cdot)\,\big|\, A_\sigma\big),
\end{aligned}
 \end{equation}
as one sees from \eqref{jointentropy}, together with a comparison with Lemma~\ref{lem-entridentshreds}. Now divide this by $|W_{R+M}|$ and make $R\to\infty$ and observe that $|W_{R+M}|/|W_R|\to 1$. For the entropy of $\xi_R$, we need to switch from the reference measure $\Pi^{\ssup\Lcal}_W({\tt Q})\otimes [\partial \Pi^{\ssup\Scal}_{W}(\xi_R)\otimes {\tt K}_W]$ to its image measure under $\widetilde\Pi_{W;M}$. This gives for the first term on the right-hand side of \eqref{entropycalc}, using \eqref{Mlargespreads},
$$
\begin{aligned}
\frac 1{|W|}&H_{\Lcal_{W}\times\Scal_W}\big(\xi_R\,\big|\, \Pi^{\ssup\Lcal}_W({\tt Q})\otimes [\partial \Pi^{\ssup\Scal}_{W}(\xi_R)\otimes {\tt K}_W]\big)\\
&=\widetilde J_{W;M}(\xi_R)+\frac 1{|W|}\int \xi_R(\d(\omega,\varpi))\,\log\frac{\d \Pi^{\ssup\Lcal}_W({\tt Q})\otimes [\partial \Pi^{\ssup\Scal}_{W}(\xi_R)\otimes {\tt K}_W]}{\d \Pi^{\ssup\Lcal}_W({\tt Q})\otimes [\partial \Pi^{\ssup\Scal}_{W}(\xi_R)\otimes {\tt K}_W]\circ\widetilde \Pi_{W;M}^{-1}}(\omega,\varpi)\\
&\leq \widetilde J_{W; M}(\xi_R)+\frac 1{|W|}\int \Pi_W(P)(\d(\omega,\varpi))\,\widetilde{\mathfrak N}_W^{\ssup{\ell}}(\omega,\varpi)\log (1+\eps)\leq \widetilde J_{W;M}(\xi_R)+C\eps,
\end{aligned}
$$
where $C$ depends only on $\rho_1+ \rho_2$.

We have to show that all other terms on the right-hand side of \eqref{entropycalc} vanish (i.e., have a vanishing upper bound) with a prefactor of $1/|W|$ as $R\to\infty$. We leave the proof of $\lim_{R\to\infty}\frac 1{|W_R|}\sum_{k\in\N}q_{W,W',k}=0$ to the reader. 

Let us turn to the next term. Recall that $n=\sum_{i\in I}(l_i+L_i)$ is random. A lower bound for $B_n$ is easy to get by estimating the $\mu$-term against the free Gaussian kernel with parameter $\beta n$, i.e., $\mu_{x_i,x_i}^{\ssup{n, W\cup W',{\rm par}}}(\mathfrak m_{W,W'}\geq 1)\leq (4\pi\beta n)^{-d/2}$, such that it is clear that $\liminf_{R\to\infty}\frac 1{|W_R|}\int\xi_R^{\ssup M}(\d \varpi)\,\log B_n\geq 0$. Concerning the other part of that term, we claim that we may estimate $q_{W,W',n}\geq \e^{-n C/R^2}$ for any large $R$ and $n$. Indeed, lower bound the event $\{\mathfrak m_{W,W'}\geq 1\}\cap \{B_0=x_i=B_{n\beta}\}\cap \{B_{k\beta}\in W\cup W'\,\forall k\in[n]\}$ by the event that the Brownian motion on the time interval $[0,\beta n]$, starting from and terminating at $x_i$, runs entirely within $W\cup W'$ and visits both $W$ and $W'$. Then, for large $n$, the probability of this event has a logarithmic asymptotics in $n$ with a parameter equal to the principal Dirichlet eigenvalue of the Laplace operator in $W\cup W'$. This eigenvalue is, as $R\to\infty$, asymptotic to some constant times $1/(R+M)^2\sim 1/R^2$. This implies that
$$
\begin{aligned}
- \frac 1{|W|}&\int \xi_R^{\ssup M}(\d\varpi)\, \log q_{W,W',n}
\leq \frac C{R^2}\,\frac 1{|W|}\int \xi_R^{\ssup M}(\d\varpi_R)\, \sum_{i\in I}(l_i+L_i)\\
&\leq \frac C{R^2}\,\frac 1{|W|}\Big(\langle \Pi_W^{\ssup\Scal}(\xi_R),\mathfrak N_W^{\ssup \ell}\rangle+\int \partial \Pi_W^{\ssup\Scal}(\xi_R)(\d \mu)\sum_{i\in I}\sum_{m\in\N} mp_{y_i,x_{\sigma(i)}}^{\ssup{W'}}(m)\Big)\\
&\leq \frac C{R^2}+\frac C{R^2}\frac 1{|W|}\int \partial \Pi_W^{\ssup\Scal}(\xi_R)(\d \mu)\sum_{i\in I}|x_{\sigma(i)}-y_i|,
\end{aligned}
$$
using the following Lemma~\ref{lem-explengthbrige}. We now estimate simply $|x_{\sigma(i)}-y_i|\leq CR$ and $|I|\leq \sum_{i\in I}l_i =\mathfrak N_W^{\ssup {\ell,\Scal}}(\varpi)$, to see that the last summand is of order $O(1/R)$ and that the right-hand side vanishes as $R\to\infty$.

\begin{lemma}[Expected length of Brownian bridge in $W'$]\label{lem-explengthbrige} For any $M\in(0,\infty)$, there is a constant $C_M\in(0,\infty)$ (depending on $d$, $\beta$ and $M$) such that, for any $R\in\N$,
\begin{equation}\label{littletimeinannulus}
\sum_{n\in\N}n p_{y,x}^{\ssup{W'}}(n)\leq C_M |x-y|,\qquad y\in W',x\in (W')^{\rm c}.
\end{equation}
\end{lemma}

\begin{proof}
For the proof, it suffices to replace $W'$ by $[-R- M,R+M]^{d-1}\times [0,M]$ (which we will still write as $W'$ in this proof), since $W_{R+M}\setminus W_R$ is a finite union of such sets, modulo shifts and orthogonal transformations. We write $x=(\underline x,x^{\ssup d})$ with $\underline x\in [-R- M,R+M]^{d-1}$ and $x^{\ssup d}\in [0,M]$. The assertion needs to be proved only for large $R$ and large $ |\underline x-\underline y|$. The intuitive reason for \eqref{littletimeinannulus} is that $\tau_{W'}$ has exponential tails as $R\to\infty$ under $\P_y(\cdot\mid B(\tau_{W'}\beta)=x)$, uniformly in $x\in W_R$ and $y\in W'$, since the $d$-th component of the Brownian bridge lies in $[0,M]$. This implies that the conditioned Gaussian random walk $(B_{n\beta})_{n\in\N}$ is forced to assume a drift into the direction $w(\underline x-\underline y)$ for some $w\in(0,\infty)$ that asymptotically depends only on $M$. 

Let us come to the details. The vector $(\underline B_{n\beta})_{n\in\N}$ of the first $d-1$ components and the $d$-th component $(B_{n\beta}^{\ssup d})_{n\in\N}$ of the Gaussian random walk are independent, and the stopping time $\tau_{W'}$ depends only on the latter; indeed, we write it as the exit time $\tau_{[0,M]}$ from $[0,M]$ for the $d$-th component. Let us denote the one-step distribution density of the Gaussian walk by $g$ and its $(d-1)$-dimensional version by $\underline g$ and with $d$-th component version $g^{\ssup d}$ (these are nothing but the usual Gaussian densities with variance parameter $2\beta$ in the respective dimensions). We now  integrate over $B_{n\beta}$ to get, in the limit as $R\to\infty$, uniformly in large $n$ and in $y\in W'$ and $x\in (W')^{\rm c}$ for $|\underline x-\underline y|\to\infty$,
$$
\begin{aligned}
\frac {\P_y(B_{\tau_{W'}\beta}\in \d x)}{\d x}p_{y,x}^{\ssup{W'}}(n+1)&=\int_{W'}\P_y(\tau_{W'}>n, B_{n\beta}\in\d z)g(z-x)\\
&\sim \E_{\underline y}\big[\underline g(\underline B_{n\beta}-\underline x)\big]\,\E_{y^{\ssup d}}\big[\1\{\tau_{[0,M]}>n\}g^{\ssup d}(B_{n\beta}^{\ssup d}-x^{\ssup d})\big]\\
&= \e^{-|\underline x-\underline y|^2/4\beta (n+1)} \,\e^{-n C_M} \e^{o(n)},
\end{aligned}
$$
where the first term on the right-hand side comes out of a quadratic  extension in the exponent, and the second from subadditivity (the probability of $\{\tau_{[0,M]}>n \}$ decays exponentially in $n$, while the $g^{\ssup d}$-expectation is bounded in $n$). Hence, the term $p_{y,x}^{\ssup{W'}}(n+1)$ is strictly maximal for $n\sim |\underline x-\underline y|/ \sqrt{4\beta C_M}$. Hence, the sum of $(n+1)p_{y,x}^{\ssup{W'}}(n+1)$ over $n$ comes mainly from the summands around that value of $n$, and hence it is asymptotic to $|\underline x-\underline y|/ \sqrt{4\beta C_M}$. This implies the assertion in \eqref{littletimeinannulus}. 

The right-hand side of the last display is equal to $\e^{-|\underline y-\underline x|( \sqrt{C_M/\beta}+o(1))}$ for this $n$. Since the $p_{y,x}^{\ssup{W'}}(n+1)$ do not decay exponentially fast for $n/|\underline x-\underline y|$ in a neighbourhood of $\sqrt{4\beta C_M}$, but away from it and sum up to one on $n$, we also see that $\P_y(B_{\tau_{W'}\beta}\in \d x)/\d x = \e^{-|\underline y-\underline x|(\sqrt{C_M/\beta}+o(1))}$ as $|\underline y-\underline x|\to\infty$.

\end{proof}

Now let us turn to the last term on the right-hand side of \eqref{entropycalc}, where we recall the definition of $A_\sigma$ in \eqref{Ameasuredef}. We are going to write the measure $(\xi_R)_{\Lcal_W\to\Tcal_W}(\omega,\cdot)=\frac{\Pi_W^{\ssup{\Lcal,\Tcal}}( \xi_R)(\d\omega,\cdot)}{\Pi_W^{\ssup{\Lcal}}( \xi_R)(\d\omega)}$ with the help of projections of the measure $\widetilde \xi=\widetilde \Pi_{\widetilde W;M}(\Pi_{\widetilde W}(P))$. This measure has a finite entropy with respect to the reference measure in the definition of $J_{\widetilde W}$, hence, it has a density $\widetilde g\colon  \Lcal_{\widetilde W}\times\Scal_{\widetilde W}\to[0,\infty)$ with respect to the reference measure $ \widetilde \nu=\Pi^{\ssup\Lcal}_{\widetilde W}({\tt Q})\otimes[\partial \Pi_{\widetilde W} ^{\ssup\Scal}(P)\otimes {\tt K}_{\widetilde W}]$. We conceive the projections  $\pi_1=\Pi_W^{\ssup{\Lcal,\Tcal}}\colon \Lcal_{\widetilde W}\times\Scal_{\widetilde W}\to\Lcal_W \times\Tcal_W$ and $\pi_2=\Pi_W^{\ssup{\Lcal}}\colon \Lcal_{\widetilde W}\times\Scal_{\widetilde W}\to\Lcal_W$ as projections from configurations in $\widetilde W$ to loop/boundary-shred-configurations in $W$, respectively to loop configurations in $W$. The factorisation lemma from measure theory implies that there are measurable functions $g_1\colon \Lcal_W\times\Tcal_W\to[0,\infty)$ and $g_2\colon \Lcal_W\to [0,\infty)$ such that $\widetilde g=g_2\circ \pi_2=g_1\circ\pi_1$ and $g_1=g_2\circ \pi_3$, where $\pi_3$ is the projection from $\Lcal_W\times \Tcal_W$ on $\Lcal_W$, that is, $g_1(\omega,\mu)=g_2(\omega)$. Therefore, they are densities of the image measures of $\widetilde \xi$ with respect to the image measures of $\widetilde \nu$ under $\pi_1$, respectively under $ \pi_2$, that is,
$$
g_1(\omega,\mu)=\frac{\d \pi_1(\widetilde \xi)}{\d\pi_1(\widetilde \nu)}(\omega,\mu)
\qquad\mbox{and}\qquad
g_2(\omega)=\frac{\d \pi_2(\xi_R)}{\d\pi_2(\widetilde \nu)}(\omega),\qquad  \omega\in\Lcal_W,\mu\in\Tcal_W.
$$
Recall from \eqref{disintegration} that $(\xi_R)_{\Lcal_W\to\Tcal_W}(\omega,\d\mu)=\frac{\Pi_W^{\ssup{\Lcal,\Tcal}}(\xi_R)(\d\omega,\d\mu)}{\Pi_W^{\ssup{\Lcal}}(\xi_R)(\d\omega)}$ and note that
$$
\begin{aligned}
\Pi_W^{\ssup{\Lcal,\Tcal}}(\xi_R)
&=\Pi_W^{\ssup{\Lcal,\Tcal}}(\widetilde\Pi_{W;M}(\Pi_W(P)))
=\Pi_W^{\ssup{\Lcal,\Tcal}}(\widetilde\Pi_{W;M}(\Pi_{\widetilde W\to W}(\Pi_{\widetilde W}(P))))\\
&=\Pi_W^{\ssup{\Lcal,\Tcal}}(\Pi_{\widetilde W\to W}(\widetilde\Pi_{\widetilde W;M}(\Pi_{\widetilde W}(P)))=\pi_1(\widetilde \xi),
\end{aligned}
$$
since $\widetilde\Pi_{W;M}\circ \Pi_{\widetilde W\to W}= \Pi_{\widetilde W\to W} \circ \widetilde\Pi_{\widetilde W;M}$ (by $\Pi_{\widetilde W\to W} $ we mean $ \Pi_W$ with domain $\Lcal_{\widetilde W}\times\Scal_{\widetilde W}$). Analogously, we see that $\Pi_W^{\ssup{\Lcal}}(\xi_R)=\pi_2(\widetilde \xi)$. Hence,
$$
\frac{(\xi_R)_{\Lcal_W\to\Tcal_W}(\omega,\d\mu)}{A_\sigma(\d\mu)}
=\frac{\pi_1(\widetilde \xi)(\d\omega,\d\mu)}{\pi_2(\widetilde \xi)(\d\omega)A_\sigma(\d\mu)}
=\frac{g_1(\omega,\mu)}{g_2(\omega)}\,\frac {\d\pi_1(\widetilde\nu)}{\d[\pi_2(\widetilde \nu)\otimes A_\sigma]}(\omega,\mu)
=\frac {\d\pi_1(\widetilde\nu)}{\d[\pi_2(\widetilde \nu)\otimes A_\sigma]}(\omega,\mu),
$$
where the last step derives from $g_1=g_2\circ \pi_3$.

Note that $\pi_2(\widetilde \nu)=\Pi_W^{\ssup\Lcal}(\Pi_{\widetilde W}^{\ssup\Lcal}({\tt Q}))=\Pi_W^{\ssup\Lcal}({\tt Q})$, since the remainder part of $\widetilde \nu$ does not create loops in $W$. On the other hand, $\pi_1(\widetilde \nu)=\Pi_W^{\ssup\Lcal}({\tt Q})\otimes \partial\Pi_W^{\ssup\Scal}(\Pi_{\widetilde W\setminus W}^{\ssup{\Lcal,\widetilde W}}({\tt Q})\otimes [\partial\Pi_{\widetilde W}^{\ssup\Scal}(P)\otimes {\tt K}_{\widetilde W}])$, where we denote by $\Pi_{\widetilde W\setminus W}^{\ssup{\Lcal,\widetilde W}}$ the projection on all the loops that start in $\widetilde W\setminus W$ and have all the particles in $\widetilde W$. Therefore, 
$$
\frac {\d\pi_1(\widetilde\nu)}{\d[\pi_2(\widetilde \nu)\otimes A_\sigma]}(\omega,\mu)
=\frac{\d \partial\Pi_W^{\ssup\Scal}(\Pi_{\widetilde W\setminus W}^{\ssup{\Lcal,\widetilde W}}({\tt Q})\otimes [\partial\Pi_{\widetilde W}^{\ssup\Scal}(P)\otimes {\tt K}_{\widetilde W}])}{\d A_\sigma}(\mu).
$$
The measure in the numerator is equal to $A_{\sigma'}$ for some random bijective $\sigma'\colon I\to I$ whose distribution depends on $\partial\Pi_{\widetilde W}^{\ssup\Scal}(P)$. Hence, the quotient is of the form $\prod_{i\in I}\P_{(\sigma')^{-1}(i)}(B(\tau_{W'\beta})\in\d x_i)/\P_{\sigma^{-1}(i)}(B(\tau_{W'\beta})\in\d x_i)$, which is bounded from above by $C_M^{|\mu|}$ for some constant $C_M$, depending only on $M$, $d$ and $\beta$. We have to integrate its logarithm with respect to $\xi_R$, which gives $\log C_M$ times the expectation of $|\mu|$ under $\partial\Pi_W^{\ssup\Scal}(\xi_R)=\partial\Pi_{W}^{\ssup\Scal}(P)$, i.e., the expected number of $W$-shreds under $P$. This is $o(|W|)$, according to Corollary~\ref{cor-shrednumberesti}. This shows that the last term on the right-hand side of \eqref{entropycalc} is $o(|W|)$ and hence negligible.

This finishes the proof of Lemma~\ref{lem-noshreds}.

\subsection{Preparation: entropy of projection image measures}
\label{sec-projentr}

Fix $W$ and recall the events $A_{W;\mathfrak r,M,L,K,T}$ defined in \eqref{ALcaldef}--\eqref{Adef} and the restricted shred-configuration set $\Scal_W^{\ssup{M,\vartheta,S}}$ defined in \eqref{SMvarthetaSdef} and the next line. In this section, we show that the entropy $J_W(\xi)$ is well approximated with the entropy $J_W(\widetilde \xi)$ of a measure $\widetilde \xi$ that approximates $\xi$ and is concentrated on the above restricted set of configurations, if the parameters are picked large. More precisely,  we estimate the entropy $J_W(\xi\circ\Pi_{W;\mathfrak r,M,L,K,T,\vartheta,S}^{-1})$ of some  projection $ \Pi_{W;\mathfrak r,M,L,K,T,\vartheta,S}$ of an arbitrary measure $\xi$ on $  A_{W;\mathfrak r,M,L,K,T}\cap \Scal_W^{\ssup{M,\vartheta,S}}$ in terms of $J_W(\xi)$. We want to show that the projected entropy is not much larger than the original entropy after taking $W\uparrow \R^d$ and $M, L,K$ and $T,\vartheta\to\infty$, uniformly in $S\in\N$. This will be used in the proof of the ergodic approximation in Lemma~\ref{lem-ergappr}, which will be used in several further proofs.

\begin{lemma}[Entropy of projection image]\label{lem-EntProjImag} Fix a box $W\Subset\R^d$ and $\mathfrak r, M,L,K,T,\vartheta\in(0,\infty)$ and $S\in\N$. Then there is a projection $\Pi=\Pi_{W;\mathfrak r,M,L,K,T,\vartheta, S}\colon\Lcal_W\times\Scal_W\to A_{W;\mathfrak r,M,L,K;T}\cap(\Lcal_W\times \Scal_W^{\ssup{M,\vartheta,S}})$ and there are $\eps^{\ssup\Lcal}_{M,W,L}$ and $\eps_{M,\vartheta,S,T}^{\ssup\Scal}\in (0,1)$ satisfying, for any $S$,
$$
\lim_{L\to\infty}\lim_{M\to\infty}\limsup_{W\uparrow \R^d}\eps_{M,W,L}^{\ssup\Lcal}=0\qquad\mbox{and}\qquad \lim_{M,\vartheta,T\to\infty}\eps_{M,\vartheta,S,T}^{\ssup\Scal}=0,
$$ 
such that, for any  $\xi\in\Mcal_1(\Lcal_W \times\Scal_W)$ with finite expected total particle number $\langle \xi,\mathfrak N^{\ssup\ell}_W\rangle$, and for any $S\in\N$,
\begin{equation}
J_W(\xi\circ\Pi^{-1})\leq J_W(\xi)+\eps^{\ssup\Lcal}_{M,W,L}+\Big(\frac 1K+\eps_{M,\vartheta,S,T}^{\ssup\Scal}\Big)\frac 1{|W|}\langle \xi,\mathfrak N^{\ssup\ell}_W\rangle.
\end{equation}
\end{lemma}

\begin{proof}
For the ease of notation,  we write $\xi^{\ssup\Lcal}$ instead of $\Pi_W^{\ssup\Lcal}(\xi)$ and  $\xi^{\ssup\Scal}$ instead of $\Pi_W^{\ssup\Scal}(\xi)$ and $\partial \xi^{\ssup\Scal}$ instead of $\partial \Pi_W^{\ssup\Scal}(\xi)$. Furthermore, we write ${\tt Q}_W$ instead of $\Pi_W^{\ssup\Lcal}({\tt Q})$. Furthermore, we drop the indices $W$, $\mathfrak r, M, L,K$, and $T,\vartheta,S$ and write $\Pi^{ \ssup\Lcal}$ and $\Pi^{\ssup\Scal}$ for the respective  projections on $A^{ \ssup\Lcal}$ and $A^{\ssup\Scal}\cap \Scal_W^{\ssup{M,\vartheta,S }}$ (to be defined in this proof). Then we put $\Pi=\Pi^{\ssup\Lcal}\otimes \Pi^{\ssup\Scal}$. We are going to use the decomposition  \eqref{jointentropy} of the entropies into a loop part and a shred part.

Note that, on any measurable space $\Omega$,
\begin{equation}\label{entropyofimages}
H(P\circ f^{-1}\mid Q\circ f^{-1})\leq H(P\mid Q),\qquad  P,Q \in\Mcal_1(\Omega), f\colon\Omega\to \Omega\mbox{ mb.},
\end{equation}
since the left-hand side is equal to the relative entropy of $P$ with respect to $Q$ on the $\sigma$-field generated by $f$, and relative entropies are increasing in the $\sigma$-field \cite[Prop.~15.5(c)]{G88}.

Therefore, we have
\begin{equation}\label{entropyofimage}
\begin{aligned}
H_{\Lcal_W\times\Scal_W}&\big(\xi\circ\Pi^{-1}\mid {\tt Q}_W\otimes [\partial (\xi\circ\Pi^{-1})\otimes {\tt K}_W]\big)\\
&=H_{\Lcal_W\times\Scal_W}\big(\xi\circ\Pi^{-1}\mid [{\tt Q}_W\otimes [\partial \xi\otimes {\tt K}_W]]\circ\Pi^{-1}\big)
+\int \d \xi\circ \Pi^{-1}\,\log\frac{\d [{\tt Q}_W\otimes [\partial \xi\otimes {\tt K}_W]]\circ\Pi^{-1}}{\d {\tt Q}_W\otimes [\partial (\xi\circ\Pi^{-1})\otimes {\tt K}_W]}\\
& \leq H_{\Lcal_W\times\Scal_W}\big(\xi\mid {\tt Q}_W\otimes [\partial \xi\otimes {\tt K}_W]\big)
+\int_{\Lcal_W} \d \xi^{\ssup\Lcal}\,\log\frac{\d [{\tt Q}_W\circ (\Pi^{\ssup\Lcal})^{-1}] }{\d {\tt Q}_W}\circ\Pi^{\ssup\Lcal}\\
&\quad +\int_{\Scal_W}\d \xi^{\ssup\Scal}\,\log\frac{\d [\partial\xi\otimes {\tt K}_W]\circ (\Pi^{ \ssup\Scal})^{-1}  }{\d [\partial\xi\otimes {\tt K}_W] }\circ \Pi^{\ssup\Scal}
-\int_{\Scal_W}\d \xi^{\ssup\Scal}\,\log\frac{\d [\partial(\xi^{\ssup\Scal}\circ (\Pi^{\ssup\Scal})^{-1})\otimes {\tt K}_W]}{\d \partial\xi^{\ssup\Scal}\otimes {\tt K}_W}\circ \Pi^{\ssup\Scal}.
\end{aligned}
\end{equation}
Note that, in the last term on the right-hand side of \eqref{entropyofimage}, we can drop the ${\tt K}_W$-terms and see that this is a relative entropy, which is nonnegative; hence we can drop it with ``$\leq$''.

Now we handle the second term  on the right-hand side of \eqref{entropyofimage}. We use the projection $\Pi=\Pi^{\ssup\Lcal}_{W;\mathfrak r,M,L,K}\colon \Lcal_W\to A^{\ssup\Lcal}_{W;\mathfrak r,M,L,K}$ defined, for all $\omega= \sum_{x\in\zeta}\delta_{(x,f_x)}$, by 
$$
\begin{aligned}
\Pi(\omega)&=\Pi^{\ssup\Lcal}_{W;\mathfrak r,M,L,K}(\omega)
=\sum_{z\in W\cap 2\mathfrak r \Z^d}\sum_{l\in [L]}\Big[\1\{|\zeta^{\ssup{M,l}}_z|\leq K\}\sum_{x\in\zeta_z^{\ssup{M,l}}}\delta_{(x,f_x)}\\
&\quad+\1\{|\zeta^{\ssup{M,l}}_z|> K\}\frac 1{\binom{|\zeta^{\ssup{M,l}}_z|}{K}}\sum_{B \subset \zeta^{\ssup{M,l}}_z\colon |B|=K}\sum_{x\in B}\delta_{(x,f_x)}\Big],
\end{aligned}
$$
where $\zeta^{\ssup{M,l}}_z=\{x\in\zeta\cap (z+W_{\mathfrak r})\colon \ell(f_x)=l , f_x\in \Ccal^{\ssup{M,\circlearrowleft}}\}$, where $\Ccal^{\ssup{M,\circlearrowleft}}$ denotes the set of all loops all of whose legs have a spread $\leq M$. In other words, $\Pi$ first removes all loops that are longer than $L$ or have a leg with spread $>M$ and then removes in each subbox $z+W_{\mathfrak r}$ and for any $l\in[L]$, if necessary, a uniformly distributed  choice of loops such that no more than $K$ loops of length $=l$ survive. This algorithm can be conceived as being carried out independently in each of these subboxes.

Recall that $\Lcal_W^{\ssup M}$ denotes the set of loop configurations in $\Lcal_W$ (i.e., with all particles in $W$) that have only legs in $\Ccal^{\ssup M}$, i.e., with all the spreads $\leq M$. With an additional super-index $l$ we restrict to configurations with all loop lengths $=l$, and the super-index $\leq L$ means the union over $l\in [L]=\{1,\dots,L\}$. With the index $W;\mathfrak r,W'$ for $W'\subset W$ we mean those configurations with particles in $W$ that have initial site in $W'$. All this notation also applies to measures ${\tt Q}$.

For $\omega$ in the image of $\Pi$, the pre-image $\Pi^{-1}(\{\omega\})$ is equal to the set of configurations of the form $\sum_{z\in W\cap 2\mathfrak r \Z^d}(\omega_z+\omega_z'+\sum_{l\in[L]}\omega_{z,l}'')$, where $\omega_z=\sum_{x\in \zeta\cap(z+W_{\mathfrak r})}\delta_{(x,f_x)}$ (the projection of $\omega$ to that subbox), and $\omega'_z\in \Lcal_{W;\mathfrak r,z+W_{\mathfrak r}}\setminus \Lcal^{\ssup{M,\leq L}}_{W;\mathfrak r,z+W_{\mathfrak r}}$ and $\omega_{z,l}''=0$ (empty configuration) if $|\zeta_z^{\ssup{M,l}}|<K$ and $\omega_{z,l}''\in \Lcal^{\ssup{M,l}}_{W;\mathfrak r,z+W_{\mathfrak r}}$ otherwise.  Note that all the configurations $\omega_z'$ and $\omega_{z,l}''$ are independent over $z$ and $l$ under ${\tt Q}_W$. For fixed $z$ and $l$, we split ${\tt Q}^{\ssup{l}}_{W;\mathfrak r,z+W_{\mathfrak r}}$ into the PPP ${\tt Q}_{W;\mathfrak r,z+W_{\mathfrak r}}^{\ssup{M,l}}$ on $\Lcal_{W;\mathfrak r,z+W_{\mathfrak r}}^{\ssup {M,l}}$  and $\Lcal_{W;\mathfrak r,z+W_{\mathfrak r}}^{\ssup {\infty,l}}\setminus \Lcal_{W;\mathfrak r,z+W_{\mathfrak r}}^{\ssup {M,l}}$. The first one has an intensity measure with total mass
$$
q_z^{\ssup{M,l}}= \frac 1l \int_{z+W_{\mathfrak r}}\d x\,  \mu_{x,x}^{\ssup{M,l,W,{\rm par}}}(\Ccal^{\ssup\circlearrowleft}),
$$
 (recall  \eqref{nudef}), where $\mu_{x,x}^{\ssup{M,l,W,{\rm par}}}$ is the restriction of $\mu_{x,x}^{\ssup{l,W,{\rm par}}}$ to loops configurations with loops in the set $\Ccal^{\ssup{M,\circlearrowleft}}=\{f\in \Ccal^{\ssup\circlearrowleft}\colon \forall i\colon f_{x,i}\in \Ccal_M\}$. These two PPPs are independent.  Therefore, the density $f$ of ${\tt Q}_W\circ\Pi^{-1}$ with respect to ${\tt Q}_W$ is the product 
\begin{equation}\label{productdensity}
f(\omega)=\prod_{z\in W\cap 2\mathfrak r \Z^d}\Big[\e^{\sum_{l=L+1}^\infty q_{z}^{\ssup l}}\prod_{l\in[L]}f_{z,l}( \omega_z^{\ssup l})\Big]
=\e^{\nu^{\ssup{W,{\rm par}}}(\Ccal_{>L}^{\ssup{\circlearrowleft}})}\prod_{z\in W\cap 2\mathfrak r \Z^d}\prod_{l\in[L]}f_{z,l}( \omega_z^{\ssup l}),
\end{equation}
with a density $f_{z,l}$ of ${\tt Q}_{W;\mathfrak r,z+W_{\mathfrak r}}^{\ssup{l}}\circ \Pi_{z+W_{\mathfrak r}}^{-1}$ with respect to ${\tt Q}_{W;\mathfrak r,z+W_{\mathfrak r}}^{\ssup{l}}$, where ${\tt Q}_{W;\mathfrak r,z+W_{\mathfrak r}}$ denotes the PPP on the set $\Lcal_{W;\mathfrak r,z+W_{\mathfrak r}}$ of loop configurations with starting sites in  $z+W_{\mathfrak r}$  and all the particles contained in $W$. The $q$-term expresses that, for $l>L$, the pre-image of the empty configuration is equal to $\Lcal^{\ssup l}_{W;\mathfrak r,z+W_{\mathfrak r}}$ and the empty configuration has probability $\e^{-q_{z}^{\ssup l}}$ under ${\tt Q}^{\ssup l}_{W;\mathfrak r,z+W_{\mathfrak r}}$. Recall that $\nu^{\ssup{W,{\rm par}}}$ denotes the intensity measure of the PPP on $W$, and we wrote $\Ccal_{>L}^{\ssup\circlearrowleft}$ for the set of loops with lengths $>L$.

Hence, the only difference between the distribution of a given $\omega_z^{\ssup l}\in \Lcal_{W;\mathfrak r,z+W_{\mathfrak r}}^{\ssup l}\setminus \Lcal_{W;\mathfrak r,z+W_{\mathfrak r}}^{\ssup {M,l}}$ under ${\tt Q}^{\ssup{l}}_{W;\mathfrak r,z+W_{\mathfrak r}}\circ\Pi^{-1}$ and under ${\tt Q}^{\ssup{l}}_{W;\mathfrak r,z+W_{\mathfrak r}}$ is the change of the Poisson parameter from $q_z^{\ssup{M,l}}$ to $q_z^{\ssup{l}}=q_z^{\ssup{\infty,l}}$. Therefore, for $ \omega_z$ with $|\zeta_z^{\ssup{M,l}}|<K$, we have
$$
f_{z,l}(\omega_z^{\ssup l})=\frac{\d {\tt Q}_{W;\mathfrak r,z+W_{\mathfrak r}}^{\ssup{M,l}}}{\d {\tt Q}^{\ssup{l}}_{W;\mathfrak r,z+W_{\mathfrak r}}}(\omega_z^{\ssup l})
=\e^{q_z^{\ssup{l}}-q_z^{\ssup{M,l}}}\Big(\frac{q_z^{\ssup{M,l}}}{q_z^{\ssup{l}}}\Big)^{|\zeta_{z,l}|},
$$
where $\zeta_{z,l}$ is the set of initial sites of the loops in $z+ W_{\mathfrak r}$ with length $l$.

For $ \omega_z^{\ssup l}$ with $|\zeta_z^{\ssup{M,l}}|=K$, in the pre-image there may be additional loops in $\Lcal_{W;\mathfrak r,z+W_\mathfrak r}^{\ssup{M,l}}$ with uniformly and independently distributed locations, such that an additional Poisson density comes in, but also a combinatorial term that assigns the $m$ points to the additional loops and the other $K$ to the existing loop:
$$
\begin{aligned}
f_{z,l}(\omega_z^{\ssup l})&=\frac{\d {\tt Q}_{W;\mathfrak r,z+W_{\mathfrak r}}^{\ssup{M,l}}}{\d {\tt Q}^{\ssup{l}}_{W;\mathfrak r,z+W_{\mathfrak r}}}(\omega_z^{\ssup l})
\sum_{m\in\N}\frac{\Poi_{q_z^{\ssup{M,l}}}(K+m)}{\Poi_{q_z^{\ssup{M,l}}}(K)\Poi_{q_z^{\ssup{M,l}}}(m)}{\tt Q}_{W;\mathfrak r,z+W_{\mathfrak r}}(\{\omega_z'\colon |\zeta_{z,l}|=m\})\\
&= \e^{q_z^{\ssup{l}}-q_z^{\ssup{M,l}}}\Big(\frac{q_z^{\ssup{M,l}}}{q_z^{\ssup{l}}}\Big)^{|\zeta_{z,l}|}\e^{q_z^{\ssup{M,l}}}
= \e^{q_z^{\ssup{l}}}\Big(\frac{q_z^{\ssup{M,l}}}{q_z^{\ssup{l}}}\Big)^{|\zeta_{z,l}|}.
\end{aligned}
$$

Recall that we picked $\mathfrak r$ such that $W$ is the disjoint (up to boundaries) union of the subboxes $z+ W_{\mathfrak  r}$ with $z\in W\cap 2\mathfrak r\Z^d$. Therefore $\sum_z \sum_{l\in\N}  q_z^{\ssup l}=\nu^{\ssup {W,{\rm par}}}(\Ccal^{\ssup\circlearrowleft})$ and $\sum_z \sum_{l\in\N}  q_z^{\ssup {M,l}}=\nu^{\ssup {W,{\rm par}}}(\Ccal^{\ssup{M,\circlearrowleft}})$, the total mass of the reference measure of the set of all loop configurations with all spreads of the legs $\leq M$. Therefore, we can summarize
$$
\begin{aligned}
f(\omega)&=\e^{\nu^{\ssup{W,{\rm par}}}(\Ccal_{>L}^{\ssup{\circlearrowleft}})} \e^{\nu^{\ssup {W,{\rm par}}}(\Ccal_{\leq L}^{\ssup\circlearrowleft})-\nu^{\ssup {W,{\rm par}}}(\Ccal_{\leq L}^{\ssup{M,\circlearrowleft}})}\prod_{z\in W\cap 2\mathfrak r \Z^d}\prod_{l\in[L]}\Big[\Big(\frac{q_z^{\ssup{M,l}}}{q_z^{\ssup{l}}}\Big)^{|\zeta_{z,l}|}\e^{q_z^{\ssup{M,l}}\1\{|\zeta_{z,l}|=K\}}\Big]\\
&\leq \e^{\eps_{M,W,L}^{\ssup\Lcal}|W|}\exp\Big\{\sum_{z\in W\cap 2\mathfrak r \Z^d}\sum_{l\in [L]}q_z^{\ssup{M,l}}\1\{|\zeta_{z,l}|=K\}\Big\}\\
&\leq \e^{\eps_{M,W,L}^{\ssup\Lcal}|W|}\exp\Big\{\frac 1K\sum_{z\in W\cap 2\mathfrak r \Z^d}\sum_{l\in [L]}q_z^{\ssup{M,l}}|\zeta_{z,l}|\Big\}\\
&\leq \e^{\eps_{M,W,L}^{\ssup\Lcal}|W|}\e^{\frac {|\zeta|}K},
\end{aligned}
$$
where we put $\eps_{M,W,L}^{\ssup\Lcal}=\frac 1{|W|}[\nu^{\ssup{W,{\rm par}}}(\Ccal_{>L}^{\ssup{\circlearrowleft}})+\nu^{\ssup {\rm par}}(W\times [\Ccal_{\leq L}^{\ssup\circlearrowleft}\setminus \Ccal_{\leq L}^{\ssup{M,\circlearrowleft}}])]$.

Hence, the second term on the right-hand side of \eqref{entropyofimage} can be estimated as follows.
$$
\begin{aligned}
\int&\d \xi^{\ssup\Lcal} \,\log \Big(\frac{\d ( {\tt Q}^{\ssup\Lcal} \circ(\Pi^{\ssup\Lcal})^{-1})}{\d  {\tt Q}^{\ssup\Lcal} }\circ\Pi^{\ssup\Lcal}\Big)
=\int\d \xi^{\ssup\Lcal} \,\log f(\Pi^{\ssup\Lcal}(\omega)) \leq |W|\eps_{M,W,L}^{\ssup\Lcal}+\frac 1K\int\d \xi^{\ssup\Lcal}\,|\zeta|\\
&\leq |W|\eps_{M,W,L}^{\ssup\Lcal}+\frac 1K\langle \xi,\mathfrak N^{\ssup\ell}_W\rangle.
\end{aligned}
$$
We leave it to the reader to prove that $\lim_{L\to\infty}\lim_{M\to\infty}\limsup_{W\uparrow \R^d}\eps_{M,W,L}^{\ssup\Lcal}=0$.

Now we turn to estimating the third term on the right-hand side of \eqref{entropyofimage}. We define $\Pi^{\ssup\Scal}\colon \Scal_W\to A_{W;\mathfrak r,M,T}^{\ssup\Scal}\cap \Scal_W^{\ssup{M,\vartheta,S}}$ by describing in words the removal procedure as follows. From a given shred configuration $\varpi=\sum_{g\in\Gamma}\delta_g$, first all shreds are removed that do not lie in 
$$
\Ccal^{\ssup {M,\vartheta,S}}=\bigcup_{k\in\N}\big\{f\in \Ccal_k\colon f_i\in\Ccal^{\ssup M}, |f(i\beta)-f((i+S)\beta)|\leq \vartheta \sqrt S\ \forall i\big\}.
$$ 
Then, with a fixed enumeration $(z_1,z_2,\dots,z_{N})$ of $W\cap 2\mathfrak r\Z^d$, we successively for $i=1,\dots, N$ remove shreds from the configuration such that the particle number in $z_i+W_{\mathfrak r}$ is not larger than $T$. We write $\Scal_W^{\neg \ssup {M,\vartheta,S}}$ for the set of all configurations of $W$-shreds that do not lie in $\Ccal^{\ssup {M,\vartheta,S}}$, i.e., the set of $\sum_{g\in\Gamma}\delta_g$  with $g\in\Ccal\setminus \Ccal^{\ssup {M,\vartheta,S}}$ for all $g\in\Gamma$.

For simplicity, we write $\Pi$ instead of $\Pi^{\ssup\Scal}$. For a given $\varpi$ in the image of $\Pi$, the pre-image $\Pi^{-1}(\{\varpi\})$ is contained in the set of all $W$-shred configurations $\varpi+\varpi'+\varpi''$ with $\varpi'\in \Scal_W^{\neg \ssup {M,\vartheta,S}}$ and $\varpi''\in \Scal_W^{\ssup {M,\vartheta,S} }$ being either the empty configuration or satisfying $\widetilde {\mathfrak N}_{z+W_{\mathfrak r}}(\varpi''+\varpi)>T$ for some $z\in W\cap 2\mathfrak r\Z^d$. Denote $ \Scal^{\ssup {M,\vartheta,S}}_{z,\varpi}=\{\varpi''\in \Scal_W^{\ssup {M,\vartheta,S} }\colon \widetilde {\mathfrak N}_{z+W_{\mathfrak r}}(\varpi'')> T- \widetilde {\mathfrak N}_{z+W_{\mathfrak r}}(\varpi)\}$, then 
$$
\Pi^{-1}(\{\varpi\})\subset \{\varpi\}\cup\Big(\{\varpi\}+(\Scal_W^{\neg \ssup{M,\vartheta,S}}\setminus\{0\})+\bigcup_{z\in W\cap 2\mathfrak r\Z^d}\Scal^{\ssup {M,\vartheta,S}}_{z,\varpi}\Big).
$$
Then, writing $\partial\xi$ instead of $\partial\Pi_W^{\ssup\Scal}(\xi)$ and $\mu$ for $\partial\varpi$ and ${\tt K}$ instead of ${\tt K}_W$, we have, using the product structure of ${\tt K}$,
$$
\begin{aligned}
(\partial \xi\otimes {\tt K})\circ \Pi^{-1}(\d \varpi)
&\leq \partial\xi(\d\mu){\tt K}(\mu,\d\varpi)\Big[1+\\
&\qquad\int_{\Tcal_W} \frac{\xi(\d\mu+\d\mu')}{\partial\xi(\d\mu)}\1\{|\mu'|\geq 1\} \int_{(\Scal_W^{\neg \ssup {M,\vartheta,S}}\setminus\{0\})+\bigcup_{z\in W\cap 2\mathfrak r\Z^d}\Scal^{\ssup {M,\vartheta,S}}_{z,\varpi}  }{\tt K}(\mu',\d\varpi')\Big].
\end{aligned}
$$
Hence, the third term on the right-hand side of \eqref{entropyofimage} can be estimated as follows (using $\log(1+x)\leq x$):
\begin{equation}\label{RHSScal}
\begin{aligned}
\int_{\Scal_W}\d \xi^{\ssup\Scal}\,&\log\frac{\d [\partial\xi\otimes {\tt K}_W]\circ (\Pi^{ \ssup\Scal})^{-1}  }{\d [\partial\xi\otimes {\tt K}_W] }\circ \Pi^{\ssup\Scal}\\
& \leq \int_{\Scal_W} \xi(\d\varpi)\int_{\Tcal_W} \frac{\partial\xi(\d\mu+\d\mu')}{\partial\xi(\d\mu)}\1\{|\mu'|\geq 1\}
{\tt K}\Big(\mu',(\Scal_W^{\neg \ssup{M,\vartheta,S}}\setminus\{0\})+\bigcup_{z\in W\cap 2\mathfrak r\Z^d}\Scal^{\ssup {M,\vartheta,S}}_{z,\Pi(\varpi)}\Big).
\end{aligned}
\end{equation}
Since $\Scal_W^{\neg \ssup {M,\vartheta,S}}$  and $\bigcup_{z\in W\cap 2\mathfrak r\Z^d}\Scal^{\ssup {M,\vartheta,S}}_{z,\Pi(\varpi)}$ are disjoint sets of $W$-shreds, its elements can be written as the sum over configurations of these two sets. Depending on these two, $\mu'$ can be split into the sum of $\widetilde \mu=\sum_{i\in \widetilde I}\delta_{(x_i,l_i,y_i)}$ and $\mu'\setminus \widetilde \mu=\sum_{i\in I'}\delta_{(x_i,l_i,y_i)}$, and  the last term in the last display  decomposes into the according product  of ${\tt K}(\widetilde \mu,\Scal_W^{\neg \ssup {M,\vartheta,S}}\setminus\{0\})$ and ${\tt K}(\mu'\setminus \widetilde \mu,\bigcup_{z\in W\cap 2\mathfrak r\Z^d}\Scal^{\ssup {M,\vartheta,S}}_{z,\Pi(\varpi)})$.  We estimate the second term as follows.
$$
\begin{aligned}
{\tt K}\Big(\mu'\setminus \widetilde \mu,\bigcup_{z\in W\cap 2\mathfrak r\Z^d}\Scal^{\ssup {M,\vartheta,S}}_{{z,\Pi(\varpi)}}(\Pi(\varpi))\Big)
&\leq {\tt K}\Big(\mu'\setminus \widetilde \mu,\bigcup_{z\in W\cap 2\mathfrak r\Z^d}\{\varpi'\colon \widetilde{\mathfrak N}_{z+W_{\mathfrak r}}(\varpi')>T-\widetilde{\mathfrak N}_{z+W_{\mathfrak r}}(\Pi(\varpi)) \}\Big)\\
&\leq \frac 1T \sum_{z\in W\cap 2\mathfrak r\Z^d}\int_{\Scal_W}{\tt K}(\mu'\setminus \widetilde \mu,\d\varpi')\, \widetilde{\mathfrak N}_{z+W_{\mathfrak r}}(\varpi'+\Pi(\varpi))\\
&\leq \frac 1T \Big[\mathfrak N_W^{\ssup{\ell,\Scal}}(\varpi)+\sum_{i\in I'\setminus \widetilde I} l_i\Big],
\end{aligned}
$$
where in the last step we used that $\Pi(\varpi)\leq \varpi$ and that $W$ is the (disjoint, up to boundaries) union of the boxes $z+W_{\mathfrak r}$, $z\in W+2\mathfrak r\Z^d$.

For  estimating the other, we use that the probability of the set of single $W$-shreds of length $k$, namely $ \Ccal_k^{\ssup {M,\vartheta,S}}$ under ${\tt K}(\delta_{(x,l,y)},\cdot)$ for any $(x,l,y)$  has a probability $\geq (1-\eps_M)^{k}(1-\eps_{\vartheta,S})^{k/S}$ for some $\eps_M$  and $\eps_{\vartheta,S}$ that vanish if $M\to\infty$ respectively $\vartheta\to\infty$ (both do not depend on $R$). The reason is that all spreads of the legs and all the differences $g(iS\beta)-g((i-1)S\beta)$ are independent. Since these $W$-shreds are independent under ${\tt K}(\widetilde \mu,\cdot)$,  we can estimate ${\tt K}(\widetilde \mu,\Scal_W^{\neg \ssup {M,\vartheta,S}}\setminus\{0\})\leq 1-\prod_{i\in \widetilde I}(1-\eps_M)^{l_i}(1-\eps_{\vartheta,S})^{l_i/S}\leq C(\eps_M+\eps_{\vartheta,S})\sum_{i\in \widetilde I} l_i$. 

We estimate now the product against the maximum of these two upper bounds (also depending on whether $\widetilde\mu=0$ or $\mu'\setminus\widetilde \mu=0$) and the other one against one, and  for simplicity the entire term against the sum of the two upper bounds. Substituting this above, we see that the right-hand side of \eqref{RHSScal} is not larger than
$$
\begin{aligned}
\int_{\Scal_W}\xi(\d \varpi)&\int_{\Tcal_W}\frac{\partial\xi(\d\mu+\d\mu')}{\partial\xi(\d\mu)}\,\Big[C(\eps_M+\eps_{\vartheta,S})\sum_{i\in \widetilde I} l_i+\frac 1T \mathfrak N_W^{\ssup{\ell,\Scal}}(\varpi)+\frac 1T \sum_{i\in I'}\ l_i\Big]\leq \eps_{M,\vartheta,S,T}^{\ssup{\Scal}}\langle \xi,\mathfrak N_W^{\ssup{\ell,\Scal}}\rangle,
\end{aligned}
$$
with an obvious choice of $\eps_{M,\vartheta,S,T}^{\ssup{\Scal}}$ such that $\lim_{M,\vartheta,T\to\infty} \eps_{M,\vartheta,S,T}^{\ssup{\Scal}}=0$ for any $S$.
\end{proof}

\subsection{Ergodic approximation}\label{sec-ergappr}

In the main result of this section, Lemma~\ref{lem-ergappr}, we prove a technical, but crucial approximation that will be used in several important other proofs. The main message is that an arbitrary shift-invariant loop-interlacement distribution can be replaced by an ergodic measure that has only loops and interlacements with additional nice properties, without changing the crucial properties (particle densities in loops and in interlacements, interaction energy, entropy) much. More precisely, in the crucial variational formula $\chi(\rho_1,\rho_2)$ in \eqref{chidefneu}, we show that the infimum can be restricted, with only a small error, to the set of ergodic random configurations such that their projection on $W=[-R,R]^d$ is concentrated on the event $A_{W;\mathfrak r,M,L,K,T}$ defined in \eqref{ALcaldef}, \eqref{AScaldef} and \eqref{Adef} and on $\Scal_W^{\ssup{M,\vartheta,S}}$ (defined  in \eqref{SMvarthetaSdef}) if $M,L,K,T,\vartheta$ are picked large enough, uniformly in $S\in \N$. Here we will use Lemma~\ref{lem-EntProjImag} to handle the entropy part of the variational formula.

Lemma~\ref{lem-ergappr} is crucial for applying several preliminary results that rely on such special properties of the random configurations, since these imply for example that the error term in the large-deviations lower bound in Lemma~\ref{lem-lowerboundlogpint} can be controlled and that the interaction energy per unit volume is bounded. More explicitly, Lemma~\ref{lem-ergappr} will be used in the proof of the lower bound of Theorem~\ref{thm-freeenergy} in Section~\ref{sec-finishlowerbound} (it is what we need in order to go on from  \eqref{hatZNlimitlowbound3})
and in the proof of Theorem~\ref{thm-specrelent} in Section~\ref{sec-limitingentropy} and was already used in the proof of Lemma~\ref{lem-propertieschi} in Section~\ref{sec-lemmaproof}.

We follow a similar ansatz as in the proof of Lemma~\ref{lem-noshreds}. Indeed, for a given loop/shred configuration in a given large box, we do not change anything in the configuration in the box, but we rewire all the ends of the shreds (the loop shreds and the interlacement shreds) in such a way that they are now connected with certain regularly placed sites at the boundary of a slightly larger box. Doing this independently in all these slightly larger subboxes (which decompose $\R^d$), they connect up with the new shreds in the neighbouring subbox. This leads to a global configuration whose crucial three quantities (entropy, interaction energy and particle densities) are hardly changed and which has much better independence properties, in particular, is ergodic.

\begin{lemma}[Ergodic approximation]\label{lem-ergappr}
Fix  $\rho_1,\rho_2\in(0,\infty)$ and fix $P\in \Mcal_1^{\ssup{\rm s}}(\Lcal\times\Scal)$ with finite $ \h^{\ssup{\Lcal,\Scal}}(P)$ and finite $\langle F_U,P\rangle$ and $\langle P,{\mathfrak N}_U^{\ssup {\ell,\Lcal}}\rangle=\rho_1, \langle P,{\mathfrak N}_U^{\ssup {\ell,\Scal}}\rangle=\rho_2$. Furthermore, fix $\eps\in(0,1)$. 

Then, for any sufficiently large $\vartheta, M,L,K,T,R$ (in this order), there is an ergodic $P'\in\Mcal_1^{\ssup{\rm s}}(\Lcal \times\Scal)$ such that $\Pi_{W_{R+M}}(P')(A_{{W_{R+M}};\mathfrak r,M,L,K,T})=1=\Pi_{W_{R+M}}^{\ssup\Scal}(P')(S_{W_{R+M}}^{\ssup{M,\vartheta,S}}),\langle P',\mathfrak N_U^{\ssup{\ell, \Lcal}}\rangle \in\Bcal_\eps(\rho_1), \langle P',\mathfrak N_U^{\ssup{\ell, \Scal}}\rangle \in\Bcal_\eps(\rho_2)$ and $\frac 1{|{W_{R+M}}|}\langle F_{{W_{R+M}},{W_{R+M}}}, \Pi_{W_{R+M}}( P')\rangle\leq \eps +\langle F_U,P\rangle$ and $J_{W_{R+M}}(\Pi_{W_{R+M}}(P'))\leq \eps +J_{W_R}(\Pi_{W_R}(P))$.
\end{lemma}

\begin{proof}
Our proof is in the spirit of the one of Lemma~\ref{lem-noshreds} in Section~\ref{sec-noshreds}. Again we abbreviate
$$
W=W_R=[-R,R]^d\qquad \mbox{and}\qquad \widetilde W=W_{R+M}=[-R-M,R+M]^d\qquad\mbox{and put}\qquad W'=\widetilde W\setminus W.
$$
As usual, $\Pi_{\widetilde W}$ is the well-known projection on $\widetilde W$. Denote by $\widetilde \Pi_{\widetilde W}=\Pi_{\widetilde W;\mathfrak r,M,L,K,T}$ the projection onto $A_{\widetilde W;\mathfrak r,M,L,K,T}$ introduced in Lemma~\ref{lem-EntProjImag} under the name $\Pi$ (with $\widetilde W$ instead of $W$).  We pick $M, K,L,\vartheta,T$ so large that the particle densities of the $\widetilde\Pi$-projection of $P$ are close enough to the ones of $P$ and the error terms in Lemma~\ref{lem-EntProjImag} are small enough, more precisely,
$|\langle \mathfrak N_{U}^{\ssup{\ell,\Lcal}}\circ\widetilde\Pi_{U},\Pi_{ U}(P) \rangle-\rho_1|\leq \eps$ and $|\langle \mathfrak N_{U}^{\ssup{\ell,\Scal}}\circ\widetilde\Pi_{U},\Pi_{ U}(P) \rangle-\rho_2|\leq \eps$ and that $\lim_{R'\to\infty}\eps^{\ssup\Lcal}_{M,[-R',R']^d,L}+(\frac 1K+\eps_{M,\vartheta,S,T}^{\ssup\Scal})(\rho_1+\rho_2+1)\leq \eps$.

Then, for any sufficiently large  $R$ we will construct an ergodic $\widetilde P\in \Mcal_1^{\ssup{\rm s}}(\Lcal\times \Scal)$ such that 
\begin{eqnarray}
\frac1{|\widetilde W|}\langle F_{\widetilde W,\widetilde W}\circ \widetilde\Pi_{\widetilde W}, \Pi_{\widetilde W}(\widetilde P)\rangle- \frac 1{|W|}\langle F_{W,W}\circ\widetilde\Pi_W,\Pi_W(P)\rangle&\leq& \eps,\label{Goal1}\\
\frac 1{|\widetilde W|}\langle [\mathfrak N_{\widetilde W}^{\ssup{\ell,\Lcal}}+ \mathfrak N_{\widetilde W}^{\ssup{\ell,\Scal}}]\circ\widetilde\Pi_{\widetilde W},\Pi_{\widetilde W}(\widetilde P) \rangle- \frac 1{| W|}\langle [\mathfrak N_{W}^{\ssup{\ell,\Lcal}}+ \mathfrak N_{W}^{\ssup{\ell,\Scal}}]\circ\widetilde\Pi_{W},\Pi_{ W}(P) \rangle|&\leq& \eps,\label{Goal2}\\
J_{\widetilde W}(\Pi_{\widetilde W}(\widetilde P))-J_{W}(\Pi_W(P)\circ\widetilde \Pi_{W}^{-1})&\leq &\eps.\label{Goal3}
\end{eqnarray}

This will end the proof of this lemma with $P'=\widetilde P\circ \widetilde \Pi_{\widetilde W}^{-1}$, since $\frac 1{|W|}\langle F_{W,W}\circ\widetilde\Pi_W,\Pi_W(P)\rangle\leq \langle F_U,P\rangle$ and, if $R$ is large enough,  $J_{\widetilde W}(\Pi_{\widetilde W}(\widetilde P)\circ\widetilde \Pi_{\widetilde W}^{-1})\leq \eps + J_{\widetilde W}(\Pi_{\widetilde W}(\widetilde P))$ (see Lemma~\ref{lem-EntProjImag} and our choice of $M,K,L;T$ above), and $J_{W}(\Pi_W(P)\circ \widetilde \Pi_{W}^{-1})\leq J_W(\Pi_W(P))+\eps$ again by Lemma~\ref{lem-EntProjImag}.

Let us now construct $\widetilde P$. This will be partially similar to the proof of Lemma~\ref{lem-noshreds}. Again, we extend the $W$-projection of $P$ to a random configuration in $\widetilde W$, but here we do it in such a way that all  $W$-shreds are extended to the boundary of $\widetilde W$ in such a way that their composition in all the boxes $z+\widetilde W$ with $z\in 2(R+M)\Z^d$ concatenates to  interlacements; this time we are not creating new loops. We do this for all shreds, also the loop shreds, since we will later show that the expected number of particles in these loops is negligible in the limit $W\uparrow \R^d$; hence it will not change much in the expected particle numbers if we turn all loops that have particles outside $W$ into interlacements.

Decompose $\R^d$ regularly into boxes $\widetilde W_z=z+\widetilde W$ with $z\in 2(R+M)\Z^d$ and assume that $R$ is much larger than $M$.  Put $\xi_R=\Pi_{W}(P)\circ\widetilde\Pi_{W}^{-1}$
 and $\psi_R=\partial\Pi_{W}^{\ssup\Scal}(\xi_R)$. We will construct a loop/shred configuration $\xi_R^{\ssup M}\in\Mcal_1(\Lcal_{\widetilde W}\times\Scal_{\widetilde W})$ that extends $\xi_R$ by appending independent Brownian $W'$-shreds  in such a way that no new  loop (or hardly any) is created in $\widetilde W$ and all $W$-shreds are connected with $\partial \widetilde W$ in a certain regular way. We will do this independently over $z\in 2(R+M)\Z^d$, such that all these added shreds in neighbouring subboxes  automatically connect up such that they are concatenated to interlacements in $\R^d$. 
 
For this, pick some deterministic  index set $\Phi$ with $|\Phi|\asymp |W|/R\asymp R^{d-1}$ and a collection of points $(X_a)_{a\in \Phi} $ in $W_{R+M}^\circ\setminus W_{R+M-1}$ that are more than one away from the corners of $W_{R+M}$, i.e., from $(R+M)\{-1,1\}^d$. Then each $X_a$ is closest to precisely one of the $2d$ neighbouring subboxes $z+W_{R+M}$ with $z\in 2(R+M)\{-1,1\}^d$ (i.e., there is precisely one $z\in 2(R+M)\{-1,1\}^d$ such that $X_a\in \Bcal_1(\widetilde W\cap(z+ \widetilde W))$, where $\Bcal_1$ denotes the 1-neighbourhood). Assume that, for any $z\in 2(R+M)\{-1,1\}^d$, the number of $a$ such that $X_a$ is closest to $z+W$ is the same. Put $Y_a^{\ssup z}=X_a$ for these $z$. 
Furthermore, put $X_a^{\ssup z}= z+X_a$ for $z\in 2(R+M)\Z^d$, and define $Y_a^{\ssup {z'} }$ as above for any neighbour $z'$ of $z$ and $X_a^{\ssup z}$. We conceive, for any $z\in 2(R+M)Z^d$, the pair $(X_a^{\ssup z},Y_a^{\ssup z})\in (z+\widetilde W,(z+\widetilde W)^{\rm c})$ as the entry-exit sites of a $(z+\widetilde W)$-shred. They lie perfectly opposite to each other close to the boundary of $z+\widetilde W$. We put $X_a^{\ssup 0}=X_a$ and $Y_a=Y_a^{\ssup 0}$.

Now let $\mu=\sum_{i\in I}\delta_{(x_i,l_i,y_i)}\in\Tcal_W$ be given such that $y_i\in \widetilde W$ for all $i\in I$. We are considering $\mu$ under $\partial\Pi_W^{\Ssup\Scal}(\xi_R)$, i.e., we can see $\mu$ as the configuration of all the starting/terminal sites and lengths of all the $W$-shreds in the configuration that is sampled under $\xi_R$. Note that $\xi_R$ has, by definition as a random configuration in the event $A_{W;\mathfrak r,M, L,K; T}$, with probability one, only a bounded number of particles in any subbox of radius $\mathfrak r$ around sites in $2\mathfrak r \Z^d$ and hence not more than $\asymp R^{d-1}$ particles close to $\partial W$. Therefore, we may assume that $|I |\leq O(R^{d-1})$.

We pick an injective map $I\ni i\mapsto a_i\in \Phi$ and connect $X_{a_i}$ with $x_i$ and $y_i$ with $Y_{a_i}$ by means of an independent Brownian $W'$-shred, for any $i\in I$. For all other $a$'s, i.e., for $a\in \Phi\setminus I_\Phi$ with $I_\Phi=\{a_i\colon i\in I\}$, we connect $X_a$ with $Y_a$ with an additional independent Brownian $W'$-shred. The distribution of  the resulting configuration is called $\xi_R^{\ssup M}$, i.e., the extension of $\xi_R$ to a loop/shred configuration in $\widetilde W$ without loops in $W'$.

Having constructed $\xi_R^{\ssup M}$ in this way, we take independent measures  $\xi_R^{\ssup {M,z}}\in\Mcal_1(\Lcal_{\widetilde W_{z}}\times\Scal_{\widetilde W_{z}})$ with $z\in 2(R+M) \Z^d$ such that their shift $\theta_z(\xi_R^{\ssup {M,z}})$ has the same distribution as $\xi_R^{\ssup M}$. Then we put 
\begin{eqnarray}
\widetilde P_{R,M}&=&\bigotimes_{z\in 2(R+M) \Z^d}\xi_R^{\ssup {M,z}}\in \Mcal_1(\Lcal_W\times\Scal_W),\\
P_{R,M}&=&\frac 1{|\widetilde W|}\int_{\widetilde W}\d x\, \theta_x(\widetilde P_{R,M})\in \Mcal_1^{\ssup {\rm s}}(\Lcal_{\widetilde W}\times\Scal_{\widetilde W}),
\end{eqnarray}
the latter which is obviously stationary. The proof of \cite[Theorem 14.12]{G88} applies and shows that $P_{R,M}$ is ergodic. We will show that $\widetilde P=P_{R,M}$ satisfies \eqref{Goal1}--\eqref{Goal3} for $R$ large enough, which ends the proof of the lemma.

Note that $P_{R,M}$ has its interlacements only from our construction from the $|\Phi|$-many connections of $X_a^{\ssup z}$ with $Y_a^{\ssup z}$ (partially via $x_i$ and $y_i$), which perfectly connect each other in the sites $Y_a^{\ssup z}=X_a^{\ssup {z'}}$, if $z+\widetilde W$ and $z'+\widetilde W$ are neighbouring subboxes. All the loops of $P_{R,M}$ are the $(z+W)$-loops for some $z\in 2(R+M)\Z^d$.

The proofs of \eqref{Goal1} and \eqref{Goal2} are more or less identical to the proofs of \eqref{Goal1a} and \eqref{Goal2a} in the proof of Lemma~\ref{lem-noshreds}; we leave the details to the reader; let us only mention that here also the following  Lemma~\ref{lem-loopshredsnegl} is used, which says that our decision to turn all $W$-loops that have particles also in $W'$ into interlacements does not noticeably change the particle densities: 
For shift-invariant loop-distributions, the expected number of particles in loops that have particles in $W$, but not all their particles, is negligible:

\begin{lemma}\label{lem-loopshredsnegl} For any $P\in\Mcal_1^{\ssup{\rm s}}(\Lcal)$ with $\langle P,\mathfrak N_U^{\ssup{\ell}}\rangle\in(0,\infty)$, we have
$$
\lim_{R\to\infty}\frac 1{|W_R|}\langle\Pi_{W_R}^{\ssup\Scal}(P),\mathfrak N_{W_R}^{\ssup\ell}\rangle=0.
$$
\end{lemma}  
 
\begin{proof} Note that $\Pi_{W_R}^{\ssup\Scal}(P)=\Pi_{W_R}(P)-\Pi_{W_R}^{\ssup\Lcal}(P)$. By shift-invariance,  for $R\in\N$, 
$$
\frac 1{|W_R|}\langle\Pi_{W_R}(P),\mathfrak N_{W_R}^{\ssup\ell}\rangle=\frac 1{|W_R|}\sum_{z\in W_R\cap \Z^d}\langle\Pi_{W_R}(P),\mathfrak N_{z+U}^{\ssup\ell}\rangle=\langle P,\mathfrak N_{U}^{\ssup\ell}\rangle.
$$
Furthermore, the normalized integral with respect to the second  is  asymptotically not smaller than $\langle P,\mathfrak N_{U}^{\ssup\ell}\rangle$. Indeed, denoting by $\mathfrak N_{\widetilde W}^{\ssup{\ell,\leq M}}$ the number of particles in all the loops $f$ that start in $\widetilde W\Subset \R^d$ and have spread $\max_{t\in[0,\beta\ell(f)]}|f(t)-f(0)|\leq M$, we see that
$$
\frac 1{|W_R|}\langle\Pi_{W_R}^{\ssup\Lcal}(P),\mathfrak N_{W_R}^{\ssup\ell}\rangle
\geq \frac 1{|W_R|}\sum_{z\in W_{R-M}\cap \Z^d}\langle\Pi_{W_R}(P),\mathfrak N_{z+U}^{\ssup{\ell,\leq M}}\rangle
=\frac {|W_{R-M}|}{|W_R|} \langle P,\mathfrak N_{U}^{\ssup{\ell,\leq M}}\rangle \to \langle P,\mathfrak N_{U}^{\ssup{\ell,\leq M}}\rangle 
$$
as $R\to\infty$, and the last term converges to $\langle P,\mathfrak N_{U}^{\ssup\ell}\rangle$ as $M\to\infty$, according to the monotonous convergence theorem.
\end{proof}  

\medskip

\textit {Proof of \eqref{Goal3}:} Analogously to the argument around \eqref{J(P)_J(widetildeP)}, we can argue also here that it suffices to show that $J_{\widetilde W}(\xi_R^{\ssup M})\leq J_W(\xi_R)+\eps$. Indeed, the estimate in  \eqref{J(P)_J(widetildeP)} is analogously valid here as well, with the difference that the second error term on the right-hand side of \eqref{JWupperboundx} is not trivial as in \eqref{J(P)_J(widetildeP)} (since $\psi=0$ there), but needs an argument, since $\psi=\partial \Pi_{3\widetilde W}^{\ssup\Scal}(\widetilde P_{R,M})$ is non-trivial: For large $R$ we have
\begin{equation}\label{J(P)_J(widetildeP)neu}
\begin{aligned}
J_{\widetilde W}(\Pi_{\widetilde W}(P_{R,M}))
&\leq J_{3 \widetilde W}(\Pi_{3 \widetilde W}(\widetilde P_{R,M}))+\frac\eps2-\langle \psi,\mathfrak p_{3 \widetilde W,M}\rangle.
\end{aligned}
\end{equation}
In order to show that the last term is negligible, we point out that all the $3 \widetilde W$-shreds $f_i'$ of $\Pi_{3 \widetilde W}(\widetilde P_{R,M})$ satisfy the condition $|f_i'(j S\beta)-f_i'((j-1)S\beta)/\leq \vartheta\sqrt S$  either almost surely for those $j$ such that $f_i'([(j-1)S\beta,jS\beta])$ lies in $z+W$ for some $z\in  2(R+M)\Z^d$ (since they come from $\xi_R$ or an independent copy of it, which  is concentrated on $A_{W_R;\Theta}^{\ssup\Scal}$) or in expectation, since these pieces are independent Brownian bridges by construction (see the argument in Remark~\ref{rem-mathfrakp_esti_for_BM}). Therefore, making $S$ large enough, we may assume that $J_{\widetilde W}(\Pi_{\widetilde W}(P_{R,M}))\leq J_{3 \widetilde W}(\Pi_{3 \widetilde W}(\widetilde P_{R,M}))+\eps$.

In order to see that $ J_{3 \widetilde W}(\Pi_{3 \widetilde W}(\widetilde P_{R,M}))=J_{\widetilde W}(\Pi_{\widetilde W}(\widetilde P_{R,M}))=J_{\widetilde W}(\xi_R^{\ssup M})$, observe that the configurations in the $3^d$ subboxes $z+\widetilde W$ with $z\in 2(R+M)\{-1,0,1\}^d$ are independent copies from each other and that therefore the density of $\Pi_{3\widetilde W}(P)(\omega,\mu,\cdot)$ with respect to ${\tt K}_{3\widetilde W}(\mu,\cdot)$,  conditional on the boundary-shred configuration $\mu$, decomposes into a product of the corresponding $(z+\widetilde W)$-versions (since $\Pi_{3\widetilde W}(P)$ decomposes into i.i.d.~parts by definition, and ${\tt K}_{3\widetilde W}$ decomposes by the Markov property), and therefore the entropy of $\Pi_{3\widetilde W}(P)$ decomposes into the sum over these entropies. Because of the prefactor $\frac1{ |\widetilde W|}$ in the definition of $J_{\widetilde W}$, we obtain that the $J_{3\widetilde W}$-value equals the $J_{\widetilde W}$-value. Hence, for proving \eqref{Goal3}, we only need to show that $J_{\widetilde W}(\xi_R^{\ssup M})\leq J_W(\xi_R)+\eps$.

The distribution of $\xi_R ^{\ssup M}$ can be written, using the Markov property, as follows. Let $(\omega,\varpi)\in\Lcal_W\times \Scal_W$ with $\varpi=\sum_{i\in I}\delta_{f_i}$ and $\mu=\partial \Pi_W^{\ssup\Scal}(\varpi)=\sum_{i\in I}\delta_{(x_i,l_i,y_i)}\in\Tcal_W$. Recall that $|I|\leq |\Phi|$ and that we did not add any $\widetilde W$-loops.  We can write, using the Markov property and the independence of all the Brownian $W'$-shreds,
\begin{equation}\label{xiRMident}
\begin{aligned}
\xi_R^{\ssup M}\big(\d(\omega',\varpi')\big)
&=\xi_R\big(\d (\omega,\varpi)\big)\,\e^{-\sum_{k\in\N} [q_{W',k}+ q_{W',W,k}]}\\
&\quad\otimes\bigotimes_{i\in I}\Big[p_{X_{a_i},x_i}^{\ssup{W'}}(\d m_i) q_{X_{a_i},x_i}^{\ssup{m_i,W'}}(\d g_i)\, p_{y_i,Y_{a_i}}^{\ssup{ W'}}(\d m_i') q_{y_i,Y_{a_i }}^{\ssup{m_i', W'}}(\d g_i')\Big]\\
&\quad\otimes \bigotimes_{a\in \Phi\setminus I_\Phi}\Big[p^{\ssup{W'}}_{X_a,Y_a}(\d n_a)q^{\ssup{n_a,W'}}_{X_a,Y_a}(\d g_a') \Big],
\end {aligned}
\end{equation}
where we remind on the definition of the $p$- and the $q$-terms in \eqref{pxyneudef} and \eqref{qneudef}. We used the notation $\omega'=\omega$ (seen as an element of $\Lcal_{\widetilde  W}$ with only $W$-loops) and $\varpi'=\sum_ {i\in I}\delta_{g_i'\diamond f_i\diamond g_i}+\sum_{a\in \Phi\setminus I_\Phi}\delta_{g_a'}$ and $\mu'=\sum_{i\in I}\delta_{(X_{a_i},L_{a_i},Y_{a_i})} +\sum_{a\in \Phi\setminus I_\Phi}\delta_{(X_a,n_a,Y_a)}$ with $L_{a_i}=m_i+\ell(f_i)+m_i'$.

Analogously, we write the distribution of the reference measure, using the above notation. Recall from the proof of Lemma~\ref{lem-noshreds} that the reference measure $\Pi^{\ssup\Lcal}_{\widetilde W}({\tt Q})$ is the convolution of $\Pi^{\ssup\Lcal}_{W}({\tt Q})$, $\Pi^{\ssup\Lcal}_{W'}({\tt Q})$ and $\Pi_{W,W'}^{\ssup\Lcal}({\tt Q})$ (recall that the latter measure is the restriction of ${\tt Q}$ to the set of loops that have all particles in $W\cup W'$ and have at least one in $W$ and one in $W'$). Hence, the loop part is similar to \eqref{Qdecomposition} without the $n$-loop, and we have
\begin{equation}\label{referencemeasure}
\begin{aligned}
\Pi_{\widetilde W}^{\ssup\Lcal}({\tt Q})&\otimes [\partial\Pi_{\widetilde W}^{\ssup\Scal}(\xi_R^{\ssup M})\otimes {\tt K}_{\widetilde W}]\big(\d(\omega',\varpi')\big)
=\Pi_{ W}^{\ssup\Lcal}({\tt Q})(\d \omega)\,\e^{-\sum_{k\in\N} [q_{W',k}+ q_{W',W,k}]}\,\e^{-\sum_{k\in\N} q_{W,W',k}}\\
&\qquad\otimes \partial\Pi_{W}^{\ssup\Scal}(\xi_R)(\d\mu)\bigotimes_{i\in I}\Big[p_{X_{a_i},x_i}^{\ssup{ W'}}(\d m_i) q_{X_{a_i},x_i}^{\ssup{m_i, W'}}(\d g_i) p_{y_i,Y_{a_i}}^{\ssup{ W'}}(\d m_i') q_{y_i,Y_{a_i }}^{\ssup{m_i',W'}}(\d g_i')\Big]\\
&\quad \otimes \bigotimes_{a\in \Phi\setminus I_\Phi}\Big[p^{\ssup{W'}}_{X_a,Y_a}(\d n_a)q^{\ssup{n_a, W'}}_{X_a,Y_a}(\d g_a') \Big]\otimes {\tt K}_W(\mu,\d \varpi).
\end{aligned}
\end{equation}
Summarizing, 
$$
\begin{aligned}
\frac{\d\xi_R^{\ssup M}}{\d \Pi_{\widetilde W}^{\ssup\Lcal}({\tt Q})\otimes [\partial\Pi_{\widetilde W}^{\ssup\Scal}(\xi_R^{\ssup M})\otimes {\tt K}_{\widetilde W}]}(\omega',\varpi')
=\frac{\d\xi_R}{\d \Pi_{W}^{\ssup\Lcal}({\tt Q})\otimes [\partial\Pi_{W}^{\ssup\Scal}(\xi_R)\otimes {\tt K}_{ W}]}(\omega,\varpi)\e^{\sum_{k\in\N}q_{W,W',k}}.
\end{aligned}
$$
Hence, we see that
$$
J_{\widetilde W}(\xi_R^{\ssup M})
=J_W(\xi_R)+\frac 1{|W|}\sum_{k\in\N}q_{W,W',k},
$$
and as in the proof of Lemma~\ref{lem-noshreds}, one sees that the last term is $\leq\eps$ for $R$ large enough. This ends the proof of \eqref{Goal3} and hence the proof of this lemma.
\end{proof}

\section{Proof of Theorem~\ref{thm-specrelent}: specific relative entropy density}\label{sec-limitingentropy}

\noindent In this section we establish a new kind of {\it specific relative entropy per volume} of ergodic loop/interlacement configurations, namely with respect to the product of $\Pi_W^{\ssup\Lcal}({\tt Q})$ and the product of the boundary-configuration and the kernel ${\tt K}_W$ defined in \eqref{kernelK}. In other words, we prove Theorem~\ref{thm-specrelent}. In particular, we identify the limit of the entropy term $J_W(\xi)$ defined in \eqref{JWdeffirst} (see also \eqref{JWdef}) appearing in the variational formula in \eqref{hatZNlimit} and \eqref{hatZNlimitlowbound3} in the limit  $R\to\infty$, i.e., in the limit $W=W_R=[-R,R]^d\uparrow  \R^d$. Recall that, in $d\geq 3$,  ${\tt K}_W$ is a regular version of the distribution of $\Pi_W^{\ssup\Scal}({\tt R})$ given $\partial\Pi_W^{\ssup\Scal}({\tt R})$ for the Brownian $\beta$-spacing interlacement PPP, ${\tt R}$, of Section~\ref{sec-InterlacePPP}, see Lemma~\ref{lem-KWident}, but we are not going to use this fact.

\begin{remark}
{The usual starting point of such a proof is to prove the superadditivity of the map $W\mapsto H_{\Lcal_{W}\times \Scal_W}(\Pi_{W}(P)\mid\Pi_{W}^{\ssup\Lcal}(\LPP)\otimes [\partial\Pi_W^{\ssup\Scal}(P)\otimes {\tt K}_W])$ on the set of all boxes $W\Subset \R^d$, but we did not succeed in doing this, as it concerns the interlacement part. Instead, we will be relying on Proposition~\ref{lem-uppboundJ_W}, which needs ergodicity, combined with Lemma~\ref{lem-ergappr}, which approximates arbitrary stationary $P$'s by ergodic ones that have several useful additional properties.
}
\hfill$\Diamond$
\end{remark}

\begin{proof}[Proof of Theorem~\ref{thm-specrelent}] 
Let $P$ be in $\Mcal_1^{\ssup{\rm s}}(\Lcal\times\Scal)$ with finite expected particle number in the unit box $U$. 

{\it (1) We first prove that the limit in \eqref{jointentropydef} exists and satisfies \eqref{hbound}.} Pick an $R\in(0,\infty)$ and pick any sequence $(R_n)_{n\in\N}$ tending to $\infty$. We would like to apply Proposition~\ref{lem-uppboundJ_W}(2) and (3) to get a lower bound for $\liminf_{n\to\infty}J_{W_{R_n }}(\Pi_{W_{R_n}}(P))$ that is close to $J_{W_R}(\Pi_{W_R}(P))$, but we can do this only for $W_R$-ergodic $P$'s that are concentrated on $\Lcal\times\Scal^{\ssup{M,\vartheta,S}}$ (see \eqref{SMvarthetaSdef}) for some $M, \vartheta, S$. Hence, we first need to replace $\Pi_{W_{R_n}}(P)$ by some measure like that, using Lemma~\ref{lem-ergappr}. Indeed, with parameters $\mathfrak r, M, L, K,T, \vartheta,S$, we apply Lemma\ \ref{lem-ergappr} (see also the proof of Lemma~\ref{lem-EntProjImag}) to see that, for any given small $\eps\in(0,1)$ and any $n$ there is an ergodic $P_n\in\Mcal_1^{\ssup{\rm s}}(\Lcal^{\ssup M}\times\Scal^{\ssup {M,\vartheta,S}})$ such that $\Pi_{W_{R_n}}^{\ssup\Scal}(P_n)$ is concentrated on $A^{\ssup{\Scal}}_{W_{R_n};M,\mathfrak r,T}\cap \Scal_{W_{R_n}}^{\ssup{\vartheta,S}}$ and satisfies $|\langle P_n-P,\mathfrak N^{\ssup{\Scal,\ell}}_{W_ {R_n}}\rangle|\leq \eps |W_{R_n}|$ such that
$$
J_{W_{R_n }}(\Pi_{W_{R_n}}(P))
\geq J_{W_{R_n}}(\Pi_{W_{R_n}}(P_n))-\eps,
$$
if the parameters $M,L,K,T,\vartheta$ are picked large enough. Now we can apply Proposition~\ref{lem-uppboundJ_W}(2). Indeed, we write $R_n=m_n R^{\ssup{m_n}}$ with $m_n\in\N$ and $m_n \to \infty$ and  $R^{\ssup{m_n}}\downarrow R$ as $n\to\infty$ and get
$$
\begin{aligned}
J_{W_{R_n }}(\Pi_{W_{R_n}}(P_n))
&\geq  J_{W_{R^{\ssup{m_n}}}}(\Pi_{W_{R^{\ssup{m_n}}}}(P_n))-\frac{C_M}{(m_n R^{\ssup{m_n}})^{d/2}}-\Big[\frac {C_{M,\mathfrak r,T}}{m_n R^{\ssup{m_n}}}+\frac{C\vartheta^2}{S}\Big]\langle P_n,\mathfrak N_U^{\ssup{\Scal,\ell}}\rangle.
\end{aligned}
$$
We can pick $S=\vartheta^3$. Now we combine the last two displays and make $n\to\infty$ and use that $|\langle P_n-P,\mathfrak N^{\ssup{\Scal,\ell}}_{W_ {R_n}}\rangle|\leq \eps |W_{R_n}|$ and that both $P$ and $P_n$ are shift-invariant to get that 
\begin{equation}\label{hexist2}
\begin{aligned}
\liminf_{n\to\infty} J_{W_{R_n }}(\Pi_{W_{R_n}}(P))
&\geq \liminf_{n\to\infty}  J_{W_{R^{\ssup{m_n}}}}(\Pi_{W_{R^{\ssup{m_n}}}}(P_n))-\frac{C}{\vartheta}\big(\eps +\langle P,\mathfrak N_U^{\ssup{\Scal,\ell}}\rangle\big)-\eps.
\end{aligned}
\end{equation}
We can make $M,K,L,T,\vartheta\to\infty$ to  obtain
\begin{equation}\label{liminfgeq}
\begin{aligned}
\liminf_{n\to\infty} J_{W_{R_n }}(\Pi_{W_{R_n}}(P))
&\geq \liminf_{M,K,L,T,\vartheta\to\infty}\liminf_{n\to\infty}  J_{W_{R^{\ssup{m_n}}}}(\Pi_{W_{R^{\ssup{m_n}}}}(P_n))
-\eps\\
&\geq J_{W_R}(\Pi_{W_R}(P))-\eps,
\end{aligned}
\end{equation}
where we used the following argument for the $J$-term. Indeed, we want to use Proposition~\ref{lem-uppboundJ_W}(3). Note that, from the construction of $P_n$ for $W_{R_n}$ in the proof of Lemma~\ref{lem-ergappr}, we see that $\Pi_{W_{R_n}}(P_n)=\Pi_{W_{R_n}}(P)\circ \widetilde \Pi_{W_{{R_n}}}^{-1}$, where $\widetilde \Pi_{W_{R_n}}$ is the projection onto $A_{W_{R_n};M,L,K,T}$ introduced in Lemma~\ref{lem-EntProjImag}. Therefore, it is easy to see that $P_n$ weakly converges (in our usual sense that all the projections $\Pi_{W_{R'}}(P_n)$ weakly converge, $R'\in\N$) towards $P$ if all the parameters $n,M,K,L,T,\vartheta$ grow to $\infty$.  Furthermore, $\Pi_{W_{R^{\ssup{m_n}}}}(P_n)=\Pi_{W_{R^{\ssup{m_n}}}}(P)\circ \widetilde \Pi_{W_{R^{\ssup{m_n}}}}^{-1}$. In order to see that, in this limit inferior the limit $P_n\Longrightarrow P$ and the limit as $R^{\ssup{m_n}}\downarrow R$ according to Proposition~\ref{lem-uppboundJ_W}(3) can be combined such that the last step of \eqref{liminfgeq} appears, use first \eqref{entropyofimages} for $f= \widetilde \Pi_{W_{R^{\ssup{m_n}}}}$ and then prove that the density $\d\nu_n\circ \widetilde \Pi_{W_{R^{\ssup{m_n}}}}^{-1}/\d \nu_n$ tends to one in the limit as $n,M,K,L,T,\vartheta\to \infty$, where $\nu_n={\tt Q}_{W_{R^{\ssup{m_n}}} }\otimes[\partial \Pi_{W_{R^{\ssup{m_n}}} }(P)\otimes {\tt K}_{W_{R^{\ssup{m_n}}}}]$ is the reference measure. We leave the proof of the latter to the reader.

Now \eqref{liminfgeq} implies that the limit inferior is not smaller than the limit superior, i.e., the limit in \eqref{jointentropydef} exists. Furthermore, from \eqref{liminfgeq} we also see that \eqref{hbound} holds.

{\it (2) Now we show that $P\mapsto\h^{\ssup{\Lcal,\Scal}}(P)$ is lower  semicontinuous.} Recall that we are using on $\Mcal_1( \Lcal\times\Shreds )$ the weak topology induced by all the $\Pi_W$, i.e., $P_n$ converges towards $P$ if and only if $\Pi_W(P_n)$ weakly converges towards $\Pi_W(P)$ for any $W\Subset\R^d$. We write $\Lcal_W\times \Scal_W$ as $[\Lcal_W\times \Tcal_W]\times \Scal_W$ and write $(\Pi_W(P))_{\Lcal_W\times \Tcal_W\to\Scal_W}(\omega,\mu,\cdot)$ for the corresponding kernel (by disintegration). Also the map $P\mapsto (\Pi_W(P))_{\Lcal_W\times \Tcal_W\to\Scal_W}(\omega,\mu,\cdot)$ is continuous for any $\omega\in \Lcal_W$ and $\mu\in\Tcal_W$.  It follows from the characterisation of the entropy in \eqref{entropyformula} and the fact that limits of lower semi-continuous functions are lower semi-continuous that relative entropies are lower semi-continuous in the first argument. With the help of writing $\partial\xi= \partial\Pi^{\ssup\Scal}_W(P)$ and $\xi_\mu(\d (\omega,\varpi))=(\Pi_W(P))_{\Tcal_W\to\Lcal_W\times\Scal_W}(\mu,\d(\omega,\varpi))$ and
\begin{equation}\label{decompScal-entropy}
\begin{aligned}
\Hcal_W(P)=|W|J_W(\Pi_W(P))&=\int_{\Tcal_W}\partial\xi(\d\mu)\, H_{\Lcal_W\times\Scal_W}\big(\xi_\mu\mid\Pi_W^{\ssup\Lcal}( {\tt Q})\otimes {\tt K}_W(\mu,\cdot)\big) 
\end{aligned}
\end{equation}
and Fatou's lemma, the assertion follows.

{\it (3) Next, we show the convexity of the map $P\mapsto J_W(\Pi_W(P))$.} This 
follows from the convexity of the map $(\mu,\nu)\mapsto H(\mu\mid \nu)$ for any relative entropy $H$, see  \cite[Proposition 15.5(d)]{G88}. Since limit of convex functions are convex, the convexity of $\h^{\ssup{\Lcal,\Scal}}(\cdot)$ on  $\Mcal_1^{\ssup{\rm s}}(\Lcal\times\Scal)$ follows. The proof of concavity given in the proof of \cite[Proposition 15.14]{G88} applies also here, using \eqref{decompScal-entropy}. 
Indeed, if one replaces $P$ by $sP_1+(1-s)P_2$ for $s\in(0,1)$ and $P_1,P_2\in\Mcal_1^{\ssup{\rm s}}(\Lcal\times\Scal)$, then $\Hcal_W(sP_1+(1-s)P_2)$ can be lower estimated (using the monotonicity of the logarithm) against $s \Hcal_W(P_1)+(1-s) \Hcal_W(P_2)+s\log s+(1-s)\log (1-s)$, and dividing by $|W|$ and letting $|W|\to \infty$ implies the concavity of  $\h^{\ssup{\Lcal,\Scal}}(\cdot)$ on the set of all ergodic $P\in \Mcal_1^{\ssup{\rm s}}(\Lcal\times\Scal)$.

{\it (4) Now we prove the compactness of the restricted level set of $\h^{\ssup{\Lcal,\Scal}}(\cdot)$, given in \eqref{levelset}.} Assume that $c\in\R$ and a sequence of compact sets $K_R\subset \Tcal_{W_R}$ with $R\in\N$ are given (we take the sequence $(R_N)_{N\in\N}$ as $\N$ without loss of generality).
Because of lower semicontinuity of $\h^{\ssup{\Lcal,\Scal}}(\cdot)$ and of $P\mapsto \langle P,\mathfrak N_U^{\ssup\ell}\rangle$, the set in \eqref{levelset} is closed, hence we only need to prove its tightness. Pick a sequence $(P_n)_{n\in\N}$ in $\Mcal_1(\Lcal\times \Shreds)$ in that set. Using \eqref{hbound}, we see that, for any $R\in\N$, the projections $\Pi_{W_R}(P_n)$ lie for any $n$ in the restricted level set 
\begin{equation}\label{restrictedlevelset}
 \big\{\xi\in\Mcal_1(\Lcal_{W_R}\times\Shreds_{W_R})\colon \partial \Pi_{W_R}^{\ssup\Scal}(\xi)\in K_R(L), J_{W_R}(\xi)\leq c'\big\},
\end{equation} 
for some $c'\in\R$. According to Lemma~\ref{lem-compactnessLoop}(1), this set is compact. Therefore, $(\Pi_{W_R}(P_n))_{n\in\N}$ is tight. This finishes the proof, because of the following argument: We can assume that $(\Pi_{W_R}(P_n))_{n\in\N}$ converges towards some $\xi^{\ssup{W_R}}\in\Mcal_1( \Lcal_{W_R}\times\Shreds_{W_R})$ as $n\to \infty$. Using a standard diagonal argument, we can assume that the $n$-subsequence works for all the $R\in\N$. Since $(\Pi_{W_R}(P_n))_{R\in\N}$ is consistent for any $n\in\N$, also $(\xi^{\ssup{W_R}})_{R\in\N}$ is consistent. Now Proposition~\ref{Prop-construction_of_P} yields the existence of a $P\in\Mcal_1(\Lcal\times\Shreds)$ such that $\Pi_{W_R}(P)=\xi^{\ssup{W_R}}$ for any $R\in\N$. It is easy to see that $P_n$ converges (along the chosen subsequence) towards $P$. 
\end{proof}

Let us close this section with making  some comments and remarks. First, let us state the following useful general inequality (which was already used in the proof of Theorem~\ref{thm-specrelent}). 

For a probability measure $\nu$ on the product of two Polish spaces  $\Xcal_1$ and $\Xcal_2$ with projections $\Pi_i\colon \Xcal_1\times\Xcal_2\to\Xcal_i$, we write $\nu_{1\to 2}$ for the Markov kernel from $\Xcal_1$ into $\Mcal_1(\Xcal_2)$ that is defined  by the disintegration formula
\begin{equation}\label{disintegration}
\nu\big(\d(x_1,x_2)\big)=\Pi_1(\nu)(\d x_1)\nu_{1\to 2 }^{\ssup{x_1}}(\d x_2).
\end{equation}
In other words, $\nu_{1\to 2}$ is a regular version of the conditional distribution of $X_2$ given $X_1$ if $(X_1,X_2)$ is a random vector with distribution $\nu$. Let us also recall the standard measure-theoretic notation $\mu\otimes K(\d(x_1,x_2))=\mu(\d x_1)\,K(x_1,\d x_2)$ and $\mu K(\d  x_2)=\int_{\Xcal_1}\mu(\d x_1)\, K(x_1,\d x_2)$ for $\mu\in\Mcal_1(\Xcal_1)$ and a kernel $K$ from  $\Xcal_1$ into $\Mcal_1(\Xcal_2)$. By $H_1, H_2$ and $H$ we denote the relative entropies on the spaces $\Mcal_1(\Xcal_1), \Mcal_1(\Xcal_2)$, and $\Mcal_1(\Xcal_1\times\Xcal_2)$, respectively.

\begin{lemma}[Entropy on products]\label{lem-entropyproducts}
For any two Polish spaces  $\Xcal_1$ and $\Xcal_2$ with projections $\Pi_i\colon \Xcal_1\times\Xcal_2\to\Xcal_i$,
\begin{equation}\label{entropyesti}
H(\nu\mid\rho)\geq H_1(\Pi_1(\nu)\mid\Pi_1(\rho))+H_2(\Pi_1(\nu)\nu_{1\to 2}\mid \Pi_1( \nu)\rho_{1 \to 2}),\qquad \nu,\rho\in\Mcal_1(\Xcal_1\times \Xcal_2).
\end{equation}
\end{lemma}

\begin{proof}We may assume that $\nu$ has a density with respect to $\rho$. Then 
\begin{equation}\label{Hdecomp}
\begin{aligned}
H(\nu\mid \rho)&=\int_{\Xcal_1}\Pi_1(\nu)(\d  x_1)\int_{\Xcal_2}\nu_{1\to 2}^{\ssup{x_1}}(\d x_2)\,\log\Big(\frac {\d(\Pi_1( \nu))}{ \d(\Pi_1(\rho))}(x_1)\frac{\d(\nu_{1\to 2}^{\ssup{x_1}})}{\d(\rho_{1\to 2}^{\ssup{x_1}})}(x_2)\Big)\\
&=H_1(\Pi_1(\nu)\mid \Pi_1(\rho))+\int_{\Xcal_1}\Pi_1(\nu)(\d x_1)\,H_2\big(\nu_{1\to 2}^{\ssup{x_1}}\mid \rho_{1\to 2}^{\ssup{x_1}}\big)\\
&\geq H_1(\Pi_1(\nu)\mid \Pi_1(\rho))+H_2\big(\Pi_1(\nu)\nu_{1\to 2}\mid \Pi_1(\nu)\rho_{1\to 2}),
\end{aligned}
\end{equation}
where in the last step we used Jensen's inequality for the convex function $(\nu,\rho)\mapsto H_2(\nu\mid\rho)$; see \cite[Proposition 15.5(d)]{G88}.
\end{proof}

\begin{remark}[$\h\geq \h^{\ssup\Lcal}$]\label{rem-hcomparison}
An application of \eqref{entropyesti} shows that $\h\geq \h^{\ssup\Lcal}$, where $\h$ was defined in \eqref{entropydensitydef}, and we define $\h^{\ssup\Lcal}(P)$ as $\h^{\ssup{\Lcal,\Scal}}(P)$ for any $P\in\Mcal_1^{\ssup{\rm s}}(\Lcal)$, where $P$ is considered as a loop-interlacement configuration with empty interlacement part. Indeed, for any $W\Subset \R^d$, observe that $\Pi_W(\Lcal)$ is the direct sum of $\Lcal_W=\Pi_W^{\ssup\Lcal}(\Lcal)$ and the set of all simple point measures that start in $W$, but have not all their particles in $W$. Using \eqref{entropyesti}, we see that 
$$
H_{\Pi_W(\Lcal)}\big(\Pi_W(P)\,\big|\, \Pi_W(\LPP)\big)
\geq H_{\Lcal_W}\big(\Pi_W^{\ssup\Lcal}(P)\,\big|\,\Pi_W^{\ssup\Lcal}(\LPP)\big),
$$
where we dropped on the right-hand side another entropy term.
\hfill$\Diamond$
\end{remark}

\section{Finish of the proof of Theorem~\ref{thm-freeenergy}: making $R\to\infty$}\label{sec-finishproof}

\noindent After all the preparations, we finish the proof of  Theorem~\ref{thm-freeenergy} in this section: the proof of the upper bound is finished in Section~\ref{sec-finishuppbound}, the one of the lower bound in Section~\ref{sec-finishlowerbound}.

\subsection{Upper bound}\label{sec-finishuppbound}
With the material from Sections~\ref{sec-limitingentropy} and \ref{sec-LDPupperbound}, we now finish the proof of the upper bound in Theorem~\ref{thm-freeenergy}, based on \eqref{hatZNlimit}.  Recall that $W=W_R=[-R,R]^d$ and that $\chi$ was defined in \eqref{chidefneu} and the set $\Mcal_1^{\ssup k}(\Lcal_W\times \Scal_W)$ in \eqref{Mvcalkdef}. Starting from \eqref{hatZNlimit}, it is sufficient to prove the following.
\begin{equation}\label{upperbound 5}
\begin{aligned}
&\liminf_{R\to\infty} \inf\Big\{\frac1{|W_R|}\langle \xi,F_{W_R,W_R}\rangle+ J_{W_R}(\xi)\colon \xi\in\bigcap_{k\in\N} \Mcal_1^{\ssup k}(\Lcal_{W_R}\times \Scal_{W_R}),
\\ 
&\quad\Pi_{W_R}^{\ssup\Lcal}(\xi)\in K_{R}^{\ssup\Lcal}(L),\Pi_{W_R}^{\ssup\Scal}(\xi)\in K_{R}^{  \ssup\Scal}(L), \langle \xi,\smfrac 1{|{W_R}|}{\mathfrak N}_{W_R}^{\ssup{\ell,\Lcal}}\rangle= \rho_1, \langle \xi,\smfrac 1{|{W_R}|}{\mathfrak N}_{W_R}^{\ssup{\ell, \Scal}}\rangle= \rho_2\Big\}\\
&\geq \inf\Big\{\langle P,F_U\rangle +\h^{\ssup{\Lcal,\Scal}}(P)\colon
P\in\Mcal^{\ssup{\rm s}}_1(\Lcal\times\Scal),\langle P,{\mathfrak N}_U^{\ssup {\ell,\Lcal}}\rangle\in\overline \Bcal_{2\delta}(\rho_1), \langle P,{\mathfrak N}_U^{\ssup {\ell,\Scal}}\rangle\in\overline \Bcal_{2\delta}(\rho_2)\Big\}, 
\end{aligned}
\end{equation}
where we recall the definition of $J_W(\xi)$ from \eqref{JWdef}. With the help of the continuity properties of $\chi$ from Lemma~\ref{lem-propertieschi}, the limit as $\delta\downarrow 0$ of the right-hand side is easily seen to coincide with $\chi(\rho_1,\rho_2)$, and the proof of \eqref{freeenergyident} is then finished. 

Recall the two sets $K_R^{\ssup\Lcal}(L)$ and $K_R^{\ssup\Scal}(L)$ from Lemmas~ \ref{lem-compactnessLoop} and \ref{lem-exponentiallytight}, respectively, and note that $L$ does not depend on $R$ and is picked large enough such that \eqref{hatZNlimit} is true.

Now we turn to the proof of \eqref{upperbound 5}. For the ease of notation, we now write $W$ instead of $W_R$.  Assume that $\xi^{\ssup {W}}$ is an approximate minimiser in the  variational formula on the left-hand side of \eqref{upperbound 5}, that is, it is amenable to that formula and approaches the infimum up to some quantity that vanishes as $R\to\infty$. Put
\begin{equation}
\Hcal=\liminf_{R\to\infty } J_{W}(\xi^{\ssup{W}})\qquad \mbox{and}\qquad \Fcal=\liminf_{R\to\infty }\frac1{|W|}\langle \xi^{\ssup{W}},F_{W,W}\rangle.
\end{equation}
Without loss of generality, $\Hcal$ and $\Fcal$ are finite. Let $\eps\in(0,1)$ be given. We need to construct a $P\in\Mcal^{\ssup{\rm s}}_1(\Lcal\times\Scal)$ that is amenable for the formula on the right-hand side of \eqref{upperbound 5} (i.e., satisfies $\langle P,{\mathfrak N}_U^{\ssup {\ell,\Lcal}}\rangle\in\overline \Bcal_{2\delta}(\rho_1),$ and $\langle P,{\mathfrak N}_U^{\ssup {\ell,\Scal}}\rangle\in\overline \Bcal_{2\delta}(\rho_2)$) and such that $\h^{\ssup{\Lcal,\Scal}}(P)\leq \Hcal+\eps$ and $\langle P,F_U\rangle\leq  \Fcal+\eps$, which will finish the proof. Let $M$ be so large that the constant $c_M$ introduced in Lemma~\ref{lem-LowerBound} satisfies $\log c_M\geq -\eps/8$. Let $R$ be so large that $C_M R^{-d/2}\leq \eps/8$, where $C_M$  is from \eqref{cRMvanishes}. Furthermore, assume that $R$ is also so large that 
\begin{equation}\label{RchoiceJF}
J_{W}(\xi^{\ssup {W}})\leq \Hcal+\eps/8\qquad\mbox{and}\qquad
\frac1{|W|}\langle \xi^{\ssup{W}},F_{W,W}\rangle\leq \Fcal+\eps/8.
\end{equation} 

In our first step, we are going to replace  $\xi^{\ssup {W}}$ by loop-shred distributions in larger and larger boxes in such a way that the entropy and energy do not really increase and the expected particle densities in loops and in shreds does not really change. (In a later step, we will use Proposition~\ref{Prop-construction_of_P} to further extend this measure to an element of $\Mcal_1(\Lcal\times\Scal)$.)  We want to use Proposition~\ref{lem-uppboundJ_W}(4) for that and Lemma~\ref{lem-lowerboundlogpint} for estimating the error term, but need first to replace $\xi^{\ssup {W}}$ by an approximation that guarantees that we are allowed to apply Proposition~\ref{lem-uppboundJ_W}(4) and that  the error term  is small. For this, we need to employ Lemma~\ref{lem-EntProjImag} with fixed parameters $\mathfrak r\in(1,\infty)$ and $S\in\N$ and large parameters $R,M,K,L, T,\vartheta$. Put $ \Theta=(\mathfrak r, M, L,K,T, \vartheta,S)$. Pick the parameters in $\Theta$ so large that the error term in Lemma~\ref{lem-EntProjImag} is $\leq \eps/4$ and pick  $\xi^{\ssup{W,\Theta,0}}\in\Mcal_1(\Lcal_{W}\times \Scal_{W})$ (namely, $\xi^{\ssup{W,\Theta,0}}=\xi^{\ssup {W}}\circ \Pi^{-1}$ with the projection $\Pi$ of Lemma~\ref{lem-EntProjImag}) that is concentrated on $A_{W;\Theta}=A_{W;\mathfrak r,M,K,L,T}\cap(\Lcal_{W}\times\Scal_{W}^{\ssup{M,\vartheta,S}})$ such that 
\begin{equation}\label{Thetagross}
J_{W;M}(\xi^{\ssup{W,\Theta,0}})\leq J_{W}(\xi^{\ssup{W}})+\frac \eps 4\qquad\mbox{and}\qquad \frac 1{ |W|}\langle F_{W,W},\xi^{\ssup{W,\Theta,0}}-\xi^{\ssup{W}}\rangle\leq \frac\eps 4 ,
\end{equation}
and $|\langle \mathfrak N_{W}^{\ssup{\Lcal,\ell}},\xi^{\ssup{W,\Theta,0}}-\xi^{\ssup{W}}\rangle|\leq \frac\eps 4 |W|$
and $|\langle \mathfrak N_{W}^{\ssup{\Scal,\ell}},\xi^{\ssup{W,\Theta,0}}-\xi^{\ssup{W}}\rangle|\leq \frac\eps 4 |W|$ and $\Pi_{W}^{\ssup{\Lcal}}(\xi^{\ssup{W,\Theta,0}})\in K_{3R}^{\ssup\Lcal}(L+\frac 14)$ and $\Pi_{W}^{\ssup{\Scal}}(\xi^{\ssup{W,\Theta,0}})\in K_{3R}^{\ssup\Scal}(L+\frac 14)$. (The fact that all these integrals w.r.t.~$\xi^{\ssup{W,\Theta,0}}$ can be made as close to the ones w.r.t.~$\xi^{\ssup{W}}$ follows from monotonicity of $\Pi$ in the cutting parameters $M, K,L, \vartheta$.)

Now we apply Proposition~\ref{lem-uppboundJ_W}(4) to $\xi^{\ssup{W,\Theta,0}}$ successively for $k=1$ (producing $\xi^{\ssup{W,\Theta,1}}$, then for $k=1 $ and $W$ replaced by $3 W$ (producing $\xi^{\ssup{W,\Theta,2}}$ and so on. We pick  the parameters $R, S$ so large that the error term in Proposition~\ref{lem-uppboundJ_W}(4) (using Lemma~\ref{lem-lowerboundlogpint} for the last term) is $\leq \frac \eps 8 3^{-k}$ for the $k$-th application of Proposition~\ref{lem-uppboundJ_W}(4). In this way, we obtain a sequence of $ \xi^{\ssup{W, \Theta,k}}\in \widetilde \Mcal_1^{\ssup k}\subset \Mcal_1(\Lcal_{3^kW}\times \Scal_{3^kW})$ such that the approximation error in the infimum is $\leq \eps/8$; more precisely,
\begin{equation}\label{xikdef}
J_{3^k W;M}(\xi^{\ssup{W,\Theta,k}})\leq \frac\eps 4+J_{W;M}(\xi^{\ssup{W,\Theta,0}})\qquad\mbox{and}\qquad \Pi_{3^k W \to 3^{k-1} W}(\xi^{\ssup{W,\Theta,k}})=\xi^{\ssup{W,\Theta,k-1}},
\end{equation}
and every $\xi^{\ssup{W,\Theta,k}}$ is $3^k W$-shift invariant. In particular, the family $(\xi^{\ssup{W,\Theta,k}})_{k\in \N}$ is consistent, and every projection of any $\xi^{\ssup{W,\Theta,k}}$ on a shift $z+W$ with $z\in 2R\Z^d$ is (after shifting back to the origin)  equal to $\xi^{\ssup{W,\Theta,0}}$.

Note that $\xi^{\ssup{W,\Theta,k}}$ is concentrated on $A_{3^k W;\Theta}$, since its projections on any discrete shift of $W$ in $3^k W$ is equal (after shifting to the origin) to $\xi^{\ssup{W,\Theta,0}}$. Therefore all shreds in it have all the properties of the shreds in the configurations in $A_{W;\Theta}$ (namely a bounded leg spread, a bound on the extension of all subparts, and an upper bound for the particle number in each of the boxes $W_{\mathfrak r}+z$). Therefore, the expected interaction $\frac 1{|3^k W|}\langle F_{3^k W,3^k W},\xi^{\ssup{W,\Theta,k}}\rangle$ is essentially equal to $\frac 1{|W|}\langle F_{W,W},\xi^{\ssup{W,\Theta,0}}\rangle$, up to the interaction between any two neighbouring boxes of radius $R$, which can be upper bounded by some constant (depending only on $M,L,K,T,\mathfrak r$ and $d, \beta, v$) times $1/R$, since it is a surface-order term. By making $R$ large enough, this error term  is below $\eps/8$. Therefore,
\begin{equation}\label{Frecursiveuppbound}
\frac 1{|3^k W|}\langle F_{3^kW,3^kW},\xi^{\ssup{W,\Theta,k}}\rangle\leq \frac \eps 4+
\frac 1{| W|}\langle F_{W,W},\xi^{\ssup{W,\Theta,0}}\rangle,\qquad k\in\N.
\end{equation}
By the same argument, but without error terms, the normalized expected particle numbers for loops and shreds, $\frac 1{|W|}\langle \mathfrak N_{W}^{\ssup{\ell,\Lcal}},\xi^{\ssup{W,\Theta,0}}\rangle$ and $\frac 1{|W|}\langle \mathfrak N_{W}^{\ssup{\ell,\Scal}},\xi^{\ssup{W,\Theta,0}}\rangle$, do not change when going from $W$ to $3^kW$ and from $\xi^{\ssup{W,\Theta,0}}$ to $\xi^{\ssup{W,\Theta,k}}$. In particular, $\frac 1{|3^k W|}\langle \mathfrak N_{3^k W}^{\ssup{\ell,\Lcal}},\xi^{\ssup{W,\Theta,k}}\rangle\in \overline B_{2\delta}(\rho_1)$ and $\frac 1{|3^k W|}\langle \mathfrak N_{3^k W}^{\ssup{\ell,\Scal}},\xi^{\ssup{W,\Theta,k}}\rangle\in \overline B_{2\delta}(\rho_2)$. 

Now we use Proposition~\ref{Prop-construction_of_P} to obtain from the consistent family $(\xi^{\ssup{W,\Theta,k}})_{k\in\N}$ some $\widetilde P\in\Mcal_1(\Lcal\times\Scal)$ that extends all the $\xi^{\ssup{W,\Theta,k}}$s. By the $W$-shift invariance of all the $\xi^{\ssup{W,\Theta,k}}$s, $P$ is invariant under any shift by a vector $\in 2R\Z^d$. Now we replace $\widetilde P$ by $P=\frac 1{|W|}\int_W\d x\, \widetilde P\circ \theta_x^{-1}$, which is shift-invariant, i.e., lies in $\Mcal_1^{\ssup{\rm s}}(\Lcal\times\Scal)$.  We are going to show that $\h^{\ssup{\Lcal,\Scal}}(P)\leq 2\eps+ J_{W}(\xi^{\ssup  {W}})$ and $\langle F_{U},P\rangle\leq 2\eps +\frac 1{|W|}\langle F_{W,W}, \xi^{\ssup  {W}}\rangle$ and that both the loop and the interlacement particle densities  of $P$ are close to $\rho_1$, respectively to $\rho_2$. This will imply \eqref{upperbound 5} and hence end the proof of the upper bound in Theorem~\ref{thm-freeenergy}.

It is relatively easy to see that the loop-particle density and the interlacement-particle densities lie in $\Bcal_{3\delta}(\rho_1)$, respectively in $\Bcal_{3\delta}(\rho_2)$. Indeed, use that $\frac 1{|3^k W|}\langle \mathfrak N_{3^k W}^{\ssup{\ell,\Lcal}},\xi^{\ssup{W,\Theta,k}}\rangle\sim \frac 1{|3^k W|}\langle \mathfrak N_{3^k W}^{\ssup{\ell,\Lcal}},P\rangle \to \langle \mathfrak N_U^{\ssup{\ell,\Lcal}},P\rangle$ as $k\to\infty$ since $\langle \mathfrak N_{3^k W}^{\ssup{\ell,\Lcal}},P\rangle$ lies between $\langle \mathfrak N_{3^k W -W}^{\ssup{\ell,\Lcal}},P\rangle$ and $\langle \mathfrak N_{3^k W+W}^{\ssup{\ell,\Lcal}},P\rangle$ and since, according to Lemma~\ref{lem-loopshredsnegl}, the expected number of particles in box-shreds of loops is negligible with respect to the box volume. A similar argument applies to the expected particle density in interlacements; we leave the details to the reader.

Now we show that $J_{3^k W;M}(\Pi_{3^k W}(P))\leq \Hcal+2\eps$ for any $k\in\N$, which implies, via Theorem~\ref{thm-specrelent}, that $\h^{\ssup{\Lcal,\Scal}}(P)\leq \Hcal+2\eps$. Indeed, using convexity of $J_{3^k W;M}$ and \eqref{JWupperboundx} in Proposition~\ref{lem-uppboundJ_W}(1) with $m=3$, we see that, for any $k\in\N$,
$$
\begin{aligned}
J_{3^k W;M}(\Pi_{3^k W}(P))
&\leq\frac 1{ |W|}\int_{W}\d x\, J_{3^k W;M}(\theta_x(\Pi_{x+3^k W}(\widetilde P))\\
&\leq \eps+J_{3^{k+1} W}(\Pi_{3^{k+1} W;M}(\widetilde P))=\eps+ J_{3^{k+1} W;M}(\xi^{\ssup {W,\Theta,k+1}})\\
&\leq  \eps+ \frac \eps 4+ J_{W;M}(\xi^{\ssup{W,\Theta,0}})
\leq \frac {3\eps}2 +J_W(\xi ^{\ssup W})\leq 2\eps +\Hcal,
\end{aligned}
$$
where we used also \eqref{xikdef} and \eqref{Thetagross} and one of the requirements for $R$.

It is only left to show that $\langle F_U,P\rangle\leq 2\eps +\Fcal$, where we recall the definition of $F_U$ from \eqref{FUdef}. We first point out that, according to the spatial ergodic theorem,
\begin{equation}\label{ergodictheorem}
\begin{aligned}
\langle F_U,P\rangle 
&= \lim_{k\to\infty}\frac 1{|3^k W|} \Big\langle \sum_{z\in\Z^d\colon z+U\subset 3^k W} F_{z+U},P\Big\rangle,
\end{aligned}
\end{equation}
since $F_U$ is bounded and has bounded reach almost surely on the support of $P$. This is seen as follows.
Recall that  every $\xi^{\ssup{W,\Theta,k}}$ is concentrated on the set $A_{3^{k} W;\Theta}$ of configurations; these have a globally bounded interaction between any box $z+U$ and its complement $(z+U)^{\rm c}$, where \lq bounded\rq\ is meant in terms of spatial distance, number of interacting loops/shreds  and in amount of interaction. More precisely, we may pick some $C\in(0,\infty)$ (only depending on the parameters in $\Theta$ and on the length of the support of the interaction potential $v$), such that $F_{z+U}(\omega,\varpi)$ is equal to the total interaction  between $z+U$ and $z+ W_C$ only, almost surely with respect to $\xi^{\ssup{W,\Theta,k}}$. And if $z+U\subset 3^k W$, then the integration of $F_{z+U}$ with respect to $P$ can be replaced by the one with respect to $\xi^{\ssup{W,\Theta,k+1}}$.

Recall from \eqref{FWWdef} that $F_{W,W}(\omega,\varpi)=\frac 12 F_{W,W}^{\ssup{\Lcal\Lcal}}(\omega)+F_{W,W}^{\ssup{\Lcal\Scal}}(\omega,\varpi)+\frac 12 F_{W,W}^{\ssup{\Scal\Scal}}(\varpi)$, whose definitions we recall from \eqref{FLL}--\eqref{FSS}. Therefore, almost surely with respect to $P$, for any $x\in W$ and any $k\in\N$, we get from splitting the interaction between $3^k W$ and $3^k W$ into the interactions between $z+U$ and $ z+ U^{\rm c}$ and sum on $z$,
$$
\begin{aligned}
F_{ 3^k W, 3^k W}
&\geq \sum_{z\in\Z^d\colon z+U\subset 3^k W}\Big[\frac 12 F_{z+U,z+U }+F_{z+U,W\cap(z+ U^{\rm c})} \Big]\\
&\geq \sum_{z\in\Z^d\colon z+U\subset W_{3^k-C}}\Big[\frac 12 F_{z+U,z+U }+F_{z+U,z+(W_C\setminus U)} \Big]\\
&=\sum_{z\in\Z^d\colon z+U\subset W_{3^k-C}}\Big( \smfrac 12 F_{ z+U, z+U}^{\ssup{\Lcal\Lcal}}+ F_{ z+U, z+U}^{\ssup{\Lcal\Scal}}+ \smfrac 12 F_{ z+U, z+U}^{\ssup{\Scal\Scal}} \\
& \qquad+F^{\ssup{\Lcal\Lcal}}_{ z+U, z+(W_C\setminus U)}+F^{\ssup{\Lcal\Scal}}_{ z+U, z+(W_C\setminus U)}+F^{\ssup{\Lcal\Scal}}_{ z+(W_C\setminus U), z+ U}+F^{\ssup{\Scal\Scal}}_{ z+U, z+(W_C\setminus U)}\Big)\\
&= \sum_{z\in\Z^d\colon z+U\subset W_{3^k R-C}}F_{ z+U},
\end{aligned}
$$
recalling the definition of $F_U$ from \eqref{FUdef}. Certainly, the sum of $F_{z+U}$ on $z\in\Z^d$ with $z+U\subset 3^kW$ but $z+U\not\subset W_{3^k R-C}$ is negligible in \eqref{ergodictheorem}, and 
$$
\frac1{|3^k W|}\Big|\Big\langle F_{3^k W,3^k W}-\frac 1{|W|}\int_W\d x \, F_{x+3^k W,x+3^k W},\xi^{\ssup{W,\Theta,k+1}}\Big\rangle\Big|\to0,\qquad k \to\infty.
$$
Therefore, 
$$
\begin{aligned}
\langle F_U,P\rangle 
&\leq \lim_{k\to\infty}\frac 1{|3^k W|}\big\langle F_{ 3^k W, 3^k W},P\big\rangle
 =\lim_{k\to\infty}\frac 1{|3^k W|}\Big\langle \int_W\d x\, F_{x+3^k W,x+3^k W},\xi^{\ssup{W,\Theta,k+1}}\Big\rangle\\
& = \lim_{k\to\infty}\frac 1{|3^k W|}\Big\langle  F_{3^k W,3^k W},\xi^{\ssup{W,\Theta,k+1}}\Big\rangle.
\end{aligned}
$$
Now using subsequently  \eqref{Frecursiveuppbound}, \eqref{Thetagross} and \eqref{RchoiceJF}, we see that $\langle F_U,P\rangle \leq \Fcal+2\eps$, which finishes the proof.

\subsection{Lower bound}\label{sec-finishlowerbound}

For finishing the proof of the lower bound in  Theorem~\ref{thm-freeenergy}, we go back to \eqref{hatZNlimitlowbound3}. Applying Lemma~\ref{lem-ergappr} and \eqref{hbound}, we see that, given any small $\eps\in(0,1)$, we can make the parameters $\vartheta, M,L,K,T$ and then $R$ (in this order) so large that the variational formula on the right-hand side of  \eqref{hatZNlimitlowbound3} is not larger than $\eps+\chi(\rho_1,\rho_2)$, also using the continuity of $\chi(\rho_1,\rho_2)$ in $(\rho_2,\rho_2)$; see Lemma~\ref{lem-propertieschi}. This gives, for any $\delta,\delta'\in(0,1)$,
\begin{equation}\label{hatZNlimitlowbound4}
\begin{aligned}
\liminf_{N\to\infty}\frac 1{|\L_N|}\log \widehat Z_{N,R,\delta}^{\ssup{\rm bc}}(\L_N,\rho_1,\rho_2)
&\geq-\frac {C_{M,L,K,T}}R-\frac {C_M}{R^{d/2}}-\Big[\frac {C_{M,\vartheta,S}}R+\frac{C\vartheta^2}S\Big](\rho_2+\delta')
\\
&\quad-\eps-\chi(\rho_1,\rho_2).
\end{aligned}
\end{equation}
We make now first $S\to\infty$ and then $R\to\infty$, which makes the first three terms disappear.  We have finished the proof of the lower bound in  Theorem~\ref{thm-freeenergy}.

\appendix

\section{The Brownian interlacement PPP with spacing $\beta$}\label{sec-InterlacePPP}

\noindent In this appendix we introduce, for $d\geq 3$, an interesting Poisson point process ${\tt R}= {\tt R}^{\ssup{u,\beta}}$ of random interlacements that has some connection to the specific relative entropy $\h^{\ssup{\Lcal,\Scal}}$ defined in \eqref{jointentropydef}, or, more directly, some connection to the kernel $ {\tt K}_W$ defined in \eqref{kernelK}. We survey its construction and some of its properties. Processes of this type were recently recognized as being instrumental in the description of the long loops in the Bose gas. In the present paper, such process does not appear in  our description of the free energy, since our entropy $\h^{\ssup {\Lcal,\Scal}}$ can be entirely characterised in terms of the kernel ${\tt K}_W$ defined in \eqref{kernelK}, but it admits enlightening interpretations. Because of Lemma~\ref{lem-KWident} below, $\h^{\ssup{\Lcal,\Scal}}$  can be seen as a kind of specific relative entropy with respect to the reference PPP ${\tt Q}\otimes {\tt R}$ of loops and interlacements, but note the crucial appearance of the boundary-particle distribution $\partial\Pi_W^{\ssup{\Scal}}(P)$ in \eqref{JWdeffirst}. Nevertheless, the process ${\tt R}$ is at the heart of the limiting description of the Bose gas and its condensate, so we decided to spend some words on it, even though it is not involved in our analysis of the interacting Bose gas.

Here is an ad-hoc definition of this process:
Let $d\geq 3$ and fix $\beta\in(0,\infty)$. The \emph{Brownian interlacement Poisson point process with particle spacing $\beta$}, or for short \emph{Brownian interlacement process}, ${\tt R}\in \Mcal_1^\ssup{{\rm s}}(\Shreds)$, is the distribution of the random interlacement point process $\varpi^{\ssup{\rm B}}=\sum_{g\in\Gamma_{\rm B}}\delta_g$, which  consists of independent two-sided Brownian motions $g\in\Ccal_\infty$ with parameter set $\R$, identified as equal modulo time shift by  $\beta$. The only free parameter of this process is the expected total amount of  particles in the unit box $U=[-\frac 12, \frac 12]^d$, which we call $u\in(0,\infty)$. We sometimes write ${\tt R}^{\ssup{u,\beta}}$ or ${\tt R}^{\ssup u}$; if we take $u=1$, then we drop this super-index.

The first construction of such a process was carried out in the seminal paper \cite{Sz10} for simple random walk on $\Z^d$; later the Brownian version was constructed in \cite{Sz13}. The version that we need in this paper is for Gaussian random walks instead, with interpolation by independent Brownian bridges. The Gaussian interlacement PPP has been already constructed in \cite{AFY21}, but without interpolation, and without stating the distributional properties that we will need; hence we are giving a short account on the construction and on the properties that we need. Alternatively, we could also refer to the Brownian version constructed in \cite{Sz13} and furnish each interlacement $g\colon\R\to\R^d$ independently with a determination of where $\beta\Z$ is in the domain $\R$, which leads presumably to the same process that we are going to construct  now, but we prefer to derive some distributional properties along our construction.

We are following the construction of the interlacement PPP for simple random walk on $\Z^d$  given in \cite[Section 1]{Sz10} and translate it to the Gaussian random walk in $\R^d$. Once we have constructed this discrete-time process, we scale it up in time by a factor of $\beta$ and  interpolate it between the times in $\beta \Z$ with independent Brownian bridges with time-interval $[0,\beta]$. See \cite{LP18} for general references to the PPP theory.

Here is a slightly more explicit definition of the Gaussian interlacement PPP: It is the PPP with a $\sigma$-finite intensity measure $\Upsilon$ on the set of trajectories $\Z\to\R^d$ that is uniquely determined by the following property: For $W\Subset\R^d$, the restriction of $\Upsilon$ to the set of trajectories that visit $W$ is the (not normalised) measure that starts with a site $\in W$ at time 0 that is \lq distributed\rq\ according to the equilibrium measure of the Gaussian random walk on $W$, and has two independent trajectories: one into the negative time axis, conditioned on never entering $W$ again, and one in the positive time axis, being just a Gaussian random walk in $\Z^d$. (By transience, it visits $W$ only finitely often.) This $W$-dependent measure turns out to be consistent in $W$ such that the described property fixes it uniquely.

Let us come to the details. By $\P_x$ we denote the probability measure under which a Gaussian random walk $(X_n)_{n\in\N_0}$ (i.e., with steps that are independent and Gaussian distributed with variance $2\beta$) starts from $x\in \R^d$. Since we are in $d\geq 3$, this is a  random variable with values in the set 
$$
\Tcal_+=  \Big\{{\mathbf x}=(x_n)_{n\in\N_0}\in (\R^d)^{\N_0}\colon \lim_{n\to \infty}x_n=\infty\Big\}.
$$
By $\Tcal=\{{\mathbf x}=(x_n)_{n\in \Z}\in(\R^d)^\Z\colon \lim_{n\to\pm\infty} x_n\}$ we denote the set of two-sided trajectories that visit any compact set only finitely often, and by $\Tcal^*$ its quotient space with respect to the identification of functions if they are time-shifts of each other, i.e., $\Tcal^*=\{\{\theta_n f\colon n \in \Z\}\colon f\in \Tcal\}$, where $\theta_n$ is the time-shift by $n$. By $\pi^*\colon \Tcal\to\Tcal^*$ we denote the canonical projection. By $\Tcal_W=\{(x_n)_{n\in\Z}\in \Tcal\colon x_n\in W\mbox{ for some }n\in\Z\}$ we denote the set of trajectories that visit $W$, and by $\Tcal_W^*$ its image under $\pi^*$.

As we now need some little theory of capacities and equilibrium measures for the Gaussian random walk, we restrict to special compact subsets, namely boxes. For any box $W\Subset\R^d$, we denote by $\e_W$  the density of the equilibrium measure of the Gaussian random walk, i.e.,
\begin{equation}
\e_W(x)=\1_W(x)\P_x(\widetilde H_W=\infty),\qquad x\in \R^d,
\end{equation}
where $\widetilde H_W=\inf\{n\in\N\colon X_n\in W\}$ is the entry time into $W$. It is clear that $\e_W(x)>0$ in the box $W$, but $=0$ outside the box. 
By $\P_x^{\ssup W}$ we denote the conditional measure $\P_x(\cdot\mid \widetilde H_W<\infty)$ given that the path visits $W$. Now we introduce the path measure $Q_W$ on $\Tcal$ as follows:
\begin{equation}\label{QWdef}
Q_W\big((X_{-n})_{n\in\N_0}\in A, X_0\in \d x, (X_{n})_{n\in\N_0}\in B\big)=\P_x^{\ssup W}(A)\e_W(x)\P_x(B)\,\d x ,\qquad x\in \R^d, A,B\subset \Tcal_+\mbox{ mb.}
\end{equation}
We also need the last time $L_W=\sup\{n\in \N_0\colon X_n\in W\}$ of a visit to $W$ and the set of finite-length trajectories
\begin{equation}
\Tcal_{W\to W}=\{{\mathbf x}=(x_n)_{n\in\{0,\dots,N\}}\in (\R^d)^{N+1}\colon N\in\N, x_0\in W, x_N\in W\}
\end{equation}
that start and terminate in $W$.

We can now formulate and prove the existence and uniqueness of the measure $\Upsilon$ that we will take as the intensity measure of our PPP.

\begin{lemma}[Existence, uniqueness and properties of $\Upsilon$]\label{lem-Upsilon}
\begin{enumerate}
\item There is a unique $\sigma$-finite measure $\Upsilon$ on $\Tcal^*$ such that
\begin{equation}
\1_{\Tcal_W^*}\,\Upsilon = \pi^*\circ Q_W,\qquad W\subset\R^d\mbox{ a box}.
\end{equation}

\item $\Upsilon$ is invariant under time-reversal and spatial shifts.

\item For any $W\Subset \R^d$ and any ${\mathbf x}=(x_n)_{n\in\{0,\dots,N\}}\in\Tcal_{W\to W}$,
\begin{equation}\label{QWdistprop}
\begin{aligned}
Q_W\big((X_{-n})_{n\in\N_0}\in A, &(X_n)_{n\in\{0,\dots,L_W\}}\in \d {\mathbf x},(X_{L_W+n})_{n\in\N_0}\in B\big)\\
&=\P_{x_0}^{\ssup W}(A)\e_W(x_0)\P_{x_0}((X_n)_{n\in\{0,\dots,N\}}\in \d{\mathbf x}) \P_{x_N}^{\ssup W}(B),\qquad  A,B\subset \Tcal_+\mbox{ mb.}
\end{aligned}
\end{equation}
\end{enumerate}
\end{lemma}

\begin{proof} We only give some few hints, as the proof  is entirely analogous to the one of \cite[Theorem 1.1]{Sz10} (which is lengthy and technical). We leave the details to the reader.

The uniqueness in (1) follows from the fact that $\Tcal^*$ is exhausted by the sets $\Tcal_{[-R,R]^d}^*$ with $R\in \N$. The proof of the existence consists of showing that the measure $\Upsilon_W=\pi^*\circ Q_W$ is consistent in $W$ in the sense that, for any $W\subset \widetilde W\Subset \R^d$, we have $\1_{\Tcal_W^*}\,\Upsilon_{\widetilde W}=\Upsilon_{W}$. Then Carath\'eodory's extension theorem yields (1). The proof of this is entirely analogous to the proof of the analogous assertion in \cite[Theorem 1.1]{Sz10}, with some obvious changes that come only from the continuous nature of $\R^d$: the fact that the equilibrium measure has a Lebesgue density and that trajectories run in $\R^d$ rather than in $\Z^d$. The proof goes via a decomposition into paths that start and end in $\widetilde W$ and run within $W$ in between, and it uses the Markov property of the random walk.
\end{proof}

Assertion (1) says that $\Upsilon$ can be described by the explicit formula \eqref{QWdef} on the set of trajectories that visit $W$, and Assertion (3) says that every path under $Q_W$ can be decomposed in time in three pieces: one piece of a Gaussian random until the first time of a visit to $W$ (having the equilibrium measure as the  \lq distribution\rq\ at this time), one piece until the last visit to $W$, and the last piece, which is a Gaussian random walk conditioned on never visiting $W$ anymore. These pieces are independent, given the two locations at the first and the last time of a visit to $W$. 

By $\Leb$ we denote the Lebesgue measure on $[0,\infty)$. Now we use the $\sigma$-finite measure $\Upsilon\otimes \Leb$ on $\Tcal^+\times [0,\infty)$ as the intensity measure of a PPP $\varpi=\sum_i\delta_{(X_i^*,v_i)}$, where we are going to use the $v$-component for controlling the spatial density of the paths. More precisely, for any $v\in(0,\infty)$, we can call the point process $\sum_{i\colon v_i\leq v}\delta_{X_i^*}$ the {\em Gaussian interlacement process with density parameter} $v$.

\begin{remark}[Construction of the Gaussian interlacement PPP.]\label{rem-PPPconstruction} 
From Assertion (3) in Lemma~\ref{lem-Upsilon}, using the standard construction of PPPs from the intensity measure, the following construction of the PPP  $\sum_{i\colon v_i\leq v}\delta_{X_i^*}$ emerges. In simple words, one can construct the PPP on the set $\Tcal_W^*$ of paths that visit the box $W$ by first sampling a Poisson-distributed variable $N$ with parameter $v {\rm cap}(W)= v\int_W \e_W(x)\,\d x$ and then $N$ random sites $X_0^{\ssup 1},\dots,X_0^{\ssup N}\in W$ with distribution $\e_W(x)\,\d x/{\rm cap}(W)$. Then one samples independently over $i=1,\dots,N$ a Gaussian random walk $(X_{-n}^{\ssup i})_{n\in\N_0}$ conditioned on never visiting $W$, and another, independent, Gaussian random walk $(X_{n}^{\ssup i})_{n\in\N_0}$ (which may return to $W$ never or several times, but eventually vanishes at $\infty$). Finally, we \lq forget\rq\ the precise time at which $X^{\ssup i}$ first enters $W$, i.e., we put $X_i^*=\pi^*(X^{\ssup i})$.
\hfill$\Diamond$
\end{remark}

Now we interpolate any Gaussian random walk to a Brownian motion by inserting independent Brownian bridges and scaling them up in time with a factor of $\beta$. For this, we will be using the concept of $K$-marking with a kernel $K$, which is well-known in PPP-theory. Introduce the map
\begin{equation}\label{randomwalkinterpol}
\Tcal_+ \ni{\mathbf x}=(x_n)_{n\in\N_0}\mapsto B^{\ssup {\mathbf x}}=(B^{\ssup {\mathbf x}}_t)_{t\in[0,\infty)}\in \Ccal,
\qquad B^{\ssup {\mathbf x}}_t =
B^{\ssup{x_n,x_{n+1}}}_{t-\beta n}\quad\mbox{if }t\in [\beta n,\beta (n+1)], n\in\N_0,
\end{equation}
and $(B^{\ssup{x_n,x_{n+1}}})_{n\in\N_0} $ is a family of independent Brownian bridges with time-interval $[0,\beta]$ from $x_n$ to $x_{n+1}$. Then, if $X$ is a Gaussian random walk with variance $2\beta$ of the steps, $B^{\ssup X}$ is a Brownian motion with generator $\Delta$. We can describe this interpolation operation in terms of distributions via the kernel $K$ from $\Tcal_+$ to $ \Ccal$ given by
\begin{equation}
K({\mathbf x},\d f)=\bigotimes_{n\in\N_0} \mu_{x_n,x_{n+1}}^{\ssup \beta}\big(\d (\theta_{\beta n} (f|_{[\beta n, \beta (n+1)]}))\big),\qquad {\mathbf x}=(x_n)_{n\in\N_0} \in\Tcal_+, f\in\Ccal,
\end{equation}
where we recall from \eqref{nnBBM} the canonical Brownian bridge measures. We use the same symbol $K$ for the variants of this kernel on the sets $\Tcal$ and $\Tcal^*$. 

From $\Upsilon$, we now proceed with the $K$-marking $\varpi^{\ssup K }=\sum_i\delta_{(X_i^*,B_i^*,v_i)}$, where the Brownian motions $B_i^*$ derive from the random walks $X_i^*$ via  \eqref{randomwalkinterpol}, independently over $i$. Its intensity measure is $\Upsilon\otimes K\otimes \Leb$. Observe that $X_i^*$ is a deterministic function of $B_i^*$ and that the Brownian motions are uniquely determined only up to time-shifts in $\beta\Z$, in contrast to the Brownian interlacement process constructed in \cite{Sz13}. 

Finally, the measure ${\tt R}^{\ssup {u,\beta}}$  that we announced is now defined as the distribution of the process $\varpi^{\ssup{K,u}}=\sum_ {i\colon u_i \leq v_u}\delta_{B_i^*}$, where $v_u\in [0,\infty)$ is determined by the requirement that $u=\E(\sum_{i\colon u_i \leq v_u}\sum_{n\in\Z}\1_{B_i^*(n\beta)\in U})=1$, where $U=[-\frac 12,\frac 12]^d$ is the unit box. It is only technical to prove that the distribution of the process $\varpi^{\ssup {K,u}}$ is invariant under time-reversal and spatial shifts, since the Brownian bridge measures have the relevant properties for that. That is, the distribution ${\tt R}^{\ssup {u,\beta}}$ of $\varpi^{\ssup {K,u}}$ lies in $\Mcal_1^{\ssup{\rm s}}( \Scal)$. This finishes the introduction of the Brownian $\beta$-spaced interlacement  process.

\begin{remark}[Construction of the Brownian $\beta$-spaced interlacement PPP.]\label{rem-PPPbetaconstruction}
Referring to the construction of Remark~\ref{rem-PPPconstruction}, the $B_i^*$ may be constructed from the $X^{\ssup i}$ by interpolating them by Brownian bridges as in \eqref{randomwalkinterpol}, and $N$ independent standard Brownian motions $B^{\ssup i}=(B^{\ssup i}_t)_{t\in[0,\infty)}$ with generator $\Delta$ arise, starting from $B^{\ssup i}_0=X_0^{\ssup i}$. Analogously, and using time-reversal properties of the Brownian bridge measures, we interpolate also the negative-time parts and obtain Brownian motions conditioned on never hitting $W$ at any time $ \in \beta\N$, independent of the positive-time parts, given the starting sites $B^{\ssup i}_0$. Finally, we \lq forget\rq\ the integer-part of $\frac 1\beta $ times the time, i.e., we put $B_i^*=\pi^*(B^{\ssup i})=\{\theta_{\beta n}(B^{\ssup i})\colon n\in\Z\}$.
\hfill$\Diamond$
\end{remark}

Now we finally give the connection between the interlacement process ${\tt R}$ and our analysis of the interacting Bose gas. Recall the kernel ${\tt K}_W$ from \eqref{kernelK} and recall from \eqref{JWdeffirst} how crucial it is in this connection.

\begin{lemma}[Identification of ${\tt K}_W$ in terms of ${\tt R}$]\label{lem-KWident} The kernel ${\tt K }_W(\mu,\cdot)$ defined in \eqref{kernelK} is equal to the kernel $\Pi_W(\SPP)_{\Tcal_W\to  \Scal_W}(\mu,\cdot)$. 
\end{lemma}

\begin{proof}
 We refer to the construction of the Brownian $\beta$-spaced interlacement process $\varpi$ from Remarks~\ref{rem-PPPconstruction} and \ref{rem-PPPbetaconstruction}. The event $\{\partial\Pi_W^{\ssup\Scal}(\varpi)=\mu\}$ on that we condition depends only on the processes $(B^{\ssup i}_t)_{t\in[0,\infty)}$, i.e., the parts from the time of the first visit to $W$ at a time $\in \beta\Z$ on. Introduce the subsequent entry and exit times (up to the factor $\beta$) $0=\tau^{\ssup i}_1<\sigma_1^{\ssup i}<\tau_2^{\ssup i}<\sigma_2^{\ssup i}<\dots< \tau_{m_i} ^{\ssup i}<\sigma_{m_i}^{ \ssup i}<\infty$ of $B^{\ssup i}$, and put $\mu=\sum_{j\in I}\delta_{(x_j,l_j,y_j)}$, then the event $\{\partial\Pi_W^{\ssup\Scal}(\varpi)=\mu\}$ is equal to the disjoint union over all bijections $h\colon J=\{(i,k)\in [N]\times \N\colon k\leq m_i\}\to I$ of the event
$$
\bigcap_{(i,k)\in J}\Big\{B^{\ssup i}_{\tau^{\ssup i}_k \beta}=x_{h(i,k)}, \sigma_k^{\ssup i}-\tau_k^{\ssup i}=l_{h(i,k)}, B^{\ssup i}_{\sigma^{\ssup i}_k \beta}=y_{h(i,k)}\Big\}.
$$
On this event, $\Pi_W^{\ssup\Scal}(\varpi)$ is the empirical measure of the paths $(B^ {\ssup i}_t)_{t\in [\beta \tau^{\ssup i}_k, \beta \sigma^{\ssup i}_k]}$, properly time-shifted, with $(i,k)\in J$. According to the strong Markov property at all these stopping times, these path pieces are conditionally independent, and their distribution is given by $q_{x_{h(i,k)},y_{h(i,k)}}^{\ssup {l_{h(i,k)},W}}$. Hence, their joint distribution is the product measure of these, which is ${\tt K}_W(\mu,\cdot)$. This finishes the proof.
\end{proof}

\bigskip

\noindent{\bf Acknowledgement.} The research was supported by the DFG-SPP 2265 {\em Random Geometric Systems}, Project P04 {\em The Statistical Mechanics of the Interlacement Poisson Point Process}. Many inspiring discussions with Stefan Adams (University of Warwick), Alexander Drewitz (University of Cologne) and Alexander Zass (Weierstrass Institute Berlin) are gratefully acknowledged, as well as some fruitful exchange with Guillaume Bellot (University of Lille) and Quirin Vogel (Universität Klagenfurt).


\begin{thebibliography}{WWW98}

\bibitem[A08]{A07}
{\sc S.~Adams},
\textit{Large deviations for empirical measures in cycles of integer partitions and their relation to systems of Bosons},  Analysis and Stochastics of Growth Processes, LMS, Oxford University Press, 148-172, (2008).

\medskip

\bibitem[ACK11]{ACK10}
{\sc S.~Adams, A.~Collevecchio} and  {\sc W.~K\"onig},
\newblock A variational formula for the free energy of an interacting many-particle system,
\newblock {\it Ann.~Probab.} {\bf 39:2}, 683--728 (2011).

\medskip



\bibitem[ADPZ11]{ADPZ11}
\newblock{ \sc S.~Adams, N.~Dirr, M.~Peletier} and {\sc J.~Zimmer},
\newblock  From a Large-Deviations Principle to the Wasserstein Gradient Flow: A New Micro-Macro Passage,
\newblock {\it Commun.~Math.~Phys.} {\bf 307:3}, 791-815 (2011), DOI 10.1007/s00220-011-1328-4.

\medskip







\bibitem[AFY21]{AFY21}
\newblock {\sc I.~Armendáriz, P.A.~Ferrari}, and  {\sc S.~ Yuhjtman},
\newblock Gaussian random permutation and the boson point process,
\newblock {\it Commun. Math. Phys.} {\bf 387:3}, 1515-1547 (2021).

\medskip


\bibitem[AKP21]{AKP19}
{\sc L.~Andreis, W.~König}, and {\sc R.I.A.~Patterson},
\newblock  A large-deviations principle for all the cluster sizes of a sparse Erd\H{o}s-R\'enyi graph,
\newblock {\em Random Structures Algorithms} {\bf 59:4}, 522-553 (2021). 

\medskip

\bibitem[AKLP23]{AKLP24}
{\sc L.~Andreis, W.~König, H.~Langhammer}, and {\sc R.I.A.~Patterson},
\newblock  A large-deviations principle for all the components in a sparse inhomogeneous random graph,
\newblock {\em Probab.~Theory Relat.~Fields} {\bf 186}, 521-620 (2023). 

\medskip



\bibitem[B07]{Bogachev2}
\newblock {\sc V.I.~Bogachev},
\newblock {\em Measure Theory}, Vol.~II,
\newblock Springer-Verlag, Berlin, Heidelberg (2007).

\medskip

\bibitem[BDM24]{BDM24}
\newblock {\sc G.~Bellot, D.~Dereudre} and {\sc M.~Ma\"ida},
\newblock DLR equations for the superstable Bose gas at any temperature and activity,
\newblock preprint arXiv:2410.10225,  {\tt https://arxiv.org/abs/2410.10225} (2024).

\medskip


\bibitem[BKM24]{BKM24}
\newblock {\sc E.~Bolthausen, W. König}, and {\sc C.~Mukherjee},
\newblock Self-repellent Brownian bridges in an interacting Bose gas,
\newblock preprint, ArXiv:2405.08753, {\tt https://arxiv.org/abs/2405.08753} (2024).

\medskip

\bibitem[BKV24]{BKV24}
\newblock {\sc T.~Bai, W.~König,} and {\sc Q.~Vogel},
\newblock Proof of off-diagonal long-range order in a mean-field trapped Bose gas via the Feynman--Kac formula
\newblock {\it Elec.~J.~Prob.} {\bf 30}, 1-26, (2025) {\tt DOI: 10.1214/25-EJP1370}.

\medskip



\bibitem[BR97]{BR97} 
\newblock {\sc O.~Bratteli} and {\sc D.~W.~Robinson},
\newblock {\it Operator Algebras and Quantum Statistical Mechanics II},
\newblock 2nd ed., Springer-Verlag (1997).

\medskip


\medskip



 
 
 
 \bibitem[CJK23]{CJK23}
\newblock {\sc O.~Collin, B.~Jahnel} and {\sc W.~König},
\newblock A micro-macro variational formula for the free energy of a many-body system with unbounded marks,
\newblock {\it Elec. J. Prob.} {\bf  28},  1-58 (2023).

\medskip

\bibitem[CN23]{CN23}
\newblock {\sc A.~Chiarini} and {\sc M.~Nitzschner},
\newblock  Lower bounds for bulk deviations for the simple random walk on $\Z^d$, $d\geq 3$,
\newblock {\it Ann.~Inst.~Henri Poincaré, Probab.~Stat.}, { to appear}, {\tt file:///C:/Users/koeni/Downloads/AIHP2408-005R1A0.pdf} (2025).

\medskip

\bibitem[DV24]{DV24}
\newblock {\sc M.~Dickson} and {\sc Q.~Vogel},
\newblock Formation of infinite loops for an interacting bosonic
loop soup,
\newblock  {\it Electron. J. Probab.} {\bf 29}, article no. 24, 1–39.
ISSN: 1083-6489 {\tt https://doi.org/10.1214/24-EJP1085} (2024).


\medskip

\bibitem[DVJ03]{DVJ03}
\newblock {\sc D.J.~Daley} and {\sc D.~Vere-Jones},
\newblock {\it  An Introduction to the Theory of Point Processes: Volume I: Elementary Theory and Methods,} Second Edition,
\newblock Springer (2003).

\medskip

\bibitem[DVJ08]{DVJ08}
\newblock {\sc D.J.~Daley} and {\sc D.~Vere-Jones},
\newblock {\it  An Introduction to the Theory of Point Processes: Volume II: General Theory and Structure,} Second Edition,
\newblock Springer (2008).

\medskip
 
\bibitem[DZ10]{DZ98}
{\sc A.~Dembo} and {\sc O.~Zeitouni},
{\it Large Deviations Techniques and Applications\/}, 
corrected printing of the 2nd edition, Springer,  Berlin Heidelberg (2010).

\medskip



\bibitem[Fe53]{F53} 
\newblock {\sc R.P.~Feynman},
\newblock Atomic theory of the $ \lambda $ transition in Helium,
\newblock {\it Phys. Rev.} {\bf 91}, 1291--1301 (1953).

\medskip





\bibitem[FKSS20a]{FKSS20}
\newblock {\sc J.~Fröhlich, A.~Knowles, B.~Schlein,} and {\sc V.~Sohinger,}
\newblock  A path-integral analysis of interacting Bose gases and loop gases. 
\newblock {\em J. Stat. Phys.} {\bf 180:1-6}, 810–831 (2020).
 
\medskip

\bibitem[G11]{G88}
\newblock {\sc H.-O.~Georgii},
\newblock {\it Gibbs Measures and Phase Transitions}, Second Edition,
\newblock De Gruyter Studies in Mathematics 9 (2011).

\medskip

\bibitem[GZ93]{GZ93}
\newblock {\sc H.-O.~Georgii} and {\sc H.~Zessin},
\newblock Large deviations and the maximum entropy principle for marked point random fields,
\newblock {\it Prob.~Theory Relat. Fields} {\bf 96}, 177--204 (1993).

\medskip

\bibitem[G94]{G94}
\newblock {\sc H.-O.~Georgii,}
\newblock Large deviations and the equivalence of ensembles for Gibbsian particle systems with superstable interaction,
\newblock {\it Prob.~Theory Relat. Fields} {\bf 99}, 171--195 (1994).

\medskip

\bibitem[G70]{G70}
\newblock {\sc J.~Ginibre,}
\newblock {\it Some Applications of Functional Integration in Statistical Mechanics, and Field Theory\/}, 
\newblock C. de Witt and R. Storaeds, Gordon and Breach, New York   (1970).

\medskip



\bibitem[KVZ25]{KVZ23}
\newblock {\sc W.\ König, Q.\ Vogel} and {\sc A. Zass},
\newblock Off-diagonal long-range order for the free Bose gas via the Feynman--Kac formula,
\newblock {\it Ann. Appl. Probab.}, to appear,  preprint arXiv:2312.07481,  {\tt https://arxiv.org/abs/2312.07481} (2023).

\medskip

\bibitem[LP18]{LP18}
\newblock {\sc G.~Last} and {\sc M.~Penrose,}
\newblock {\em Lectures on the Poisson Process.}
\newblock Cambridge University Press, Cambridge (2018).

\medskip 



\bibitem[LW04]{LW04}
\newblock {\sc G.F. Lawler} and {\sc W. Werner,}
\newblock  The Brownian loop soup, {\it Probab. Theory Related
Fields} {\bf 128:4}, 565-588 (2004).

\medskip

\bibitem[NPZ13]{NPZ13}
\newblock {\sc B.~Nehring, S.~Poghosyan} and {\sc H.~Zessin},
\newblock On the construction of point processes in statistical mechanics,
\newblock {\it J. Math. Phys.} {\bf 54} (6), 063302 (2013).

\medskip

\bibitem[PRV13]{PRV13}
\newblock {\sc M.A. Peletier, D.R.M. Renger,} and {\sc M. Veneroni}, \newblock Variational formulation of the Fokker--Planck equation with decay: A particle approach.
\newblock {\em Communications in Contemporary Mathematics} {\bf 15(05)}, 1350017 (2013).

\medskip

\bibitem[PS01]{PS01}
\newblock {\sc C.~Pethick} and {\sc H.~Smith},
\newblock {\em  Bose–Einstein Condensation of Dilute Gases}. Cambridge University Press, Cambridge (2001).

\medskip

\bibitem[PS03]{PS03}
\newblock {\sc L.~Pitaevskii} and {\sc S.~Stringari},
\newblock {\em Bose–Einstein Condensation}. Oxford Science Publications, Oxford (2003).

\medskip

\bibitem[QT23]{QT23}
\newblock {\sc A.~Quitmann} and {\sc L.~Taggi},
\newblock Macroscopic loops in the Bose gas, Spin $O(N)$ and related models,
\newblock {\it Commun.~Math.~Phys.} { \bf 400}, 2081–2136 (2023).

\medskip

\bibitem[Raf09]{R09}
\newblock {\sc M.~Rafler},
\newblock {\it Gaussian Loop- and Polya Processes: A Point Process Approach},
\newblock PhD Thesis, University of Potsdam (2009).

\medskip

\bibitem[Rue69]{Rue69} 
\newblock {\sc D.~Ruelle},
\newblock {\it Statistical Mechanics: Rigorous Results},
\newblock W.A.~Benjamin, Inc., (1969).



\medskip

\bibitem[S\"u93]{S93} 
\newblock {\sc A.~S\"ut\H o},
\newblock Percolation transition in the Bose gas,
\newblock {\it J. Phys. A: Math. Gen.} {\bf 26}, 4689--4710 (1993).


\medskip

\bibitem[S\"u02]{S02} {\sc A.~S\"ut\H o},
\newblock Percolation transition in the Bose gas: II,
\newblock {\it J. Phys. A: Math. Gen.} {\bf 35}, 6995--7002 (2002).

\medskip





\bibitem[Sz10]{Sz10}
\newblock {\sc A.-S.~Sznitman},
\newblock Vacant set of random interlacements and percolation,
\newblock  {\it Ann. Math.} {\bf 171}, 2039--2087 (2010).

\medskip

\bibitem[Sz13]{Sz13}
\newblock {\sc A.-S.~Sznitman},
\newblock On scaling limits and Brownian interlacements,
\newblock {\it Bull. Braz. Math. Soc., New Series} {\bf 44(4)}, 555--592, (2013); special issue of the Bulletin of the Brazilian Mathematical Society--IMPA 60 years.

\medskip

\bibitem[Sz23]{Sz23}
\newblock {\sc A.-S.~Sznitman},
\newblock  On bulk deviations for the local behavior of random interlacements,
\newblock {\em Ann. Scient. Éc. Norm. Sup.}, série 4, t. 56, 801-858 (2023).

\medskip





\bibitem[U06]{U06}
\newblock {\sc D. Ueltschi}, 
\newblock Relation between Feynman cycles and off-diagonal long-range order, 
\newblock {\em Phys. Rev. Lett.} {\bf 97}, 170601 (2006).


\medskip

\bibitem[V23]{V23}
\newblock {\sc Q.~Vogel},
\newblock Emergence of interlacements from the finite volume Bose soup,
\newblock {\em Stoch. Proc. Appl.} {\bf 166} Dezember 2023, 104227 {\tt https://doi.org/10.1016/j.spa.2023.104227} (2023).

\medskip




\bibitem[Z22]{Z22}
\newblock {\sc A.~Zass},
\newblock Gibbs Point Processes on Path Space: Existence, Cluster Expansion and Uniqueness,
\newblock {\it Markov Process. Rel. Fields} {\bf 28} (3), 329--364 (2022).


\end{thebibliography}
\end{document}